\newif\ifLTEX
\LTEXtrue

\documentclass[a4paper]{amsart}

\usepackage{mathtools,amssymb}
\usepackage[abbrev]{amsrefs}
\usepackage{xparse}

\ifLTEX
\usepackage{graphicx}
\else
\usepackage[dvipdfmx]{graphicx}
\fi
 
\usepackage{hyperref}
\usepackage{subcaption}
\usepackage{enumitem}
\usepackage{amsthm}
\usepackage{pdfpages}
\usepackage[all]{xy} 
\usepackage{float}
\usepackage{marginfix}

\newtheorem{thm}{Theorem}[section]
\newtheorem{lem}[thm]{Lemma}

\newtheorem{cor}[thm]{Corollary}
\newtheorem{prop}[thm]{Proposition}

\theoremstyle{definition}

\newcommand{\R}{\mathbb{R}}

\newcommand{\Z}{\mathbb{Z}}

\newcommand{\M}{\mathrm{Mod}_{0,2n+2}}
\newcommand{\Mb}{\mathrm{Mod}_{0,2n+1}^1}
\newcommand{\Mp}{\mathrm{Mod}_{0,2n+2;\ast }}
\newcommand{\LM}{\mathrm{LMod}_{2n+2}}
\newcommand{\LMb}{\mathrm{LMod}_{2n+1}^1}
\newcommand{\LMp}{\mathrm{LMod}_{2n+2,\ast }}
\newcommand{\PM}{\mathrm{PMod}_{0,2n+2}}
\newcommand{\PMn}{\mathrm{PMod}_{0,n}}
\newcommand{\Mg}{\mathrm{Mod}_{g}}
\newcommand{\Mgb}{\mathrm{Mod}_{g}^1}
\newcommand{\Mgp}{\mathrm{Mod}_{g,\ast }}
\newcommand{\SM}{\mathrm{SMod}_{g;k}}
\newcommand{\SMb}{\mathrm{SMod}_{g;k}^1}
\newcommand{\SMp}{\mathrm{SMod}_{g,\ast ;k}}
\newcommand{\B}{\mathcal{B}}

\numberwithin{equation}{section}

\allowdisplaybreaks
\title[Finite presentation for balanced superelliptic mapping class group]{Finite presentations for the balanced superelliptic mapping class groups}

\author[S. Hirose]{Susumu Hirose}
\address{
(Susumu Hirose)
Department of Mathematics, Faculty of Science and Technology, Tokyo University of Science, 2641 Yamazaki, Noda-shi, Chiba, 278-8510 Japan
}
\email{hirose\_susumu@ma.noda.tus.ac.jp}

\author[G.~Omori]{Genki Omori}
\address{
(Genki Omori)
Department of Mathematics, Faculty of Science and Technology, Tokyo University of Science, 2641 Yamazaki, Noda-shi, Chiba, 278-8510 Japan
}
\email{omori\_genki@ma.noda.tus.ac.jp}

\subjclass[2010]{57S05, 57M07, 57M05, 20F05}

\date{\today}

\begin{document}
\maketitle
\begin{abstract}
The balanced superelliptic mapping class group is the normalizer of the transformation group of the balanced superelliptic covering space in the mapping class group of the total surface.  
We give finite presentations for the balanced superelliptic mapping class groups of closed surfaces, surfaces with one marked point, and surfaces with one boundary component. 
To give these presentations, we construct finite presentations for corresponding liftable mapping class groups in a different generating set from Ghaswala-Winarski's presentation in~\cite{Ghaswala-Winarski1}. 
\end{abstract}

\section{Introduction}


Let $\Sigma _{g}$ be a connected closed oriented surface of genus $g\geq 0$.
For integers $n\geq 1$ and $k\geq 2$ with $g=n(k-1)$, the \textit{balanced superelliptic covering map} $p=p_{g,k}\colon \Sigma _g\to \Sigma _0$ is a branched covering map with $2n+2$ branch points $p_1,\ p_2,\ \dots ,\ p_{2n+2}\in \Sigma _{0}$, their unique preimages $\widetilde{p}_1,\ \widetilde{p}_2,\ \dots ,\ \widetilde{p}_{2n+2}\in \Sigma _{g}$, and the covering transformation group generated by the \textit{balanced superelliptic rotation} $\zeta =\zeta _{g,k}$ of order $k$ (precisely defined in Section~\ref{section_bscov} and see Figure~\ref{fig_bs_periodic_map}). 
When $k=2$, $\zeta =\zeta _{g,2}$ coincides with the hyperelliptic involution, and for $k\geq 3$, the balanced superelliptic covering space was introduced by Ghaswala and Winarski~\cite{Ghaswala-Winarski2}.

For subsets $A$ and $B$ of $\Sigma _g$ $(g\geq 0)$, the {\it mapping class group} $\mathrm{Mod}(\Sigma _{g}; A, B)$ of the tuple $(\Sigma _{g}, A, B)$ is the group of isotopy classes of orientation-preserving self-homeomorphisms on $\Sigma _{g}$ which preserve $A$ setwise and $B$ pointwise. 
Put $\B =\{ p_1,\ p_2,\ \dots ,\ p_{2n+2}\} \subset \Sigma _{0}$ and we denote by $D$ a disk neighborhood of $p_{2n+2}$ in $\Sigma _0-\B $, by $\widetilde{D}$ the preimage of $D$ with respect to $p$, and by $\Sigma _g^1$ (resp. $\Sigma _0^1$) the complement of the interior of $\widetilde{D}$ (resp. $D$) in $\Sigma _g$ (resp. $\Sigma _0$). 
Under the situation above, we define 
\begin{align*}
\M &=\mathrm{Mod}(\Sigma _{0}; \B , \emptyset ),\quad \Mp =\mathrm{Mod}(\Sigma _{0}; \B -\{ p_{2n+2}\} , \{ p_{2n+2}\} ),\\
\Mb &=\mathrm{Mod}(\Sigma _{0}^1; \B -\{ p_{2n+2}\}, \partial \Sigma _{0}^1),
\end{align*}
and for $g\geq 1$, 
\begin{align*}
\Mg &=\mathrm{Mod}(\Sigma _g; \emptyset , \emptyset ),\quad \Mgp =\mathrm{Mod}(\Sigma _g; \emptyset , \{ \widetilde{p}_{2n+2}\}),\text{ and}\\
\Mgb &=\mathrm{Mod}(\Sigma _g^1; \emptyset , \partial \Sigma _{g}^1).
\end{align*}
We regard $\Mp $ as the subgroup of $\M $ which consists of elements preserving the point $p_{2n+2}$. 

For $g=n(k-1)\geq 1$, an orientation-preserving self-homeomorphism $\varphi $ on $\Sigma _{g}$ (resp. on $\Sigma _{g}^1$ whose restriction to $\partial \Sigma _{g}^1$ is the identity map) is \textit{symmetric} for $\zeta =\zeta _{g,k}$ if $\varphi \left< \zeta \right> \varphi ^{-1}=\left< \zeta \right> $ (resp. $\varphi \left< \zeta |_{\Sigma _g^1} \right> \varphi ^{-1}=\left< \zeta |_{\Sigma _g^1} \right> $). 
We abuse notation and denote simply $\zeta |_{\Sigma _g^1}=\zeta $. 
The \textit{balanced superelliptic mapping class group} (or the \textit{symmetric mapping class group}) $\SM $ (resp. $\SMp$ and $\SMb$) is the subgroup of $\Mg $ (resp. $\Mgp$ and $\Mgb$) which consists of elements represented by symmetric homeomorphisms. 
Ghaswala and McLeay~\cite{Ghaswala-McLeay} showed that any element in $\SMb $ are represented by a homeomorphism which commutes with $\zeta $. 
Birman and Hilden~\cite{Birman-Hilden3} showed that $\SM $ (resp. $\SMp$ and $\SMb$) coincides with the group of symmetric isotopy classes of symmetric homeomorphisms on $\Sigma _g$ (resp. $(\Sigma _g, \widetilde{p}_{2n+2})$ and $\Sigma _g^1$). 

An orientation-preserving self-homeomorphism $\varphi $ on $\Sigma _{0}$ (resp. on $\Sigma _{0}^1$ whose restriction to $\partial \Sigma _{0}^1$ is the identity map) is \textit{liftable} with respect to 
$p=p_{g,k}$ if there exists an orientation-preserving self-homeomorphism $\widetilde{\varphi }$ on $\Sigma _{g}$ (resp. on $\Sigma _{g}^1$ whose restriction to $\partial \Sigma _{g}^1$ is the identity map) such that $p\circ \widetilde{\varphi }=\varphi \circ p$ (resp. $p|_{\Sigma _g^1}\circ \widetilde{\varphi }=\varphi \circ p|_{\Sigma _g^1}$), namely, the following diagrams commute: 
\[
\xymatrix{
\Sigma _g \ar[r]^{\widetilde{\varphi }} \ar[d]_p &  \Sigma _{g} \ar[d]^p & \Sigma _g^1 \ar[r]^{\widetilde{\varphi }}\ar[d]_{p|_{\Sigma _g^1}}  &  \Sigma _{g}^1\ar[d]^{p|_{\Sigma _g^1}} \\
\Sigma _{0}  \ar[r]_{\varphi } &\Sigma _{0}, \ar@{}[lu]|{\circlearrowright} & \Sigma _{0}^1  \ar[r]_{\varphi } &\Sigma _{0}^1. \ar@{}[lu]|{\circlearrowright}
}
\] 
We also abuse notation and denote simply $p|_{\Sigma _g^1}=p$. 
The \textit{liftable mapping class group} $\mathrm{LMod}_{2n+2;k}$ (resp. $\mathrm{LMod}_{2n+2,\ast ;k}$ and $\mathrm{LMod}_{2n+1;k}^1$) is the subgroup of $\M $ (resp. $\Mp $ and $\Mb $) which consists of elements represented by liftable homeomorphisms for $p_{g,k}$. 
Birman and Hilden~\cite{Birman-Hilden2} proved that $\mathrm{LMod}_{2n+1;k}^1$ is isomorphic to $\SMb $ and $\mathrm{LMod}_{2n+2;k}$ is isomorphic to the quotient of $\SM $ by $\left< \zeta _{g,k}\right> $.  
By Proposition~\ref{prop_smod_char}, we also show that $\mathrm{LMod}_{2n+2,\ast ;k}$ is isomorphic to the quotient of $\SMp$ by $\left< \zeta _{g,k}\right> $. 

When $k=2$, $\zeta =\zeta _{g,2}$ is the hyperelliptic involution and corresponding symmetric mapping class groups are called the hyperelliptic mapping class groups.  
In this case, $\mathrm{LMod}_{2n+2;2}$ (resp. $\mathrm{LMod}_{2n+2,\ast ;2}$ and $\mathrm{LMod}_{2n+1;2}^1$) is equal to $\M $ (resp. $\Mp $ and $\Mb $) and a finite presentation for $\mathrm{SMod}_{g;2}$ was given by Birman and Hilden~\cite{Birman-Hilden1}. 

When $k\geq 3$, in Lemma~3.6 of \cite{Ghaswala-Winarski1}, Ghaswala and Winarski gave a necessary and sufficient condition for lifting a homeomorphism on $\Sigma _0$ with respect to $p_{g,k}$ (see also Lemma~\ref{lem_GW}). 
By their necessary and sufficient condition, we see that the liftability of a self-homeomorphism on $(\Sigma _0, \B )$ or $(\Sigma _0^1, \B -\{ p_{2n+2}\} )$ does not depend on $k\geq 3$ (the liftability depends on only the action on $\B $). 
Hence we omit ``$k$'' in the notation of the liftable mapping class groups for $k\geq 3$ (i.e. we express $\mathrm{LMod}_{2n+2;k}=\LM $, $\mathrm{LMod}_{2n+2,\ast ;k}=\LMp $, and $\mathrm{LMod}_{2n+1;k}^1=\LMb $ for $k\geq 3$).  
Ghaswala and Winarski constructed a finite presentation for $\LM $ and calculated the integral first homology group of $\LM $\footnote{The calculation of the first homology group in \cite{Ghaswala-Winarski1} includes a mistake. The error was corrected in \cite{Ghaswala-Winarski1-1}}.

In this paper, we give another finite presentation for $\LM $ in a different generating set from Ghaswala-Winarski's presentation in~\cite{Ghaswala-Winarski1} and finite presentations for the groups $\LMp$, $\LMb$, $\SM $, $\SMp $, and $\SMb $ (Theorems~\ref{thm_pres_lmodb}, \ref{thm_pres_lmodp}, \ref{thm_pres_lmod}, \ref{thm_pres_smodb}, \ref{thm_pres_smodp}, and \ref{thm_pres_smod}), and calculate their integral first homology groups. 
The integral first homology group $H_1(G)$ of a group $G$ is isomorphic to the abelianization of $G$. 
Put $\Z _l=\Z /l\Z $ for an integer $l\geq 2$. 
The results about the integral first homology groups of the liftable mapping class groups and the balanced superelliptic mapping class groups are as follows:

\begin{thm}\label{thm_abel_lmod}
For $n\geq 1$,
\begin{enumerate}
\item $H_1(\LM )\cong \left\{ \begin{array}{ll}
 \Z \oplus \Z _2\oplus \Z _{2}&\text{if }  n \text{ is odd},   \\
 \Z \oplus \Z _{2}&\text{if }  n \text{ is even},
 \end{array} \right.$\\
\item $H_1(\LMp )\cong \left\{ \begin{array}{ll}
 \Z \oplus \Z _2&\text{if }  n=1,   \\
 \Z ^2\oplus \Z _{n}&\text{if }  n\geq 2 \text{ is even},\\
 \Z ^2\oplus \Z _{2n}&\text{if }  n\geq 3 \text{ is odd},
 \end{array} \right.$\\
\item $H_1(\LMb )\cong \left\{ \begin{array}{ll}
 \Z ^2&\text{if }  n=1,   \\
 \Z ^3&\text{if }  n\geq 2.
 \end{array} \right.$\\
\end{enumerate}
\end{thm}

We remark that Theorem~\ref{thm_abel_lmod} for $\LM $ was given by Theorem~1.1 in \cite{Ghaswala-Winarski1-1}. 

\begin{thm}\label{thm_abel_smod}
For $n\geq 1$ and $k\geq 3$ with $g=n(k-1)$,
\begin{enumerate}
\item $H_1(\SM )\cong \left\{ \begin{array}{ll}
 \Z \oplus \Z _2\oplus \Z _2&\text{if }  n\geq 1\text{ is odd and }k\text{ is odd},   \\
 \Z \oplus \Z _2\oplus \Z _4&\text{if }  n\geq 1\text{ is odd and }k\text{ is even},   \\ 
 \Z \oplus \Z _2 &\text{if }  n\geq 2\text{ is even},
 \end{array} \right.$\\
\item $H_1(\SMp )\cong \left\{ \begin{array}{ll}
 \Z \oplus \Z _{2k}&\text{if }  n=1,   \\
 \Z ^2\oplus \Z _{kn}&\text{if }  n\geq 2 \text{ is even},\\
 \Z ^2\oplus \Z _{2kn}&\text{if }  n\geq 3 \text{ is odd},
 \end{array} \right.$\\
\item $H_1(\SMb )\cong \left\{ \begin{array}{ll}
 \Z ^2&\text{if }  n=1,   \\
 \Z ^3&\text{if }  n\geq 2.
 \end{array} \right.$\\
\end{enumerate}
\end{thm}

The explicit generators of the first homology groups in Theorem~\ref{thm_abel_lmod}~and~\ref{thm_abel_smod} are given in their proof in Section~\ref{section_abel-lmod} and~\ref{section_abel-smod}. 
To give presentations for the liftable mapping class groups in Theorem~\ref{thm_pres_lmodb}, \ref{thm_pres_lmodp}, and \ref{thm_pres_lmod}, we construct a new finite presentation for the pure mapping class group of a 2-sphere in Proposition~\ref{prop_pres_pmod}. 
A finite presentation for the pure mapping class group of a 2-sphere was already given by Ghaswala and Winarski~\cite{Ghaswala-Winarski1}. 
They used a finite presentation for the pure braid group by Margalit and McCammond~\cite{Margalit-McCammond} to construct the presentation for the pure mapping class group of a 2-sphere.

Contents of this paper are as follows. 
In Section~\ref{Preliminaries}, we review the definition of the balanced superelliptic covering map $p_{g,k}$ and Ghaswala-Winarski's necessary and sufficient condition for lifting a homeomorphism on $\Sigma _0$ with respect to $p_{g,k}$. 
We introduce some liftable homeomorphisms on $\Sigma _0$ for $p_{g,k}$ in Section~\ref{section_liftable-element} and review some relations among their liftable elements in Section~\ref{section_relations_liftable-elements}. 
In Section~\ref{section_pmod}, we give a new finite presentation for the pure mapping class group of a 2-sphere (Proposition~\ref{prop_pres_pmod}).
In Section~\ref{section_birman-exact-seq}, we observe the group structures of the groups $\LMp $, $\LMb $, and $\SMb $ via group extensions (Propositions~\ref{prop_exact_lmodp}, \ref{prop_exact_lmodb}, and~\ref{prop_exact_smodb}) and characterize $\SMp $ as the image of the forgetful map $\mathcal{F}\colon \SMp \to \SM $ (Proposition~\ref{prop_smod_char}). 
In Section~\ref{section_lmod}, we give finite presentations for the liftable mapping class groups (Theorem~\ref{thm_pres_lmodb}, \ref{thm_pres_lmodp}, and \ref{thm_pres_lmod}) and calculate their integral first homology groups in Section~\ref{section_abel-lmod}. 
Finally, in Section~\ref{section_smod}, we give finite presentations for the balanced superelliptic mapping class groups (Theorem~\ref{thm_pres_smodb}, \ref{thm_pres_smodp}, and \ref{thm_pres_smod}) and calculate their integral first homology groups in Section~\ref{section_abel-smod}. 
In Section~\ref{section_lifts}, we construct explicit lifts of liftable homeomorphisms introduced in Section~\ref{section_liftable-element} with respect to $p_{g,k}$. 

\section{Preliminaries}\label{Preliminaries}

Throughout this paper, we assume that $n$ is a positive integer. 

\subsection{The balanced superelliptic covering space}\label{section_bscov}

For an integer $k\geq 2$ with $g=n(k-1)$, we describe the surface $\Sigma _{g}$ as follows. 
We take the unit 2-sphere $S^2=S(1)$ in $\R ^3$ and $n$ mutually disjoint parallel copies $S(2),\ S(3),\ \dots ,\ S(n+1)$ of $S(1)$ by translations along the x-axis such that 
\[
\max \bigl( S(i)\cap (\R \times \{ 0\}\times \{ 0\} )\bigr) <\min \bigl( S(i+1)\cap (\R \times \{ 0\}\times \{ 0\} )\bigr)
\]
for $1\leq i\leq n$ (see Figure~\ref{fig_bs_periodic_map}). 
Let $\zeta$ be the $(-\frac{2\pi }{k})$-rotation of $\R ^3$ on the $x$-axis. 
Then we remove $2k$ disjoint open disks in $S(i)$ for $2\leq i\leq n$ and $k$ disjoint open disks in $S(i)$ for $i\in \{ 1,\ n+1\}$ which are setwisely preserved by the action of $\zeta $, and connect $k$ boundary components of the punctured $S(i)$ and ones of the punctured $S(i+1)$ by $k$ annuli such that the union of the $k$ annuli is preserved by the action of $\zeta $ for each $1\leq i\leq n$ as in Figure~\ref{fig_bs_periodic_map}. 
Since the union of the punctured $S(1)\cup S(2)\cup \cdots \cup S(n+1)$ and the attached $n\times k$ annuli is homeomorphic to $\Sigma _{g=n(k-1)}$, 
we regard this union as $\Sigma _g$. 

By the construction above, the action of $\zeta $ on $\R ^3$ induces the action on $\Sigma _g$, the quotient space $\Sigma _g/\left< \zeta \right>$ is homeomorphic to $\Sigma _0$, and the quotient map $p=p_{g,k}\colon \Sigma _g\to \Sigma _0$ is a branched covering map with $2n+2$ branch points in $\Sigma _0$. 
We call the branched covering map $p\colon \Sigma _g\to \Sigma _0$ the \textit{balanced superelliptic covering map}. 
Denote by $\widetilde{p}_1,\ \widetilde{p}_2,\ \dots ,\ \widetilde{p}_{2n+2}\in \Sigma _g$ the fixed points of $\zeta $ such that $\widetilde{p}_i<\widetilde{p}_{i+1}$ in $\R \times \{ 0\} \times \{ 0\} =\R$ for $1\leq i\leq 2n+1$, by $p_i\in \Sigma _0$ for $1\leq i\leq 2n+2$ the image of $\widetilde{p}_i$ by $p$ (i.e. $p_1,\ p_2,\ \dots ,\ p_{2n+2}\in \Sigma _g$ are branch points of $p$), and by $\B $ the set of the branch points $p_1,\ p_2,\ \dots ,\ p_{2n+2}$. 
Let $D$ be a 2-disk in $(\Sigma _0-\B )\cup \{ p_{2n+2}\}$ whose interior includes the point $p_{2n+2}$ and $\widetilde{D}$ the preimage of $D$ with respect to $p$. 
Denote by $\Sigma _g^1$ the complement of interior of $\widetilde{D}$ in $\Sigma _g$ and by $\Sigma _0^1$ the complement of interior of $D$ in $\Sigma _0$. 
Then we also call the restriction $p|_{\Sigma _g^1}\colon \Sigma _g^1\to \Sigma _0^1$ the balanced superelliptic covering map and we denote simply $p|_{\Sigma _g^1}=p$. 

\begin{figure}[h]
\includegraphics[scale=1.5]{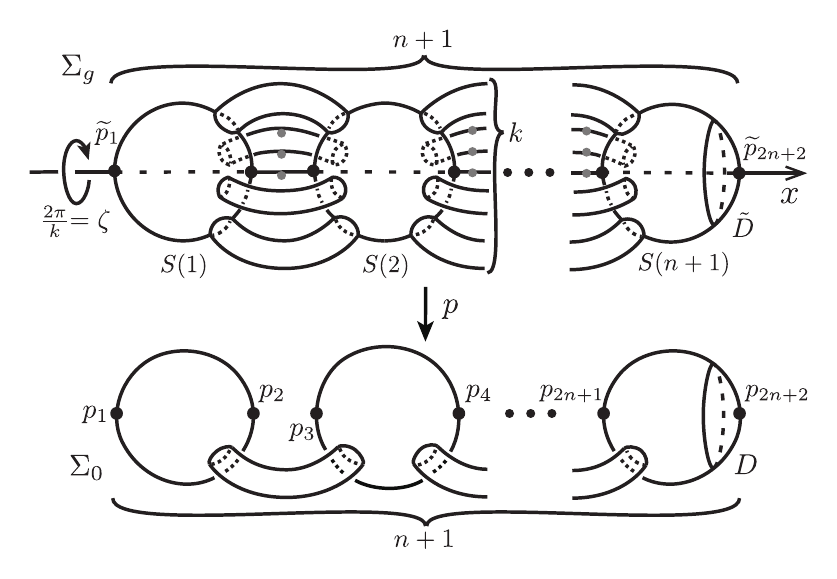}
\caption{The balanced superelliptic covering map $p=p_{g,k}\colon \Sigma _g\to \Sigma _0$.}\label{fig_bs_periodic_map}
\end{figure}

\subsection{Liftable elements for the balanced superelliptic covering map}\label{section_liftable-element}

In this section, we introduce some liftable homeomorphisms on $\Sigma _0$ for the balanced superelliptic covering map $p=p_{g,k}$ for $k\geq 3$. 
For maps or mapping classes $f$ and $g$, the product $gf$ means that $f$ apply first, and we often abuse notation and denote a homeomorphism and its isotopy class by the same symbol.  

Let $l_i$ $(1\leq i\leq 2n+1)$ be an oriented simple arc on $\Sigma _0$ whose endpoints are $p_i$ and $p_{i+1}$ as in Figure~\ref{fig_path_l}. 
Put $L=l_1\cup l_2\cup \cdots \cup l_{2n+1}$, $L^1=L\cap \Sigma _0^1$, and $\B ^1=\B -\{ p_{2n+2}\}$. 
The isotopy class of a homeomorphism $\varphi $ on $\Sigma _0$ (resp. $\Sigma _0^1$) relative to $\B $ (resp. $\B ^1\cup \partial \Sigma _0^1$) is determined by the isotopy class of the image of $L$ (resp. $L^1$) by $\varphi $ relative to  $\B $ (resp. $\B ^1\cup \partial \Sigma _0^1$). 
We identify $\Sigma _0$ with the surface on the lower side in Figure~\ref{fig_path_l} by some homeomorphism of $\Sigma _0$. 

\begin{figure}[h]
\includegraphics[scale=1.5]{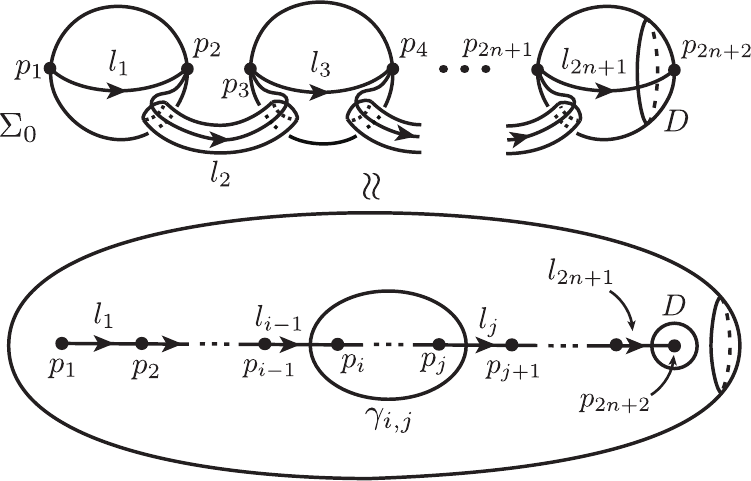}
\caption{A natural homeomorphism of $\Sigma _0$, arcs $l_1,\ l_2,\ \dots ,\ l_{2n+1}$, and a simple closed curve $\gamma _{i,j}$ on $\Sigma _0$ for $1\leq i<j\leq 2n+2$.}\label{fig_path_l}
\end{figure}

Let $l$ be a simple arc on $\Sigma _0$ whose endpoints lie in $\B $. 
A regular neighborhood $\mathcal{N}$ of $l$ in $\Sigma _0$ is homeomorphic to a 2-disk. 
Then the half-twist $\sigma [l]$ is a self-homeomorphism on $\Sigma _0$ which is described as the result of anticlockwise half-rotation of $l$ in $\mathcal{N}$ as in Figure~\ref{fig_sigma_l}. 
We define $\sigma _i=\sigma [l_i]$ for $1\leq i\leq 2n+1$. 
As a well-known result, $\M $ is generated by $\sigma _1$, $\sigma _2,\ \dots $, $\sigma _{2n+1}$ 
(see for instance Section~9.1.4 in \cite{Farb-Margalit}). 
Let $\mathrm{Map}(\B )$ be the group of self-bijections on $\B $. 
Since $\mathrm{Map}(\B )$ is naturally identified with the symmetric group $S_{2n+2}$ of degree $2n+2$, the action of $\M $ on $\B$ induces the surjective homomorphism
\[
\Psi \colon \M \to S_{2n+2}
\]
given by $\Psi (\sigma _i)=(i\ i+1)$.  

\begin{figure}[h]
\includegraphics[scale=1.1]{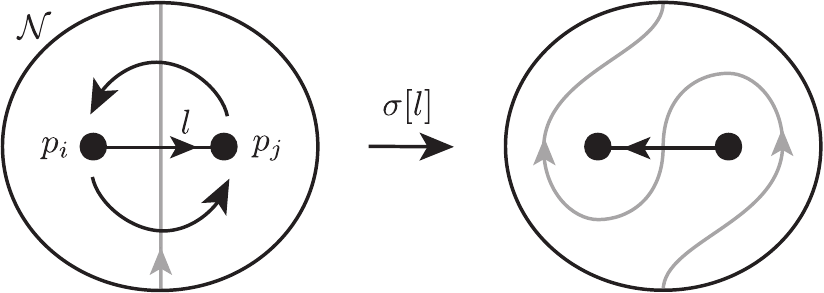}
\caption{The half-twist $\sigma [l]$ along an arc $l$.}\label{fig_sigma_l}
\end{figure}

Put $\B _o=\{ p_1,\ p_3,\ \dots ,\ p_{2n+1}\}$ and $\B _e=\{ p_2,\ p_4,\ \dots ,\ p_{2n+2}\}$. 
An element $\sigma $ in $S_{2n+2}$ is \textit{parity-preserving} if $\sigma (\B _o)=\B _o$, and is \textit{parity-reversing} if $\sigma (\B _o)=\B _e$. 
An element $f$ in $\M $ is \textit{parity-preserving} (resp. \textit{parity-reversing}) if $\Psi (f)$ is \textit{parity-preserving} (resp. \textit{parity-reversing}). 
Let $W_{2n+2}$ be the subgroup of $S_{2n+2}$ which consists of parity-preserving or parity-reversing elements, $S_{n+1}^o$ (resp. $S_{n+1}^e$) the subgroup of $S_{2n+2}$ which consists of elements whose restriction to $\B _e$ (resp. $\B _o$) is the identity map.  
Note that $S_{n+1}^o$ (resp. $S_{n+1}^e$) is a subgroup of $W_{2n+2}$, isomorphic to $S_{n+1}$, and generated by transpositions $(1\ 3)$, $(3\ 5),\ \dots $, $(2n-1\ 2n+1)$ (resp. $(2\ 4)$, $(4\ 6),\ \dots $, $(2n\ 2n+2)$). 
Then we have the following exact sequence:
\begin{eqnarray}\label{exact1}
1\longrightarrow S_{n+1}^o\times S_{n+1}^e\longrightarrow W_{2n+2}\stackrel{\pi }{\longrightarrow }\Z _2\longrightarrow 1, 
\end{eqnarray}
where the homomorphism $\pi \colon W_{2n+2}\to \Z _2$ is defined by $\pi (\sigma )=0$ if $\sigma $ is parity-preserving and $\pi (\sigma )=1$ if $\sigma $ is parity-reversing. 
Ghaswala and Winarski~\cite{Ghaswala-Winarski1} proved the following lemma. 
\begin{lem}[Lemma~3.6 in \cite{Ghaswala-Winarski1}]\label{lem_GW}
Let $\LM $ be the liftable mapping class group for the balanced superelliptic covering map $p_{g,k}$ for $n\geq 1$ and $k\geq3$ with $g=n(k-1)$. 
Then we have
\[
\LM =\Psi ^{-1}(W_{2n+2}).
\]
\end{lem}
Lemma~\ref{lem_GW} implies that a mapping class $f\in \M $ lifts with respect to $p_{g,k}$ if and only if $f$ is parity-preserving or parity-reversing (in particular, when $k\geq 3$, the liftability of a homeomorphism on $\Sigma _0$ or $\Sigma _0^1$ does not depend on $k$). 
The \textit{pure mapping class group} $\PM $ is the kernel of the homomorphism $\Psi \colon \M \to S_{2n+2}$. 
Since all elements in $\PM $ is parity-preserving, $\PM $ is a subgroup of $\LM $ and we have the following exact sequence:  
\begin{eqnarray}\label{exact2}
1\longrightarrow \PM \longrightarrow \LM \stackrel{\Psi }{\longrightarrow }W_{2n+2}\longrightarrow 1. 
\end{eqnarray}

We will introduce some explicit liftable elements as follows (their explicit lifts to $\Sigma _g$ are given in Section~\ref{section_lifts}). 
\\

\paragraph{\emph{The Dehn twist $t_{i,j}$}}

For a simple closed curve $\gamma $ on $\Sigma _g$ $(g\geq 0)$, we denote by $t_\gamma $ the right-handed Dehn twist along $\gamma $. 
Let $\gamma _{i,j}$ for $1\leq i<j\leq 2n+2$ be a simple closed curve on $\Sigma _0-(\B \cup D)$ such that $\gamma _{i,j}$ surrounds the $j-i+1$ points $p_i$, $p_{i+1},\ \dots $,~$p_j$ as in Figure~\ref{fig_path_l}. 
We have $\gamma _{i,2n+2}=\gamma _{1,i-1}$ for $i\geq 3$ as an isotopy class of a curve. 
Then we define $t_{i,j}=t_{\gamma _{i,j}}$ for $1\leq i<j\leq 2n+2$. 
Note that $t_{1,2n+1}=t_{1,2n+2}=t_{2,2n+2}=1$ and $t_{i,2n+2}=t_{1,i-1}$ for $i\geq 3$ in $\M $. 
Since the Dehn twist $t_{i,j}$ preserves $\B$ pointwise, i.e. $t_{i,j}$ lies in $\PM $, $t_{i,j}$ lifts with respect to $p_{g,k}$ by Lemma~\ref{lem_GW}. \\

\paragraph{\emph{The half-twists $a_i$ and $b_i$}}

Let $\alpha _i$ (resp. $\beta _i$) for $1\leq i\leq n$ be a simple arc whose endpoints are $p_{2i-1}$ and $p_{2i+1}$ (resp. $p_{2i}$ and $p_{2i+2}$) as in Figure~\ref{fig_arcs_a_ib_i}. 
Then we define $a_i=\sigma [\alpha _i]$ and $b_i=\sigma [\beta _i]$ for $1\leq i\leq n$. 
Note that $a_i=\sigma _{2i}\sigma _{2i-1}\sigma _{2i}^{-1}$ and $b_i=\sigma _{2i+1}\sigma _{2i}\sigma _{2i+1}^{-1}$ for $1\leq i\leq n$. 
Since $\Psi (a_i)=(2i-1\ 2i+1)$ and $\Psi (b_i)=(2i\ 2i+2)$ for $1\leq i\leq n$, the mapping classes $a_i$ and $b_i$ are parity-preserving. 
Thus, by Lemma~\ref{lem_GW}, $a_i$ and $b_i$ lift with respect to $p_{g,k}$. 
The generating set of Ghaswala-Winarski's finite presentation for $\LM $ in~\cite{Ghaswala-Winarski1} consists of $a_i^{-1}$, $b_i^{-1}$ for $1\leq i\leq n$, some Dehn twists, and one parity-reversing element. \\

\begin{figure}[h]
\includegraphics[scale=1.3]{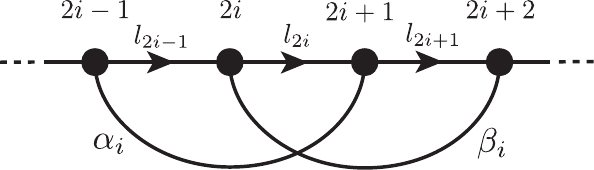}
\caption{Arcs $\alpha _i$ and $\beta _i$ $(1\leq i\leq n)$ on $\Sigma _0$.}\label{fig_arcs_a_ib_i}
\end{figure}

\paragraph{\emph{The half-rotation $h_i$}}

Let $\mathcal{N}$ be a regular neighborhood of $l_i\cup l_{i+1}$ in $(\Sigma _0-\B )\cup \{ p_{i},\ p_{i+1},\ p_{i+2}\}$ for $1\leq i\leq 2n$. 
$\mathcal{N}$ is homeomorphic to a 2-disk. 
Then we denote by $h_i$ the self-homeomorphism on $\Sigma _0$ which is described as the result of anticlockwise half-rotation of $l_i\cup l_{i+1}$ in $\mathcal{N}$ as in Figure~\ref{fig_h_i}. 
Note that $h_i=\sigma _i\sigma _{i+1}\sigma _i$ for $1\leq i\leq 2n$. 
Since $\Psi (h_i)=(i\ i+2)$ for $1\leq i\leq 2n$, the mapping class $h_i$ is parity-preserving. 
Thus, by Lemma~\ref{lem_GW}, $h_i$ lifts with respect to $p_{g,k}$. 

\begin{figure}[h]
\includegraphics[scale=0.95]{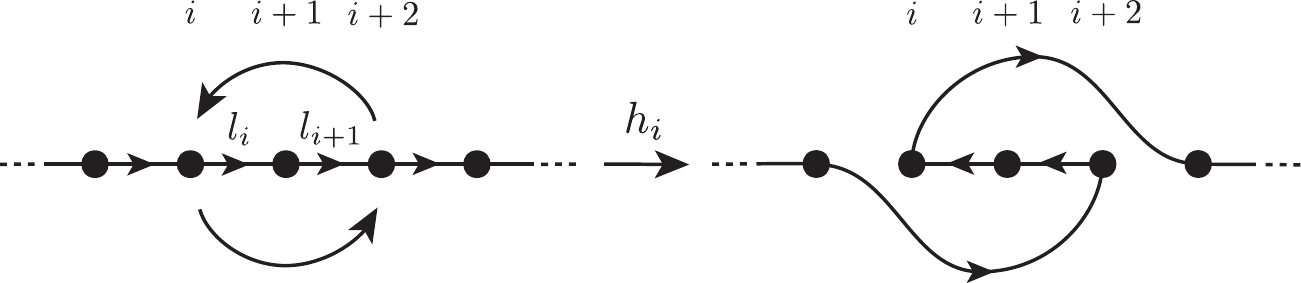}
\caption{The mapping class $h_i$ $(1\leq i\leq 2n)$ on $\Sigma _0$.}\label{fig_h_i}
\end{figure}

\paragraph{\emph{The parity-reversing element $r$}}

Let $r$ be the self-homeomorphism on $\Sigma _0$ which is described as the result of the $\pi $-rotation of $\Sigma _0$ on the axis as in Figure~\ref{fig_r}. 
We see that $r(l_{i})=l_{2n+2-i}^{-1}$ for $1\leq i\leq 2n+1$ and $\Psi (r)=(1\ 2n+2)(2\ 2n+1)\cdots (n+1\ n+2)$. 
Thus $r$ is parity-reversing and lifts with respect to $p_{g,k}$ by Lemma~\ref{lem_GW}. \\

\begin{figure}[h]
\includegraphics[scale=1.3]{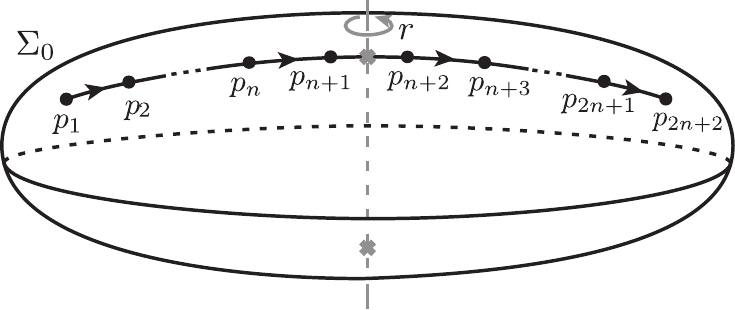}
\caption{The mapping class $r$ on $\Sigma _0$.}\label{fig_r}
\end{figure}

We can regard a mapping class in $\M $ which has a representative fixing $D$ pointwise as an element in $\Mb $. 
We regard the mapping classes $t_{i,j}$ for $1\leq i<j\leq 2n+2$, $a_i$ for $1\leq i\leq n$, $b_i$ for $1\leq i\leq n-1$, and $h_i$ for $1\leq i\leq 2n-1$ as elements in $\Mb $.

\subsection{Basic relations among liftable elements for the balanced superelliptic covering map}\label{section_relations_liftable-elements}

In this section, we review relations among the elements $h_i$, $t_{i,j}$, and $r$ introduced in Section~\ref{section_liftable-element}. 
Recall that we regard $\LMp $ as a subgroup of $\LM $ and for mapping classes $[\varphi ]$ and $[\psi ]$, $[\psi ][\varphi ]=[\psi \circ \varphi ]$, i.e. $[\varphi ]$ apply first. 

\paragraph{\emph{Commutative relations}}

For elements $f$ and $h$ in a group $G$, $f\rightleftarrows h$ if $f$ commutes with $h$, namely, the relation $fh=hf$ holds in $G$. 
We call such a relation a \textit{commutative relation}. 
For mapping classes $f$ and $h$, if their representatives have mutually disjoint supports, then we have $f\rightleftarrows h$ in the mapping class group. 
We have the following lemma. 

\begin{lem}\label{lem_comm_rel}
We have the following relations:
\begin{enumerate}
\item $t_{i,j} \rightleftarrows t_{k,l}$ \quad for $j<k$, $k\leq i<j\leq l$, or $l<i$ in $\LM $ and $\Mb$, 
\item $h_i \rightleftarrows h_j$ \quad for $j-i\geq 3$ in $\LM $ and for $j-i\geq 3$ and $j\not=2n$ in $\Mb $, 
\item $h_k \rightleftarrows t_{i,j}$ \quad for $k+2<i$, $i\leq k<k+2\leq j$, or $j<k$ in $\LM $ and for ``$k+2<i$, $i\leq k<k+2\leq j$, or $j<k$'' and $k\not=2n$ in $\Mb $. 
\end{enumerate}
\end{lem}

\paragraph{\emph{Conjugation relations}}

For elements $f_1$, $f_2$, and $h$ in a group $G$, we call the relation $hf_1h^{-1}=f_2$, which is equivalent to $hf_1=f_2h$, in $G$ a \textit{conjugation relation}. 
In particular, in the case of $f_1=t_\gamma $ or $f_1=\sigma [l]$, and $h$ is a mapping class, the relations $ht_\gamma =t_{h(\gamma )}h$ and $h\sigma [l]=\sigma [h(l)]h$ hold in the mapping class group. 
By the relations above, we have the following lemma. 

\begin{lem}\label{lem_conj_rel}
The following relations hold in $\LMp $ and $\Mb $:
\begin{enumerate}
\item $h_i ^{\pm 1}t_{i,i+1}=t_{i+1,i+2}h_i^{\pm 1}$ \quad for $1\leq i\leq 2n-1$, 
\item $h_i ^{\varepsilon }h_{i+1}^{\varepsilon }t_{i,i+1}=t_{i+2,i+3}h_i^{\varepsilon }h_{i+1}^{\varepsilon }$ \quad for $1\leq i\leq 2n-2$ and $\varepsilon \in \{ 1, -1\}$,
\item $h_{i}^{\varepsilon }h_{i+1}^{\varepsilon }h_{i+2}^{\varepsilon }h_{i}^{\varepsilon }=h_{i+2}^{\varepsilon }h_{i}^{\varepsilon }h_{i+1}^{\varepsilon }h_{i+2}^{\varepsilon }$ \quad for $1\leq i\leq 2n-3$ and $\varepsilon \in \{ 1, -1\}$,
\end{enumerate}
and the following relations hold in $\LM $:
\begin{enumerate}
\item $h_{2n}^{\pm 1}t_{2n,2n+1}=t_{2n+1,2n+2}h_{2n}^{\pm 1}$, 
\item $h_{2n-1}^{\varepsilon }h_{2n}^{\varepsilon }t_{2n-1,2n}=t_{2n+1,2n}h_{2n-1}^{\varepsilon }h_{2n}^{\varepsilon }$ \quad for $\varepsilon \in \{ 1, -1\}$,
\item $h_{2n-2}^{\varepsilon }h_{2n-1}^{\varepsilon }h_{2n}^{\varepsilon }h_{2n-1}^{\varepsilon }=h_{2n}^{\varepsilon }h_{2n-2}^{\varepsilon }h_{2n-1}^{\varepsilon }h_{2n}^{\varepsilon }$\quad for $\varepsilon \in \{ 1, -1\}$,
\item $rh_i=h_{2n-i+1}r$ \quad for $1\leq i\leq 2n$, 
\item $rt_{i,j}=t_{2n-j+3,2n-i+3}r$ \quad for $1\leq i<j\leq 2n+2$. 
\end{enumerate}
\end{lem}

\paragraph{\emph{Lifts of relations in the symmetric group}}

Recall that we have the surjective homomorphism $\Psi \colon \LM \to W_{2n+2}$ by Lemma~\ref{lem_GW} and the relations $\Psi (h_i)=(i\ i+2)$ for $1\leq i\leq 2n$ and $\Psi (r)=(1\ 2n+2)(2\ 2n+1)\cdots (n+1\ n+2)$. 
By the exact sequence~(\ref{exact1}), $W_{2n+2}$ is a group extension of $S_{n+1}^o\times S_{n+1}^e$ by $\Z _2$. 
For an element $\sigma $ in $\M $, we define $\bar{\sigma }=\Psi (\sigma )$. 
Then the following relations hold in $W_{2n+2}$: 
\begin{itemize}
\item $\bar{h}_i^2=1$ \quad for $1\leq i\leq 2n$, 
\item $\bar{h}_i\bar{h}_{i+1}\bar{h}_i^{-1}\bar{h}_{i+1}^{-1}=1$ \quad for $1\leq i\leq 2n-1$, 
\item $\bar{h}_i\bar{h}_{i+2}\bar{h}_i\bar{h}_{i+2}^{-1}\bar{h}_i^{-1}\bar{h}_{i+2}^{-1}=1$ \quad for $1\leq i\leq 2n-2$, 
\item $\bar{r}^2=1$, 
\end{itemize}
By the exact sequence~(\ref{exact1}), products $h_i^2$ $(1\leq i\leq 2n)$, $h_ih_{i+1}h_i^{-1}h_{i+1}^{-1}$ $(1\leq i\leq 2n-1)$, $h_ih_{i+2}h_ih_{i+2}^{-1}h_i^{-1}h_{i+2}^{-1}$ $(1\leq i\leq 2n-2)$, and $r^2$ lie in $\PM $.
We have the following lemma. 

\begin{lem}\label{lem_lift_W}
The following relations hold in $\LMp $ and $\Mb $:
\begin{enumerate}
\item $h_i^2=t_{i,i+2}$ \quad for $1\leq i\leq 2n-1$, 
\item $h_ih_{i+1}h_i^{-1}h_{i+1}^{-1}=t_{i,i+1}t_{i+2,i+3}^{-1}$ \quad for $1\leq i\leq 2n-2$,
\item $h_ih_{i+2}h_ih_{i+2}^{-1}h_i^{-1}h_{i+2}^{-1}=t_{i+1,i+2}t_{i+2,i+3}^{-1}$ \quad for $1\leq i\leq 2n-3$,
\end{enumerate}
and the following relations hold in $\LM $:
\begin{enumerate}
\item $h_{2n}^2=t_{2n,2n+2}=t_{1,2n}$, 
\item $h_{2n-1}h_{2n}h_{2n-1}^{-1}h_{2n}^{-1}=t_{2n-1,2n}t_{2n+1,2n+2}^{-1}=t_{2n-1,2n}t_{1,2n}^{-1}$,
\item $h_{2n-2}h_{2n}h_{2n-2}h_{2n}^{-1}h_{2n-2}^{-1}h_{2n}^{-1}=t_{2n-1,2n}t_{2n,2n+1}^{-1}$,
\item $r^2=1$.
\end{enumerate}
\end{lem}

\begin{proof}
The relations~(1) and (4) clearly hold. 
We will show the relations~(2) and (3) from a point of view of the spherical braid group (we can also prove by observing the image of $L$ and $L^1$ by these mapping classes). 
Let $SB_{2n+2}$ be the \textit{spherical braid group} of $2n+2$ strands. 
For $b_1$, $b_2\in SB_{2n+2}$, the product $b_1b_2$ is a braid as on the right-hand side in Figure~\ref{fig_braid_product_def}. 
Then we have the surjective homomorphism $\Gamma \colon SB_{2n+2}\to \M $ which maps the half-twist about $i$-th and $(i+1)$-st strands as on the left-hand side in Figure~\ref{fig_braid_product_def} to the half-twist $\sigma _i\in \M $. 
We abuse notation and denote simply the half-twist in $SB_{2n+2}$ about $i$-th and $(i+1)$-st strands by $\sigma _i$, the full-twist in $SB_{2n+2}$ about strands between $i$-th strand and $j$-th strand by $t_{i,j}$, and $h_i=\sigma _i\sigma _{i+1}\sigma _i$ for $1\leq i\leq 2n$. 
The braid $h_ih_{i+1}h_i^{-1}h_{i+1}^{-1}$ lies on the left-hand side in Figure~\ref{fig_braid_rel_calc} and we can see that $h_ih_{i+1}h_i^{-1}h_{i+1}^{-1}$ is isotopic to $t_{i,i+1}t_{i+2,i+3}^{-1}$. 
We can also show that $h_ih_{i+2}h_ih_{i+2}^{-1}h_i^{-1}h_{i+2}^{-1}=t_{i+1,i+2}t_{i+2,i+3}^{-1}$ in $SB_{2n+2}$ as on the right-hand side in Figure~\ref{fig_braid_rel_calc}. 
We have completed the proof of Lemma~\ref{lem_lift_W}. 
\end{proof}

\begin{figure}[h]
\includegraphics[scale=0.75]{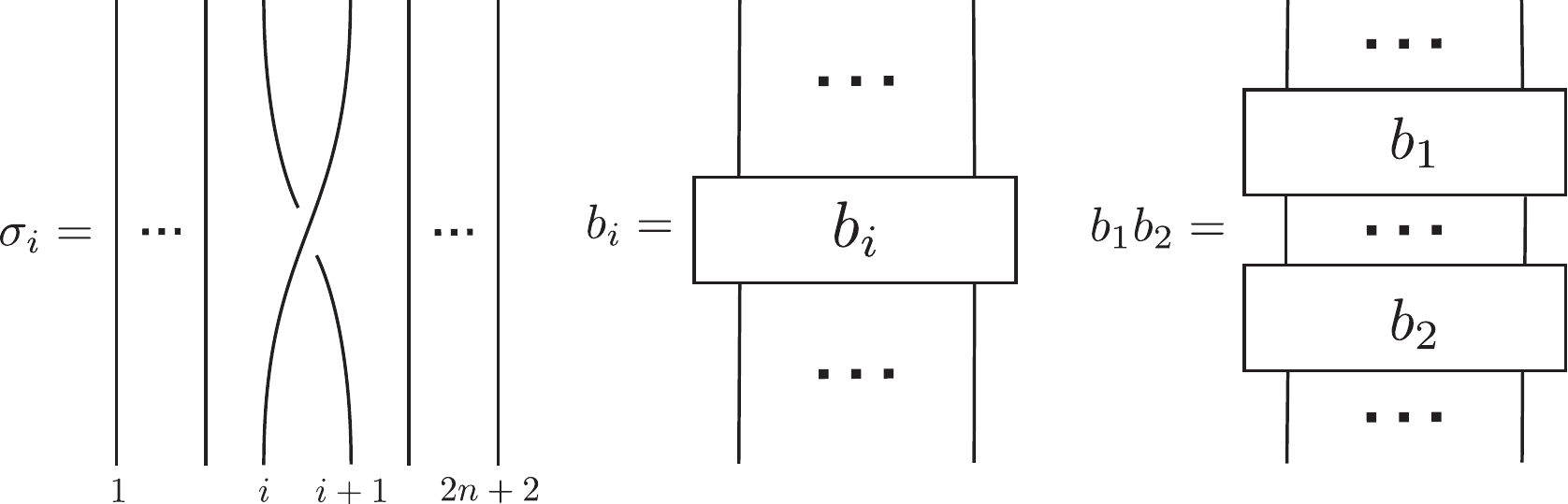}
\caption{The half-twist $\sigma_i\in SB_{2n+2}$ $(1\leq i\leq 2n+1)$ and the product $b_1b_2$ in $SB_{2n+2}$ for $b_i$ $(i=1,\ 2)$.}\label{fig_braid_product_def}
\end{figure}

\begin{figure}[h]
\includegraphics[scale=0.72]{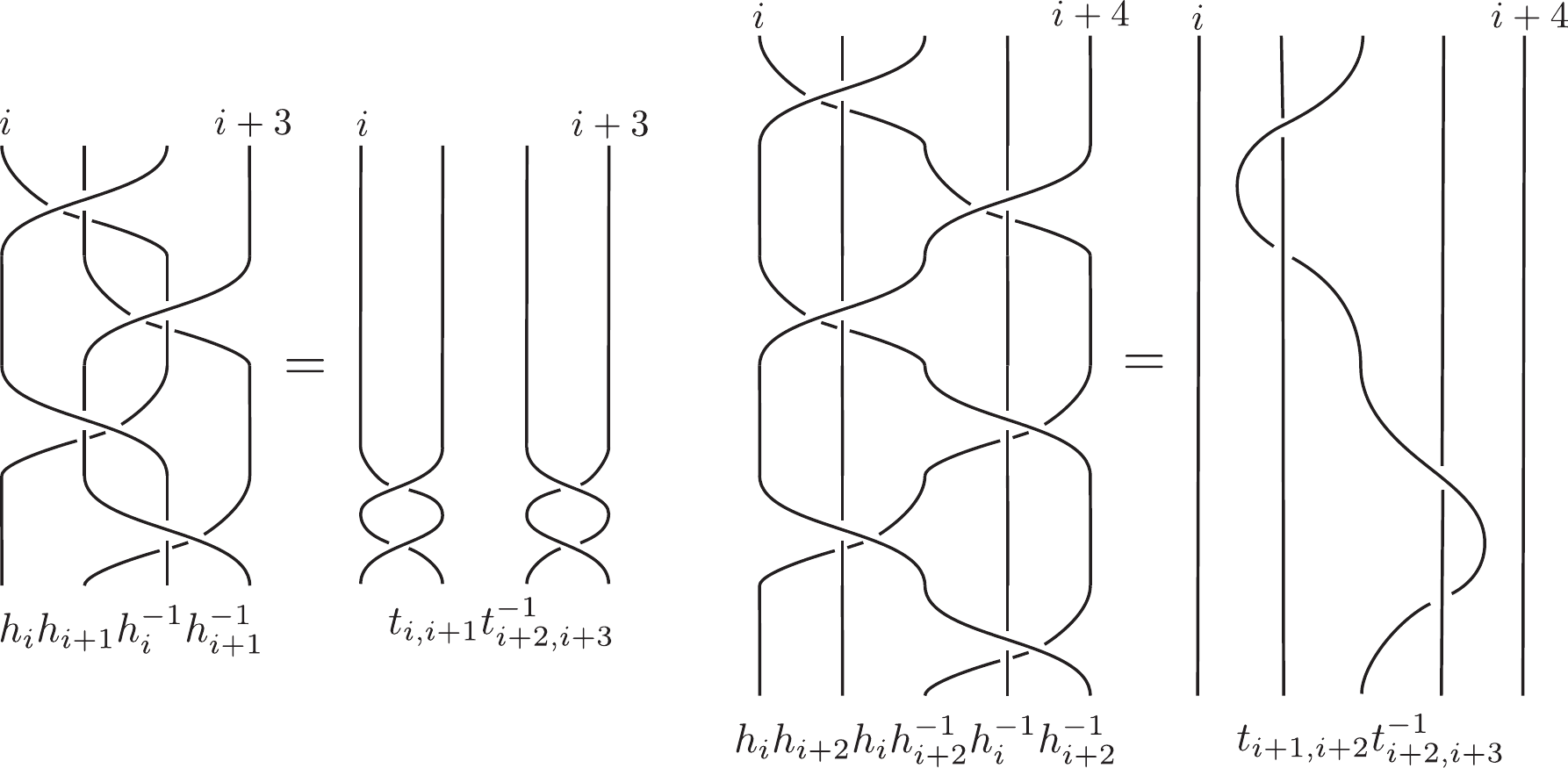}
\caption{Relations $h_ih_{i+1}h_i^{-1}h_{i+1}^{-1}=t_{i,i+1}t_{i+2,i+3}^{-1}$ and $h_ih_{i+2}h_ih_{i+2}^{-1}h_i^{-1}h_{i+2}^{-1}=t_{i+1,i+2}t_{i+2,i+3}^{-1}$ in $SB_{2n+2}$.}\label{fig_braid_rel_calc}
\end{figure}

Remark that, up to commutative relations in Lemma~\ref{lem_comm_rel} and conjugation relations in Lemma~\ref{lem_comm_rel}, the relation~(2) in Lemma~\ref{lem_lift_W}, i.e. $h_ih_{i+1}h_i^{-1}h_{i+1}^{-1}=t_{i,i+1}t_{i+2,i+3}^{-1}$, is equivalent to the relation $h_ih_{i+1}t_{i,i+1}=t_{i,i+1}h_{i+1}h_i$ for $1\leq i\leq 2n-1$, and the relation~(3) in Lemma~\ref{lem_lift_W}, i.e. $h_ih_{i+2}h_ih_{i+2}^{-1}h_i^{-1}h_{i+2}^{-1}=t_{i+1,i+2}t_{i+2,i+3}^{-1}$, is equivalent to the relation $h_ih_{i+2}h_it_{i+1,i+2}^{-1}=t_{i+2,i+3}^{-1}h_{i+2}h_ih_{i+2}$ for $1\leq i\leq 2n-2$. \\

\paragraph{\emph{An expression of $t_{i,j}$ for $j-i\geq 3$ by product of $h_k$ and $t_{l,l+1}$}}

By Lemma~\ref{lem_lift_W}, we have $t_{i,i+2}=h_i^2$ for $1\leq i\leq 2n$. 
In the case of $j-i\geq 3$, an expression of $t_{i,j}$ by product of $h_k$ $(1\leq k\leq 2n)$ and $t_{l,l+1}$ $(1\leq l\leq 2n+1)$ is given by the following lemma. 

\begin{lem}\label{lem_t_ij-braid}
The following relations hold in $\LM $:
\[
t_{i,j}=\left\{
		\begin{array}{ll}
		h_i^2 \quad \text{for }j-i=2,\\
		t_{j-1,j}^{-\frac{j-i-3}{2}}\cdots t_{i+2,i+3}^{-\frac{j-i-3}{2}}t_{i,i+1}^{-\frac{j-i-3}{2}}(h_{j-2}\cdots h_{i+1}h_{i})^{\frac{j-i+1}{2}}\\ \text{for }2n+1\geq j-i \geq 3\text{ is odd},\\
		t_{j-2,j-1}^{-\frac{j-i-2}{2}}\cdots t_{i+2,i+3}^{-\frac{j-i-2}{2}}t_{i,i+1}^{-\frac{j-i-2}{2}}h_{j-2}\cdots h_{i+2}h_{i}h_{i}h_{i+2}\cdots h_{j-2}\\ \cdot (h_{j-3}\cdots h_{i+1}h_{i})^{\frac{j-i}{2}}\quad \text{for }2n\geq j-i \geq 4\text{ is even}.		
		\end{array}
		\right.\\
\]
For $j\leq 2n+1$, the relations above hold in $\LMp $ and $\Mb $. 
\end{lem}

We call the relation in Lemma~\ref{lem_t_ij-braid} for $1\leq i<j\leq 2n+2$ the \textit{relation}~$(T_{i,j})$.  
Recall that we have $t_{2n+1,2n+2}=t_{1,2n}$ and $t_{1,2n+1}=t_{2,2n+2}=t_{1,2n+2}=1$ in $\LM $. 
Thus the relation
\[
t_{2n-1,2n}^{-n+1}\cdots t_{3,4}^{-n+1}t_{1,2}^{-n+1}h_{2n-1}\cdots h_3h_1h_1h_3\cdots h_{2n-1}(h_{2n-2}\cdots h_{2}h_{1})^{n}=1
\]
holds in $\LMp $ and the relations
\[
t_{2n,2n+1}^{-n+1}\cdots t_{4,5}^{-n+1}t_{2,3}^{-n+1}h_{2n}\cdots h_4h_2h_2h_4\cdots h_{2n}(h_{2n-1}\cdots h_{3}h_{2})^{n}=1 
\]
and
\[
t_{1,2n}^{-n+1}t_{2n-1,2n}^{-n+1}\cdots t_{3,4}^{-n+1}t_{1,2}^{-n+1}(h_{2n}\cdots h_{2}h_{1})^{n+1}=1
\]
hold in $\LM $. 
We will give a precise proof of Lemma~\ref{lem_t_ij-braid} in the proof of Lemma~\ref{lem_t_ij_pres1} by proving that the relations~$(T_{i,j})$ are equivalent to lantern relations up to the relations in Lemmas~\ref{lem_comm_rel}, \ref{lem_conj_rel}, and \ref{lem_lift_W}. 
However, we can check that the relations in Lemma~\ref{lem_t_ij-braid} holds in $\LM $ and $\Mb $ from a point of view of the spherical braid group as a similar argument in the proof of Lemma~\ref{lem_lift_W}. 
Recall the surjective homomorphism $\Gamma \colon SB_{2n+2}\to \M $ defined in the proof of Lemma~\ref{lem_lift_W}. 
For instance, $t_{i,i+5}$ is expressed by the braid as on the right-hand side in Figure~\ref{fig_t_ij-braid-even} and the product $t_{i,i+1}^{-1}t_{i+2,i+3}^{-1}t_{i+4,i+5}^{-1}(h_{i+3}h_{i+2}h_{i+1}h_i)^3$ is the braid as on the left-hand side in Figure~\ref{fig_t_ij-braid-even}. 
By the deformation by an isotopy as in Figure~\ref{fig_t_ij-braid-even}, we can show that the braid $t_{i,i+5}$ coincides with the braid $t_{i,i+1}^{-1}t_{i+2,i+3}^{-1}t_{i+4,i+5}^{-1}(h_{i+3}h_{i+2}h_{i+1}h_i)^3$. 
For the case of $t_{i,i+6}$, we also have a deformation by an isotopy which gives $t_{i,i+6}=t_{i,i+1}^{-2}t_{i+2,i+3}^{-2}t_{i+4,i+5}^{-2}h_{i+4}h_{i+2}h_ih_ih_{i+2}h_{i+4}(h_{i+3}h_{i+2}h_{i+1}h_i)^3$ as in Figure~\ref{fig_t_ij-braid-odd}. \\

\begin{figure}[h]
\includegraphics[scale=0.72]{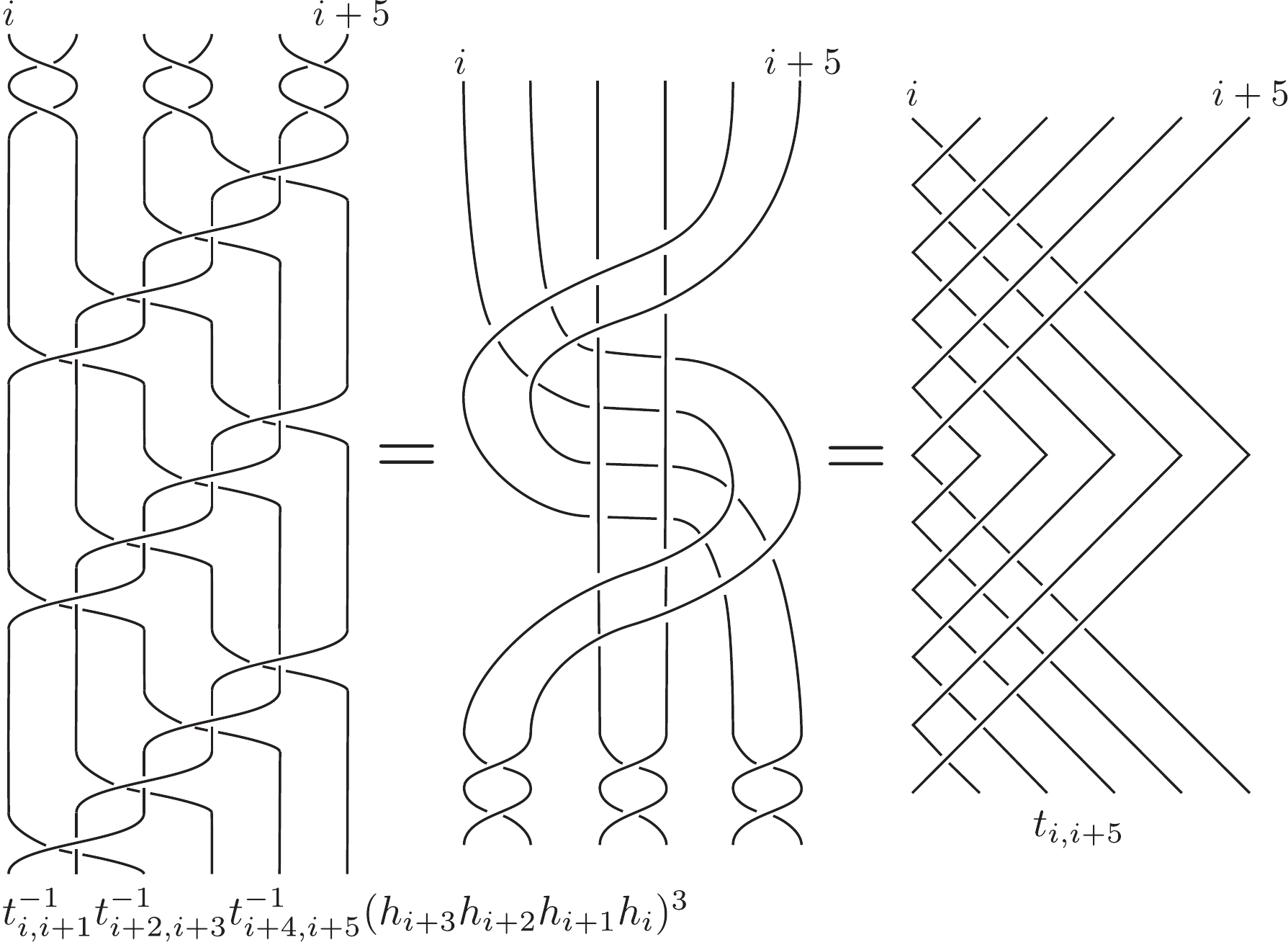}
\caption{The relation $t_{i,i+1}^{-1}t_{i+2,i+3}^{-1}t_{i+4,i+5}^{-1}(h_{i+3}h_{i+2}h_{i+1}h_i)^3=t_{i,i+5}$ in $SB_{2n+2}$.}\label{fig_t_ij-braid-even}
\end{figure}

\begin{figure}[h]
\includegraphics[scale=0.65]{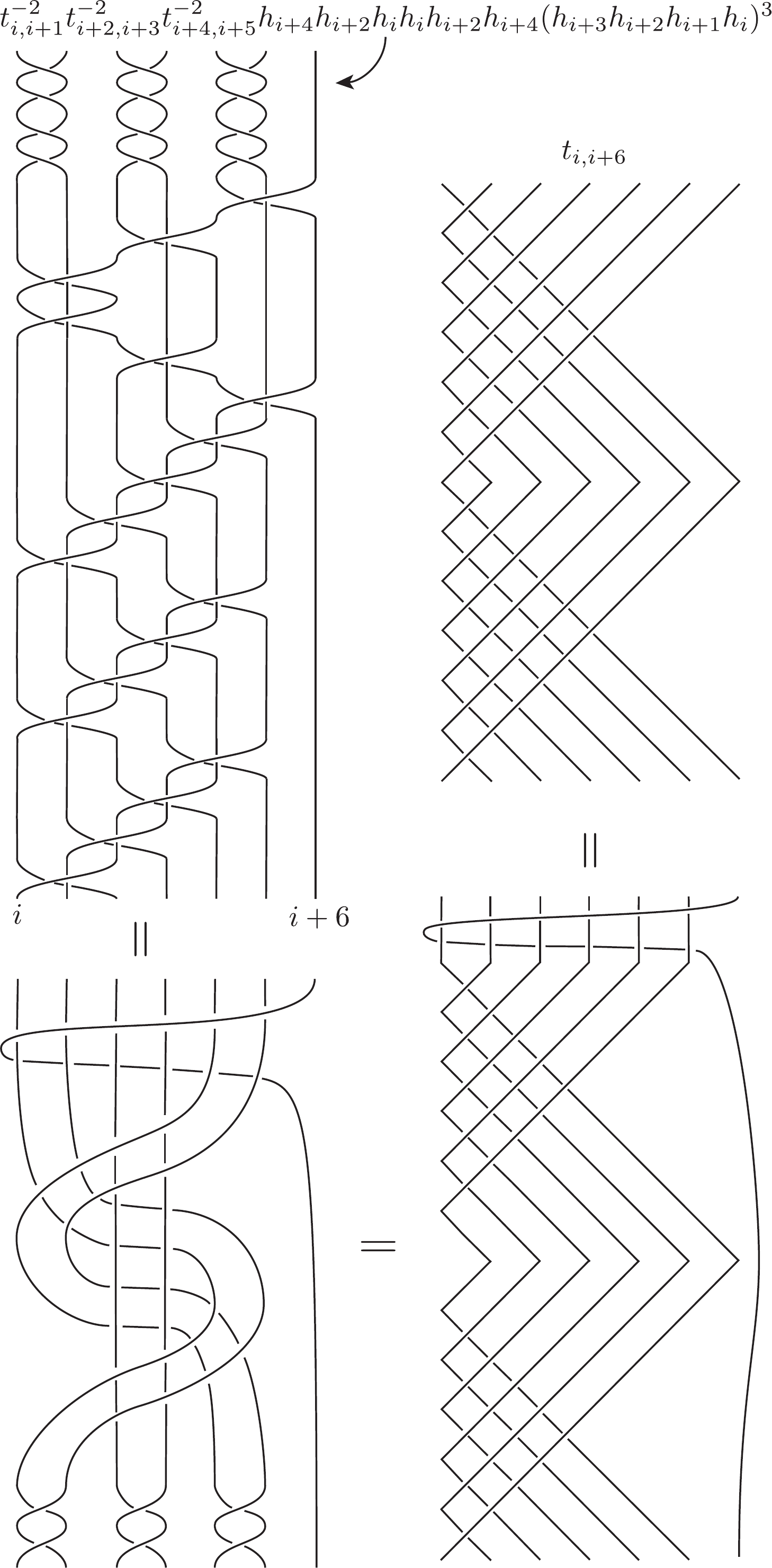}
\caption{The relation $t_{i,i+1}^{-2}t_{i+2,i+3}^{-2}t_{i+4,i+5}^{-2}h_{i+4}h_{i+2}h_ih_ih_{i+2}h_{i+4}\\ \cdot (h_{i+3}h_{i+2}h_{i+1}h_i)^3=t_{i,i+6}$ in $SB_{2n+2}$.}\label{fig_t_ij-braid-odd}
\end{figure}

\paragraph{\emph{Pentagonal relations}}

Pentagonal relations are given in the following lemma.

\begin{lem}[Pentagonal relation]\label{lem_pentagon_rel}
Let $\Sigma $ be a 2-sphere with five boundary components and $\alpha $, $\beta $, $\gamma $, $\delta $, and $\varepsilon $ simple closed curves on $\Sigma $, each one of which separates $\Sigma $ into two 2-spheres with three and four boundary components as on the left-hand side in Figure~\ref{fig_pentagon-rel_def}. 
Then we have the following relation in $\mathrm{Mod}(\Sigma ,\emptyset ,\partial \Sigma )$:
\[
t_\varepsilon ^{-1}t_\gamma t_\beta t_\alpha t_\delta ^{-1}=t_\delta ^{-1}t_\alpha t_\beta t_\gamma t_\varepsilon ^{-1}.
\]
\end{lem}

We call the relation in Lemma~\ref{lem_pentagon_rel} a \textit{pentagonal relation}. 
Lemma~\ref{lem_pentagon_rel} is proved after the next lemma. 

We consider a pentagon whose vertices are labeled in turn by $\alpha $, $\beta $, $\gamma $, $\delta $, and $\varepsilon $ and take the path $(\delta ,\alpha ,\beta ,\gamma ,\varepsilon )$ which passes through $\delta $, $\alpha $, $\beta $, $\gamma $, and $\varepsilon $ sequentially as on the right-hand side in Figure~\ref{fig_pentagon-rel_def}. 
Note that two vertices in the pentagon are jointed by an edge if and only if these labeling two simple closed curves are not disjoint. 
The path $(\delta ,\alpha ,\beta ,\gamma ,\varepsilon )$ or its inverse is corresponding to the pentagonal relation $t_\varepsilon ^{-1}t_\gamma t_\beta t_\alpha t_\delta ^{-1}=t_\delta ^{-1}t_\alpha t_\beta t_\gamma t_\varepsilon ^{-1}$. 
By the $\frac{2\pi }{5}$-rotation of the path $(\delta ,\alpha ,\beta ,\gamma ,\varepsilon )$, we obtain the path $(\gamma ,\varepsilon ,\alpha ,\beta ,\delta )$ in the pentagon and the corresponding pentagonal relation 
\[
t_\delta ^{-1}t_\beta t_\alpha t_\varepsilon t_\gamma ^{-1}=t_\gamma ^{-1}t_\varepsilon t_\alpha t_\beta t_\delta ^{-1}.
\] 
We can check that this pentagonal relation, corresponding to $(\gamma ,\varepsilon ,\alpha ,\beta ,\delta )$, is equivalent to the pentagonal relation corresponding to $(\delta ,\alpha ,\beta ,\gamma ,\varepsilon )$ up to commutative relations. 

\begin{figure}[h]
\includegraphics[scale=0.65]{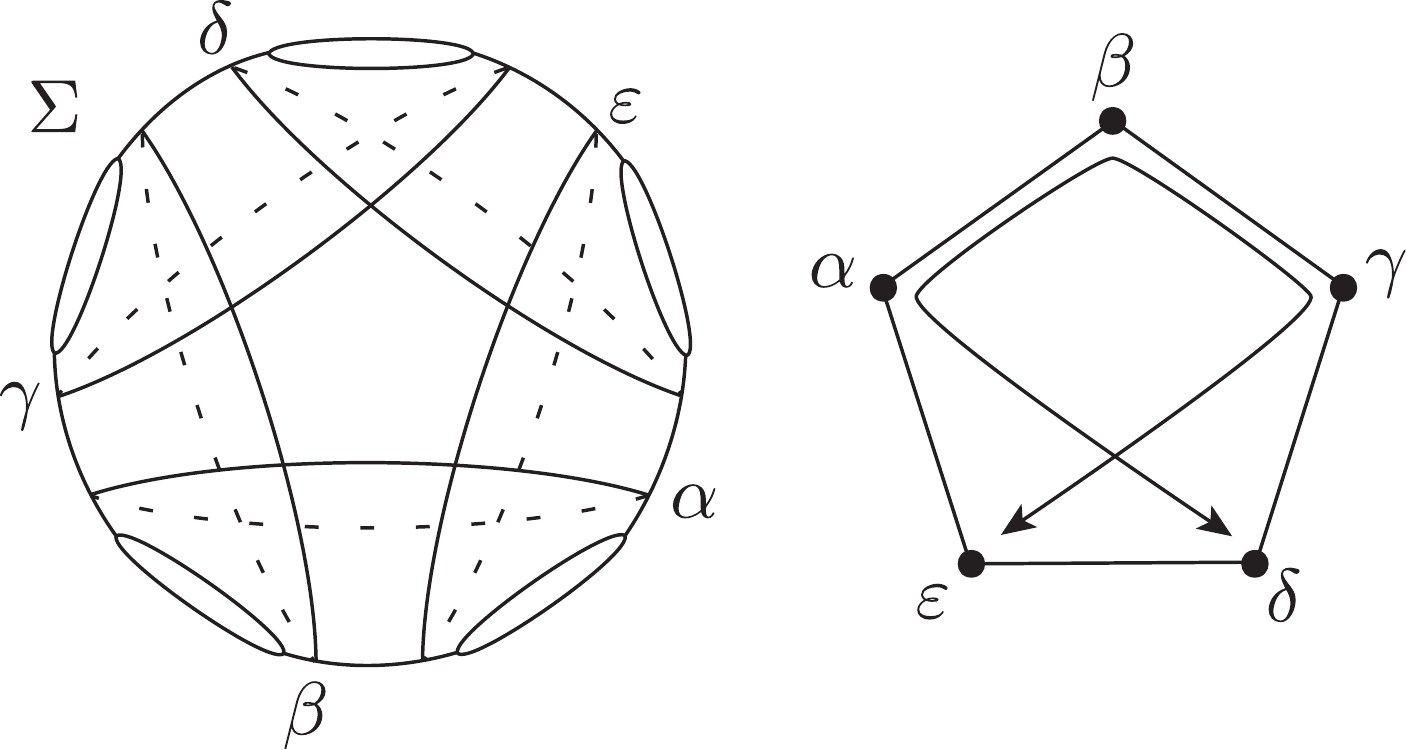}
\caption{Simple closed curves $\alpha $, $\beta $, $\gamma $, $\delta $, and $\varepsilon $ on $\Sigma $ and a pentagon whose vertices are labeled by these curves.}\label{fig_pentagon-rel_def}
\end{figure}

Throughout this paragraph and Section~\ref{section_pmod}, we consider the pure mapping class group $\PMn $ of 2-sphere with $n$ marked point not only for even $n$, but also for any positive integer $n$. 
In odd $n$ cases, we can similarly define the $n$ marked points set $\B =\{ p_1, p_2, \dots , p_{n}\}$ in $\Sigma _0$, the simple closed curve $\gamma _{i,j}$ on $\Sigma _0-\B $ for $1\leq i<j\leq n$ as in Figure~\ref{fig_path_l}, and the Dehn twist $t_{i,j}$ along $\gamma _{i,j}$. 
Recall that $t_{i,j}$ lies in $\PMn $ for $1\leq i<j\leq n$. 
For $1\leq i<j<k<l<m\leq n$, simple closed curves $\gamma _{i,k-1}$, $\gamma _{j,l-1}$, $\gamma _{k,m-1}$, $\gamma _{i,l-1}$, and $\gamma _{j,m-1}$ are lie on $\Sigma _0$ as in Figure~\ref{fig_pentagon-rel-p_ijklm}. 
We call the pentagonal relation corresponding to the path $(\gamma _{i,l-1}, \gamma _{i,k-1}, \gamma _{j,l-1}, \gamma _{k,m-1}, \gamma _{j,m-1})$, i.e. 
\[
t_{j,m-1}^{-1}t_{k,m-1}t_{j,l-1}t_{i,k-1}t_{i,l-1}^{-1}=t_{i,l-1}^{-1}t_{i,k-1}t_{j,l-1}t_{k,m-1}t_{j,m-1}^{-1},
\]
the \textit{pentagonal relation~$(P_{i,j,k,l,m})$} or the \textit{relation~$(P_{i,j,k,l,m})$}.

\begin{figure}[h]
\includegraphics[scale=0.87]{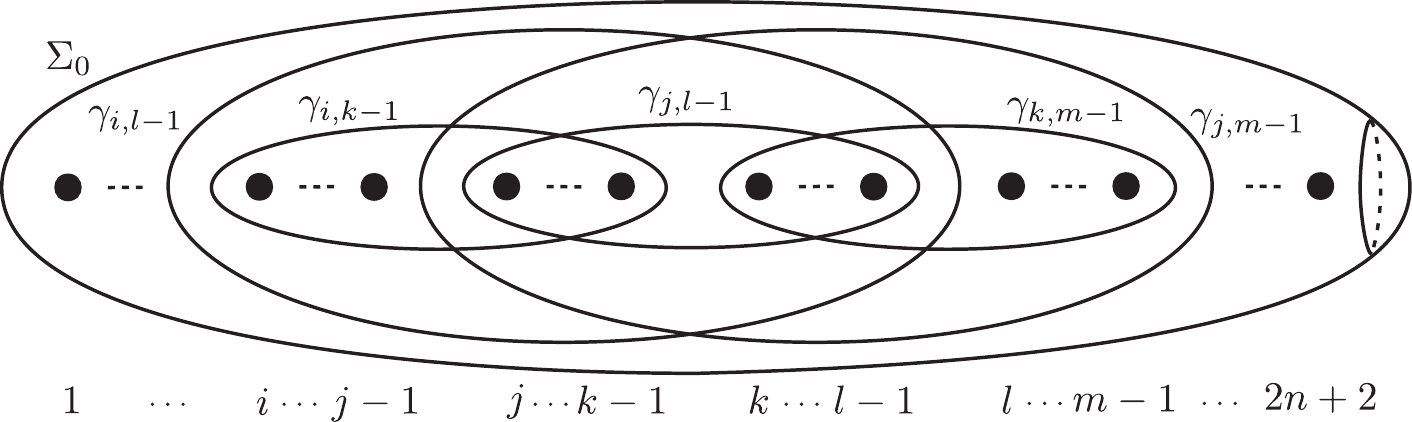}
\caption{Simple closed curves $\gamma _{i,k-1}$, $\gamma _{j,l-1}$, $\gamma _{k,m-1}$, $\gamma _{i,l-1}$, and $\gamma _{j,m-1}$ on $\Sigma _0$.}\label{fig_pentagon-rel-p_ijklm}
\end{figure}

We review the lantern relation in the next lemma to prove Lemma~\ref{lem_pentagon_rel}. 

\begin{lem}[Lantern relation]\label{lem_lantern_rel}
Let $\Sigma $ be a 2-sphere with four boundary components, $\alpha $, $\beta $, and $\gamma $ simple closed curves on $\Sigma $ each one of which separates $\Sigma $ into two 2-spheres with three boundary components as in Figure~\ref{fig_lantern-rel_def}, and $\delta _i$ $(i=1,2,3,4)$ the boundary curve of $\Sigma $. 
Then we have the following relation in $\mathrm{Mod}(\Sigma ,\emptyset ,\partial \Sigma )$:
\[
t_\alpha t_\beta t_\gamma =t_{\delta _1}t_{\delta _2}t_{\delta _3}t_{\delta _4}.
\]
\end{lem}

\begin{figure}[h]
\includegraphics[scale=0.8]{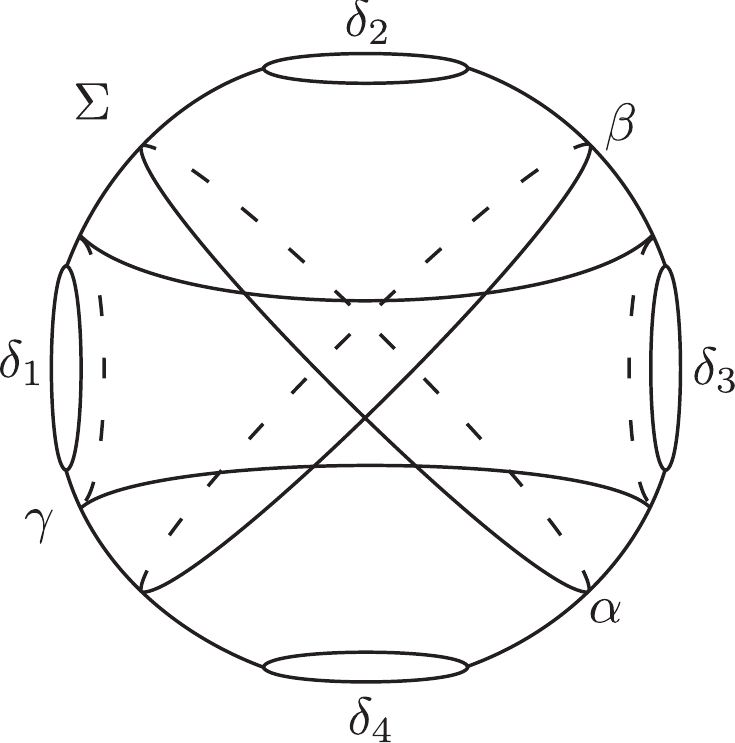}
\caption{Simple closed curves $\alpha $, $\beta $, $\gamma $, and $\delta _i$ $(i=1,2,3,4)$ on $\Sigma $.}\label{fig_lantern-rel_def}
\end{figure}

\begin{proof}[Proof of Lemma~\ref{lem_pentagon_rel}]
Let $\alpha $, $\beta $, $\gamma $, $\delta $, and $\varepsilon $ be simple closed curves on $\Sigma $ as in Figure~\ref{fig_pentagon-rel_def} and $\Sigma ^\prime $ the subsurface of $\Sigma $ which is cut off by $\alpha $ and homeomorphic to a 2-sphere with four boundary components. 
Denote by $\delta _{1}$, $\delta _{2}$, and $\delta _{3}$ the boundary curves of $\Sigma ^\prime $, each one of which is not $\alpha $, and by $\gamma _1$ and $\gamma _2$ the simple closed curves on $\Sigma ^\prime $ as in Figure~\ref{fig_pentagon-rel_proof}. 
By Lemma~\ref{lem_lantern_rel}, we have the two lantern relations
\[
t_{\delta }t_\gamma t_{\gamma _1}=t_\alpha t_{\delta _1}t_{\delta _2}t_{\delta _3} \quad \text{and}\quad t_\gamma t_{\delta }t_{\gamma _2}=t_\alpha t_{\delta _1}t_{\delta _2}t_{\delta _3}.
\] 
These relations are equivalent to the relations 
\begin{eqnarray}
t_\gamma ^{-1}t_{\delta }^{-1}t_\alpha &=&t_{\delta _1}^{-1}t_{\delta _2}^{-1}t_{\delta _3}^{-1}t_{\gamma _1} \quad \text{and}\label{rel1_pentagon_proof}\\
t_{\delta }^{-1}t_\gamma ^{-1}t_\alpha &=&t_{\delta _1}^{-1}t_{\delta _2}^{-1}t_{\delta _3}^{-1}t_{\gamma _2}. \label{rel2_pentagon_proof}
\end{eqnarray}
Since $t_\beta t_\varepsilon ^{-1}t_{\gamma _2}t_\varepsilon t_\beta ^{-1}=t_{t_\beta t_\varepsilon ^{-1}(\gamma _2)}=t_{\gamma _1}$ and $t_\beta t_\varepsilon ^{-1}$ commutes with $t_{\delta _i}$ $(i=1,2,3)$, by the conjugation of the relation~(\ref{rel2_pentagon_proof}) by $t_\beta t_\varepsilon ^{-1}$, we have the relation
\begin{eqnarray}
t_\beta t_\varepsilon ^{-1}t_{\delta }^{-1}t_\gamma ^{-1}t_\alpha t_\varepsilon t_\beta ^{-1}=t_{\delta _1}^{-1}t_{\delta _2}^{-1}t_{\delta _3}^{-1}t_{\gamma _1}.\label{rel3_pentagon_proof}
\end{eqnarray}
Thus, by the relation~(\ref{rel1_pentagon_proof}) and~(\ref{rel3_pentagon_proof}), we have 
\begin{eqnarray*}
&&t_\beta t_\varepsilon ^{-1}t_{\delta }^{-1}\underset{\rightarrow }{\underline{t_\gamma ^{-1}}}t_\alpha t_\varepsilon t_\beta ^{-1}=t_\gamma ^{-1}t_{\delta }^{-1}t_\alpha \\
&\stackrel{\text{COMM}}{\Longleftrightarrow}&t_\beta t_\varepsilon ^{-1}t_{\delta }^{-1}t_\alpha \underline{t_\varepsilon t_\gamma ^{-1}t_\beta ^{-1}}=\underline{t_\gamma ^{-1}}t_{\delta }^{-1}t_\alpha \\
&\stackrel{\text{CONJ}}{\Longleftrightarrow}&t_\gamma t_\beta \underset{\leftarrow }{\underline{t_\varepsilon ^{-1}}}\ \underset{\rightarrow }{\underline{t_{\delta }^{-1}}}t_\alpha =t_{\delta }^{-1}t_\alpha t_\beta t_\gamma t_\varepsilon ^{-1}\\
&\stackrel{\text{COMM}}{\Longleftrightarrow}&t_\varepsilon ^{-1}t_\gamma t_\beta t_\alpha t_{\delta }^{-1}=t_{\delta }^{-1}t_\alpha t_\beta t_\gamma t_\varepsilon ^{-1},
\end{eqnarray*}
where ``COMM'' and ``CONJ'' mean deformations of relations by commutative relations and conjugations of relations, respectively, ``$A_1\underset{\rightarrow }{\underline{A}}A_2$'' (resp. ``$A_1\underset{\leftarrow }{\underline{A}}A_2$'') means that deforming the word $A_1AA_2$ by moving $A$ right (resp. left), and denote $A^\prime \underline{A}=B\stackrel{\text{CONJ}}{\Longleftrightarrow }A^\prime =BA^{-1}$ and $\underline{A}A^\prime =B\stackrel{\text{CONJ}}{\Longleftrightarrow }A^\prime =A^{-1}B$. 
We have completed the proof of Lemma~\ref{lem_pentagon_rel}. 
\end{proof}

\begin{figure}[h]
\includegraphics[scale=0.8]{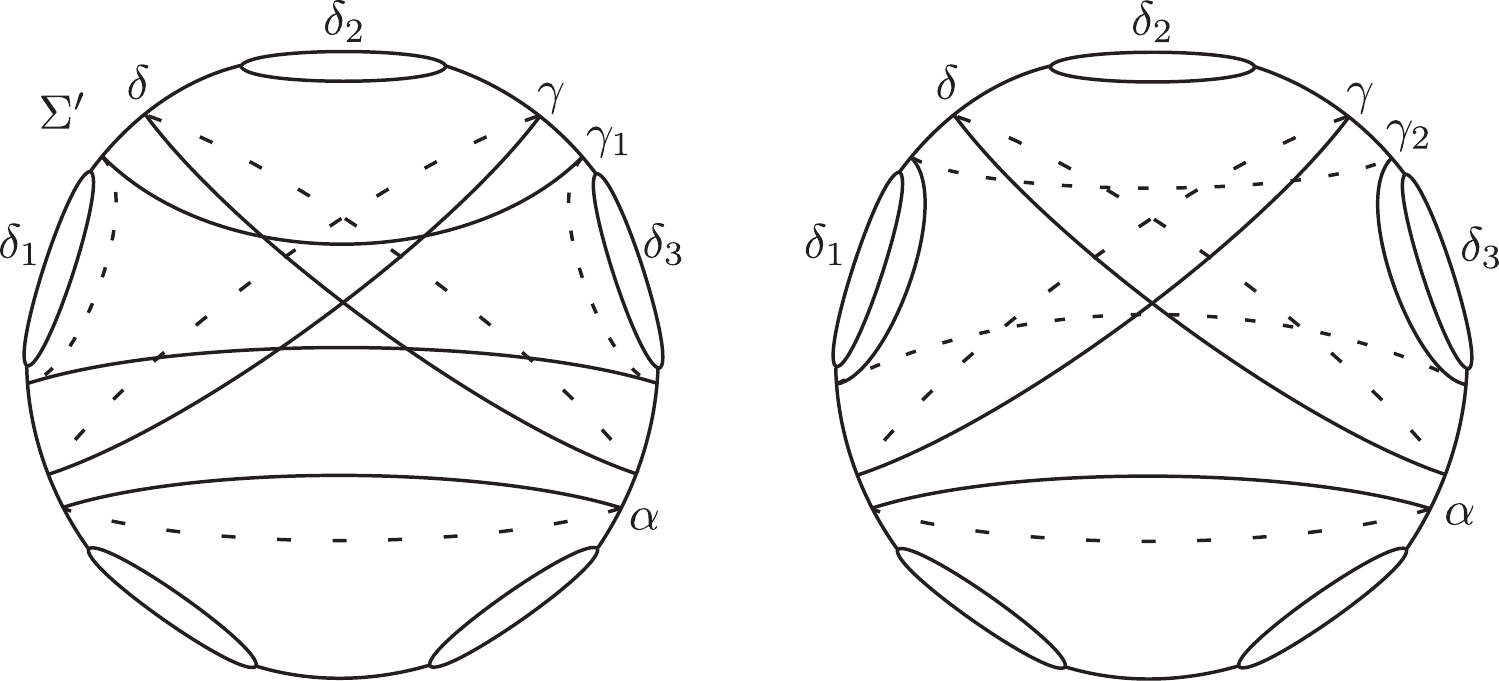}
\caption{Simple closed curves $\gamma _1$, $\gamma _2$, and $\delta _i$ $(i=1,2,3)$ on $\Sigma ^\prime $.}\label{fig_pentagon-rel_proof}
\end{figure}

\section{A finite presentation for the pure mapping class group of a 2-sphere}\label{section_pmod}

In this section, we consider the pure mapping class group $\PMn $ of a 2-sphere with $n$ marked point not only for even $n$, but also for any positive integer $n$. 
Recall that $\B =\{ p_1, p_2, \dots , p_{n}\}$ is the $n$ marked points set in $\Sigma _0$ and $t_{i,j}$ is the Dehn twist along the simple closed curve $\gamma _{i,j}$ on $\Sigma _0-\B $ for $1\leq i<j\leq n$ as in Figure~\ref{fig_path_l}. 
We will prove the following proposition in this section. 

\begin{prop}\label{prop_pres_pmod}
The group $\PMn $ admits the presentation with generators $t_{i,j}$ for $1\leq i<j\leq n-1$ with $(i,j)\not= (1,n-1)$, and the following defining relations: 
\begin{enumerate}
\item commutative relations $t_{i,j} \rightleftarrows t_{k,l}$ \quad for $j<k$, $k\leq i<j\leq l$, or $l<i$, 
\item pentagonal relations $(P_{i,j,k,l,m})$ \quad for $1\leq i<j<k<l<m\leq n$,\\ 
i.e. $t_{j,m-1}^{-1}t_{k,m-1}t_{j,l-1}t_{i,k-1}t_{i,l-1}^{-1}=t_{i,l-1}^{-1}t_{i,k-1}t_{j,l-1}t_{k,m-1}t_{j,m-1}^{-1}$.
\end{enumerate}
\end{prop}

The relation~(1) and (2) in Proposition~\ref{prop_pres_pmod} are reviewed in Lemma~\ref{lem_comm_rel} and after Lemma~\ref{lem_pentagon_rel}, respectively. 

To prove Proposition~\ref{prop_pres_pmod} and theorems in Sections~\ref{section_lmod} and \ref{section_smod}, we review a relationship between a group extension and group presentations as follows. 
Let $G$ be a group and let $H=\bigl< X\mid R\bigr>$ and $Q=\bigl< Y\mid S\bigr>$ be presented groups which have the short exact sequence 
\[
1\longrightarrow H\stackrel{\iota }{\longrightarrow }G\stackrel{\nu }{\longrightarrow }Q\longrightarrow 1.
\]
We take a preimage $\widetilde{y}\in G$ of $y\in Q$ with respect to $\nu $ for each $y\in Q$. 
Then we put $\widetilde{X}=\{ \iota (x) \mid x\in X\} \subset G$ and $\widetilde{Y}=\{ \widetilde{y} \mid y\in Y\} \subset G$. 
Denote by $\widetilde{r}$ the word in $\widetilde{X}$ which is obtained from $r\in R$ by replacing each $x\in X$ by $\iota (x)$ and also denote by $\widetilde{s}$ the word in $\widetilde{Y}$ which is obtained from $s\in S$ by replacing each $y\in Y$ by $\widetilde{y}$. 
We note that $\widetilde{r}=1$ in $G$. 
Since $\widetilde{s}\in G$ is an element in $\ker \nu =\iota (H)$ for each $s\in S$, there exists a word $v_{s}$ in $\widetilde{X}$ such that $\widetilde{s}=v_{s}$ in $G$. 
Since $\iota (H)$ is a normal subgroup of $G$, for each $x\in X$ and $y\in Y$, there exists a word $w_{x,y}$ in $\widetilde{X}$ such that $\widetilde{y}\iota (x)\widetilde{y}^{-1}=w_{x,y}$ in $G$. 
The next lemma follows from an argument of the combinatorial group theory (for instance, see \cite[Proposition~10.2.1, p139]{Johnson}).

\begin{lem}\label{presentation_exact}
Under the situation above, the group $G$ admits the presentation with the generating set $\widetilde{X}\cup \widetilde{Y}$ and following defining relations:
\begin{enumerate}
 \item[(A)] $\widetilde{r}=1$ \quad for $r\in R$,
 \item[(B)] $\widetilde{s}=v_{s}$ \quad for $s\in S$,
 \item[(C)] $\widetilde{y}\iota (x)\widetilde{y}^{-1}=w_{x,y}$ \quad for $x\in X$ and $y\in Y$.
\end{enumerate} 
\end{lem}

\begin{proof}[Proof of Proposition~\ref{prop_pres_pmod}]
We proceed by induction on $n=|\B |\geq 1$. 
Put $\B ^1=\{ p_1, p_2, \dots ,p_{n-1}\}$ and $\Sigma _{0,n-1}=\Sigma _0-\B ^1$. 
Then we regard $\mathrm{PMod}_{0,n-1}$ as the group $\mathrm{Mod}(\Sigma _0, \emptyset , \B ^1 )$. 
Let $\mathcal{F}\colon \PMn \to \mathrm{PMod}_{0,n-1}$ be the forgetful homomorphism which is defined by the correspondence $[\varphi ]\mapsto [\varphi ]$. 
Birman~\cite{Birman} constructed the following exact sequence (see also Section~4.2 in~\cite{Farb-Margalit}): 
\begin{eqnarray}\label{birman-exact1}
\pi _1(\Sigma _{0,n-1},p_n ) \stackrel{\Delta }{\longrightarrow }\PMn \stackrel{\mathcal{F}}{\longrightarrow }\mathrm{PMod}_{0,n-1}\longrightarrow 1, 
\end{eqnarray}
where for $\gamma \in \pi _1(\Sigma _{0,n-1},p_n )$, $\Delta (\gamma )$ is a self-homeomorphism on $\Sigma _0$ which is the result of pushing $p_{n}$ along $\gamma $ once and for $\gamma _1$, $\gamma _2\in \pi _1(\Sigma _{0,n-1},p_n )$, the product $\gamma _1\gamma _2$ means $\gamma _1\gamma _2(t)=\gamma _2(2t)$ for $0\leq t\leq \frac{1}{2}$ and $\gamma _1\gamma _2(t)=\gamma _1(2t-1)$ for $\frac{1}{2}\leq t\leq 1$. 
We remark that $\pi _1(\Sigma _{0,n-1},p_n )$ is trivial for $n\in \{ 1, 2\} $ and $\mathrm{PMod}_{0,0}=\mathrm{Mod}(\Sigma _0, \emptyset , \emptyset )$ is also trivial. 
Hence the groups $\mathrm{PMod}_{0,1}$ and $\mathrm{PMod}_{0,2}$ are trivial. 
For $n\geq 3$, $\pi _1(\Sigma _{0,n-1},p_n )$ is generated by loops $\omega _1$, $\omega _2,\ \dots $, $\omega _{n-2}$ on $\Sigma _0-\B ^1 $ based at $p_{n}$ as in Figure~\ref{fig_gen_pi_1}.
In the case of $n=3$, $\pi _1(\Sigma _{0,2},p_3 )$ is the infinite cyclic group and the image of the homomorphism $\Delta \colon \pi _1(\Sigma _{0,2},p_3 ) \to \mathrm{PMod}_{0,3}$ is trivial. 
By the exact sequence~(\ref{birman-exact1}), we show that $\mathrm{PMod}_{0,3}$ is trivial. 
In the cases $n\in \{ 1,2,3\}$, the Dehn twist along any simple closed curve on $\Sigma _0-\B $ is trivial in $\PMn $. 
Thus, Proposition~\ref{prop_pres_pmod} is true for the cases $n\in \{ 1,2,3\}$. 
In the case of $n\geq 4$, since $\pi _1(\Sigma _{0,n-1},p_n )$ is isomorphic to the free group of rank $n-2$, the center of $\pi _1(\Sigma _{0,n-1},p_n )$ is trivial for $n\geq 4$. 
Thus, the homomorphism $\Delta $ is injective by Corollary~1.2 in~\cite{Birman}, and we will apply Lemma~\ref{presentation_exact} to the exact sequence~(\ref{birman-exact1}). 

The base case of this induction is the case $n=4$. 
We remark that the presentation for $\mathrm{PMod}_{0,4}$ in Proposition~\ref{prop_pres_pmod} has generators $t_{1,2}$ and $t_{2,3}$, and no relations. 
Since $\mathrm{PMod}_{0,3}$ is trivial and $\pi _1(\Sigma _{0,n-1},p_n )$ is isomorphic to the free group of rank $2$ and generated by $\omega _1$ and $\omega _2$, $\mathrm{PMod}_{0,4}$ is isomorphic to the free group of rank $2$ and generated by $\Delta (\omega _1)$ and $\Delta (\omega _2)$ by the exact sequence~(\ref{birman-exact1}). 
We can show that $\Delta (\omega _1)=t_{2,3}^{-1}$ and $\Delta (\omega _2)=t_{1,2}$, and thus, Proposition~\ref{prop_pres_pmod} is true for the case $n=4$.

We assume $n\geq 5$. 
By the inductive hypothesis, $\mathrm{PMod}_{0,n-1}$ admits the presentation with generators $t_{i,j}$ for $1\leq i<j\leq n-2$ with $(i,j)\not= (1,n-2)$, and the following defining relations: 
\begin{enumerate}
\item Commutative relations $t_{i,j} \rightleftarrows t_{k,l}$ \quad for $j<k$, $k\leq i<j\leq l$, or $l<i$, 
\item Pentagonal relations $(P_{i,j,k,l,m})$ \quad for $1\leq i<j<k<l<m\leq n-1$,\\ 
i.e. $t_{j,m-1}^{-1}t_{k,m-1}t_{j,l-1}t_{i,k-1}t_{i,l-1}^{-1}=t_{i,l-1}^{-1}t_{i,k-1}t_{j,l-1}t_{k,m-1}t_{j,m-1}^{-1}$.
\end{enumerate}
We naturally regard the generator $t_{i,j}$ of this presentation for $\mathrm{PMod}_{0,n-1}$ as an element of $\PMn $ and denote by $X_{n-1}$ the generating set of the presentation above for $\mathrm{PMod}_{0,n-1}$. 
Recall that $\pi _1(\Sigma _{0,n-1},p_n )$ is isomorphic to the free group of rank $n-2$ and generated by $\omega _1$, $\omega _2,\ \dots $, $\omega _{n-2}$. 
By applying Lemma~\ref{presentation_exact} to the exact sequence~(\ref{birman-exact1}) and the finite presentations for $\mathrm{PMod}_{0,n-1}$ and $\pi _1(\Sigma _{0,n-1},p_n )$, we have the finite presentation for $\PMn $ whose generating set is $X_{n-1}\cup \{ \Delta (\omega _1), \Delta (\omega _2), \dots , \Delta (\omega _{n-2})\}$ and the defining relations are as follows: 
\begin{enumerate}
\item[(A)] no relations,
\item[(B)]
\begin{enumerate}
\item[(1)] commutative relations $t_{i,j} \rightleftarrows t_{k,l}$ \quad for $t_{i,j}, t_{k,l}\in X_{n-1}$ and $j<k$, $k\leq i<j\leq l$, or $l<i$, 
\item[(2)] pentagonal relations $(P_{i,j,k,l,m})$ \quad for $1\leq i<j<k<l<m\leq n-1$,
\end{enumerate}
\item[(C)] $t_{i,j}\Delta (\omega _k)t_{i,j}^{-1}=w_{k;i,j}$ \quad for $1\leq k\leq n-2$, $1\leq i<j\leq n-2$ and $(i,j)\not= (1,n-2)$,
\end{enumerate}
where $w_{k;i,j}$ is some product of $\Delta (\omega _1),\ \Delta (\omega _2),\ \dots ,\ \Delta (\omega _{n-2})$. 
We can show that 
\[
\Delta (\omega _i)=t_{1,i}t_{i+1,n-1}^{-1}\quad  \text{for } 1\leq i\leq n-3  \quad \text{and}\quad \Delta (\omega _{n-2})=t_{1,n-2}.
\] 
By substituting $\Delta (\omega _i)$ in generators and relations of the presentation for $\PMn $ above for the corresponding product $t_{1,i}t_{i+1,n-1}^{-1}$ or $t_{1,n-2}$ and using Tietze transformations, we obtain a presentation for $\PMn $ whose generators are $t_{i,j}$ for $1\leq i<j\leq n-1$ with $(i,j)\not= (1,n-1)$. 

It is enough for completing the proof of Proposition~\ref{prop_pres_pmod} to prove that the relation~(C) in the presentation for $\PMn $ above is obtained from commutative relations and the pentagonal relations $(P_{i,j,k,l,m})$ for $1\leq i<j<k<l<m\leq n$. 
Since $t_{i,j} \Delta (\omega _k)t_{i,j}^{-1}=\Delta (t_{i,j}(\omega _k))$, we will express $t_{i,j}(\omega _k)$ by a product of $\omega _1$, $\omega _2,\ \dots $, $\omega _{n-2}$ in $\pi _1(\Sigma _{0,n-1},p_n )$ for $1\leq k\leq n-2$, $1\leq i<j\leq n-2$ and $(i,j)\not= (1,n-2)$. 
In the cases $j\leq k$ or $k\leq i-1$, since the simple closed curve $\gamma _{i,j}$ is disjoint from $\omega _k$, we have $t_{i,j}(\omega _k)=\omega _k$. 
Thus, in this case, the relation~(C) coincides with the relation $t_{i,j}\Delta (\omega _k)t_{i,j}^{-1}=\Delta (\omega _k)$ and is obtained from commutative relations. 

In the case $i\leq k\leq j-1$,  $\gamma _{i,j}$ and $\omega _k$ transversely intersect at two points and $t_{i,j}(\omega _k)$ is a loop based at $p_n$ as in Figure~\ref{fig_action_pi_1-basis}. 
As a product of  $\omega _1$, $\omega _2,\ \dots $, $\omega _{n-2}$ in $\pi _1(\Sigma _{0,n-1},p_n )$, we have $t_{i,j}(\omega _k)=\omega _{j}\omega _{i-1}^{-1}\omega _{k}\omega _{j}^{-1}\omega _{i-1}$. 
Since the relations $t_{i,j} \Delta (\omega _k)t_{i,j}^{-1}=\Delta (\omega _{j})\Delta (\omega _{i-1})^{-1}\Delta (\omega _{k})\Delta (\omega _{j})^{-1}\Delta (\omega _{i-1})$, $\Delta (\omega _{i-1})=t_{1,i-1}t_{i,n-1}^{-1}$, $\Delta (\omega _k)=t_{1,k}t_{k+1,n-1}^{-1}$, and 
\[
\Delta (\omega _{j})=\left\{
		\begin{array}{ll}
		t_{1,j}t_{j+1,n-1}^{-1} &\text{for }2\leq j\leq n-3,\\
		t_{1,n-2} &\text{for }j=n-2,		\end{array}
		\right.\\
\]
hold in $\PMn $, we have the relation
\begin{eqnarray*}
&&t_{i,j}t_{1,k}t_{k+1,n-1}^{-1}t_{i,j}^{-1}\\
&=&\left\{
		\begin{array}{ll}
		t_{1,j}t_{j+1,n-1}^{-1}\cdot t_{i,n-1}\underset{\rightarrow }{\underline{t_{1,i-1}^{-1}}}\cdot t_{1,k}t_{k+1,n-1}^{-1}\cdot t_{j+1,n-1}t_{1,j}^{-1}\cdot t_{1,i-1}t_{i,n-1}^{-1} \\ \text{for }2\leq j\leq n-3,\\
		t_{1,n-2}\cdot t_{i,n-1}\underset{\rightarrow }{\underline{t_{1,i-1}^{-1}}}\cdot t_{1,k}t_{k+1,n-1}^{-1}\cdot t_{1,n-2}^{-1}\cdot t_{1,i-1}t_{i,n-1}^{-1} \quad \text{for }j=n-2\\
\end{array}
		\right.\\
&\stackrel{\text{COMM}}{=}&\left\{
		\begin{array}{ll}
		t_{1,j}\underset{\rightarrow }{\underline{t_{j+1,n-1}^{-1}}}t_{i,n-1}t_{1,k}t_{k+1,n-1}^{-1}t_{j+1,n-1}t_{1,j}^{-1}t_{i,n-1}^{-1} \quad \text{for }2\leq j\leq n-3,\vspace{0.1cm}\\
		t_{1,n-2}t_{i,n-1}t_{1,k}t_{k+1,n-1}^{-1}t_{1,n-2}^{-1}t_{i,n-1}^{-1} \quad \text{for }j=n-2\\
\end{array}
		\right.\\
&\stackrel{\text{COMM}}{=}&t_{1,j}t_{i,n-1}t_{1,k}t_{k+1,n-1}^{-1}t_{1,j}^{-1}t_{i,n-1}^{-1}.
\end{eqnarray*}
Thus, we have
\begin{eqnarray*}
&&t_{i,j}t_{1,k}t_{k+1,n-1}^{-1}t_{i,j}^{-1}=t_{1,j}t_{i,n-1}\underline{t_{1,k}t_{k+1,n-1}^{-1}}t_{1,j}^{-1}t_{i,n-1}^{-1}\\
&\stackrel{\text{COMM}}{\Longleftrightarrow }&t_{i,j}t_{1,k}\underline{t_{k+1,n-1}^{-1}t_{i,j}^{-1}}=\underline{t_{1,j}t_{i,n-1}t_{k+1,n-1}^{-1}}t_{1,k}t_{1,j}^{-1}t_{i,n-1}^{-1}\\
&\stackrel{\text{CONJ}}{\Longleftrightarrow }&t_{k+1,n-1}t_{i,n-1}^{-1}\underset{\rightarrow }{\underline{t_{1,j}^{-1}}}t_{i,j}t_{1,k}=\underline{t_{1,k}t_{1,j}^{-1}}t_{i,n-1}^{-1}t_{i,j}t_{k+1,n-1}\\
&\stackrel{\text{COMM}}{\Longleftrightarrow }&\underline{t_{k+1,n-1}t_{i,n-1}^{-1}}t_{i,j}t_{1,k}t_{1,j}^{-1}=t_{1,j}^{-1}t_{1,k}\underset{\rightarrow }{\underline{t_{i,n-1}^{-1}}}t_{i,j}t_{k+1,n-1}\\
&\stackrel{\text{COMM}}{\Longleftrightarrow }&t_{i,n-1}^{-1}t_{k+1,n-1}t_{i,j}t_{1,k}t_{1,j}^{-1}=t_{1,j}^{-1}t_{1,k}t_{i,j}t_{k+1,n-1}t_{i,n-1}^{-1}.
\end{eqnarray*}
Therefore the relation~(C) for $1\leq i\leq k\leq j-1\leq n-3$ is equivalent to the pentagonal relation~$(P_{1,i,k+1,j+1,n})$ up to commutative relations and we have completed the proof of Proposition~\ref{prop_pres_pmod}. 
\end{proof}

\begin{figure}[h]
\includegraphics[scale=1.3]{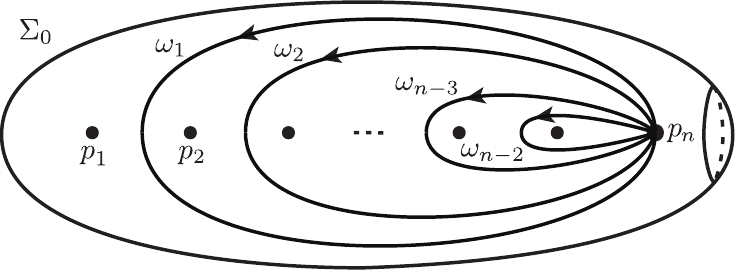}
\caption{Loops $\omega _1$, $\omega _2,\ \dots $, $\omega _{n-2}$ on $\Sigma _0-\B ^1$ based at $p_{n}$.}\label{fig_gen_pi_1}
\end{figure}

\begin{figure}[h]
\includegraphics[scale=0.9]{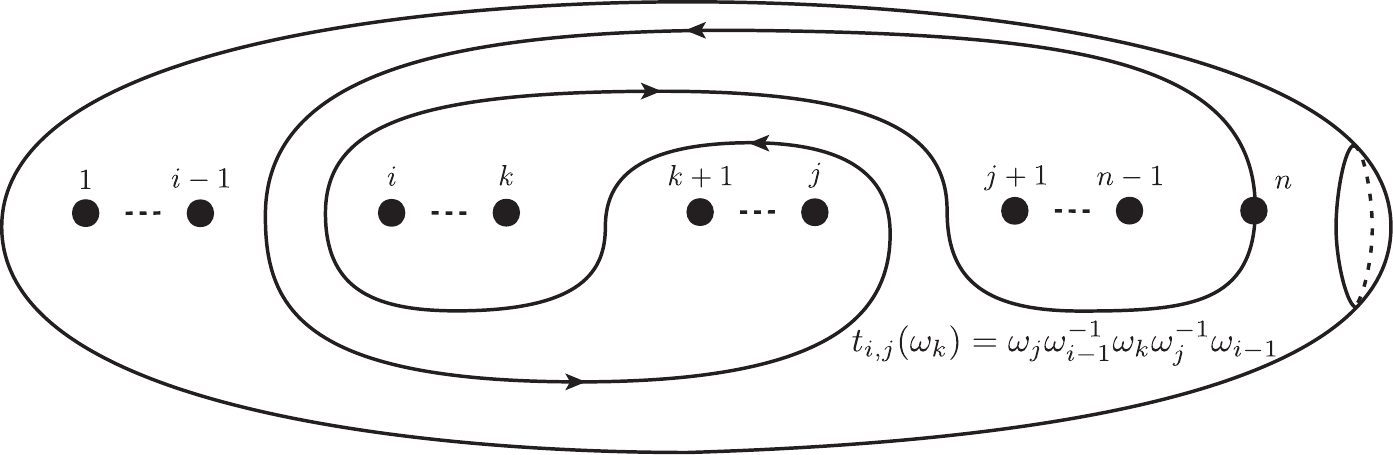}
\caption{A loop $t_{i,j}(\omega _k)=\omega _{j}\omega _{i-1}^{-1}\omega _{k}\omega _{j}^{-1}\omega _{i-1}$ for $i\leq k\leq j-1$.}\label{fig_action_pi_1-basis}
\end{figure}

\section{Group structures of liftable mapping class groups and balanced superelliptic mapping class groups via group extensions}\label{section_birman-exact-seq}

\subsection{Group structures of liftable mapping class groups with either one marked point or one boundary component}\label{section_exact-seq_lmod}

In this section, we observe the group structures of $\LMp $ and $\LMb $ via group extensions. 
Recall that, under the identification of the symmetric group $S_{2n+2}$ with the group of self-bijections on $\B =\{ p_1, p_2, \dots ,p_{2n+2}\}$, the image of the liftable mapping class group $\LM $ by the surjective homomorphism $\Psi \colon \M \to S_{2n+2}$ coincides with the group $W_{2n+2}\cong (S_{n+1}^o\times S_{n+1}^e)\rtimes \Z _2$ by Lemma~\ref{lem_GW} (see Section~\ref{section_liftable-element}). 
Let $W_{2n+2;\ast }$ be the subgroup of $W_{2n+2}$ which consists of elements preserving $p_{2n+2}$ and $S_{n}^e$ the subgroup of $W_{2n+2}$ which consists of elements whose restriction to $\B ^o\cup \{ p_{2n+2}\}$ is the identity map. 
By the definitions, all elements in $W_{2n+2;\ast }$ are parity-preserving and $S_{n}^e$ coincides with the intersection $S_{n+1}^e\cap W_{2n+2;\ast }$. 
Note that $S_{n}^e$ is isomorphic to $S_{n}$ and generated by transpositions $(2\ 4)$, $(4\ 6),\ \dots $, $(2n-2\ 2n)$. 
Then we have the following lemma. 

\begin{lem}\label{lem_W_ast}
We have $W_{2n+2;\ast }=S_{n+1}^o\times S_n^e$.
\end{lem}

\begin{proof}
Since $S_{n+1}^o\times S_n^e$ is generated by transpositions $\{ (i\ i+2)\mid 1\leq i\leq 2n-1\}$ and these transpositions are parity-preserving and fix $p_{2n+2}$, the group $S_{n+1}^o\times S_n^e$ is a subgroup of $W_{2n+2;\ast }$. 

For any $\sigma \in W_{2n+2;\ast}$, we define $\sigma _o,\ \sigma _e\in W_{2n+2;\ast }$ as follows: 
\[
\sigma _o(i)=\left\{
		\begin{array}{ll}
		\sigma (i) &\text{for }i\in \B _o,\\
		i &\text{for }i\in \B _e, 
		\end{array}
		\right.\\
\sigma _e(i)=\left\{
		\begin{array}{ll}
		\sigma (i) &\text{for }i\in \B _e,\\
		i &\text{for }i\in \B _o. 
		\end{array}
		\right.\\
\]
Since $\sigma $ preserves $p_{2n+2}$, by the definitions, we have $\sigma _o\in S_{n+1}^o$, $\sigma _e\in S_n^e$ and $\sigma =\sigma _o\sigma _e\in S_{n+1}^o\times S_n^e$. 
We have completed the proof of Lemma~\ref{lem_W_ast}. 
\end{proof}

Recall that we regard $\LMp $ as the subgroup of $\LM $ which consists of elements fixing the point $p_{2n+2}$. 
The next lemma determines the image of $\LMp $ with respect to $\Psi $. 

\begin{lem}\label{image_psi_lmodp}
We have $\Psi (\LMp )=W_{2n+2;\ast }$.
\end{lem}

\begin{proof}
We take a mapping class $f\in \LMp $. 
By Lemma~\ref{lem_GW}, the image $\Psi (f)$ lies in $W_{2n+2}$. 
Since $f(p_{2n+2})=p_{2n+2}$, $\Psi (f)$ is parity-preserving and fixes $p_{2n+2}$. 
Thus we have $\Psi (f)\in W_{2n+2;\ast }$ and $\Psi (\LMp )$ is a subgroup of $W_{2n+2;\ast }$. 

By Lemma~\ref{lem_W_ast}, $W_{2n+2;\ast }$ is generated by $(i\ i+2)$ for $1\leq i\leq 2n-1$. 
Recall that $h_i\in \LM$ for $1\leq i\leq 2n$ is the half-rotation as in Figure~\ref{fig_h_i} and for $1\leq i\leq 2n-1$, $h_i$ lies in $\LMp $. 
Since $\Psi (h_i)=(i\ i+2)$ for $1\leq i\leq 2n$, the restriction of $\Psi $ to $\LMp $ is a surjection on to $W_{2n+2;\ast }$. 
We have completed the proof of Lemma~\ref{image_psi_lmodp}. 
\end{proof}

Recall that all elements in the pure mapping class group $\PM $ is liftable with respect to the balanced superelliptic covering map $p=p_{g,k}$ and fixes $p_{2n+2}$. 
Thus $\PM $ is a subgroup of $\LMp $. 
By Lemma~\ref{image_psi_lmodp} and restricting the exact sequence~(\ref{exact2}) in Section~\ref{section_liftable-element} to $\LMp $, we have the following proposition. 

\begin{prop}\label{prop_exact_lmodp}
We have the following exact sequence: 
\begin{eqnarray}\label{exact_lmodp}
1\longrightarrow \PM \longrightarrow \LMp \stackrel{\Psi }{\longrightarrow }W_{2n+2;\ast }\longrightarrow 1. 
\end{eqnarray}
\end{prop}

The inclusion $\iota \colon \Sigma _0^1\hookrightarrow \Sigma _0=\Sigma _0^1\cup D$ induces the surjective homomorphism $\iota _\ast \colon \Mb \to \Mp $ whose kernel is the infinite cyclic group generated by the Dehn twist $t_{\partial D}$ along $\partial D$. 
The homomorphism $\iota _\ast$ is called the \textit{capping homomorphism}. 
We have the following lemma. 

\begin{lem}\label{surj_capping_lmod}
We have $\iota _\ast (\LMb )=\LMp $.
\end{lem}

\begin{proof}
Assume that $g=n(k-1)$. 
For any $f\in \LMb $, there exists a representative $\varphi \in f$ and a self-homeomorphism $\widetilde{\varphi }$ on $\Sigma _g^1$ whose restriction to $\partial \Sigma _g^1$ is the identity map such that $p\circ \widetilde{\varphi }=\varphi \circ p$. 
Denote by $\varphi ^\prime $ (resp. $\widetilde{\varphi }^\prime $) the self-homeomorphism on $\Sigma _{0}$ (resp.  $\Sigma _{g}$) which is the extension of $\varphi ^\prime $ (resp. $\widetilde{\varphi }^\prime $) by the identity map on $D$ (resp. $\widetilde{D}$). 
By the definitions, we have $p\circ \widetilde{\varphi }^\prime =\varphi ^\prime \circ p$, i.e. $\varphi ^\prime $ is liftable with respect to $p\colon \Sigma _g\to \Sigma _0$, and $\varphi (p_{2n+2})=p_{2n+2}$. 
Thus, by the definition of $\iota _\ast $, we have $\iota _\ast (f)=[\varphi ^\prime ]\in \LMp $ and $\iota _\ast (\LMb )\subset \LMp $. 

For any $f\in \LMp $, we take a liftable representative $\varphi \in f$ with respect to $p\colon \Sigma _g\to \Sigma _0$. 
Since $\varphi $ fixes $p_{2n+2}\in D$, there exists an isotopy $\varphi _t\colon \Sigma _0\to \Sigma _0$ $(t\in [0,1])$ such that $\varphi _t(\B )=\B $ for $t\in [0,1]$, $\varphi _0=\varphi $, and $\varphi _1|_{D}=\mathrm{id}_{D}$. 
Since the liftability with respect to $p$ of a self-homeomorphism on $\Sigma _0$ is preserved by isotopies fixing $\B $, $\varphi _1$ is also liftable with respect to $p$. 
Hence there exists a self-homeomorphism $\widetilde{\varphi }_1$ on $\Sigma _g$ such that $p\circ \widetilde{\varphi }_1=\varphi _1\circ p$. 
Since $\varphi _1|_{D}=\mathrm{id}_{D}$, the restriction $\widetilde{\varphi }_1|_{\widetilde{D}}\colon \widetilde{D}\to \widetilde{D}$ is a covering transformation of the $k$-fold branched covering space $p|_{\widetilde{D}}\colon \widetilde{D}\to D$, namely $\widetilde{\varphi }_1|_{\widetilde{D}}=\zeta ^l|_{\widetilde{D}}$ for some $0\leq l\leq k-1$. 
We define $\varphi ^\prime =\varphi _1|_{\Sigma _0^1}\colon \Sigma _0^1\to \Sigma _0^1$ and $\widetilde{\varphi }^\prime =(\zeta ^{-l}\circ \widetilde{\varphi }_1)|_{\Sigma _g^1}\colon \Sigma _g^1\to \Sigma _g^1$. 
By the definitions, $\widetilde{\varphi }^\prime $ is the lift of $\varphi ^\prime $ with respect to $p$ whose restriction to $\partial \Sigma _g^1$ is the identity map. 
Thus, the mapping class $[\varphi ^\prime ]$ lies in $\LMb $ and we have $\iota _\ast ([\varphi ^\prime ])=[\varphi _1]=[\varphi _0=\varphi ]=f$ in $\LMp $. 
Therefore, $\iota _\ast (\LMb )\supset \LMp $ and we have completed the proof of Lemma~\ref{surj_capping_lmod}.  
\end{proof}

The next proposition gives the group extension of $\LMp $ by the capping homomorphism. 

\begin{prop}\label{prop_exact_lmodb}
We have the following exact sequence: 
\begin{eqnarray}\label{exact_lmodb}
1\longrightarrow \left< t_{\partial D}\right>  \longrightarrow \LMb \stackrel{\iota _\ast }{\longrightarrow }\LMp \longrightarrow 1. 
\end{eqnarray}
\end{prop}

\begin{proof}
Recall that we have the exact sequence
\[
1\longrightarrow \left< t_{\partial D}\right>  \longrightarrow \Mb \stackrel{\iota _\ast }{\longrightarrow }\Mp \longrightarrow 1. 
\]
By Lemma~\ref{surj_capping_lmod} and restricting the exact sequence above to $\LMb $, we have the following exact sequence: 
\[
1\longrightarrow \left< t_{\partial D}\right> \cap \LMb \longrightarrow \LMb \stackrel{\iota _\ast }{\longrightarrow }\LMp \longrightarrow 1. 
\]
Thus, it is enough for completing the proof of Proposition~\ref{prop_exact_lmodb} to prove that $t_{\partial D}$ is liftable with respect to $p$. 
Let $\zeta ^\prime $ be a self-homeomorphism on $\Sigma _g$ which is described as a result of a $(-\frac{2\pi }{k})$-rotation of $\Sigma _g^1\subset \Sigma _g$ fixing the disk $\widetilde{D}$ pointwise (for instance, in the case of $k=3$, $\zeta ^\prime $ is a $(-\frac{2\pi }{3})$-rotation as on the top in Figure~\ref{fig_lift_t_partial-d}). 
Since the restriction of $\zeta ^\prime $ to $\widetilde{D}$ is identity, we regard $\zeta ^\prime $ as an element in $\SMb $. 
Recall that $L=l_1\cup l_2\cup \cdots \cup l_{2n+1}$ as in Figure~\ref{fig_path_l} and denote by $\widetilde{L}$ the preimage of $L$ with respect to $p$. 
We remark that the isotopy class of a self-homeomorphism on $\Sigma _0$ (resp. $\Sigma _g$) relative to $D\cup \B $ (resp. $\widetilde{D}$) is determined by the isotopy class of the image of $L$ (resp. $\widetilde{L}$) relative to $D\cup \B $ (resp. $\widetilde{D}$). 
We can show that $p(\zeta ^\prime (\widetilde{L}))$ is isotopic to $t_{\partial D}(L)$ relative to $D\cup \B $ (in the case $k=3$, see Figure~\ref{fig_lift_t_partial-d}). 
Thus, $t_{\partial D}$ is liftable with respect to $p$ and we can take a lift of $t_{\partial D}$ by $\zeta ^\prime $. 
Therefore, $\left< t_{\partial D}\right> \cap \LMb =\left< t_{\partial D}\right> $ and we have completed the proof of Proposition~\ref{prop_exact_lmodb}. 
\end{proof}

\begin{figure}[h]
\includegraphics[scale=1.55]{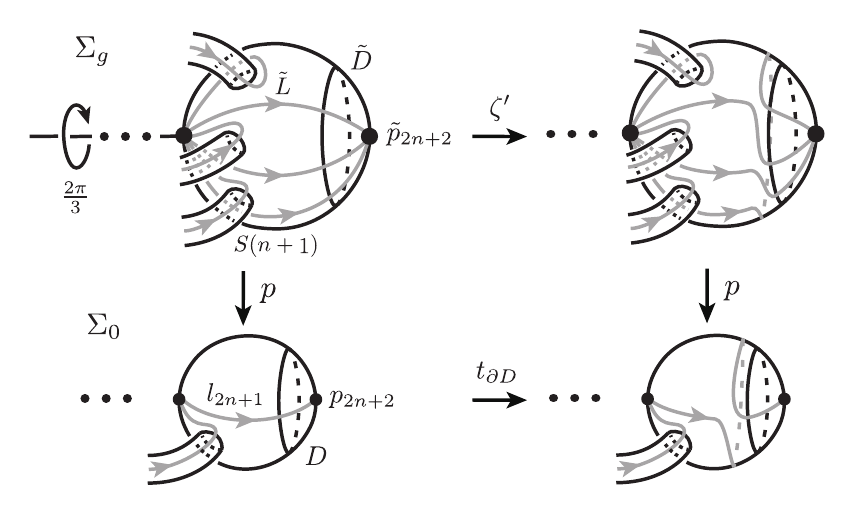}
\caption{The images of $L$ by $t_{\partial D}$ and of $\widetilde{L}$ by $\zeta ^\prime $ for $k=3$.}\label{fig_lift_t_partial-d}
\end{figure}

\subsection{The forgetful homomorphism and the capping homomorphism for the balanced superelliptic mapping class groups}\label{section_exact-seq_smod}

In this section, we will observe the forgetful homomorphism and the capping homomorphism for the balanced superelliptic mapping class groups. 
Suppose that $n\geq 1$, $k\geq 3$, and $g=n(k-1)$. 
Let $\mathcal{F}\colon \Mgp \to \Mg $ be the forgetful homomorphism which is defined by the correspondence $[\varphi ]\mapsto [\varphi ]$. 
Recall that $\widetilde{p}_{2n+2}\in \Sigma _g$ is the preimage of $p_{2n+2}\in \Sigma _0$ with respect to $p_{g,k}$ and the product $\gamma _1\gamma _2$ for $\gamma _1$, $\gamma _2\in \pi _1(\Sigma _{g}, \widetilde{p}_{2n+2})$ means $\gamma _1\gamma _2(t)=\gamma _2(2t)$ for $0\leq t\leq \frac{1}{2}$ and $\gamma _1\gamma _2(t)=\gamma _1(2t-1)$ for $\frac{1}{2}\leq t\leq 1$. 
Birman~\cite{Birman} constructed the following exact sequence (see also Section~4.2 in~\cite{Farb-Margalit}): 
\begin{eqnarray}\label{birman_exact_seq}
1\longrightarrow \pi _1(\Sigma _g, \widetilde{p}_{2n+2})  \stackrel{\Delta }{\longrightarrow }\Mgp \stackrel{\mathcal{F}}{\longrightarrow }\Mg \longrightarrow 1, 
\end{eqnarray}
where $\Delta (\gamma )$ for $\gamma \in \pi _1(\Sigma _{g}, \widetilde{p}_{2n+2})$ is a self-homeomorphism on $\Sigma _g$ which is described as the result of pushing $p_{2n+2}$ along $\gamma $ once. 
As a generalization of Theorem~3.1 in~\cite{Brendle-Margalit}, we have the following lemma. 

\begin{lem}\label{forget_smod_inj}
For $g\geq 2$, the restriction $\mathcal{F}|_{\SMp }\colon \SMp \to \Mg $ is injective. 
\end{lem}

\begin{proof}
We will prove this lemma by a similar argument of Brendle and Margalit~\cite{Brendle-Margalit}. 
By the exact sequence~(\ref{birman_exact_seq}), we have $\ker \mathcal{F}|_{\SMp }=\SMp \cap \Delta (\pi _1(\Sigma _g, \widetilde{p}_{2n+2}))$. 
Thus, it is enough for completing the proof of Lemma~\ref{forget_smod_inj} to show that $\zeta \Delta (\alpha )\not= \Delta (\alpha )\zeta ^l$ in $\SMp $ for any $1\leq l\leq k-1$ and a nontrivial element $\alpha \in \pi _1(\Sigma _g, \widetilde{p}_{2n+2})$. 
For $2\leq l\leq k-1$, since 
\[
\mathcal{F}(\zeta \Delta (\alpha ))=\zeta \not= \zeta ^l=\mathcal{F}(\Delta (\alpha )\zeta ^l),
\]
we have $\zeta \Delta (\alpha )\not= \Delta (\alpha )\zeta ^l$ in $\SMp $. 
Hence we will show that $\zeta \Delta (\alpha )\zeta ^{-1}\not= \Delta (\alpha )$ in $\SMp $ for any nontrivial element $\alpha \in \pi _1(\Sigma _g, \widetilde{p}_{2n+2})$. 
By the injectivity of $\Delta $ and $\zeta \Delta (\alpha )\zeta ^{-1}=\Delta (\zeta (\alpha ))$, this condition is equivalent to $\zeta (\alpha )\not= \alpha $ in $\pi _1(\Sigma _g, \widetilde{p}_{2n+2})$. 

Denote by $\Sigma _g^\prime $ a $4nk$-gon which is obtained from $\Sigma _g$ by cutting $\Sigma _g$ along $p^{-1}(l_1\cup l_2 \cup \cdots \cup l_{2n})$. 
We take a hyperbolic structure on $\Sigma _g$ such that $\zeta $ is an isometry on $\Sigma _g$ and $\Sigma _g^\prime $ is realized as a fundamental domain of the universal covering $\widetilde{p}\colon \mathbb{H}^2\to \Sigma _g$ which is a hyperbolic $4nk$-gon in the hyperbolic plane $\mathbb{H}^2$.  
Let $\widetilde{\widetilde{p}_{2n+2}}\in \mathbb{H}^2$ be the center point of $\Sigma _g^\prime \subset \mathbb{H}^2$ and $\widetilde{\zeta }$ the $(-\frac{2\pi }{k})$-rotation of $\mathbb{H}^2$ around $\widetilde{\widetilde{p}_{2n+2}}$. 
By the construction of $\Sigma _g^\prime $ and the definition of $\widetilde{\zeta }$, the isometry $\widetilde{\zeta }$ acts on $\Sigma _g^\prime \subset \mathbb{H}^2$ and $\widetilde{\zeta }$ is a lift of $\zeta $ with respect to $\widetilde{p}\colon \mathbb{H}^2\to \Sigma _g$. 
Since $\widetilde{\widetilde{p}_{2n+2}}$ is the unique fixed point of $\widetilde{\zeta }$ 
and any point in $\widetilde{p}^{-1}(\{ \widetilde{p}_1, \widetilde{p}_2, \dots , \widetilde{p}_{2n+1}\} )$ is not fixed by $\widetilde{\zeta }$, the point $\widetilde{\widetilde{p}_{2n+2}}$ is a lift of $\widetilde{p}_{2n+2}$ with respect to $\widetilde{p}\colon \mathbb{H}^2\to \Sigma _g$. 

Let $\alpha \colon [0,1]\to \Sigma _{g}$ be a loop based at $\widetilde{p}_{2n+2}$ which is non-trivial in $\pi _1(\Sigma _g, \widetilde{p}_{2n+2})$ and $\widetilde{\alpha }\colon [0,1]\to \mathbb{H}^2$ a lift of $\alpha $ with respect to $\widetilde{p}$ such that the origin of $\widetilde{\alpha }$ is $\widetilde{\widetilde{p}_{2n+2}}$. 
Since $\pi _1(\Sigma _g, \widetilde{p}_{2n+2})$ freely acts on $\widetilde{p}^{-1}(\widetilde{p}_{2n+2})$ as the covering transformation group of $\widetilde{p}$, $\alpha \cdot \widetilde{\widetilde{p}_{2n+2}}=\widetilde{\alpha }(1)$ does not coincide with $\widetilde{\widetilde{p}_{2n+2}}$.   
Assume that $\zeta \circ \alpha =\alpha $. 
Then the composition $\widetilde{\zeta }\circ \widetilde{\alpha }$ is also a lift of $\zeta \circ \alpha =\alpha $ with respect to $\widetilde{p}$ such that its origin is $\widetilde{\widetilde{p}_{2n+2}}$ since $\widetilde{\zeta }$ fixes $\widetilde{\widetilde{p}_{2n+2}}$. 
Since we have $\widetilde{\zeta }\circ \widetilde{\alpha }=\widetilde{\alpha }$ by the unique lifting property, we show that $\widetilde{\zeta }(\widetilde{\alpha }(1))=\widetilde{\alpha }(1)$, namely $\widetilde{\zeta }$ has distinct two fixed points $\widetilde{\widetilde{p}_{2n+2}}$ and $\widetilde{\alpha }(1)$. 
This contradicts the fact that $\widetilde{\zeta }$ has the unique fixed point.  
Therefore,  we have $\zeta \circ \alpha \not= \alpha $ and have completed the proof of Lemma~\ref{forget_smod_inj}. 
\end{proof}

We remark that the image of the restriction $\mathcal{F}|_{\SMp }$ is included in $\SM $ and $\mathcal{F}|_{\SMp }\colon \SMp \to \SM $ is not surjective. 

Let $\theta \colon \SM \to \LM $ be the surjective homomorphism with the kernel $\left< \zeta \right> $ which is obtained from the Birman-Hilden correspondence~\cite{Birman-Hilden2}, namely $\theta $ is defined as follows: for $f\in \SM $ and a symmetric representative $\varphi \in f$ for $\zeta $, we define $\hat{\varphi }\colon \Sigma _0\to \Sigma _0$ by $\hat{\varphi }(x)=p(\varphi (\widetilde{x}))$ for some $\widetilde{x}\in p^{-1}(x)$. 
Then $\theta (f)=[\hat{\varphi }]\in \LM $. 
We have the following proposition. 

\begin{prop}\label{prop_smod_char}
For  $g\geq 2$, we have $\mathcal{F}(\SMp )=\theta ^{-1}(\LMp )$. 
\end{prop}

\begin{proof}
For any $f\in \SMp $, we take a symmetric representative $\varphi \in f$ and the corresponding homeomorphism $\hat{\varphi }\colon \Sigma _0\to \Sigma _0$ defined above. 
Since $\varphi $ preserves the point $\widetilde{p}_{2n+2}$ by the definition of $\SMp $ and $\varphi \circ p=p\circ \hat{\varphi }$ holds, $\hat{\varphi }$ also preserves the point $p(\widetilde{p}_{2n+2})=p_{2n+2}$. 
Thus, we have $\theta (f)=[\hat{\varphi }]\in \LMp $ and show that $\mathcal{F}(\SMp )\subset \theta ^{-1}(\LMp )$. 

For $f\in \SM $, we assume that $\theta (f)\in \LMp $. 
Then there exist a representative $\varphi \in \theta (f)$ with $\varphi (p_{2n+2})=p_{2n+2}$ and a homeomorphism $\widetilde{\varphi }\colon \Sigma _g\to \Sigma _g$ such that $p\circ \widetilde{\varphi }=\varphi \circ p$. 
Then, since $\widetilde{p}_{2n+2}$ is the unique preimage of $p_{2n+2}$ with respect to $p$, we can see that $\widetilde{\varphi }(\widetilde{p}_{2n+2})=\widetilde{p}_{2n+2}$. 
Since $f[\widetilde{\varphi }]^{-1}$ lies in the kernel of $\theta $, there exists $0\leq l\leq k-1$ such that $f[\widetilde{\varphi }]^{-1}=\zeta ^l$. 
Hence we have $f=\zeta ^l[\widetilde{\varphi }]\in \SM $. 
Since $\zeta $ and $\widetilde{\varphi }$ preserve the point $\widetilde{p}_{2n+2}$, the composition $\zeta ^l\circ \widetilde{\varphi }$ also preserves $\widetilde{p}_{2n+2}$. 
Thus we have the mapping class $\zeta ^l[\widetilde{\varphi }]\in \SMp $ and $f=\mathcal{F}(\zeta ^l[\widetilde{\varphi }])\in \mathcal{F}(\SMp )$. 
Therefore, we have completed the proof of Proposition~\ref{prop_smod_char}. 
\end{proof}

By Lemma~\ref{forget_smod_inj} and Proposition~\ref{prop_smod_char}, we regard $\SMp $ as a subgroup of $\SM $ which consists of elements represented by symmetric elements preserving $\widetilde{p}_{2n+2}$.  
The next lemma is given by Proposition~2.1 in \cite{Ghaswala-McLeay}. 

\begin{lem}\label{lem_comm_smodb}
For $g\geq 2$, any element in $\SMb $ has representative which commutes with $\zeta $. 
\end{lem}

The inclusion $\widetilde{\iota }\colon \Sigma _g^1\hookrightarrow \Sigma _g=\Sigma _g^1\cup \widetilde{D}$ induces the surjective homomorphism $\widetilde{\iota }_\ast \colon \Mgb \to \Mgp $ whose kernel is the infinite cyclic group generated by the Dehn twist $t_{\partial \widetilde{D}}$ along $\partial \widetilde{D}$. 
The homomorphism $\widetilde{\iota }_\ast$ is also called the \textit{capping homomorphism}. 
We have the following lemma. 

\begin{lem}\label{surj_capping_smod}
For $g\geq 2$, $\widetilde{\iota }_\ast (\SMb )=\SMp $.
\end{lem}

\begin{proof}
For any $f\in \SMb $, by Lemma~\ref{lem_comm_smodb}, there exists a representative $\varphi $ of $f$ such that $\varphi $ commutes with $\zeta $. 
Denote by $\varphi ^\prime $ the self-homeomorphism on $\Sigma _g=\Sigma _g^1\cup \widetilde{D}$ which is the extension of $\varphi $ by the identity map on $\widetilde{D}$. 
By the definition, $\varphi ^\prime $ commutes with $\zeta $ and preserves $\widetilde{p}_{2n+2}$. 
Thus $\widetilde{\iota }_\ast (f)=[\varphi ^\prime ]\in \SMp $, namely we have $\widetilde{\iota }_\ast (\SMb )\subset \SMp $. 

Recall that $\theta \colon \SM \to \LM $ is the surjective homomorphism with the kernel $\left< \zeta \right> $ which is obtained from the Birman-Hilden correspondence~\cite{Birman-Hilden2} and we regard $\SMp $ as a subgroup of $\SM $ by Lemma~\ref{forget_smod_inj}. 
By Proposition~\ref{prop_smod_char}, the restriction $\theta \colon \SMp \to \LMp $ is surjective. 
By the result of Birman and Hilden \cite{Birman-Hilden2}, we have an isomorphism $\theta ^1\colon \SMb \to \LMb $ which is similarly defined to $\theta $ as before Proposition~\ref{prop_smod_char}. 
Then we have the following commutative diagram: 
\[
\xymatrix{
\SMb \ar[r]^{\widetilde{\iota }_\ast } \ar[d]_{\theta ^1} &  \SMp \ar[d]^\theta \\
\LMb  \ar[r]_{\iota _\ast } &\LMp . \ar@{}[lu]|{\circlearrowright} 
}
\] 
We also recall that $\zeta ^\prime $ is the self-homeomorphism on $\Sigma _g$ which is defined in the proof of Proposition~\ref{prop_exact_lmodb} (see also Figure~\ref{fig_lift_t_partial-d}). 
Remark that the restriction of $\zeta ^\prime $ to $\widetilde{D}$ is the identity map and $\widetilde{\iota }_\ast ([\zeta ^\prime ])=[\zeta ]\in \SMp $. 

We take any $f\in \SMp $. 
By Lemma~\ref{surj_capping_lmod}, $\iota _\ast $ is surjective, there exists $\hat{f}^\prime \in \LMb $ such that $\iota _\ast (\hat{f}^\prime )=\theta (f)$. 
Put $f_0^\prime =(\theta^1)^{-1}(\hat{f}^\prime )\in \SMb $. 
Since we have
\begin{eqnarray*}
\theta (\widetilde{\iota}_\ast (f_0^\prime )\cdot f^{-1})&=&\theta (\widetilde{\iota}_\ast (f_0^\prime ))\cdot \theta (f)^{-1}=\iota _\ast (\theta ^1(f_0^\prime ))\cdot \theta (f)^{-1}\\
&=&\iota _\ast (\hat{f}^\prime )\cdot \theta (f)^{-1}=\theta (f)\cdot \theta (f)^{-1}\\
&=&1, 
\end{eqnarray*}
there exists an integer $0\leq l\leq k-1$ such that $\widetilde{\iota}_\ast (f_0^\prime )\cdot f^{-1}=[\zeta ]^l\in \SMp $. 
Thus we have $[\zeta ^\prime]^{-l}f_0^\prime \in \SMb $ and 
\begin{eqnarray*}
\widetilde{\iota}_\ast ([\zeta ^\prime]^{-l}f_0^\prime )&=&\widetilde{\iota}_\ast ([\zeta ^\prime])^{-l}\widetilde{\iota}_\ast (f_0^\prime )=[\zeta ]^{-l}[\zeta ]^{l}f\\
&=&f.
\end{eqnarray*}
Therefore we have $\widetilde{\iota }_\ast (\SMb )\supset \SMp $ and have completed the proof of Lemma~\ref{surj_capping_smod}. 
\end{proof}

The capping homomorphism $\widetilde{\iota }_\ast \colon \Mgb \to \Mgp $ induces the exact sequence
\[
1\longrightarrow \left< t_{\partial \widetilde{D}}\right>  \longrightarrow \Mgb \stackrel{\widetilde{\iota }_\ast }{\longrightarrow }\Mgp \longrightarrow 1. 
\]
Since $t_{\partial \widetilde{D}}$ commutes with $\zeta $, we have $\left< t_{\partial \widetilde{D}}\right> \cap \SMb =\left< t_{\partial \widetilde{D}}\right> $. 
By Lemma~\ref{surj_capping_smod} and restricting the exact sequence above to $\SMb $, we have the following proposition. 

\begin{prop}\label{prop_exact_smodb}
For $g\geq 2$, we have the following exact sequence: 
\begin{eqnarray}\label{exact_smodb}
1\longrightarrow \left< t_{\partial \widetilde{D}}\right>  \longrightarrow \SMb \stackrel{\widetilde{\iota }_\ast }{\longrightarrow }\SMp \longrightarrow 1. 
\end{eqnarray}
\end{prop}

\section{Presentations for the liftable mapping class groups}\label{section_lmod}

Recall that we assume the $n$ is a positive integer throughout this paper.  

\subsection{Main theorems for the liftable mapping class groups}\label{section_main_thm_lmod}

Recall that for elements $f$ and $h$ in a group $G$, $f\rightleftarrows h$ means the commutative relation $fh=hf$ in $G$. 
Main theorems of this paper for the liftable mapping class groups are as follows. 

\begin{thm}\label{thm_pres_lmodb}
The group $\LMb $ admits the presentation with generators $h_i$ for $1\leq i\leq 2n-1$ and $t_{i,j}$ for $1\leq i<j\leq 2n+1$, 
and the following defining relations: 
\begin{enumerate}
\item commutative relations
\begin{enumerate}
\item $h_i \rightleftarrows h_j$ \quad for $j-i\geq 3$,
\item $t_{i,j} \rightleftarrows t_{k,l}$ \quad for $j<k$, $k\leq i<j\leq l$, or $l<i$, 
\item $h_k \rightleftarrows t_{i,j}$ \quad for  $k+2<i$, $i\leq k<k+2\leq j$, or $j<k$,
\end{enumerate}
\item relations among $h_i$'s and $t_{j,j+1}$'s: 
\begin{enumerate}
\item $h_i ^{\pm 1}t_{i,i+1}=t_{i+1,i+2}h_i^{\pm 1}$ \quad for $1\leq i\leq 2n-1$, 
\item $h_i ^{\varepsilon }h_{i+1}^{\varepsilon }t_{i,i+1}=t_{i+2,i+3}h_i^{\varepsilon }h_{i+1}^{\varepsilon }$ \quad for $1\leq i\leq 2n-2$ and $\varepsilon \in \{ 1, -1\}$,
\item $h_{i}^{\varepsilon }h_{i+1}^{\varepsilon }h_{i+2}^{\varepsilon }h_{i}^{\varepsilon }=h_{i+2}^{\varepsilon }h_{i}^{\varepsilon }h_{i+1}^{\varepsilon }h_{i+2}^{\varepsilon }$ \quad for $1\leq i\leq 2n-3$ and $\varepsilon \in \{ 1, -1\}$,
\item $h_ih_{i+1}t_{i,i+1}=t_{i,i+1}h_{i+1}h_i$ \quad for $1\leq i\leq 2n-2$,
\item $h_ih_{i+2}h_it_{i+1,i+2}^{-1}=t_{i+2,i+3}^{-1}h_{i+2}h_ih_{i+2}$ \quad for $1\leq i\leq 2n-3$,
\end{enumerate}
\item the pentagonal relations $(P_{i,j,k,l,m})$ \quad for $1\leq i<j<k<l<m\leq 2n+2$,\\ 
i.e. $t_{j,m-1}^{-1}t_{k,m-1}t_{j,l-1}t_{i,k-1}t_{i,l-1}^{-1}=t_{i,l-1}^{-1}t_{i,k-1}t_{j,l-1}t_{k,m-1}t_{j,m-1}^{-1}$,
\item 
\begin{enumerate}
\item the relations~$(T_{i,j})$ for $1\leq i<j\leq 2n+1$, i.e. \\
$t_{i,j}=\left\{
		\begin{array}{ll}
		h_i^2 \quad \text{for }j-i=2,\\
		t_{j-1,j}^{-\frac{j-i-3}{2}}\cdots t_{i+2,i+3}^{-\frac{j-i-3}{2}}t_{i,i+1}^{-\frac{j-i-3}{2}}(h_{j-2}\cdots h_{i+1}h_{i})^{\frac{j-i+1}{2}}\\ \text{for }2n-1\geq j-i \geq 3\text{ is odd},\\
		t_{j-2,j-1}^{-\frac{j-i-2}{2}}\cdots t_{i+2,i+3}^{-\frac{j-i-2}{2}}t_{i,i+1}^{-\frac{j-i-2}{2}}h_{j-2}\cdots h_{i+2}h_{i}h_{i}h_{i+2}\cdots h_{j-2}\\ \cdot (h_{j-3}\cdots h_{i+1}h_{i})^{\frac{j-i}{2}}\quad \text{for }2n\geq j-i \geq 4\text{ is even}. 
		\end{array}
		\right.$
\end{enumerate}
\end{enumerate}
\end{thm}

Remark that by the relation~(1) (b) in Theorem~\ref{thm_pres_lmodb}, if $j-i$ is even, then $t_{i,i+1}$ commutes with $t_{j,j+1}$. 
The relations~(2) (a)--(c) in Theorem~\ref{thm_pres_lmodb} are the relations~(1)--(3) in Lemma~\ref{lem_conj_rel}, i.e. conjugation relations, and the relations~(2) (d) and (e) in Theorem~\ref{thm_pres_lmodb} are equivalent to the relations~(2) and (3) in Lemma~\ref{lem_lift_W} up to the relations~(2) (a) and (b) in Theorem~\ref{thm_pres_lmodb}. 

\begin{thm}\label{thm_pres_lmodp}
The group $\LMp $ admits the presentation which is obtained from the finite presentation for $\LMb $ in Theorem~\ref{thm_pres_lmodb} by adding the relation 
\begin{enumerate}
\item[(4)\ (b)] $t_{1,2n+1}=1$. 
\end{enumerate}
\end{thm}

\begin{thm}\label{thm_pres_lmod}
The group $\LM $ admits the presentation which is obtained from the finite presentation for $\LMp $ in Theorem~\ref{thm_pres_lmodp} by adding generators $h_{2n}$ and $r$, and the following relations:
\begin{enumerate}
\item[(4)]
\begin{enumerate}
\item[(c)] $t_{1,2n}^{-n+1}t_{2n-1,2n}^{-n+1}\cdots t_{3,4}^{-n+1}t_{1,2}^{-n+1}(h_{2n}\cdots h_{2}h_{1})^{n+1}=1$, 
\end{enumerate}
\item[(5)] relations among $r$ and $h_i$'s or $t_{i,j}$'s: 
\begin{enumerate}
\item $r^2=1$, 
\item $rh_i=h_{2n-i+1}r$ \quad for $1\leq i\leq 2n$, 
\item $rt_{i,i+1}=t_{2n-i+2,2n-i+3}r$ \quad for $2\leq i\leq 2n$, 
\item $rt_{1,j}=t_{1,2n-j+2}r$ \quad for $2\leq j\leq 2n$. 
\end{enumerate}
\end{enumerate}
\end{thm}

The defining relations above are introduced in Lemma~\ref{lem_comm_rel}, \ref{lem_conj_rel}, \ref{lem_lift_W}, \ref{lem_t_ij-braid}, and \ref{lem_pentagon_rel}. 
By Theorems~\ref{thm_pres_lmodb}, \ref{thm_pres_lmodp}, and \ref{thm_pres_lmod}, we have the following corollary. 

\begin{cor}\label{cor_gen_lmod}
\begin{enumerate}
\item The groups $\LMb $ and $\LMp$ are generated by $h_1,\ h_2,\ \dots ,\ h_{2n-1}$, and $t_{1,2}$, 
\item The group $\LM $ is generated by $h_1,\ h_3,\ \dots ,\ h_{2n-1},\ t_{1,2}$, and $r$. 
\end{enumerate}
\end{cor}

\begin{proof}
By Theorem~\ref{thm_pres_lmodb}, $\LMb$ is generated by $h_i$ for $1\leq i\leq 2n-1$ and $t_{i,j}$ for $1\leq i<j\leq 2n+1$. 
By the relation~(4)~(a) in Theorem~\ref{thm_pres_lmodb}, $t_{i,j}$ for $1\leq i<j\leq 2n+1$ is a product of $h_i$ $(1\leq i\leq 2n-1)$ and $t_{i,i+1}$ $(1\leq i\leq 2n)$. 
By the relation~(2)~(a) in Theorem~\ref{thm_pres_lmodb}, we show that $t_{i,i+1}$ is a product of $h_1,$ $h_2, \dots ,$ $h_{2n-1}$, and $t_{1,2}$. 
By Lemma~\ref{surj_capping_lmod}, we have $\iota _\ast (\LMb )=\LMp $. 
Therefore, $\LMb $ and $\LMp $ are generated by $h_1,\ h_2,\ \dots ,\ h_{2n-1}$, and $t_{1,2}$. 

By Theorem~\ref{thm_pres_lmod}, $\LM$ is generated by $h_i$ for $1\leq i\leq 2n$, $t_{i,j}$ for $1\leq i<j\leq 2n+1$, and $r$. 
By a similar argument in the case~(1), every $t_{i,j}$ is a product of $h_i$ $(1\leq i\leq 2n)$ and $t_{1,2}$. 
By the relation~(2)~(d) in Theorem~\ref{thm_pres_lmod}, we show that $h_2,\ h_4,\ \dots ,\ h_{2n}$ are products of $h_1,\ h_3,\ \dots ,\ h_{2n-1},\ t_{1,2}$, and $r$. 
Therefore, $\LM $ is generated by $h_1,\ h_3,\ \dots ,\ h_{2n-1},\ t_{1,2}$, and $r$, and we have completed Corollary~\ref{cor_gen_lmod}. 
\end{proof}

\subsection{Proofs of presentations for the liftable mapping class groups}\label{section_proof_lmod}

Throughout this section, a label in a deformation of a relation indicates one of relations in Theorems~\ref{thm_pres_lmodb}, \ref{thm_pres_lmodp}, and \ref{thm_pres_lmod}. 
For instance, $A\stackrel{\text{(1)(a)}}{=}B$ (resp. $A=B\stackrel{\text{(1)(a)}}{\Longleftrightarrow }A^\prime= B^\prime $) means that $B$ (resp. the relation $A^\prime =B^\prime $) is obtained from $A$ (resp. the relation $A=B$) by the relation~(1) (a) in Theorem~\ref{thm_pres_lmodb}. 
First, we prove the following lemma. 

\begin{lem}\label{tech_rel1}
In $\LM $, the relations
\begin{enumerate}
\item 
\begin{enumerate}
\item[(a$^\prime $)] $h_i \rightleftarrows h_{2n}$ \quad for $1\leq i\leq 2n-3$,
\item[(c$^\prime $)] $h_{2n} \rightleftarrows t_{i,j}$ \quad for $1\leq i<j\leq 2n-1$,
\end{enumerate}
\item 
\begin{enumerate}
\item[(a$^\prime $)] $h_{2n}^{\pm 1}t_{2n,2n+1}=t_{1,2n}h_{2n}^{\pm 1}$, 
\item[(b$^\prime $)] $h_{2n-1}^{\varepsilon }h_{2n}^{\varepsilon }t_{2n-1,2n}=t_{1,2n}h_{2n-1}^{\varepsilon }h_{2n}^{\varepsilon }$\quad for $\varepsilon \in \{ 1, -1\}$,
\item[(c$^\prime $)] $h_{2n-2}^{\varepsilon }h_{2n-1}^{\varepsilon }h_{2n}^{\varepsilon }h_{2n-2}^{\varepsilon }=h_{2n}^{\varepsilon }h_{2n-2}^{\varepsilon }h_{2n-1}^{\varepsilon }h_{2n}^{\varepsilon }$\quad for $\varepsilon \in \{ 1, -1\}$,
\item[(d$^\prime $)] $h_{2n-1}h_{2n}t_{2n-1,2n}=t_{2n-1,2n}h_{2n}h_{2n-1}$,
\item[(e$^\prime $)] $h_{2n-2}h_{2n}h_{2n-2}t_{2n-1,2n}^{-1}=t_{2n,2n+1}^{-1}h_{2n+2}h_{2n}h_{2n+2}$,
\end{enumerate}
\item[(4)] 
\begin{enumerate}
\item[(a$^\prime $)] $t_{1,2n-1}=h_{2n}^2$
\end{enumerate}
\end{enumerate}
is obtained from the relations~(1), (2), (4) (a), and (5) in Theorem~\ref{thm_pres_lmod}. 
\end{lem}

\begin{proof}
By the conjugations by $r$, the relations~(1) (a$^\prime $), (c$^\prime $), (2) (a$^\prime $), (b$^\prime $), (c$^\prime $), and (4) (a$^\prime $) are obtained from the relations~(1) (a), (c), (2) (a), (b), (c), and (4) (a) up to the relations~(5) (b), (c), (d). 
For instance, for the relation~(2) (b$^\prime $), we have
\begin{eqnarray*}
&&h_{2n-1}^{\varepsilon }h_{2n}^{\varepsilon }t_{2n-1,2n}=t_{1,2n}h_{2n-1}^{\varepsilon }h_{2n}^{\varepsilon }\\
&\stackrel{\text{CONJ}}{\Longleftrightarrow }&rh_{2n-1}^{\varepsilon }h_{2n}^{\varepsilon }t_{2n-1,2n}r^{-1}=rt_{1,2n}h_{2n-1}^{\varepsilon }h_{2n}^{\varepsilon }r^{-1}\\
&\stackrel{\text{(5) (b),(c),(d)}}{\Longleftrightarrow }&h_{2}^{\varepsilon }h_{1}^{\varepsilon }\underline{t_{3,4}}=\underline{t_{1,2}}h_{2}^{\varepsilon }h_{1}^{\varepsilon }\\
&\stackrel{\text{CONJ}}{\Longleftrightarrow }&t_{1,2}^{-1}h_{2}^{\varepsilon }h_{1}^{\varepsilon }=h_{2}^{\varepsilon }h_{1}^{\varepsilon }t_{3,4}^{-1}\\
&\stackrel{\text{inverse}}{\Longleftrightarrow }&h_{1}^{-\varepsilon }h_{2}^{-\varepsilon }t_{1,2}=t_{3,4}h_{1}^{-\varepsilon }h_{2}^{-\varepsilon }\\
&\Longleftrightarrow &\text{the relation~(2) (b) for }i=1.
\end{eqnarray*}

For the relation~(2) (d$^\prime $),  we have
\begin{eqnarray*}
&&h_{2n-1}h_{2n}t_{2n-1,2n}=t_{2n-1,2n}h_{2n}h_{2n-1}\\
&\stackrel{\text{CONJ}}{\Longleftrightarrow }&rh_{2n-1}h_{2n}t_{2n-1,2n}r^{-1}=rt_{2n-1,2n}h_{2n}h_{2n-1}r^{-1}\\
&\stackrel{\text{(5) (b),(c),(d)}}{\Longleftrightarrow }&\underline{h_{2}h_{1}t_{3,4}}=\underline{t_{3,4}h_{1}h_{2}}\\
&\stackrel{\text{(2)(b)}}{\Longleftrightarrow }&t_{1,2}h_{2}h_{1}=h_{1}h_{2}t_{1,2}\\
&\Longleftrightarrow &\text{the relation~(2) (d) for }i=1.
\end{eqnarray*}
Hence the relation~(2) (d$^\prime $) is obtained from the relation~(2). 
By a similar argument, we can check that the relation~(2) (e$^\prime $) is obtained from the relations~(1) and (2). 
We have completed the proof of Lemma~\ref{tech_rel1}. 
\end{proof}

By Lemma~\ref{tech_rel1}, the relations~(1) (a$^\prime $), (c$^\prime $), (2) (a$^\prime $)--(e$^\prime $), and (4) (a$^\prime $) in Lemma~\ref{tech_rel1} are obtained from the relations in Theorems~\ref{thm_pres_lmod}. 
For conveniences, throughout this paper, we call the relations~(1) (a$^\prime $), (c$^\prime $), (2) (a$^\prime $)--(e$^\prime $), and (4) (a$^\prime $) in Lemma~\ref{tech_rel1} the relations~(1) (a), (c), (2) (a)--(e), and (4) (a), respectively.

Recall that $W_{2n+2}$ is the subgroup of the symmetric group $S_{2n+2}$ which consists of parity-preserving or parity-reversing elements, we define $\bar{h}_i=(i\ i+2)$ for $1\leq i\leq 2n$ and $\bar{r}=(1\ 2n+2)(2\ 2n+1)\cdots (n+1\ n+2)$ (see Sections~\ref{section_liftable-element} and \ref{section_relations_liftable-elements}). 
First we give a finite presentation for $W_{2n+2}$ which is slightly different from Ghaswala-Winarski's presentation in Lemma~4.2 of~\cite{Ghaswala-Winarski1}. 

\begin{lem}\label{pres_w_2n+2}
The group $W_{2n+2}$ admits the presentation with generators $\bar{h}_i$ for $1\leq i\leq 2n$ and $\bar{r}$, and the following defining relations: 
\begin{enumerate}
\item $\bar{h}_i^2=1$ \quad for $1\leq i\leq 2n$, 
\item $\bar{h}_i\bar{h}_{j}\bar{h}_i^{-1}\bar{h}_{j}^{-1}=1$ \quad for ``$j-i=1$ or $j-i\geq 3$'' and $1\leq i<j\leq 2n$, 
\item $\bar{h}_i\bar{h}_{i+2}\bar{h}_i\bar{h}_{i+2}^{-1}\bar{h}_i^{-1}\bar{h}_{i+2}^{-1}=1$ \quad for $1\leq i\leq 2n-2$, 
\item $\bar{r}^2=1$, 
\item $\bar{r}\bar{h}_i=\bar{h}_{2n-i+1}\bar{r}$ for $1\leq i\leq 2n$.
\end{enumerate}
\end{lem}

\begin{proof}
We will apply Lemma~\ref{presentation_exact} to the exact sequence~(\ref{exact1}). 
Since the subgroup $S_{n+1}^o\times S_{n+1}^e$ of $W_{2n+2}$ is isomorphic to $S_{n+1}\times S_{n+1}$, the group $S_{n+1}^o\times S_{n+1}^e$ has the presentation with generators $\bar{h}_i$ for $1\leq i\leq 2n$ and the following defining relations: 
\begin{itemize}
\item $\bar{h}_i^2=1$ \quad for $1\leq i\leq 2n$, 
\item $\bar{h}_i\bar{h}_{j}\bar{h}_i^{-1}\bar{h}_{j}^{-1}=1$ \quad for ``$j-i=1$ or $j-i\geq 3$'' and $1\leq i<j\leq 2n$, 
\item $\bar{h}_i\bar{h}_{i+2}\bar{h}_i\bar{h}_{i+2}^{-1}\bar{h}_i^{-1}\bar{h}_{i+2}^{-1}=1$ \quad for $1\leq i\leq 2n-2$. 
\end{itemize}
Since $\Z _2$ in the exact sequence~(\ref{exact1}) is generated by an element which is represented by a parity-reversing element, $\pi (\bar{r})$ generates the group $\Z _2$.  
By applying Lemma~\ref{presentation_exact} to the exact sequence~(\ref{exact1}) and presentations for $S_{n+1}^o\times S_{n+1}^e$ and $\Z _2$ above, we have the presentation for $W_{2n+2}$ whose generators are $\bar{h}_i$ for $1\leq i\leq 2n$ and $\bar{r}$ and the defining relations are as follows: 
\begin{enumerate}
\item[(A)] 
\begin{enumerate}
\item[(1)] $\bar{h}_i^2=1$ \quad for $1\leq i\leq 2n$, 
\item[(2)] $\bar{h}_i\bar{h}_{j}\bar{h}_i^{-1}\bar{h}_{j}^{-1}=1$ \quad for ``$j-i=1$ or $j-i\geq 3$'' and $1\leq i<j\leq 2n$, 
\item[(3)] $\bar{h}_i\bar{h}_{i+2}\bar{h}_i\bar{h}_{i+2}^{-1}\bar{h}_i^{-1}\bar{h}_{i+2}^{-1}=1$ \quad for $1\leq i\leq 2n-2$,
\end{enumerate}
\item[(B)] $\bar{r}^2=1$, 
\item[(C)] $\bar{r}\bar{h}_i\bar{r}^{-1}=\bar{h}_{2n-i+1}$ for $1\leq i\leq 2n$.
\end{enumerate}
This presentation is equivalent to the presentation in Lemma~\ref{pres_w_2n+2}. 
Therefore, we have completed the proof of Lemma~\ref{pres_w_2n+2}. 
\end{proof}

By Lemma~\ref{lem_W_ast}, we have $W_{2n+2;\ast }=S_{n+1}^o\times S_{n}^e$. 
Since $S_{n+1}^o\times S_{n}^e$ is isomorphic to $S_{n+1}\times S_{n}$, we have the following lemma. 

\begin{lem}\label{pres_w_2n+2ast}
For $n\geq 1$, $W_{2n+2;\ast }$ admits the presentation with generators $\bar{h}_i$ for $1\leq i\leq 2n-1$ and the following defining relations: 
\begin{enumerate}
\item $\bar{h}_i^2=1$ \quad for $1\leq i\leq 2n-1$, 
\item $\bar{h}_i\bar{h}_{j}\bar{h}_i^{-1}\bar{h}_{j}^{-1}=1$ \quad for ``$j-i=1$ or $j-i\geq 3$'' and $1\leq i<j\leq 2n-1$, 
\item $\bar{h}_i\bar{h}_{i+2}\bar{h}_i\bar{h}_{i+2}^{-1}\bar{h}_i^{-1}\bar{h}_{i+2}^{-1}=1$ \quad for $1\leq i\leq 2n-3$. 
\end{enumerate}
\end{lem}

Applying Lemma~\ref{presentation_exact} to the exact sequence~(\ref{exact_lmodp}) in Proposition~\ref{prop_exact_lmodp}, we have the following proposition. 

\begin{prop}\label{prop_pres_lmodp} 
The group $\LMp $ admits the presentation with generators $h_i$ for $1\leq i\leq 2n-1$ and $t_{i,j}$ for $1\leq i\leq j\leq 2n+1$, and the following defining relations: 
\begin{enumerate}
\item commutative relations
\begin{enumerate}
\item $h_i \rightleftarrows h_j$ \quad for $j-i\geq 3$,
\item $t_{i,j} \rightleftarrows t_{k,l}$ \quad for $j<k$, $k\leq i<j\leq l$, or $l<i$, 
\item $h_k \rightleftarrows t_{i,j}$ \quad for $k+2<i$, $i\leq k<k+2\leq j$, or $j<k$,
\end{enumerate}
\item pentagonal relations $(P_{i,j,k,l,m})$ \quad for $1\leq i<j<k<l<m\leq 2n+2$,\\ 
i.e. $t_{j,m-1}^{-1}t_{k,m-1}t_{j,l-1}t_{i,k-1}t_{i,l-1}^{-1}=t_{i,l-1}^{-1}t_{i,k-1}t_{j,l-1}t_{k,m-1}t_{j,m-1}^{-1}$,
\item lifts of relations in the symmetric group
\begin{enumerate}
\item $h_i^2=t_{i,i+2}$ \quad for $1\leq i\leq 2n-1$, 
\item $h_ih_{i+1}h_i^{-1}h_{i+1}^{-1}=t_{i,i+1}t_{i+2,i+3}^{-1}$ \quad for $1\leq i\leq 2n-2$,
\item $h_ih_{i+2}h_ih_{i+2}^{-1}h_i^{-1}h_{i+2}^{-1}=t_{i+1,i+2}t_{i+2,i+3}^{-1}$ \quad for $1\leq i\leq 2n-3$,
\end{enumerate}
\item (a) $h_kt_{i,j}h_k^{-1}\\ =\left\{
		\begin{array}{ll}
		t_{i+1,j}t_{i-1,i}t_{i-2,j}t_{i-1,j}^{-1}t_{i-2,i}^{-1} &\text{for }k=i-2\text{ and }3\leq i<j\leq 2n+1,\\
		t_{i+2,j}t_{i-1,j}t_{i-1,i}t_{i+1,j}^{-1}t_{i-1,i+1}^{-1} &\text{for }k=i-1\text{ and }2\leq i<j-1\leq 2n,\\
		t_{i-1,i} &\text{for }k=i-1\text{ and }2\leq i=j-1\leq 2n,\\
		t_{i,j-1}t_{j,j+1}t_{i,j+2}t_{i,j+1}^{-1}t_{j,j+2}^{-1} &\text{for }k=j\text{ and }1\leq i<j\leq 2n-1,\\
		t_{i,j-2}t_{j,j+1}t_{i,j+1}t_{i,j-1}^{-1}t_{j-1,j+1}^{-1} &\text{for }k=j-1\text{ and }2\leq i+1<j\leq 2n,\\
		t_{j,j+1} &\text{for }k=j-1\text{ and }2\leq i+1=j\leq 2n,
\end{array}
		\right.$
\item $t_{i,i}=t_{1,2n+1}=1$ \quad for $1\leq i\leq 2n+1$. 
\end{enumerate}
\end{prop}

\begin{proof}
By Proposition~\ref{prop_pres_pmod} and adding the generator $t_{1,2n+1}$ and the relation $t_{1,2n+1}=1$, $\PM $ admits the presentation with generators $t_{i,j}$ $(1\leq i<j\leq 2n+1)$ and the following defining relations: 
\begin{enumerate}
\item $t_{i,j} \rightleftarrows t_{k,l}$ \quad for $j<k$, $k\leq i<j\leq l$, or $l<i$, 
\item pentagonal relations $(P_{i,j,k,l,m})$ \quad for $1\leq i<j<k<l<m\leq 2n+1$, 
\item $t_{1,2n+1}=1$, 
\end{enumerate}
and by Lemma~\ref{pres_w_2n+2ast},  $W_{2n+2;\ast }$ admits the presentation with generators $\bar{h}_i$ $(1\leq i\leq 2n-1)$ and the following defining relations: 
\begin{enumerate}
\item $\bar{h}_i^2=1$ \quad for $1\leq i\leq 2n-1$, 
\item $\bar{h}_i\bar{h}_{j}\bar{h}_i^{-1}\bar{h}_{j}^{-1}=1$ \quad for ``$j-i=1$ or $j-i\geq 3$'' and $1\leq i<j\leq 2n-1$, 
\item $\bar{h}_i\bar{h}_{i+2}\bar{h}_i\bar{h}_{i+2}^{-1}\bar{h}_i^{-1}\bar{h}_{i+2}^{-1}=1$ \quad for $1\leq i\leq 2n-3$. 
\end{enumerate}
Since $\Psi (h_i)=\bar{h}_i$ for $1\leq i\leq 2n-1$ and the relations $h_i^2=t_{i,i+2}$ $(1\leq i\leq 2n-1)$, $h_i \rightleftarrows h_j$ $(j-i>3)$, $h_ih_{i+1}h_i^{-1}h_{i+1}^{-1}=t_{i,i+1}t_{i+2,i+3}^{-1}$ $(1\leq i\leq 2n-2)$, and $h_ih_{i+2}h_ih_{i+2}^{-1}h_i^{-1}h_{i+2}^{-1}=t_{i+1,i+2}t_{i+2,i+3}^{-1}$ $(1\leq i\leq 2n-3)$ hold in $\LMp $ by Lemma~\ref{lem_comm_rel} and~\ref{lem_lift_W}, applying Lemma~\ref{presentation_exact} to the exact sequence~(\ref{exact_lmodp}) in Proposition~\ref{prop_exact_lmodp}, we show that $\LMp $ has the presentation whose generators are $h_i$ $(1\leq i\leq 2n-1)$ and $t_{i,j}$ $(1\leq i\leq j\leq 2n+1)$ and defining relations are as follows:  
\begin{enumerate}
\item[(A)] 
\begin{enumerate}
\item $t_{i,j} \rightleftarrows t_{k,l}$ \quad for $j<k$, $k\leq i<j\leq l$, or $l<i$, 
\item pentagonal relations $(P_{i,j,k,l,m})$ \quad for $1\leq i<j<k<l<m\leq 2n+2$, 
\item $t_{1,2n+1}=1$, 
\end{enumerate}
\item[(B)]
\begin{enumerate}
\item $h_i^2=t_{i,i+2}$ \quad for $1\leq i\leq 2n-1$, 
\item $h_i \rightleftarrows h_j$ \quad for $j-i>3$,
\item $h_ih_{i+1}h_i^{-1}h_{i+1}^{-1}=t_{i,i+1}t_{i+2,i+3}^{-1}$ \quad for $1\leq i\leq 2n-2$,
\item $h_ih_{i+2}h_ih_{i+2}^{-1}h_i^{-1}h_{i+2}^{-1}=t_{i+1,i+2}t_{i+2,i+3}^{-1}$ \quad for $1\leq i\leq 2n-3$,
\end{enumerate}
\item[(C)] $h_kt_{i,j}h_k^{-1} =w_{i,j;k}$ \quad for $1\leq i<j\leq 2n+1$, $(i,j)\not= (1,2n+1)$, and $1\leq k\leq 2n-1$, 
\end{enumerate}
where $w_{i,j;k}$ is a word in $\{ t_{i,j}^{\varepsilon }\mid 1\leq i<j\leq 2n+1,\ (i,j)\not= (1,2n+1),\ \varepsilon \in \{ \pm 1\} \}$. 
The relations~(A) (a) and (B) (b) above coincide with the relations~(1) (b) and (a) in Proposition~\ref{prop_pres_lmodp}, the relation~(A) (b) above coincides with the relation~(2) in Proposition~\ref{prop_pres_lmodp}, and the relations~(B) (a), (c), and (d) above coincide with the relations~(3) (a), (b), and (c) in Proposition~\ref{prop_pres_lmodp}, respectively. 
We add generators $t_{i,i}$ $(1\leq i\leq 2n+1)$ and relations $t_{i,i}=1$ $(1\leq i\leq 2n+1)$ to this finite presentation for $\LMp$ above. 
The relations~(A) (c) above and $t_{i,i}=1$ for $1\leq i\leq 2n+1$ coincide with the relations~(5) in Proposition~\ref{prop_pres_lmodp}. 
It is sufficient for completing the proof of Proposition~\ref{prop_pres_lmodp} to prove that the relations~(C) above coincide with the relations~(1) (c) and (4) (a) in Proposition~\ref{prop_pres_lmodp}. 

In the cases of $k+2<i$, $i\leq k<k+2\leq j$, or $j<k$, the simple closed curve $\gamma _{i,j}$ is isotopic to a simple closed curve on $\Sigma _0-\B $ which is disjoint from the support of $h_k$. 
Hence, in these cases, we have $w_{i,j;k}=t_{i,j}$ and the relation~(C) coincides with the relation~(1) (c) in Proposition~\ref{prop_pres_lmodp}. 
In the case of $k=i-1$ and $i=j-1$, we have $h_{i-1}(\gamma _{i,i+1})=\gamma _{i-1,i}$. 
Hence, by the conjugation relation, we have $w_{i,i+1;i-1}=t_{i-1,i}$ and the relation~(C) above for $k=i-1$ and $i=j-1$ coincides with the relation~(4) (a) for $k=i-1$ and $i=j-1$ in Proposition~\ref{prop_pres_lmodp}. 
Similarly, in the case of $k=j-1$ and $i+1=j$, the relation~(C) above for $k=j-1$ and $i=j-1$ coincides with the relation~(4) (a) for $k=j-1$ and $i=j-1$ in Proposition~\ref{prop_pres_lmodp}. 

In the case of $k=i-2$, $h_k(\gamma _{i,j})$ is a simple closed curve on $\Sigma _0-\B $ as on the upper left-hand side in Figure~\ref{fig_proof_pres_lantern_rel} and we have $h_kt_{i,j}h_k^{-1}=t_{h_k(\gamma _{i,j})}$ by the conjugation relation. 
By the lantern relation (see on the upper left-hand side in Figure~\ref{fig_proof_pres_lantern_rel}), we have
\begin{eqnarray*}
&&h_kt_{i,j}h_k^{-1}\cdot t_{i-2,i}t_{i-1,j}=t_{i+1,j}t_{i-1,i}t_{i-2,j}\\
&\Leftrightarrow &h_kt_{i,j}h_k^{-1}=t_{i+1,j}t_{i-1,i}t_{i-2,j}t_{i-1,j}^{-1}t_{i-2,i}^{-1}. 
\end{eqnarray*}
Thus, in this case, we have $w_{i,j;k}=t_{i+1,j}t_{i-1,i}t_{i-2,j}t_{i-1,j}^{-1}t_{i-2,i}^{-1}$ and the relation~(C) above for $k=i-2$ coincides with the relation~(4) (a) for $k=i-2$ in Proposition~\ref{prop_pres_lmodp}. 

In the case of $k=i-1$ and $i<j-1$, $h_k(\gamma _{i,j})$ is a simple closed curve on $\Sigma _0-\B $ as on the upper right-hand side in Figure~\ref{fig_proof_pres_lantern_rel}. 
By the lantern relation and the conjugation relation (see on the upper right-hand side in Figure~\ref{fig_proof_pres_lantern_rel}), we have
\begin{eqnarray*}
&&h_kt_{i,j}h_k^{-1}\cdot t_{i-1,i+1}t_{i+1,j}=t_{i+2,j}t_{i-1,j}t_{i-1,i}\\
&\Leftrightarrow &h_kt_{i,j}h_k^{-1}=t_{i+2,j}t_{i-1,j}t_{i-1,i}t_{i+1,j}^{-1}t_{i-1,i+1}^{-1}. 
\end{eqnarray*}
Thus, in this case, we have $w_{i,j;k}=t_{i+2,j}t_{i-1,j}t_{i-1,i}t_{i+1,j}^{-1}t_{i-1,i+1}^{-1}$ and the relation~(C) above for $k=i-1$ and $i<j-1$ coincides with the relation~(4) (a) for $k=i-1$ and $i<j-1$ in Proposition~\ref{prop_pres_lmodp}. 

Similarly, by using the lantern relations as on the lower left-hand side for the case $k=j$ and the lower right-hand side for the case $k=j-1$ and $i+1<j$ in Figure~\ref{fig_proof_pres_lantern_rel}, we have $w_{i,j;k}=t_{i,j-1}t_{j,j+1}t_{i,j+2}t_{i,j+1}^{-1}t_{j,j+2}^{-1}$ for $k=j$ and $w_{i,j;k}=t_{i,j-2}t_{j,j+1}t_{i,j+1}t_{i,j-1}^{-1}t_{j-1,j+1}^{-1}$ for $k=j-1$ and $i+1<j$. 
These relations coincide with the relation~(4) (a) for the case $k=j-1$ and the case $k=j-1$ and $i+1<j$ in Proposition~\ref{prop_pres_lmodp}, respectively. 
Therefore $\LMp $ admits the presentation in Proposition~\ref{prop_pres_lmodp} and we have completed the proof of Proposition~\ref{prop_pres_lmodp}. 
\end{proof}

\begin{figure}[h]
\includegraphics[scale=0.76]{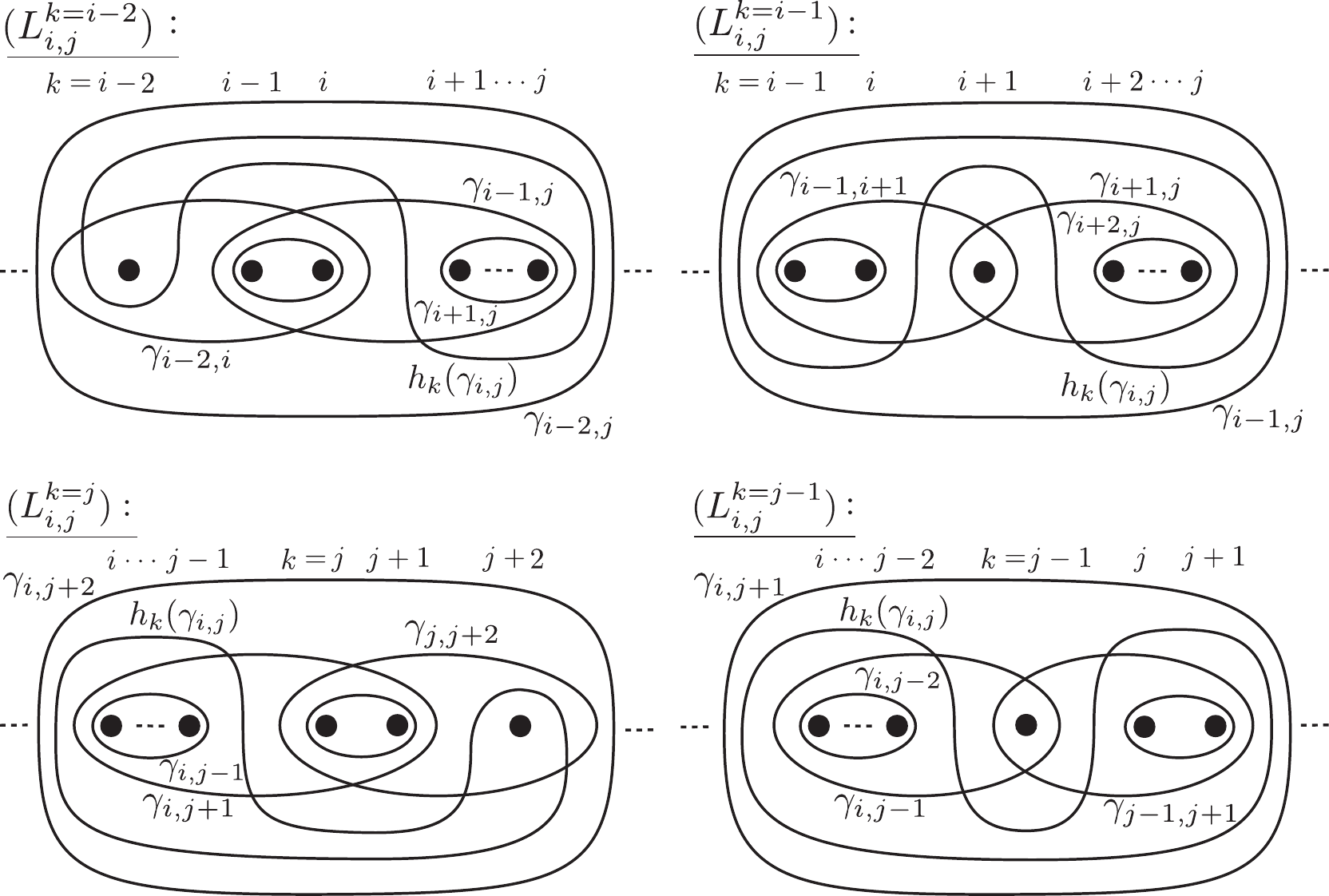}
\caption{Lantern relations~$(L_{i,j}^{k=i-2})$, $(L_{i,j}^{k=i-1})$, $(L_{i,j}^{k=j})$, and $(L_{i,j}^{k=j-1})$ in the proof of Proposition~\ref{prop_pres_lmodp}.}\label{fig_proof_pres_lantern_rel}
\end{figure}

Applying Lemma~\ref{presentation_exact} to the exact sequence~(\ref{exact2}), we have the following proposition. 

\begin{prop}\label{prop_pres_lmod} 
The group $\LM $ admits the presentation which is obtained from the finite presentation for $\LMp $ in Proposition~\ref{prop_pres_lmodp} by adding generators $h_{2n}$, $t_{1,0}$, and $r$, and the following relations:
\begin{enumerate}
\item[(3)] lifts of relations in the symmetric group
\begin{enumerate}
\item[(d)] $r^2=1$,  
\item[(e)] $rh_i=h_{2n-i+1}r$ \quad for $1\leq i\leq 2n$, 
\end{enumerate}
\item[(4)]
\begin{enumerate}
\item[(b)] $h_{2n}t_{i,j}h_{2n}^{-1}\\ =\left\{
		\begin{array}{ll}
		t_{i,2n-1}t_{2n,2n+1}t_{1,i-1}t_{i,2n+1}^{-1}t_{1,2n-1}^{-1}\\ \text{for }j=2n\text{ and }1\leq i\leq 2n-1,\\
		t_{i,2n-1}t_{1,2n}t_{1,i-1}t_{i,2n}^{-1}t_{1,2n-1}^{-1} \quad \text{for }j=2n+1\text{ and }1\leq i\leq 2n-1,\\
		t_{1,2n} \quad \text{for }i+1=j=2n+1,
\end{array}
		\right.$
\item[(c)] $rt_{i,j}r^{-1} =\left\{
		\begin{array}{ll}
		t_{2n-j+3,2n-i+3} &\text{for }2\leq i<j\leq 2n+1,\\
		t_{1,2n-j+2} &\text{for }i=1\text{ and }2\leq j\leq 2n, \\
\end{array}
		\right.$
\end{enumerate}
\item[(5)] $t_{1,0}=1$. 
\end{enumerate}
\end{prop}

\begin{proof}
By Lemma~\ref{pres_w_2n+2}, $W_{2n+2}$ admits the presentation with generators $\bar{h}_i$ $(1\leq i\leq 2n)$ and $\bar{r}$, and the following defining relations: 
\begin{enumerate}
\item $\bar{h}_i^2=1$ \quad for $1\leq i\leq 2n$, 
\item $\bar{h}_i\bar{h}_{j}\bar{h}_i^{-1}\bar{h}_{j}^{-1}=1$ \quad for ``$j-i=1$ or $j-i\geq 3$'' and $1\leq i<j\leq 2n$, 
\item $\bar{h}_i\bar{h}_{i+2}\bar{h}_i\bar{h}_{i+2}^{-1}\bar{h}_i^{-1}\bar{h}_{i+2}^{-1}=1$ \quad for $1\leq i\leq 2n-2$, 
\item $\bar{r}^2=1$, 
\item $\bar{r}\bar{h}_i=\bar{h}_{2n-i+1}\bar{r}$ for $1\leq i\leq 2n$.
\end{enumerate}
The relations $r^2=1$, $rh_i=h_{2n-i+1}r$ $(1\leq i\leq 2n)$, $rt_{i,j}r^{-1}=t_{2n-j+3,2n-i+3}$ $(1\leq i<j\leq 2n+1)$, and $t_{i,2n+2}=t_{1,i-1}$ $(3\leq i\leq 2n+1)$ hold in $\LM $ by Lemma~\ref{lem_conj_rel} and~\ref{lem_lift_W}. 
By applying Lemma~\ref{presentation_exact} to the exact sequence~(\ref{exact2}) and an argument in the proof of Proposition~\ref{prop_pres_lmodp}, we show that $\LM $ has the presentation whose generators are $h_i$ $(1\leq i\leq 2n)$, $r$, and $t_{i,j}$ $(1\leq i\leq j\leq 2n+1)$ and defining relations are as follows:  
\begin{enumerate}
\item[(A)] 
\begin{enumerate}
\item $t_{i,j} \rightleftarrows t_{k,l}$ \quad for $j<k$, $k\leq i<j\leq l$, or $l<i$, 
\item pentagonal relations $(P_{i,j,k,l,m})$ \quad for $1\leq i<j<k<l<m\leq 2n+2$, 
\item $t_{1,2n+1}=1$, 
\end{enumerate}
\item[(B)]
\begin{enumerate}
\item $h_i^2=t_{i,i+2}$ \quad for $1\leq i\leq 2n-1$, 
\item $h_{2n}^2=t_{1,2n-1}$, 
\item $h_i \rightleftarrows h_j$ \quad for $j-i>3$,
\item $h_ih_{i+1}h_i^{-1}h_{i+1}^{-1}=t_{i,i+1}t_{i+2,i+3}^{-1}$ \quad for $1\leq i\leq 2n-2$,
\item $h_{2n-1}h_{2n}h_{2n-1}^{-1}h_{2n}^{-1}=t_{2n-1,2n}t_{1,2n}^{-1}$, 
\item $h_ih_{i+2}h_ih_{i+2}^{-1}h_i^{-1}h_{i+2}^{-1}=t_{i+1,i+2}t_{i+2,i+3}^{-1}$ \quad for $1\leq i\leq 2n-2$,
\item $r^2=1$, 
\item $rh_i=h_{2n-i+1}r$ for $1\leq i\leq 2n$, 
\end{enumerate}
\item[(C)] 
\begin{enumerate}
\item $h_kt_{i,j}h_k^{-1} =w_{i,j;k}$ \quad for $1\leq i<j\leq 2n+1$, $(i,j)\not= (1,2n+1)$, and $1\leq k\leq 2n$, 
\item $rt_{i,j}r^{-1} =\left\{
		\begin{array}{ll}
		t_{2n-j+3,2n-i+3} &\text{for }2\leq i<j\leq 2n+1,\\
		t_{1,2n-j+2} &\text{for }i=1\text{ and }2\leq j\leq 2n, \\
\end{array}
		\right.$
\end{enumerate}
\end{enumerate}
where $w_{i,j;k}$ is a word in $\{ t_{i,j}^{\varepsilon }\mid 1\leq i<j\leq 2n+1,\ (i,j)\not= (1,2n+1),\ \varepsilon \in \{ \pm 1\} \}$. 

We add generators $t_{i,i}$ $(1\leq i\leq 2n+1)$ and $t_{1,0}$, and relations $t_{i,i}=t_{1,0}=1$ $(1\leq i\leq 2n+1)$ to this finite presentation for $\LM $ above (these relations are included in the relations~(5) in Proposition~\ref{prop_pres_lmodp} and (5) in Proposition~\ref{prop_pres_lmod}). 
Remark that for $k\leq 2n-1$, $w_{i,j;k}$ is calculated in the proof of Proposition~\ref{prop_pres_lmodp} and coincides with $t_{i,j}$ or a right side of the relations~(4) (a) in Proposition~\ref{prop_pres_lmodp}. 
The relations~(A) (a), (B) (c) for $j\leq 2n-1$, and (C) (a) for $k+2<i$, $i\leq k<k+2\leq j$, or $j<k\leq 2n-1$ above coincide with the relations~(1) (b), (a), and (c) in Proposition~\ref{prop_pres_lmodp}, the relation~(A) (b) above coincides with the relation~(2) in Proposition~\ref{prop_pres_lmodp}, the relations~(B) (a), (d), and (f) for $i\leq 2n-3$ above coincide with the relations~(3) (a), (b), and (c) in Proposition~\ref{prop_pres_lmodp}, the relations~(B) (g), (h), and (C) (b) above coincide with the relations~(3) (d), (e), and (4) (c) in Proposition~\ref{prop_pres_lmod}, the relation~(C) (a) for $k\in \{ i-2,\ i-1,\ j,\ j+1\}$ and $k\leq 2n-1$ coincides with the relation~(4) (a) in Proposition~\ref{prop_pres_lmodp}, and the relations~(A) (c) and $t_{i,i}=t_{1,0}=1$ for $1\leq i\leq 2n+1$ above coincide with the relations~(5) in Proposition~\ref{prop_pres_lmodp}, respectively. 

Since by Lemma~\ref{tech_rel1}, the relations~(B) (b), (c) for $j=2n$, (e), and (f) for $i=2n-2$ above are obtained from the relations~(B) (a), (c) for $j\leq 2n-1$, (d), (f) for $i\leq 2n-3$, (h), and (C) (b) above, we remove these relations from the presentation for $\LM $ above and the resulting presentation also gives a presentation for $\LM $. 
Hence, it is sufficient for completing the proof of Proposition~\ref{prop_pres_lmod} to prove that the relations~(C) (a) for $k=2n$ above are obtained from the relations in Proposition~\ref{prop_pres_lmodp} and \ref{prop_pres_lmod}. 

In the case of $j\leq 2n-1$, we have $h_{2n}t_{i,j}h_{2n}^{-1} =t_{i,j}$ and this relation is equivalent to the relation~(1) (c) for $k=1$ in Proposition~\ref{prop_pres_lmodp} up to the relations~(3) (e) and (4) (c) in Proposition~\ref{prop_pres_lmod}. 
In the case of $i=2n$ and $j=2n+1$, by the conjugation relation, we have $h_{2n}t_{2n,2n+1}h_{2n}^{-1} =t_{1,2n}$ and this relation coincides with the relation~(4) (b) for $i+1=j=2n+1$ in Proposition~\ref{prop_pres_lmod}. 

In the cases $j=2n$ or ``$j=2n+1$ and $i\leq 2n-1$'', by the lantern relations as on the lower side in Figure~\ref{fig_proof_pres_lantern_rel} and an argument as in the proof of Proposition~\ref{prop_pres_lmodp}, we have $h_{2n}t_{i,j}h_{2n}^{-1}=t_{i,2n-1}t_{2n,2n+1}t_{1,i-1}t_{i,2n+1}^{-1}t_{1,2n-1}^{-1}$ for $j=2n$ and $1\leq i\leq 2n-1$, and $h_{2n}t_{i,j}h_{2n}^{-1}=t_{i,2n-1}t_{1,2n}t_{1,i-1}t_{i,2n}^{-1}t_{1,2n-1}^{-1}$ for $j=2n+1$ and $1\leq i\leq 2n-1$. 
These relations coincide with the relations~(4) (b) for ``$j=2n$ and $1\leq i\leq 2n-1$'' or ``$j=2n+1$ and $1\leq i\leq 2n-1$'' in Proposition~\ref{prop_pres_lmod}.  
Therefore we have completed the proof of Proposition~\ref{prop_pres_lmod}. 
\end{proof}

For conveniences, we call the relation~(4) (a) in Proposition~\ref{prop_pres_lmodp} for $k=i-2$ and $3\leq i<j\leq 2n+1$ as on the upper left-hand side in Figure~\ref{fig_proof_pres_lantern_rel} the \textit{relation}~$(L_{i,j}^{k=i-2})$, for $k=i-1$ and $2\leq i<j-1\leq 2n$ as on the upper right-hand side in Figure~\ref{fig_proof_pres_lantern_rel} the \textit{relation}~$(L_{i,j}^{k=i-1})$, for $k=j$ and $1\leq i<j\leq 2n-1$ as on the lower left-hand side in Figure~\ref{fig_proof_pres_lantern_rel} the \textit{relation}~$(L_{i,j}^{k=j})$, and for $k=j-1$ and $2\leq i+1<j\leq 2n$ as on the lower right-hand side in Figure~\ref{fig_proof_pres_lantern_rel} the \textit{relation}~$(L_{i,j}^{k=j-1})$, respectively.  
Remark that ``$k=i-2$'', ``${k=i-1}$'', ``$k=j$'', and ``$k=j-1$''  in $(L_{i,j}^{k=i-2})$, $(L_{i,j}^{k=i-1})$, $(L_{i,j}^{k=j})$, and $(L_{i,j}^{k=j-1})$ are symbols and not variables, respectively. 
Similarly, we also call the relation~(4) (b) in Proposition~\ref{prop_pres_lmod} for $j=2n$ and $1\leq i\leq 2n-1$ the \textit{relation}~$(L_{i,2n}^{k=j})$ and for $j=2n+1$ and $1\leq i\leq 2n-1$ the \textit{relation}~$(L_{i,2n+1}^{k=j-1})$, respectively. 
As corollaries of Propositions~\ref{prop_pres_lmodp} and~\ref{prop_pres_lmod}, we have the following two corollaries. 

\begin{cor}\label{cor_pres_lmodp} 
The group $\LMp $ admits the presentation with generators $h_i$ for $1\leq i\leq 2n-1$ and $t_{i,j}$ for $1\leq i\leq j\leq 2n+1$, and the following defining relations: 
\begin{enumerate}
\item commutative relations
\begin{enumerate}
\item $h_i \rightleftarrows h_j$ \quad for $j-i\geq 3$,
\item $t_{i,j} \rightleftarrows t_{k,l}$ \quad for $j<k$, $k\leq i<j\leq l$, or $l<i$, 
\item $h_k \rightleftarrows t_{i,j}$ \quad for $k+2<i$, $i\leq k<k+2\leq j$, or $j<k$,
\end{enumerate}
\item relations among $h_i$'s and $t_{j,j+1}$'s: 
\begin{enumerate}
\item $h_i ^{\pm 1}t_{i,i+1}=t_{i+1,i+2}h_i^{\pm 1}$ \quad for $1\leq i\leq 2n-1$, 
\item $h_i ^{\varepsilon }h_{i+1}^{\varepsilon }t_{i,i+1}=t_{i+2,i+3}h_i^{\varepsilon }h_{i+1}^{\varepsilon }$ \quad for $1\leq i\leq 2n-2$ and $\varepsilon \in \{ 1, -1\}$,
\item $h_{i}^{\varepsilon }h_{i+1}^{\varepsilon }h_{i+2}^{\varepsilon }h_{i}^{\varepsilon }=h_{i+2}^{\varepsilon }h_{i}^{\varepsilon }h_{i+1}^{\varepsilon }h_{i+2}^{\varepsilon }$ \quad for $1\leq i\leq 2n-3$ and $\varepsilon \in \{ 1, -1\}$,
\item $h_ih_{i+1}t_{i,i+1}=t_{i,i+1}h_{i+1}h_i$ \quad for $1\leq i\leq 2n-2$,
\item $h_ih_{i+2}h_it_{i+1,i+2}^{-1}=t_{i+2,i+3}^{-1}h_{i+2}h_ih_{i+2}$ \quad for $1\leq i\leq 2n-3$,
\end{enumerate}
\item pentagonal relations $(P_{i,j,k,l,m})$ \quad for $1\leq i<j<k<l<m\leq 2n+2$,\\ 
i.e. $t_{j,m-1}^{-1}t_{k,m-1}t_{j,l-1}t_{i,k-1}t_{i,l-1}^{-1}=t_{i,l-1}^{-1}t_{i,k-1}t_{j,l-1}t_{k,m-1}t_{j,m-1}^{-1}$,
\item 
\begin{enumerate} 
\item $h_i^2=t_{i,i+2}$ \quad for $1\leq i\leq 2n-1$, 
\item the relations~$(L_{i,j}^{k=i-2})$\quad for $3\leq i<j\leq 2n+1$, \\ i.e. $h_{i-2}t_{i,j}h_{i-2}^{-1}=t_{i+1,j}t_{i-1,i}t_{i-2,j}t_{i-1,j}^{-1}t_{i-2,i}^{-1}$, 
\item the relations~$(L_{i,j}^{k=i-1})$\quad for $2\leq i<j-1\leq 2n$, \\ i.e. $h_{i-1}t_{i,j}h_{i-1}^{-1}=t_{i+2,j}t_{i-1,j}t_{i-1,i}t_{i+1,j}^{-1}t_{i-1,i+1}^{-1}$, 
\item the relations~$(L_{i,j}^{k=j})$\quad for $1\leq i<j\leq 2n-1$, \\ i.e. $h_{j}t_{i,j}h_{j}^{-1}=t_{i,j-1}t_{j,j+1}t_{i,j+2}t_{i,j+1}^{-1}t_{j,j+2}^{-1}$,
\item the relations~$(L_{i,j}^{k=j-1})$\quad for $2\leq i+1<j\leq 2n$, \\ i.e. $h_{j-1}t_{i,j}h_{j-1}^{-1}=t_{i,j-2}t_{j,j+1}t_{i,j+1}t_{i,j-1}^{-1}t_{j-1,j+1}^{-1}$, 
\item $t_{i,i}=t_{1,2n+1}=1$ \quad for $1\leq i\leq 2n+1$. 
\end{enumerate}
\end{enumerate}
\end{cor}

Remark that the relations~(1), (2), and (3) in Corollary~\ref{cor_pres_lmodp} coincide with the relations~(1), (2), and (3) in Theorem~\ref{thm_pres_lmodp}. 

\begin{cor}\label{cor_pres_lmod} 
The group $\LM $ admits the presentation which is obtained from the finite presentation for $\LMp $ in Corollary~\ref{cor_pres_lmodp} by adding generators $h_{2n}$, $t_{1,0}$, and $r$, and the following relations:
\begin{enumerate}
\item[(4)]
\begin{enumerate}
\item[(g)] the relations~$(L_{i,2n}^{k=j})$ \quad for $1\leq i\leq 2n-1$, \\ i.e. $h_{2n}t_{i,2n}h_{2n}^{-1}=t_{i,2n-1}t_{2n,2n+1}t_{1,i-1}t_{i,2n+1}^{-1}t_{1,2n-1}^{-1}$, 
\item[(h)] the relations~$(L_{i,2n+1}^{k=j-1})$ \quad for $1\leq i\leq 2n-1$, \\ i.e. $h_{2n}t_{i,2n+1}h_{2n}^{-1}=t_{i,2n-1}t_{1,2n}t_{1,i-1}t_{i,2n}^{-1}t_{1,2n-1}^{-1}$, 
\item[(i)] $h_{2n}t_{2n,2n+1}h_{2n}^{-1}=t_{1,2n}$, 
\item[(j)] $t_{1,0}=1$, 
\end{enumerate}
\item[(5)] relations among $r$ and $h_i$'s or $t_{i,j}$'s: 
\begin{enumerate}
\item $r^2=1$, 
\item $rh_i=h_{2n-i+1}r$ \quad for $1\leq i\leq 2n$, 
\item $rt_{i,j}=t_{2n-j+3,2n-i+3}r$ \quad for $2\leq i<j\leq 2n+1$, 
\item $rt_{1,j}=t_{1,2n-j+2}r$ \quad for $2\leq j\leq 2n$. 
\end{enumerate}
\end{enumerate}
\end{cor}

Remark that the relations (5) (a), (b), and (d) in Corollary~\ref{cor_pres_lmod} coincide with the relations~(5) (a), (b), and (d) in Theorem~\ref{thm_pres_lmod} and the relation (5) (c) in Theorem~\ref{thm_pres_lmod} coincides with the relation~(5) (c) in Corollary~\ref{cor_pres_lmod} for the case $j=i+1$. 

\begin{proof}[Proof of Corollary~\ref{cor_pres_lmodp}]
We consider the finite presentation which is obtained from the finite presentation for $\LMp $ in Proposition~\ref{prop_pres_lmodp} by adding the relation~(2) in Corollary~\ref{cor_pres_lmodp}. 
Since $\LMp $ admits the relations~(2) in Corollary~\ref{cor_pres_lmodp} by Lemma~\ref{lem_conj_rel} and \ref{lem_lift_W}, the group which is obtained from this finite presentation is also isomorphic to $\LMp $. 
The generating set of the presentation in Proposition~\ref{prop_pres_lmodp} coincides with one of  in Corollary~\ref{cor_pres_lmodp} and the relations~(1), (2), (3) (a), and (5) in Proposition~\ref{prop_pres_lmodp} coincide with the relations~(1), (3), (4) (a), and (f) in Corollary~\ref{cor_pres_lmodp}. 
The relations~(3) (b) and (c) in Proposition~\ref{prop_pres_lmodp} are equivqlent to the relations~(2) (d) and (e) in Corollary~\ref{cor_pres_lmodp} up to the relations~(1), (2) (a), and (b) in Corollary~\ref{cor_pres_lmodp}. 
The relation~(4) (a) for the case ``$k=i-1$ and $2\leq i=j-1\leq 2n$'' or ``$k=j-1$ and $2\leq i+1=j\leq 2n$'' in Proposition~\ref{prop_pres_lmodp} coincides with the relation~(2) (a) in Corollary~\ref{cor_pres_lmodp} and ones of other cases coincide with the relations~(4) (b)--(e) in Corollary~\ref{cor_pres_lmodp} (see before Corollary~\ref{cor_pres_lmodp}). 
Thus, by Teitze transformations, $\LMp $ has the presentation which is obtained from the presentation considered in the top of this proof by removing the relations~(3) (b) and (c) in Proposition~\ref{prop_pres_lmodp}, and this presentation coincides with the presentation in Corollary~\ref{cor_pres_lmodp}. 
Therefore, we have completed the proof of Corollary~\ref{cor_pres_lmodp}. 
\end{proof}

\begin{proof}[Proof of Corollary~\ref{cor_pres_lmod}]
We consider the finite presentation which is obtained from the finite presentation for $\LM $ in Proposition~\ref{prop_pres_lmod} by adding the relation~(2) in Corollary~\ref{cor_pres_lmod}. 
The group which is obtained from this finite presentation is isomorphic to $\LM $. 
The generating set of the presentation in Proposition~\ref{prop_pres_lmod} coincides with one of  in Corollary~\ref{cor_pres_lmod} and the relations~(3) (d), (e), (4) (b), (c), and (5) in Proposition~\ref{prop_pres_lmod} coincide with the relations~(5) (a), (b), (4) (g)--(i), (c) and (d), and (4) (j) in Corollary~\ref{cor_pres_lmod}. 
Thus, by Teitze transformations and an argument in the proof of Corollary~\ref{cor_pres_lmodp}, $\LMp $ has the presentation which is obtained from the presentation considered in the top of this proof by removing the relations~(3) (b) and (c) in Proposition~\ref{prop_pres_lmod}, and this presentation coincides with the presentation in Corollary~\ref{cor_pres_lmod}. 
Therefore, we have completed the proof of Corollary~\ref{cor_pres_lmod}. 
\end{proof}

Recall that the relation~$(T_{i,j})$ for $1\leq i<j\leq 2n+2$ is reviewed in Lemma~\ref{lem_t_ij-braid}. 
The relation~(4) (a) in Theorem~\ref{thm_pres_lmodb} is the relation~$(T_{i,j})$ for $1\leq i<j\leq 2n+1$ and the relations~(4) (b) and (c) in Theorem~\ref{thm_pres_lmod} are obtained from the relations~$(T_{1,2n+1})$, $(T_{1,2n+2})$, and $t_{1,2n+1}=t_{1,2n+2}=1$. 
The following four lemmas are proved in Section~\ref{section_proof_t_ij_pres}. 

\begin{lem}\label{lem_t_ij_pres1}
\begin{enumerate}
\item The relation~$(L_{i,j}^{k=j})$ for $1\leq i<j\leq 2n-1$ is equivalent to the relation~$(T_{i,j+2})$ up to the relations~(1) and (2)  in Theorem~\ref{thm_pres_lmodb}, $(T_{i^\prime ,j^\prime })$ for $i\leq i^\prime <j^\prime \leq j+2$ and $j^\prime -i^\prime <j-i+2$, and $t_{l,l}=1$ for $1\leq l\leq 2n+1$. 
\item The relation~$(L_{1,2n}^{k=j})$ is equivalent to the relation
\[
t_{1,2n}^{-n+1}t_{2n-1,2n}^{-n+1}\cdots t_{3,4}^{-n+1}t_{1,2}^{-n+1}(h_{2n}\cdots h_{2}h_{1})^{n+1}=1
\]
up to the relations~(1) and (2) in Theorem~\ref{thm_pres_lmod}, $(T_{i^\prime ,j^\prime })$ for $1\leq i^\prime <j^\prime \leq 2n+1$ and $j^\prime -i^\prime <2n+1$, and $t_{l,l}=t_{1,0}=1$ for $1\leq l\leq 2n-1$. 
\end{enumerate}
\end{lem}

We remark that Lemma~\ref{lem_t_ij_pres1} also give Lemma~\ref{lem_t_ij-braid}. 

\begin{lem}\label{lem_L_ij_k=j-1}
\begin{enumerate}
\item The relation~$(L_{i,j}^{k=j-1})$ for $2\leq i+1<j\leq 2n$ is obtained from the relations~(1) and (2) in Theorem~\ref{thm_pres_lmodb}, $(T_{j-1,j+1})$, and $(L_{i,j-1}^{k=j})$. 
\item The relation~$(L_{i,2n+1}^{k=j-1})$ for $1\leq i\leq 2n-1$ is obtained from the relations~(1) and (2) in Theorem~\ref{thm_pres_lmod}, $(T_{2n,2n+2})$, and $(L_{i,2n}^{k=j})$. 
\end{enumerate}
\end{lem}

\begin{lem}\label{lem_L_ij_k=i-2}
The relation~$(L_{i,j}^{k=i-2})$ for $3\leq i<j\leq 2n+1$ is equivalent to the relation~$(T_{i-2,j})$ up to the relations~(1) and (2) in Theorem~\ref{thm_pres_lmodb}, $(T_{i^\prime ,j^\prime })$ for $i-2\leq i^\prime <j^\prime \leq j$ and $j^\prime -i^\prime <j-i$, and $t_{l,l}=1$ for $1\leq l\leq 2n+1$. 
\end{lem}

\begin{lem}\label{lem_L_ij_k=i-1}
The relation~$(L_{i,j}^{k=i-1})$ for $2\leq i<j-1\leq 2n$ is obtained from the relations~(1) and (2) in Theorem~\ref{thm_pres_lmodb}, $(T_{i-1,i+1})$, and $(L_{i+1,j}^{k=i-2})$. 
\end{lem}

\begin{proof}[Proof of Theorem~\ref{thm_pres_lmodp}]
We consider the finite presentation which is obtained from the finite presentation for $\LMp $ in Corollary~\ref{cor_pres_lmodp} by adding the relations~(4) (a) for $j-i\geq 3$ and (b) in Theorem~\ref{thm_pres_lmodp}. 
Since $\LMp $ admits the relations~(4) (a) for $j-i\geq 3$ and (b) in Theorem~\ref{thm_pres_lmodp}, the group which is obtained from this finite presentation is also isomorphic to $\LMp $. 
The generators of this presentation in Theorem~\ref{thm_pres_lmodp} are obtained from the generators of the presentation in Corollary~\ref{cor_pres_lmodp} by removing $t_{i,i}$ for $1\leq i\leq 2n+1$, and the relations~(1)--(3) in Corollary~\ref{cor_pres_lmodp} coincide with the relations~(1)--(3) in Theorem~\ref{thm_pres_lmodp}. 
The relation~(4) (a) in Corollary~\ref{cor_pres_lmodp} coincides with the relations~(4) (a) in Theorem~\ref{thm_pres_lmodp} for $j-i=2$. 

By Lemmas~\ref{lem_t_ij_pres1}~(1) and \ref{lem_L_ij_k=i-2}, the relation~(4) (d) and (b) in Corollary~\ref{cor_pres_lmodp} are obtained from the relations in Theorem~\ref{thm_pres_lmodp} and $t_{l,l}=1$ for $1\leq l\leq 2n+1$, respectively. 
By these facts and Lemmas~\ref{lem_L_ij_k=j-1} and \ref{lem_L_ij_k=i-1}, the relation~(4) (e) and (c) in Corollary~\ref{cor_pres_lmodp} are also obtained from the relations in Theorem~\ref{thm_pres_lmodp} and $t_{l,l}=1$ for $1\leq l\leq 2n+1$.    
Thus, by applying Tietze transformations to the finite presentation for $\LMp $ which was considered in the top of this proof, we have the finite presentation for $\LMp $ which is obtained from the presentation in Theorem~\ref{thm_pres_lmodp} by adding the generators $t_{i,i}$ for $1\leq i\leq 2n+1$ and the relations $t_{i,i}=1$ for $1\leq i\leq 2n+1$. 
Since $t_{i,i}$ for $1\leq i\leq 2n+1$ does not appear in the relations in Theorem~\ref{thm_pres_lmodp}, by applying Tietze transformations again, $\LMp $ admits the presentation which is the result of removing the generators $t_{i,i}$ for $1\leq i\leq 2n+1$ and the relations $t_{i,i}=1$ for $1\leq i\leq 2n+1$ from this presentation for $\LMp $. 
Therefore we have completed the proof of Theorem~\ref{thm_pres_lmodp}.  
\end{proof}

We prepare the following two lemmas which are proved in Sections~\ref{section_tech_rel} and \ref{section_tech_rel19}, respectively. 

\begin{lem}\label{cor_tech_rel8}
For $1\leq i<j\leq 2n+1$ such that $j-i\geq 3$ is odd, the relation 
\begin{align*}
&t_{j-1,j}^{-\frac{j-i-3}{2}}\cdots t_{i+2,i+3}^{-\frac{j-i-3}{2}}t_{i,i+1}^{-\frac{j-i-3}{2}}(h_{j-2}\cdots h_{i+1}h_i)^{\frac{j-i+1}{2}}\\
=&t_{i,i+1}^{-\frac{j-i-3}{2}}t_{i+2,i+3}^{-\frac{j-i-3}{2}}\cdots t_{j-1,j}^{-\frac{j-i-3}{2}}(h_ih_{i+1}\cdots h_{j-2})^{\frac{j-i+1}{2}}
\end{align*}
is obtained from the relations~(1) and (2) in Theorem~\ref{thm_pres_lmodb}. 
\end{lem}

\begin{lem}\label{tech_rel19}
For $1\leq i<j\leq 2n+1$ such that $j-i\geq 4$ is even, the relation
\begin{align*}
&t_{j-2,j-1}^{-\frac{j-i-2}{2}}\cdots t_{i+2,i+3}^{-\frac{j-i-2}{2}}t_{i,i+1}^{-\frac{j-i-2}{2}}h_{j-2}\cdots h_{i+2}h_{i}h_{i}h_{i+2}\cdots h_{j-2}(h_{j-3}\cdots h_{i+1}h_{i})^{\frac{j-i}{2}}\\
=&t_{i+1,i+2}^{-\frac{j-i-2}{2}}t_{i+3,i+4}^{-\frac{j-i-2}{2}}\cdots t_{j-1,j}^{-\frac{j-i-2}{2}}h_{i}h_{i+2}\cdots h_{j-2}h_{j-2}\cdots h_{i+2}h_{i}(h_{i+1}h_{i+2}\cdots h_{j-2})^{\frac{j-i}{2}}
\end{align*}
is obtained from the relations~(1) and (2) in Theorem~\ref{thm_pres_lmodb}.  
\end{lem}

\begin{proof}[Proof of Theorem~\ref{thm_pres_lmod}]
We consider the finite presentation which is obtained from the finite presentation for $\LM $ in Corollary~\ref{cor_pres_lmod} by adding the relations~(4) (a) for $j-i\geq 3$, (b), and (c) in Theorem~\ref{thm_pres_lmod}. 
The group which is obtained from this finite presentation is isomorphic to $\LM $. 
Remark that the relations~(1)--(3), (4) (a), (5) (a), (b), (c) for $j=i+1$, and (d) in Corollary~\ref{cor_pres_lmod} coincide with the relations~(1)--(3), (4) (a) for $j-i=2$, and (5) (a), (b), (c), and (d) in Theorem~\ref{thm_pres_lmod}, respectively. 
By an argument in the proof of Theorem~\ref{thm_pres_lmodp}, the relation~(4) (b)--(e) in Corollary~\ref{cor_pres_lmod} are obtained from the relations in Theorem~\ref{thm_pres_lmod} and the relations $t_{l,l}=1$ for $1\leq l\leq 2n+1$. 
Thus, we will show that the relations~(4) (g), (h), (i) and (5) (c) for $2\leq i<j-1\leq 2n$ in Corollary~\ref{cor_pres_lmod} are obtained from the relations in Theorem~\ref{thm_pres_lmod} and the relations $t_{l,l}=t_{1,0}=1$ for $1\leq l\leq 2n+1$. 

First, for the relation~(5) (c) for $2\leq i<j-1\leq 2n$ in Corollary~\ref{cor_pres_lmod}, we have
\begin{eqnarray*}
&&rt_{i,j}r^{-1}\\
&\overset{\text{(4)(a)}}{\underset{}{=}}&\left\{
		\begin{array}{ll}
		rh_i^2r^{-1} \quad \text{for }j-i=2,\\
		r(t_{j-1,j}^{-\frac{j-i-3}{2}}\cdots t_{i+2,i+3}^{-\frac{j-i-3}{2}}t_{i,i+1}^{-\frac{j-i-3}{2}}(h_{j-2}\cdots h_{i+1}h_{i})^{\frac{j-i+1}{2}})r^{-1}\\ \text{for }2n-1\geq j-i \geq 3\text{ is odd},\\
		r(t_{j-2,j-1}^{-\frac{j-i-2}{2}}\cdots t_{i+2,i+3}^{-\frac{j-i-2}{2}}t_{i,i+1}^{-\frac{j-i-2}{2}}h_{j-2}\cdots h_{i+2}h_{i}h_{i}h_{i+2}\cdots h_{j-2}\\ \cdot (h_{j-3}\cdots h_{i+1}h_{i})^{\frac{j-i}{2}})r^{-1}\quad \text{for }2n-2\geq j-i \geq 4\text{ is even},\\	\end{array}
		\right. \\
&\overset{\text{(5)(b)}}{\underset{\text{(5)(c)}}{=}}&\left\{
		\begin{array}{ll}
		h_{2n-i+1}^2 \quad \text{for }j-i=2,\\
		t_{2n-j+3,2n-j+4}^{-\frac{j-i-3}{2}}\cdots t_{2n-i,2n-i+1}^{-\frac{j-i-3}{2}}t_{2n-i+2,2n-i+3}^{-\frac{j-i-3}{2}}\\
		\cdot (h_{2n-j+3}\cdots h_{2n-i}h_{2n-i+1})^{\frac{j-i+1}{2}}\quad \text{for }2n-1\geq j-i \geq 3\text{ is odd},\\
		t_{2n-j+4,2n-j+5}^{-\frac{j-i-2}{2}}\cdots t_{2n-i,2n-i+1}^{-\frac{j-i-2}{2}}t_{2n-i+2,2n-i+3}^{-\frac{j-i-2}{2}}\\
		\cdot h_{2n-j+3}\cdots h_{2n-i-1}h_{2n-i+1}h_{2n-i+1}h_{2n-i-1}\cdots h_{2n-j+3}\\ \cdot (h_{2n-j+4}\cdots h_{2n-i}h_{2n-i+1})^{\frac{j-i}{2}}\quad \text{for }2n-2\geq j-i \geq 4\text{ is even},\\	
		\end{array}
		\right. \\
&\overset{\text{Lem.~\ref{cor_tech_rel8}}}{\underset{\text{Lem.~\ref{tech_rel19}}}{=}}&\left\{
		\begin{array}{ll}
		h_{2n-i+1}^2 \quad \text{for }j-i=2,\\
		t_{2n-i+2,2n-i+3}^{-\frac{j-i-3}{2}}t_{2n-i,2n-i+1}^{-\frac{j-i-3}{2}}\cdots t_{2n-j+3,2n-j+4}^{-\frac{j-i-3}{2}}\\
		\cdot (h_{2n-i+1}h_{2n-i}\cdots h_{2n-j+3})^{\frac{j-i+1}{2}}\quad \text{for }2n-1\geq j-i \geq 3\text{ is odd},\\
		t_{2n-i+1,2n-i+2}^{-\frac{j-i-2}{2}}t_{2n-i-1,2n-i}^{-\frac{j-i-2}{2}}\cdots t_{2n-j+3,2n-j+4}^{-\frac{j-i-2}{2}}\\
		\cdot h_{2n-i+1}h_{2n-i-1}\cdots h_{2n-j+3}h_{2n-j+3}\cdots h_{2n-i-1}h_{2n-i+1}\\ \cdot (h_{2n-i}h_{2n-i-1}\cdots h_{2n-j+3})^{\frac{j-i}{2}}\quad \text{for }2n-2\geq j-i \geq 4\text{ is even},\\	
		\end{array}
		\right. \\
&\overset{\text{(4)(a)}}{\underset{}{=}}&t_{2n-j+3,2n-i+3}.
\end{eqnarray*}
Thus the relation~(5) (c) for $2\leq i<j-1\leq 2n$ in Corollary~\ref{cor_pres_lmod} is obtained from the relations in Theorem~\ref{thm_pres_lmod} (in particular, all the relations $rt_{i,j}r^{-1}=t_{2n-j+3,2n-i+3}$ are obtained from the relations in Theorem~\ref{thm_pres_lmod}).

By Lemma~\ref{tech_rel1}, the relation~(4) (i) is obtained from the relations in Theorem~\ref{thm_pres_lmod}. 
By Lemma~\ref{lem_t_ij_pres1}~(2), the relation~(4) (g) for $i=1$ (i.e. the relation~$(L_{1,2n}^{k=j})$) is obtained from the relations in Theorem~\ref{thm_pres_lmod} and $t_{l,l}=t_{1,0}=1$ for $1\leq l\leq 2n+1$. 
Using the relations~(5) in Theorem~\ref{thm_pres_lmod} and all $rt_{i,j}r^{-1}=t_{2n-j+3,2n-i+3}$, the conjugation of the relation~(4) (g) for $2\leq i\leq 2n$ (i.e. the relation~$(L_{i,2n}^{k=j})$) by $r$ coincides with the relation~$(L_{3,2n-i+3}^{k=i-2})$. 
Since $2n-i+3\leq 2n-2+3=2n+1$, by Lemma~\ref{lem_L_ij_k=i-2}, the relation~$(L_{3,2n-i+3}^{k=i-2})$ is also obtained from the relations in Theorem~\ref{thm_pres_lmod} and $t_{l,l}=1$ for $1\leq l\leq 2n+1$. 
Hence the relations~(4) (g) are obtained from the relations in Theorem~\ref{thm_pres_lmod} and $t_{l,l}=1$ for $1\leq l\leq 2n+1$. 
By this fact and Lemma~\ref{lem_L_ij_k=j-1}~(2), the relations~(4) (h) are obtained from the relations in Theorem~\ref{thm_pres_lmod} and $t_{l,l}=1$ for $1\leq l\leq 2n+1$. 

By an argument above, the relations~(4) (g), (h), (i) and (5) (c) for $2\leq i<j-1\leq 2n$ in Corollary~\ref{cor_pres_lmod} are obtained from the relations in Theorem~\ref{thm_pres_lmod} and the relations $t_{l,l}=t_{1,0}=1$ for $1\leq l\leq 2n+1$. 
Thus, by applying Tietze transformations to the finite presentation for $\LM $ which was considered in the top of this proof, we have the finite presentation for $\LM $ which is obtained from the presentation in Theorem~\ref{thm_pres_lmod} by adding the generators $t_{i,i}$ for $1\leq i\leq 2n+1$ and $t_{1,0}$ and the relations~$t_{i,i}=t_{1,0}=1$ for $1\leq i\leq 2n+1$. 
Since the generators $t_{i,i}$ for $1\leq i\leq 2n+1$ and $t_{1,0}$ do not appear in the relations in Theorem~\ref{thm_pres_lmod}, by applying Tietze transformations again, $\LM $ admits the presentation which is the result of removing the generators $t_{i,i}$ for $1\leq i\leq 2n+1$ and $t_{1,0}$ and the relations~$t_{i,i}=t_{1,0}=1$ for $1\leq i\leq 2n+1$ from this presentation for $\LM $. 
Therefore we have completed the proof of Theorem~\ref{thm_pres_lmod}.  
\end{proof}

\begin{proof}[Proof of Theorem~\ref{thm_pres_lmodb}]
We will apply Lemma~\ref{presentation_exact} to the exact sequence~(\ref{exact_lmodb}) in Proposition~\ref{prop_exact_lmodb} and the finite presentation for $\LMp $ in Theorem~\ref{thm_pres_lmodp}. 
The Dehn twist $t_{\partial D}$ coincides with $t_{1,2n+1}$ by the definition, is non-trivial in $\LMb $, and commutes with any element in $\LMb $. 
The generators which appear in the relations~(1)--(3), and (4) (a) in Theorem~\ref{thm_pres_lmodp} are supported on the complement of the disk $D$ in $\Sigma _0$ and these relations hold up to isotopies fixing $D$ pointwise.  
Hence we regard relations~(1)--(3), and (4) (a) in Theorem~\ref{thm_pres_lmodp} as the relations in $\LMb $ and these relations coincide with the relations~(1)--(3), and (4) (a) in Theorem~\ref{thm_pres_lmodb}. 
By applying Lemma~\ref{presentation_exact} to the exact sequence~(\ref{exact_lmodb}) in Proposition~\ref{prop_exact_lmodb} and the finite presentation for $\LMp $ in Theorem~\ref{thm_pres_lmodp}, we have the presentation for $\LMb $ whose generators are $h_{i}$ for $1\leq i\leq 2n-1$, $t_{i,j}$ for $1\leq i<j\leq 2n+1$, and $t_{D}$, and defining relations are the relations~(1)--(3), and (4) (a) in Theorem~\ref{thm_pres_lmodb}, $t_D=t_{1,2n+1}$, $t_{D}\rightleftarrows h_i$ for $1\leq i\leq 2n-1$, $t_{D}\rightleftarrows t_{i,j}$ for $1\leq i<j\leq 2n+1$. 
By Tietze transformations, $\LMb $ admits the presentation which is the result of replacing $t_D$ by $t_{1,2n+1}$ in the commutative relations above and removing the generator $t_D$ and the relation $t_D=t_{1,2n+1}$ from this presentation for $\LM $. 
The relations $t_{1,2n+1}\rightleftarrows h_i$ for $1\leq i\leq 2n-1$ and $t_{1,2n+2}\rightleftarrows t_{i,j}$ for $1\leq i<j\leq 2n+1$ coincide with the relations~(1) (c) and (b) in Theorem~\ref{thm_pres_lmodb}, respectively. 
Therefore we have completed the proof of Theorem~\ref{thm_pres_lmodb}.   
\end{proof}

\subsection{The first homology groups of the liftable mapping class groups}\label{section_abel-lmod}

Recall the the integral first homology group $H_1(G)$ of a group $G$ is isomorphic to the abelianization of $G$. 
In this section, we prove Theorem~\ref{thm_abel_lmod} by using the finite presentations in Theorems~\ref{thm_pres_lmodb}, \ref{thm_pres_lmodp}, and \ref{thm_pres_lmod}. 
For conveniences, we denote the equivalence class in $H_1(\LMb )$ (resp. $H_1(\LMp )$ and $H_1(\LM )$) of an element $h$ in $\LMb$ (resp. $\LMp $ and $\LM $) by $h$. 

\begin{proof}[Proof of Theorem~\ref{thm_abel_lmod} for $\LMb $]
We will calculate $H_1(\LMb )$ by using the finite presentation for $\LMb $ in Theorem~\ref{thm_pres_lmodb}. 
When $n=1$, the presentations for $\LMb $ in Theorem~\ref{thm_pres_lmodb} has the generators $h_1$, $t_{1,2}$, $t_{2,3}$, and $t_{1,3}$ and the relation~(1) (b) $t_{i,i+1}\rightleftarrows t_{1,3}$ $(i=1,2)$, (c) $h_1\rightleftarrows t_{1,3}$, (2) (a) $h_1^{\pm 1}t_{1,2}=t_{2,3}h_1^{\pm 1}$, and (4) (a) $t_{1,3}=h_1^2$. 
Hence, as a presentation for an abelian group (i.e. we omit commutative relations), we have
\begin{eqnarray*}
H_1(\mathrm{LMod}_3^1)
&\cong &\left< h_1, t_{1,2}, t_{2,3}, t_{1,3} \middle| t_{1,2}=t_{2,3}, t_{1,3}=h_1^2\right> \\
&\cong &\left< h_1, t_{1,2} \middle| \right> \\
&\cong &\Z [h_1]\oplus \Z [t_{1,2}].
\end{eqnarray*}
Thus, $H_1(\mathrm{LMod}_3^1)$ is isomorphic to $\Z ^2$. 

Assume that $n\geq 2$. 
The relations~(2) (a) and (b) in Theorem~\ref{thm_pres_lmodb} are equivalent to the relations $t_{i,i+1}=t_{i+1,i+2}$ $(1\leq i\leq 2n-1)$ in $H_1(\LMb )$. 
The relations~(2) (c) and (2) (e) in Theorem~\ref{thm_pres_lmodb} are equivalent to the relations $h_{i}=h_{i+2}$ $(1\leq i\leq 2n-3)$ in $H_1(\LMb )$ up to the relations $t_{i,i+1}=t_{i+1,i+2}$ $(1\leq i\leq 2n-1)$ in $H_1(\LMb )$.  
The relations~(3) and (2)~(d) in Theorem~\ref{thm_pres_lmodb} are equivalent to the trivial relation in $H_1(\LMb )$. 
By the relation~(4) (a) in Theorem~\ref{thm_pres_lmodb}, $t_{i,j}$ for $1\leq i<j\leq 2n+1$ and $j-i\geq 2$ is a product of $h_{i}$ for $1\leq i\leq 2n-1$ and $t_{i,i+1}$ for $1\leq i\leq 2n$. 
Hence, as a presentation for an abelian group, we have
\begin{eqnarray*}
&&H_1(\LMb )\\
&\cong &\left< h_i\ (1\leq i\leq 2n-1), t_{i,j}\ (1\leq i<j\leq 2n+1) \middle| \text{the relations~(2), (4) (a)}\right> \\
&\cong &\left< h_i\ (1\leq i\leq 2n-1), t_{i,i+1}\ (1\leq i\leq 2n) \middle| \text{the relations~(2)}\right> \\
&\cong &\left< h_1, h_2, t_{1,2} \middle| \right> \\
&\cong &\Z [h_1]\oplus \Z [h_2]\oplus \Z [t_{1,2}].
\end{eqnarray*}
Therefore,  $H_1(\LMb )$ for $n\geq 2$ is isomorphic to $\Z ^3$ and we have completed the proof of Theorem~\ref{thm_abel_lmod} for $\LMb $.   
\end{proof}

\begin{proof}[Proof of Theorem~\ref{thm_abel_lmod} for $\LMp $]
We will calculate $H_1(\LMp )$ by using the finite presentations for $\LMp $ in Theorem~\ref{thm_pres_lmodp}. 
By Theorem~\ref{thm_pres_lmodp}, the presentation for $\LMp $ is obtained from the presentation for $\LMb $ in Theorem~\ref{thm_pres_lmodb} by adding the relation $t_{1,2n+1}=1$. 
When $n=1$, as a presentation for an abelian group, we have
\begin{eqnarray*}
H_1(\mathrm{LMod}_{4,\ast })
&\cong &\left< h_1, t_{1,2}, t_{2,3}, t_{1,3} \middle| t_{1,2}=t_{2,3}, t_{1,3}=h_1^2, t_{1,3}=1\right> \\
&\cong &\left< h_1, t_{1,2} \middle| h_1^2=1\right> \\
&\cong &\Z [t_{1,2}]\oplus \Z _2[h_1].
\end{eqnarray*}
Thus, $H_1(\mathrm{LMod}_{4,\ast })$ is isomorphic to $\Z \oplus \Z _2$. 

Assume that $n\geq 2$. 
By an argument similar to the proof of Theorem~\ref{thm_abel_lmod} for $\LMb $, the relations~(1), (2), and (3) in Theorem~\ref{thm_pres_lmodp} are equivalent to the relations $t_{i,i+1}=t_{i+1,i+2}$ for $1\leq i\leq 2n-1$ and $h_{i}=h_{i+2}$ for $1\leq i\leq 2n-1$ in $H_1(\LMp )$. 
By the relation~(4) (a) in Theorem~\ref{thm_pres_lmodp}, we have $t_{1,2n+1}=t_{1,2}^{-n(n-1)}h_1^{n(n+1)}h_2^{n(n-1)}$ in $H_1(\LMp )$. 
Up to these relations, the relation~(4) (b) in Theorem~\ref{thm_pres_lmodp} is equivalent to the relation $t_{1,2}^{-n(n-1)}h_1^{n(n+1)}h_2^{n(n-1)}=1$ in $H_1(\LMp )$. 
Hence, as a presentation for an abelian group, we have
\begin{eqnarray*}
H_1(\LMp )
&\cong &\left< h_1, h_2, t_{1,2} \middle| t_{1,2}^{-n(n-1)}h_1^{n(n+1)}h_2^{n(n-1)}=1 \right> \\
&\cong &\left< h_1, h_2, t_{1,2}, X \middle| (h_1^2X^{n-1})^n=1, X=t_{1,2}^{-1}h_1h_2 \right> \\
&\cong &\left< h_1, h_2, X \middle| (h_1^2X^{n-1})^n=1 \right> .
\end{eqnarray*}
When $n$ is odd, we have
\begin{eqnarray*}
\left< h_1, h_2, X \middle| (h_1^2X^{n-1})^n=1 \right> 
&\cong &\left< h_1, h_2, X, Y \middle| Y^{2n}=1, Y=h_1X^{\frac{n-1}{2}}\right> \\
&\cong &\left< h_2, X, Y \middle| Y^{2n}=1\right> \\
&\cong &\Z [h_2]\oplus \Z [X]\oplus \Z _{2n}[Y], 
\end{eqnarray*}
where $X=t_{1,2}^{-1}h_1h_2$ and $Y=h_1X^{\frac{n-1}{2}}=t_{1,2}^{-\frac{n-1}{2}}h_1^{\frac{n+1}{2}}h_2^{\frac{n-1}{2}}$. 
Thus, $H_1(\LMp )$ for odd $n\geq 3$ is isomorphic to $\Z ^2\oplus \Z _{2n}$. 
When $n$ is even, we have
\begin{eqnarray*}
\left< h_1, h_2, X \middle| (h_1^2X^{n-1})^n=1 \right> 
&\cong &\left< h_1, h_2, X, Y \middle| (Y^2X^{-1})^n=1, Y=h_1X^{\frac{n}{2}} \right> \\
&\cong &\left< h_2, X, Y \middle| (Y^2X^{-1})^n=1 \right> \\
&\cong &\left< h_2, X, Y, Z \middle| Z^n=1, Z=Y^2X^{-1} \right> \\
&\cong &\left< h_2, Y, Z \middle| Z^n=1 \right> \\
&\cong &\Z [h_2]\oplus \Z [Y]\oplus \Z _{n}[Z], 
\end{eqnarray*}
where $Y=h_1X^{\frac{n}{2}}=t_{1,2}^{-\frac{n}{2}}h_1^{\frac{n+2}{2}}h_2^{\frac{n}{2}}$ and $Z=Y^2X^{-1}=t_{1,2}^{-n+1}h_1^{n+1}h_2^{n-1}$. 
Therefore, $H_1(\LMp )$ for even $n$ is isomorphic to $\Z ^2\oplus \Z _{n}$ and we have completed the proof of Theorem~\ref{thm_abel_lmod} for $\LMp $.   
\end{proof}

We give an another proof of Theorem~1.1 in \cite{Ghaswala-Winarski1-1} as follows. 

\begin{proof}[Proof of Theorem~\ref{thm_abel_lmod} for $\LM $]
We will calculate $H_1(\LM )$ by using the finite presentations for $\LM $ in Theorem~\ref{thm_pres_lmod}. 
By an argument similar to the proof of Theorem~\ref{thm_abel_lmod} for $\LMb $, the relations~(1), (2), and (3) in Theorem~\ref{thm_pres_lmod} are equivalent to the relations $t_{i,i+1}=t_{i+1,i+2}$ for $1\leq i\leq 2n-1$ and $h_{i}=h_{i+2}$ for $1\leq i\leq 2n-1$ in $H_1(\LM )$. 
Since the relations~(5) (b) and (c) in Theorem~\ref{thm_pres_lmod} are equivalent to the relations $h_i=h_{2n-i+1}$ for $1\leq i\leq 2n$ and $t_{i,i+1}=t_{2n-i+2,2n-i+3}$ for $2\leq i\leq 2n$, we have $h_i=h_j$ for any $1\leq i<j\leq 2n$ in $H_1(\LM )$. 
By an argument in the proof of Theorem~\ref{thm_abel_lmod} for $\LMp $, the relation~(4) (b) in Theorem~\ref{thm_pres_lmod} is equivalent to the relation
\[
1=t_{1,2}^{-n(n-1)}h_1^{n(n+1)}h_2^{n(n-1)}=t_{1,2}^{-n(n-1)}h_1^{2n^2}=(t_{1,2}^{-n+1}h_{1}^{2n})^n
\]
in $H_1(\LM )$.
Since we have $t_{1,2n}=t_{2n+1,2n+2}$ and $t_{2n+1,2n}=h_{2n}\cdots h_2h_1t_{1,2}h_1^{-1}h_2^{-1}\cdots h_{2n}^{-1}$ in $\LM $, the relation $t_{1,2n}=t_{1,2}$ holds in $H_1(\LM )$. 
Thus, the relation~(4) (c) in Theorem~\ref{thm_pres_lmod} is equivalent to the relation 
\[
1=t_{1,2}^{-(n-1)(n+1)}h_1^{2n(n+1)}=(t_{1,2}^{-n+1}h_{1}^{2n})^{n+1}
\]
in $H_1(\LM )$. 
By these two relations in $H_1(\LM )$, we obtain the relation $t_{1,2}^{-n+1}h_{1}^{2n}=1$ in $H_1(\LM )$. 
We can see that these two relations are obtained from the relation $t_{1,2}^{-n+1}h_{1}^{2n}=1$ in $H_1(\LM )$.  
We remark that when $n=1$, the relation $t_{1,2}^{-n+1}h_{1}^{2n}=1$ in $H_1(\LM )$ is equivalent to the relation $h_1^2=1$. 

The relation~(5) (d) in Theorem~\ref{thm_pres_lmod} is equivalent to the relation $t_{1,j}=t_{1,2n-j+2}$ for $2\leq j\leq 2n$ in $H_1(\LM )$. 
We remark that when $n=1$, this relation is equivalent to the trivial relation in $H_{1}(\LM )$.  
By the relation~(4) (a), we have 
\[
t_{1,j}=\left\{
		\begin{array}{ll}
		t_{2i-1,2i}^{-(i-2)}\cdots t_{3,4}^{-(i-2)}t_{1,2}^{-(i-2)}(h_{2i-2}\cdots h_{2}h_{1})^{i}\quad \text{for }j=2i\geq 4,\\
		t_{2i-1,2i}^{-(i-1)}\cdots t_{3,4}^{-(i-1)}t_{1,2}^{-(i-1)}h_{2i-1}\cdots h_{3}h_{1}h_{1}h_{3}\cdots h_{2i-1}(h_{2i-2}\cdots h_{2}h_{1})^{i}\\ \text{for }j=2i+1\geq 5		
		\end{array}
		\right.\\
\]
in $\LM $, and the relations $t_{1,j}=t_{1,2n-j+2}$ for $2\leq j\leq 2n$ in $H_1(\LM )$ are equivalent to the following relations:  
\begin{itemize}
\item[(A)] $t_{1,2}=t_{1,2}^{-n(n-2)}h_{1}^{2n(n-1)}\Leftrightarrow 1=t_{1,2}^{-(n-1)^2}h_{1}^{2n(n-1)}=(t_{1,2}^{-n+1}h_{1}^{2n})^{n-1}$\\ for $j=2$, 
\item[(B)] $h_1^2=t_{1,2}^{-(n-1)(n-2)}h_1^{2(n-1)^2}\Leftrightarrow 1=t_{1,2}^{-(n-1)(n-2)}h_1^{2n(n-2)}=(t_{1,2}^{-n+1}h_{1}^{2n})^{n-2}$\\ for $j=3$, 
\item[(C)] $t_{1,2}^{-i(i-2)}h_{1}^{2i(i-1)}=t_{1,2}^{-(n-i-1)(n-i+1)}h_{1}^{2(n-i)(n-i+1)}$\\
$\Leftrightarrow 1=t_{1,2}^{-n^2+2in-2i+1}h_1^{2(n^2-(2i-1)n)}=(t_{1,2}^{-n+1}h_{1}^{2n})^{n-2i+1}$\\ for $j=2i$ and $2\leq i\leq n$, and 
\item[(D)] $t_{1,2}^{-i(i-1)}h_{1}^{2i^2}=t_{1,2}^{-(n-i)(n-i-1)}h_{1}^{2(n-i)^2}\\ 
\Leftrightarrow 1=t_{1,2}^{-n^2+(2i+1)n-2i}h_1^{2n(n-2i)}=(t_{1,2}^{-n+1}h_{1}^{2n})^{n-2i}$\\ for $j=2i+1$ and $2\leq i\leq n-1$
\end{itemize}
 in $H_1(\LM )$. 
Hence the relations~(A)--(D) in $H_1(\LM )$ above are obtained from the relation $t_{1,2}^{-n+1}h_{1}^{2n}=1$ in $H_1(\LM )$. 

When $n=1$, by an argument above, we have 
\begin{eqnarray*}
H_1(\mathrm{LMod}_4 )
&\cong & \left< h_1, t_{1,2}, r \middle| r^2=1, h_{1}^{2}=1 \right> \\
&\cong & \Z [t_{1,2}]\oplus \Z _2[h_1]\oplus \Z _2[r].
\end{eqnarray*}
as a presentation for an abelian group. 
Thus, $H_1(\mathrm{LMod}_4 )$ is isomorphic to $\Z \oplus \Z _2^2$. 

Assume $n\geq 2$. 
By an argument above, as a presentation for an abelian group, we have 
\begin{eqnarray*}
H_1(\LM )
&\cong &\left< h_1, t_{1,2}, r \middle| r^2=1, t_{1,2}^{-n+1}h_{1}^{2n}=1 \right> \\
&\cong &\left< h_1, t_{1,2}, r, X \middle| r^2=1, t_{1,2}X^n=1, X=t_{1,2}^{-1}h_1^2 \right> \\
&\cong &\left< h_1,  r, X \middle| r^2=1, X^{n-1}h_{1}^2=1 \right> .
\end{eqnarray*}
When $n$ is odd, we have
\begin{eqnarray*}
&&\left< h_1,  r, X \middle| r^2=1, X^{n-1}h_{1}^2=1 \right> \\
&\cong &\left< h_1,  r,  X, Y \middle| r^2=1, Y^2=1, Y=X^{\frac{n-1}{2}}h_1 \right> \\
&\cong &\left< r,  X, Y \middle| r^2=1, Y^2=1 \right> \\
&\cong & \Z [X]\oplus \Z _2[r]\oplus \Z _2[Y],
\end{eqnarray*}
where $X=t_{1,2}^{-1}h_1^2$ and $Y=X^{\frac{n-1}{2}}h_1=t_{1,2}^{-\frac{n-1}{2}}h_1^{n}$. 
Thus $H_1(\LM )$ for odd $n\geq 3$ is isomorphic to $\Z \oplus \Z _2^2$.  
When $n$ is even, we have
\begin{eqnarray*}
&&\left< h_1,  r, X \middle| r^2=1, X^{n-1}h_{1}^2=1 \right> \\
&\cong &\left< h_1,  r,  X, Y \middle| r^2=1, X^{-1}Y^2=1, Y=X^{\frac{n}{2}}h_1 \right> \\
&\cong &\left< r,  X, Y \middle| r^2=1, X^{-1}Y^2=1 \right> \\
&\cong &\left< r,  Y \middle| r^2=1 \right> \\
&\cong & \Z [Y]\oplus \Z _2[r],
\end{eqnarray*}
where $Y=X^{\frac{n}{2}}h_1=t_{1,2}^{-\frac{n}{2}}h_1^{n+1}$. 
Therefore, $H_1(\LM )$ for even $n$ is isomorphic to $\Z \oplus \Z _2$ and we have completed the proof of Theorem~\ref{thm_abel_lmod} for $\LM $.   
\end{proof}

\subsection{Proof of Lemma~\ref{cor_tech_rel8}}\label{section_tech_rel}

In this section, we prove Lemma~\ref{cor_tech_rel8} and introduce some useful relations in liftable mapping class groups which are obtained from relations in Theorems~\ref{thm_pres_lmodb}, \ref{thm_pres_lmodp}, and \ref{thm_pres_lmod} and are also used in Sections~\ref{section_tech_rel19} and \ref{section_proof_t_ij_pres}.  
After this, throughout Section~\ref{section_lmod}, a label of a relation indicates one of relations in Theorems~\ref{thm_pres_lmodb}, \ref{thm_pres_lmodp}, and \ref{thm_pres_lmod}. 
Remark that $t_{2n+1,2n+2}=t_{1,2n}$. 
The next lemma is immediately obtained by using the relations~(1), (2) (a), (b), and $t_{2n+1,2n+2}=t_{1,2n}$. 

\begin{lem}\label{tech_rel0}
\begin{enumerate}
\item[(A)] Up to the relations~(1), (2) (a), (b), and $t_{2n+1,2n+2}=t_{1,2n}$, the relation~(2)~(d), i.e. $h_ih_{i+1}t_{i,i+1}=t_{i,i+1}h_{i+1}h_i$ for $1\leq i\leq 2n-1$, is equivalent to the following relations:
\begin{align*}
&h_ih_{i+1}t_{i,i+1}=h_{i+1}h_it_{i+2,i+3}, \quad t_{i+2,i+3}h_ih_{i+1}=t_{i,i+1}h_{i+1}h_i, \\
&h_ih_{i+1}=h_{i+1}h_it_{i,i+1}^{-1}t_{i+2,i+3}, \quad h_ih_{i+1}= t_{i,i+1}t_{i+2,i+3}^{-1}h_{i+1}h_i, \\
&h_{i+1}h_i=h_ih_{i+1}t_{i,i+1}t_{i+2,i+3}^{-1}, \quad  h_{i+1}h_i=t_{i,i+1}^{-1}t_{i+2,i+3}h_ih_{i+1}. 
\end{align*}
\item[(B)] Up to the relations~(1), (2) (a), (b), and $t_{2n+1,2n+2}=t_{1,2n}$, the relation~(2) (e), i.e. $h_ih_{i+2}h_it_{i+1,i+2}^{-1}=t_{i+2,i+3}^{-1}h_{i+2}h_ih_{i+2}$ for $1\leq i\leq 2n-2$, is equivalent to the following relations:
\begin{align*}
&t_{i+1,i+2}^{-1}h_ih_{i+2}h_i=t_{i+2,i+3}^{-1}h_{i+2}h_ih_{i+2}, \quad h_ih_{i+2}h_it_{i+1,i+2}^{-1}=h_{i+2}h_ih_{i+2}t_{i+2,i+3}^{-1},\\
&t_{i+1,i+2}^{-1}h_ih_{i+2}h_i=h_{i+2}h_ih_{i+2}t_{i+2,i+3}^{-1}, \quad h_ih_{i+2}h_i=t_{i+2,i+3}^{-1}h_{i+2}h_ih_{i+2}t_{i+1,i+2},\\
&h_ih_{i+2}h_i=t_{i+1,i+2}h_{i+2}h_ih_{i+2}t_{i+2,i+3}^{-1}, \quad h_ih_{i+2}h_i=t_{i+1,i+2}t_{i+2,i+3}^{-1}h_{i+2}h_ih_{i+2},\\
&h_ih_{i+2}h_i=h_{i+2}h_ih_{i+2}t_{i+2,i+3}^{-1}t_{i+1,i+2}. 
\end{align*} 
\end{enumerate}
\end{lem}

We also call the relations in~(A) and (B) of Lemma~\ref{tech_rel0} the relations~(2) (d) and (2) (e), respectively. 

\begin{lem}\label{tech_rel2}
For $1\leq i<j\leq 2n+2$ in the case of $\LM $, $1\leq i<j\leq 2n+1$ in the case of $\LMb $, and $1\leq l\leq 2n+1$ such that $j-i\geq 3$ is odd and $l-i$ is even, the relation
\[
(h_{j-2}\cdots h_{i+1}h_i)^{\frac{j-i+1}{2}}\rightleftarrows t_{l,l+1}
\]
is obtained from the relations~(1), (2), and $t_{2n+1,2n+2}=t_{1,2n}$. 
\end{lem}

\begin{proof}
In the cases of $l\leq i-2$ or $j+1\leq l$, the relation $(h_{j-2}\cdots h_{i+1}h_i)^{\frac{j-i+1}{2}}\rightleftarrows t_{l,l+1}$ is obtained from the relations~(1) (c). 
Since $j-i$ is odd and $l-i$ is even, $j-l$ is odd and we have $l\not \in \{ i+1,\ j\}$. 
Hence we suppose that $i\leq l\leq j-1$. 
By the relations~(2) (a), (b), 
we have the relations $h_lt_{l,l+1}=t_{l+1,l+2}h_l$, $h_{l-1}h_{l-2}t_{l,l+1}=t_{l-2,l-1}h_{l-1}h_{l-2}$, $h_{2n}t_{2n,2n+1}=t_{1,2n}h_{2n}$, and $h_{2n}h_{2n-1}t_{1,2n}=t_{2n-1,2n}h_{2n}h_{2n-1}$. 
Thus, by also using the relations~(1) (c), 
for $j\leq 2n+1$ and $i\leq l\leq j-1$ with even $l-i$, we have
\[
(h_{j-2}\cdots h_{i+1}h_i)t_{l,l+1}=\left\{
		\begin{array}{ll}
		t_{j-1,j}(h_{j-2}\cdots h_{i+1}h_i) & \text{for }l=i,\\
		t_{l-2,l-1}(h_{j-2}\cdots h_{i+1}h_i) &  \text{for }i+2\leq l \leq j-1,
		\end{array}
		\right. 
\]
and when $j=2n+2$, for $i\leq l\leq 2n+1$ with even $l-i$, we have
\[
(h_{2n}\cdots h_{i+1}h_i)t_{l,l^\prime }=\left\{
		\begin{array}{ll}
		t_{1,2n}(h_{2n}\cdots h_{i+1}h_i) & \text{for }(l,l^\prime )=(i,i+1),\\
		t_{l-2,l-1}(h_{2n}\cdots h_{i+1}h_i) &  \text{for }i+2\leq l\leq 2n-1\\
		 &\text{and }l^\prime =l+1,\\
		t_{2n-1,2n}(h_{2n}\cdots h_{i+1}h_i) &  \text{for }(l,l^\prime )=(1,2n).\\
		\end{array}
		\right. 
\]
We can see that the subgroup of $\LM $ or $\LMp $ generated by $h_{j-2}\cdots h_{i+1}h_i$ acts on the $\frac{j-i+1}{2}$ points set $\{ t_{i,i+1},\ t_{i+2,i+3},\ \dots ,\ t_{j-1,j-2}\}$ by conjugations and the action of $h_{j-2}\cdots h_{i+1}h_i$ has order $\frac{j-i+1}{2}$. 
Therefore the relation $(h_{j-2}\cdots h_{i+1}h_i)^{\frac{j-i+1}{2}}t_{l,l+1}=t_{l,l+1}(h_{j-2}\cdots h_{i+1}h_i)^{\frac{j-i+1}{2}}$ holds in $\LMb $ for $j\leq 2n+1$ and $\LM $ and is obtained from the relations~(1), (2), and $t_{2n+1,2n+2}=t_{1,2n}$. 
We have completed the proof of Lemma~\ref{tech_rel2}. 
\end{proof}

\begin{lem}\label{tech_rel3}
For $1\leq i<j\leq 2n+2$ in the case of $\LM $, $1\leq i<j\leq 2n+1$ in the case of $\LMb $, $1\leq l\leq 2n+1$, and a product $A$ of generators for $\LM $ (resp. for $\LMb $) in Theorem~\ref{thm_pres_lmod} (resp. in Theorem~\ref{thm_pres_lmod}) commuting with $t_{l+1,l+2}$, we assume that $j-i\geq 4$ and $l-i$ are even and $l\not= j$. 
Then, the relation
\[
h_{j-2}\cdots h_{i+2}h_iAh_ih_{i+2}\cdots h_{j-2}\rightleftarrows t_{l,l+1}
\]
is obtained from the relations~(1), (2), and $t_{2n+1,2n+2}=t_{1,2n}$. 
\end{lem}

\begin{proof}
In the cases of $l\leq i-2$ or $j+1\leq l$, the relation $h_{j-2}\cdots h_{i+2}h_iAh_ih_{i+2}\cdots h_{j-2}\rightleftarrows t_{l,l+1}$ is obtained from the relations~(1) (c) and $t_{2n+1,2n+2}=t_{1,2n}$. 
Since $j-i$ and $l-i$ are even, $j-l$ is also even and we have $l\not \in \{ i-1,\ j-1\}$. 
Hence we suppose that $i\leq l\leq j-2$. 
Then, up to the relation $t_{2n+1,2n+2}=t_{1,2n}$, we have
\begin{eqnarray*}
&&h_{j-2}\cdots h_{i+2}h_iAh_ih_{i+2}\cdots h_{j-2}\cdot \underset{\leftarrow }{\underline{t_{l,l+1}}}\\
&\overset{\text{(1)(c)}}{\underset{}{=}}&h_{j-2}\cdots h_{i+2}h_ih_iAh_{i+2}\cdots h_{l-2}\underline{h_{l}\cdot t_{l,l+1}}\cdot h_{l+2}\cdots h_{j-2}\\
&\overset{\text{(2)(a)}}{\underset{}{=}}&h_{j-2}\cdots h_{i+2}h_iAh_ih_{i+2}\cdots h_{l-2}\cdot \underset{\leftarrow }{\underline{t_{l+1,l+2}}}\cdot h_{l}h_{l+2}\cdots h_{j-2}\\
&\overset{\text{(1)}}{\underset{}{=}}&h_{j-2}\cdots h_{l+2}\underline{h_{l}\cdot t_{l+1,l+2}}\cdot h_{l-2}\cdots h_{i+2}h_iAh_ih_{i+2}\cdots h_{j-2}\\
&\overset{\text{(2)(a)}}{\underset{}{=}}&h_{j-2}\cdots h_{l+2}\cdot \underset{\leftarrow }{\underline{t_{l,l+1}}}\cdot h_{l}h_{l-2}\cdots h_{i+2}h_iAh_ih_{i+2}\cdots h_{j-2}\\
&\overset{\text{(1)(c)}}{\underset{}{=}}&t_{l,l+1}\cdot h_{j-2}\cdots h_{i+2}h_iAh_ih_{i+2}\cdots h_{j-2}.
\end{eqnarray*}
Therefore, the relation $h_{j-2}\cdots h_{i+2}h_iAh_ih_{i+2}\cdots h_{j-2}\rightleftarrows t_{l,l+1}$ is obtained from the relations~(1) (c), (2) (a), and $t_{2n+1,2n+2}=t_{1,2n}$, and we have completed the proof of Lemma~\ref{tech_rel3}. 
\end{proof}

As a corollary of Lemma~\ref{tech_rel3}, we have the following corollary. 

\begin{cor}\label{cor_tech_rel3}
For $1\leq i<j\leq 2n+2$ in the case of $\LM $, $1\leq i<j\leq 2n+1$ in the case of $\LMb $, and $1\leq l\leq 2n+1$ such that $j-i\geq 4$ and $l-i$ are even and $l\not= j$, the relation
\[
h_{j-2}\cdots h_{i+2}h_ih_ih_{i+2}\cdots h_{j-2}\rightleftarrows t_{l,l+1}
\]
is obtained from the relations~(1), (2), and $t_{2n+1,2n+2}=t_{1,2n}$. 
\end{cor}

\begin{lem}\label{tech_rel4}
For $1\leq i<j\leq 2n+2$ in the case of $\LM $, $1\leq i<j\leq 2n+1$ in the case of $\LMb $, $1\leq l\leq 2n-1$, and a product $A$ of generators for $\LM $ (resp. for $\LMb $) in Theorem~\ref{thm_pres_lmod} (resp. in Theorem~\ref{thm_pres_lmod}), we assume that $j-i\geq 4$ and $l-i$ are even and $l\not \in \{ i-2,\ j\} $, and each generator which appears in $A$ commutes with $h_{l+1}t_{l+2,l+3}^{-1}h_{l+2}t_{l+1,l+2}$. 
Then, the relation
\[
h_{j-2}\cdots h_{i+2}h_iAh_ih_{i+2}\cdots h_{j-2}\rightleftarrows h_{l+1}h_l
\]
is obtained from the relations~(1), (2), and $t_{2n+1,2n+2}=t_{1,2n}$. 
\end{lem}

\begin{proof}
In the cases of $l+1\leq i-3$ or $j+1\leq l$, the relation $h_{j-2}\cdots h_{i+2}h_iAh_ih_{i+2}\cdots h_{j-2}\rightleftarrows h_{l+1}h_l$ is obtained from the relation~(1). 
Since $j-i$ and $l-i$ are even, $j-l$ is also even and we have $l\not \in \{ i-3,\ i-1,\ j-1\}$. 
Hence we suppose that $i\leq l\leq j-2$. 
By Lemma~\ref{tech_rel0}, up to the relations~(2) (a), (b), and $t_{2n+1,2n+2}=t_{1,2n}$, the relation~(2) (d) is equivalent to the relations $h_lh_{l+1}=h_{l+1}h_lt_{l,l+1}^{-1}t_{l+2,l+3}$ and $h_{l+1}h_l=h_lh_{l+1}t_{l+2,l+3}^{-1}t_{l,l+1}$, and the relation~(2) (e) is equivalent to the relation $h_lh_{l+2}h_l=t_{l+2,l+3}^{-1}h_{l+2}h_{l}h_{l+2}t_{l+1,l+2}$. 
Then, up to the relation $t_{2n+1,2n+2}=t_{1,2n}$, we have
\begin{eqnarray*}
&&h_{j-2}\cdots h_{i+2}h_iAh_ih_{i+2}\cdots h_{j-2}\cdot \underset{\leftarrow }{\underline{h_{l+1}h_l}}\\
&\overset{\text{(1)(a)}}{\underset{}{=}}&h_{j-2}\cdots h_{i+2}h_iAh_ih_{i+2}\cdots h_{l}\underline{h_{l+2}\cdot h_{l+1}}h_l\cdot h_{l+4}\cdots h_{j-2}\\
&\overset{\text{(2)(d)}}{\underset{}{=}}&h_{j-2}\cdots h_{i+2}h_iAh_ih_{i+2}\cdots h_{l-2}\underline{h_{l}\cdot h_{l+1}}h_{l+2}t_{l+3,l+4}^{-1}t_{l+1,l+2}\cdot h_l\cdot h_{l+4}\cdots h_{j-2}\\
&\overset{\text{(2)(d)}}{\underset{}{=}}&h_{j-2}\cdots h_{i+2}h_iAh_ih_{i+2}\cdots h_{l-2}\cdot h_{l+1}h_l\underline{t_{l,l+1}^{-1}t_{l+2,l+3}\cdot h_{l+2}}t_{l+3,l+4}^{-1}t_{l+1,l+2}h_l\\ 
&&\cdot h_{l+4}\cdots h_{j-2}\\
&\overset{\text{(1)(c)}}{\underset{\text{(2)(a)}}{=}}&h_{j-2}\cdots h_{i+2}h_iAh_ih_{i+2}\cdots h_{l-2}\cdot h_{l+1}h_lh_{l+2}\underline{t_{l,l+1}^{-1}t_{l+1,l+2}h_l}\cdot h_{l+4}\cdots h_{j-2}\\
&\overset{\text{(2)(a)}}{\underset{}{=}}&h_{j-2}\cdots h_{i+2}h_iAh_ih_{i+2}\cdots h_{l-2}\cdot h_{l+1}\underline{h_lh_{l+2}h_lt_{l+1,l+2}^{-1}}t_{l,l+1}\cdot h_{l+4}\cdots h_{j-2}\\
&\overset{\text{(2)(e)}}{\underset{}{=}}&h_{j-2}\cdots h_{i+2}h_iAh_ih_{i+2}\cdots h_{l-2}\cdot h_{l+1}t_{l+2,l+3}^{-1}h_{l+2}\underline{h_lh_{l+2}t_{l,l+1}}\cdot h_{l+4}\cdots h_{j-2}\\
&\overset{\text{(1)(c)}}{\underset{\text{(2)(a)}}{=}}&h_{j-2}\cdots h_{i+2}h_iAh_ih_{i+2}\cdots h_{l-2}\cdot \underset{\leftarrow }{\underline{h_{l+1}t_{l+2,l+3}^{-1}h_{l+2}t_{l+1,l+2}}}h_lh_{l+2}\cdot h_{l+4}\cdots h_{j-2}\\
&\overset{\text{(1)}}{\underset{}{=}}&h_{j-2}\cdots h_{l+2}\underline{h_{l}\cdot h_{l+1}t_{l+2,l+3}^{-1}}h_{l+2}t_{l+1,l+2}\cdot h_{l-2}\cdots h_{i+2}h_iAh_ih_{i+2}\cdots h_{j-2}\\
&\overset{\text{(2)(d)}}{\underset{}{=}}&h_{j-2}\cdots h_{l+2}\cdot h_{l+1}\underline{h_{l}t_{l,l+1}^{-1}}h_{l+2}t_{l+1,l+2}\cdot h_{l-2}\cdots h_{i+2}h_iAh_ih_{i+2}\cdots h_{j-2}\\
&\overset{\text{(2)(a)}}{\underset{}{=}}&h_{j-2}\cdots h_{l+4}\underline{h_{l+2}\cdot h_{l+1}t_{l+1,l+2}^{-1}}h_{l}h_{l+2}t_{l+1,l+2}\cdot h_{l-2}\cdots h_{i+2}h_iAh_ih_{i+2}\cdots h_{j-2}\\
&\overset{\text{(2)(d)}}{\underset{}{=}}&h_{j-2}\cdots h_{l+4}\cdot h_{l+1}\underline{h_{l+2}t_{l+3,l+4}^{-1}}h_{l}h_{l+2}t_{l+1,l+2}\cdot h_{l-2}\cdots h_{i+2}h_iAh_ih_{i+2}\cdots h_{j-2}\\
&\overset{\text{(2)(a)}}{\underset{}{=}}&h_{j-2}\cdots h_{l+4}\cdot h_{l+1}\underline{t_{l+2,l+3}^{-1}h_{l+2}h_{l}h_{l+2}t_{l+1,l+2}}\cdot h_{l-2}\cdots h_{i+2}h_iAh_ih_{i+2}\cdots h_{j-2}\\
&\overset{\text{(2)(e)}}{\underset{}{=}}&h_{j-2}\cdots h_{l+4}\cdot \underset{\leftarrow }{\underline{h_{l+1}h_l}}h_{l+2}h_l\cdot h_{l-2}\cdots h_{i+2}h_iAh_ih_{i+2}\cdots h_{j-2}\\
&\overset{\text{(1)(a)}}{\underset{}{=}}&h_{l+1}h_l\cdot h_{j-2}\cdots h_{i+2}h_iAh_ih_{i+2}\cdots h_{j-2}.
\end{eqnarray*}
Therefore, the relation $h_{j-2}\cdots h_{i+2}h_iAh_ih_{i+2}\cdots h_{j-2}\rightleftarrows h_{l+1}h_l$ is obtained from the relations~(1), (2), and $t_{2n+1,2n+2}=t_{1,2n}$, and we have completed the proof of Lemma~\ref{tech_rel4}. 
\end{proof}

As a corollary of Lemma~\ref{tech_rel4}, we have the following. 

\begin{cor}\label{cor_tech_rel4}
For $1\leq i<j\leq 2n+2$ in the case of $\LM $ and $1\leq i<j\leq 2n+1$ in the case of $\LMb $ such that $j-i\geq 4$ is even, the relation
\[
h_{j-2}\cdots h_{i+2}h_ih_ih_{i+2}\cdots h_{j-2}\rightleftarrows (h_{j-3}\cdots h_{i+1}h_i)^{\frac{j-i}{2}}
\]
is obtained from the relations~(1), (2), and $t_{2n+1,2n+2}=t_{1,2n}$. 
\end{cor}

\begin{lem}\label{tech_rel5}
For $1\leq i<j\leq 2n$ in the case of $\LM $ and $1\leq i<j\leq 2n-1$ in the case of $\LMb $ such that $j-i\geq 3$ is odd, the relations
\begin{enumerate}
\item $\left\{
\begin{array}{ll}
	h_ih_{i+2}\cdots h_{j-1}(h_{j-2}\cdots h_{i+1}h_i)^{\frac{j-i+1}{2}}=(h_{j-1}\cdots h_{i+1}h_i)^{\frac{j-i+1}{2}},\\
	(h_ih_{i+1}\cdots h_{j-2})^{\frac{j-i+1}{2}}h_{j-1}\cdots h_{i+2}h_i=(h_ih_{i+1}\cdots h_{j-1})^{\frac{j-i+1}{2}},\\
\end{array}
		\right.$
\end{enumerate}
\begin{enumerate}
\item[(2)] $\left\{
\begin{array}{ll}
	h_{i+1}h_{i+3}\cdots h_{j}(h_{j-1}\cdots h_{i+1}h_i)^{\frac{j-i+1}{2}}=(h_{j}\cdots h_{i+1}h_i)^{\frac{j-i+1}{2}},\\
	(h_ih_{i+1}\cdots h_{j-1})^{\frac{j-i+1}{2}}h_{j}\cdots h_{i+3}h_{i+1}=(h_ih_{i+1}\cdots h_{j})^{\frac{j-i+1}{2}}\\
\end{array}
		\right.$
\end{enumerate}
are obtained from the relations~(1) and (2). 
\end{lem}

\begin{proof}
For $i\leq l\leq j-3$, we have the relations $h_l(h_{j-1}h_{j-2}\cdots h_{i+1}h_i)=(h_{j-1}h_{j-2}\cdots h_{i+1}h_i)h_{l+2}$ and $(h_ih_{i+1}\cdots h_{j-2}h_{j-1})h_l=h_{l+2}(h_ih_{i+1}\cdots h_{j-2}h_{j-1})$ by the relations~(1) (a) and (2) (c). 
Then we have
\begin{eqnarray*}
&&h_ih_{i+2}\cdots h_{j-1}(h_{j-2}\cdots h_{i+1}h_i)^{\frac{j-i+1}{2}}\\
&=&h_ih_{i+2}\cdots h_{j-5}\underline{h_{j-3}(h_{j-1}h_{j-2}\cdots h_{i+1}h_i)}(h_{j-2}\cdots h_{i+1}h_i)^{\frac{j-i-1}{2}}\\
&\overset{\text{(1)(a)}}{\underset{\text{(2)(c)}}{=}}&h_ih_{i+2}\cdots h_{j-5}(h_{j-1}h_{j-2}\cdots h_{i+1}h_i)h_{j-1}(h_{j-2}\cdots h_{i+1}h_i)^{\frac{j-i-1}{2}}\\
&=&h_ih_{i+2}\cdots h_{j-7}\underline{h_{j-5}(h_{j-1}h_{j-2}\cdots h_{i+1}h_i)^2}(h_{j-2}\cdots h_{i+1}h_i)^{\frac{j-i-3}{2}}\\
&\overset{\text{(1)(a)}}{\underset{\text{(2)(c)}}{=}}&h_ih_{i+2}\cdots h_{j-7}(h_{j-1}h_{j-2}\cdots h_{i+1}h_i)^3(h_{j-2}\cdots h_{i+1}h_i)^{\frac{j-i-5}{2}}\\
&\vdots &\\
&\overset{\text{(1)(a)}}{\underset{\text{(2)(c)}}{=}}&h_i(h_{j-1}h_{j-2}\cdots h_{i+1}h_i)^{\frac{j-i-1}{2}}(h_{j-2}\cdots h_{i+1}h_i)\\
&\overset{\text{(1)(a)}}{\underset{\text{(2)(c)}}{=}}&(h_{j-1}h_{j-2}\cdots h_{i+1}h_i)^{\frac{j-i+1}{2}}\\
\end{eqnarray*}
and
\begin{eqnarray*}
&&(h_ih_{i+1}\cdots h_{j-2})^{\frac{j-i+1}{2}}h_{j-1}\cdots h_{i+2}h_i\\
&=&(h_ih_{i+1}\cdots h_{j-2})^{\frac{j-i-1}{2}}\underline{(h_ih_{i+1}\cdots h_{j-2}h_{j-1})h_{j-3}}h_{j-5}\cdots h_{i+2}h_i\\
&\overset{\text{(1)(a)}}{\underset{\text{(2)(c)}}{=}}&(h_ih_{i+1}\cdots h_{j-2})^{\frac{j-i-1}{2}}h_{j-1}(h_ih_{i+1}\cdots h_{j-2}h_{j-1})h_{j-5}\cdots h_{i+2}h_i\\
&=&(h_ih_{i+1}\cdots h_{j-2})^{\frac{j-i-3}{2}}\underline{(h_ih_{i+1}\cdots h_{j-2}h_{j-1})^2h_{j-5}}h_{j-7}\cdots h_{i+2}h_i\\
&\overset{\text{(1)(a)}}{\underset{\text{(2)(c)}}{=}}&(h_ih_{i+1}\cdots h_{j-2})^{\frac{j-i-5}{2}}(h_ih_{i+1}\cdots h_{j-2}h_{j-1})^3h_{j-7}\cdots h_{i+2}h_i\\
&\vdots &\\
&\overset{\text{(1)(a)}}{\underset{\text{(2)(c)}}{=}}&(h_ih_{i+1}\cdots h_{j-2}h_{j-1})^{\frac{j-i+1}{2}}.\\
\end{eqnarray*}
Thus the relations $h_ih_{i+2}\cdots h_{j-1}(h_{j-2}\cdots h_{i+1}h_i)^{\frac{j-i+1}{2}}=(h_{j-1}\cdots h_{i+1}h_i)^{\frac{j-i+1}{2}}$ and $(h_ih_{i+1}\cdots h_{j-2})^{\frac{j-i+1}{2}}h_{j-1}\cdots h_{i+2}h_i=(h_ih_{i+1}\cdots h_{j-1})^{\frac{j-i+1}{2}}$ are obtained from the relations~(1) (a) and (2) (c). 

Since the relations $h_l(h_{j}h_{j-1}\cdots h_{i+1}h_i)=(h_{j}h_{j-1}\cdots h_{i+1}h_i)h_{l+2}$ and $(h_ih_{i+1}\cdots h_{j-2}h_{j-1})h_l=h_{l+2}(h_ih_{i+1}\cdots h_{j-2}h_{j-1})$ for $i\leq l\leq j-2$ are obtained from the relations~(1) (a) and (2) (c), by a similar argument above, the relations $h_{i+1}h_{i+3}\cdots h_{j}(h_{j-1}\cdots h_{i+1}h_i)^{\frac{j-i+1}{2}}=(h_{j}\cdots h_{i+1}h_i)^{\frac{j-i+1}{2}}$ and $(h_ih_{i+1}\cdots h_{j-1})^{\frac{j-i+1}{2}}h_{j}\cdots h_{i+3}h_{i+1}=(h_ih_{i+1}\cdots h_{j})^{\frac{j-i+1}{2}}$ are also obtained from the relations~(1) (a) and (2) (c). 
Therefore we have completed the proof of Lemma~\ref{tech_rel5}. 
\end{proof}

By an argument similar to the proof of Lemma~\ref{tech_rel5}, we also have the following lemma. 

\begin{lem}\label{tech_rel5-1}
For $1\leq i<j\leq 2n$ in the case of $\LM $ and $1\leq i<j\leq 2n-1$ in the case of $\LMb $ such that $j-i\geq 3$ is odd, the relations
\begin{enumerate}
\item $\left\{
\begin{array}{ll}
	h_{j-1}\cdots h_{i+2}h_{i}(h_{i+1}h_{i+2}\cdots h_{j-1})^{\frac{j-i+1}{2}}=(h_{i}h_{i+1}\cdots h_{j-1})^{\frac{j-i+1}{2}},\\
	(h_{j-1}\cdots h_{i+2}h_{i+1})^{\frac{j-i+1}{2}}h_{i}h_{i+2}\cdots h_{j-1}=(h_{j-1}\cdots h_{i+1}h_{i})^{\frac{j-i+1}{2}},\\
\end{array}
		\right.$
\end{enumerate}
\begin{enumerate}
\item[(2)] $\left\{
\begin{array}{ll}
	h_{j-1}\cdots h_{i+2}h_{i}(h_{i+1}h_{i+2}\cdots h_{j})^{\frac{j-i+1}{2}}=(h_{i}h_{i+1}\cdots h_{j})^{\frac{j-i+1}{2}},\\
	(h_{j}\cdots h_{i+2}h_{i+1})^{\frac{j-i+1}{2}}h_{i}h_{i+2}\cdots h_{j-1}=(h_{j}\cdots h_{i+1}h_{i})^{\frac{j-i+1}{2}}\\
\end{array}
		\right.$
\end{enumerate}
are obtained from the relations~(1) and (2). 
\end{lem}

\begin{lem}\label{tech_rel6}
For $1\leq i<j\leq 2n+2$ in the case of $\LM $ and $1\leq i<j\leq 2n+1$ in the case of $\LMb $ such that $j-i\geq 3$ is odd, the relation
\begin{align*}
 & h_{j-2}\cdots h_{i+1}h_i(h_{i+1}h_{i+3}\cdots h_{j-2})\\
=&t_{j-1,j}^{\frac{j-i-3}{2}}t_{i+1,i+2}^{-1}t_{i+3,i+4}^{-1}\cdots t_{j-2,j-1}^{-1}h_{j-2}\cdots h_{i+3}h_{i+1}t_{i,i+1}(h_{i+1}h_{i+3}\cdots h_{j-2})\\
&\cdot h_{j-3}\cdots h_{i+2}h_{i}
\end{align*}
is obtained from the relations~(1), (2), and $t_{2n+1,2n+2}=t_{1,2n}$. 
\end{lem}

\begin{proof}
In the case of $j-i=3$, the relation in Lemma~\ref{tech_rel6} is equivalent to the relation~(2) (d) up to the relation~(2) (a). 
We assume that $j-i\geq 5$ is odd. 
Then we have
\begin{eqnarray*}
&&h_{j-2}\cdots h_{i+1}\underline{h_i(h_{i+1}}h_{i+3}\cdots h_{j-2})\\
&\overset{\text{(2)(d)}}{\underset{}{=}}&h_{j-2}\cdots h_{i+2}h_{i+1}\cdot t_{i+2,i+3}^{-1}t_{i,i+1}h_{i+1}\underset{\rightarrow }{\underline{h_i}}\cdot h_{i+3}h_{i+5}\cdots h_{j-2}\\
&\overset{\text{(1)(a)}}{\underset{}{=}}&h_{j-2}\cdots h_{i+2}h_{i+1}\cdot t_{i+2,i+3}^{-1}t_{i,i+1}(h_{i+1}h_{i+3}\cdots h_{j-2})\cdot h_i.
\end{eqnarray*}
We will show that the relation
\begin{eqnarray*}
&&h_{j-2}\cdots h_{i+2}h_{i+1}\cdot t_{i+2,i+3}^{-1}t_{i,i+1}(h_{i+1}h_{i+3}\cdots h_{j-2}) h_i\\
&=&h_{j-2}\cdots h_{i+4}h_{i+3}\cdot t_{i+3,i+4}\cdot t_{i+1,i+2}^{-1}t_{i+4,i+5}^{-1}h_{i+1}t_{i,i+1}(h_{i+1}h_{i+3}\cdots h_{j-2})h_{i+2}h_i
\end{eqnarray*}
is obtained from the relations~(1), (2), and $t_{2n+1,2n+2}=t_{1,2n}$ as follows: 
\begin{eqnarray*}
&&h_{j-2}\cdots h_{i+3}\underline{h_{i+2}h_{i+1}}\cdot t_{i+2,i+3}^{-1}t_{i,i+1}(h_{i+1}h_{i+3}\cdots h_{j-2}) h_i\\
&\overset{\text{(2)(d)}}{\underset{}{=}}&h_{j-2}\cdots h_{i+4}h_{i+3}\cdot t_{i+1,i+2}^{-1}t_{i+3,i+4}h_{i+1}\underline{h_{i+2}\cdot t_{i+2,i+3}^{-1}t_{i,i+1}}(h_{i+1}h_{i+3}\cdots h_{j-2})\\
&&\cdot h_i\\
&\overset{\text{(2)(a)}}{\underset{\text{(1)(c)}}{=}}&h_{j-2}\cdots h_{i+4}h_{i+3}\cdot t_{i+1,i+2}^{-1}t_{i+3,i+4}h_{i+1}t_{i+3,i+4}^{-1}t_{i,i+1}\underline{h_{i+2}(h_{i+1}}h_{i+3}\cdots h_{j-2})\\
&&\cdot h_i\\
&\overset{\text{(2)(d)}}{\underset{}{=}}&h_{j-2}\cdots h_{i+4}h_{i+3}\cdot t_{i+1,i+2}^{-1}t_{i+3,i+4}h_{i+1}\underset{\rightarrow }{\underline{t_{i+3,i+4}^{-1}}}t_{i,i+1}\cdot t_{i+1,i+2}^{-1}t_{i+3,i+4}h_{i+1}\underline{h_{i+2}}\\
&&\underline{\cdot h_{i+3}}h_{i+5}\cdots h_{j-2}\cdot h_i\\
&\overset{\text{(2)(d)}}{\underset{\text{(1)(b)}}{=}}&h_{j-2}\cdots h_{i+4}h_{i+3}\cdot \underline{t_{i+1,i+2}^{-1}t_{i+3,i+4}}h_{i+1}t_{i,i+1}\underline{t_{i+1,i+2}^{-1}h_{i+1}}\\
&&\cdot t_{i+2,i+3}t_{i+4,i+5}^{-1}h_{i+3}\underset{\rightarrow }{\underline{h_{i+2}}}\cdot h_{i+5}h_{i+7}\cdots h_{j-2}\cdot h_i\\
&\overset{\text{(2)(d)}}{\underset{\text{(1)(b)}}{=}}&h_{j-2}\cdots h_{i+4}h_{i+3}\cdot t_{i+3,i+4}t_{i+1,i+2}^{-1}h_{i+1}t_{i,i+1}h_{i+1}\\
&&\cdot \underset{\leftarrow }{\underline{t_{i+4,i+5}^{-1}}}h_{i+3}\cdot h_{i+5}h_{i+7}\cdots h_{j-2}\cdot h_{i+2}h_i\\
&\overset{\text{(1)(c)}}{\underset{\text{(1)(b)}}{=}}&h_{j-2}\cdots h_{i+4}h_{i+3}\cdot t_{i+3,i+4}\cdot t_{i+1,i+2}^{-1}t_{i+4,i+5}^{-1}h_{i+1}t_{i,i+1}(h_{i+1}h_{i+3}\cdots h_{j-2})h_{i+2}h_i.
\end{eqnarray*}
When $j-i=5$, the last expression coincides with
\[
h_{i+3}\cdot t_{i+3,i+4}\cdot t_{i+1,i+2}^{-1}t_{i+4,i+5}^{-1}h_{i+1}t_{i,i+1}(h_{i+1}h_{i+3}\cdots h_{j-2})h_{i+2}h_i.
\]

For the case that $j-i\geq 7$ is odd, we will prove that for odd $1\leq l\leq j-i-2$, the relation
\begin{align*}
&h_{i+l+1}h_{i+l}\cdot t_{i+l,i+l+1}^{\frac{l+1}{2}}t_{i+1,i+2}^{-1}t_{i+3,i+4}^{-1}\cdots t_{i+l-2,i+l-1}^{-1}\cdot t_{i+l+1,i+l+2}^{-1}\\
&\cdot h_{i+l-2}\cdots h_{i+3}h_{i+1}t_{i,i+1}(h_{i+1}h_{i+3}\cdots h_{j-2})\\
=&t_{i+l+2,i+l+3}^{\frac{l+3}{2}}t_{i+1,i+2}^{-1}t_{i+3,i+4}^{-1}\cdots t_{i+l,i+l+1}^{-1}\cdot t_{i+l+3,i+l+4}^{-1}\\
&\cdot h_{i+l}\cdots h_{i+3}h_{i+1}t_{i,i+1}(h_{i+1}h_{i+3}\cdots h_{j-2})h_{i+l+1}
\end{align*}
is obtained from the relations~(1), (2), and $t_{2n+1,2n+2}=t_{1,2n}$ as follows: 
\begin{eqnarray*}
&&\underline{h_{i+l+1}h_{i+l}}\cdot t_{i+l,i+l+1}^{\frac{l+1}{2}}t_{i+1,i+2}^{-1}t_{i+3,i+4}^{-1}\cdots t_{i+l-2,i+l-1}^{-1}\cdot t_{i+l+1,i+l+2}^{-1}\\
&&\cdot h_{i+l-2}\cdots h_{i+3}h_{i+1}t_{i,i+1}(h_{i+1}h_{i+3}\cdots h_{j-2})\\
&\overset{\text{(2)(d)}}{\underset{}{=}}&t_{i+l,i+l+1}^{-1}t_{i+l+2,i+l+3}\underline{h_{i+l}h_{i+l+1}\cdot t_{i+l,i+l+1}^{\frac{l+1}{2}}t_{i+1,i+2}^{-1}t_{i+3,i+4}^{-1}\cdots t_{i+l-2,i+l-1}^{-1}}\\
&&\cdot t_{i+l+1,i+l+2}^{-1}h_{i+l-2}\cdots h_{i+3}h_{i+1}t_{i,i+1}(h_{i+1}h_{i+3}\cdots h_{j-2})\\
&\overset{\text{(2)(b)}}{\underset{\text{(1)(c)}}{=}}&t_{i+l,i+l+1}^{-1}\cdot t_{i+l+2,i+l+3}^{\frac{l+3}{2}}t_{i+1,i+2}^{-1}t_{i+3,i+4}^{-1}\cdots t_{i+l-2,i+l-1}^{-1}h_{i+l}\underline{h_{i+l+1}}\\
&&\underline{\cdot t_{i+l+1,i+l+2}^{-1}}h_{i+l-2}\cdots h_{i+3}h_{i+1}t_{i,i+1}(h_{i+1}h_{i+3}\cdots h_{j-2})\\
&\overset{\text{(2)(a)}}{\underset{\text{(1)}}{=}}&t_{i+l,i+l+1}^{-1}\cdot t_{i+l+2,i+l+3}^{\frac{l+3}{2}}t_{i+1,i+2}^{-1}t_{i+3,i+4}^{-1}\cdots t_{i+l-2,i+l-1}^{-1}h_{i+l}t_{i+l+2,i+l+3}^{-1}\\
&&\cdot h_{i+l-2}\cdots h_{i+3}h_{i+1}t_{i,i+1}(h_{i+1}h_{i+3}\cdots h_{i+l-2}\cdot \underline{h_{i+l+1}\cdot h_{i+l}}h_{i+l+2}\cdots h_{j-2})\\
&\overset{\text{(2)(d)}}{\underset{\text{(1)}}{=}}&\underset{\rightarrow }{\underline{t_{i+l,i+l+1}^{-1}}}\cdot t_{i+l+2,i+l+3}^{\frac{l+3}{2}}t_{i+1,i+2}^{-1}t_{i+3,i+4}^{-1}\cdots t_{i+l-2,i+l-1}^{-1}h_{i+l}t_{i+l+2,i+l+3}^{-1}\\
&&\cdot h_{i+l-2}\cdots h_{i+3}h_{i+1}t_{i,i+1}(h_{i+1}h_{i+3}\cdots h_{i+l-2}\\
&&\cdot t_{i+l,i+l+1}^{-1}t_{i+l+2,i+l+3}h_{i+l}\underline{h_{i+l+1}\cdot h_{i+l+2}}h_{i+l+4}\cdots h_{j-2})\\
&\overset{\text{(2)(d)}}{\underset{}{=}}&t_{i+l+2,i+l+3}^{\frac{l+3}{2}}t_{i+1,i+2}^{-1}t_{i+3,i+4}^{-1}\cdots t_{i+l,i+l+1}^{-1}h_{i+l}t_{i+l+2,i+l+3}^{-1}\\
&&\cdot h_{i+l-2}\cdots h_{i+3}h_{i+1}t_{i,i+1}(h_{i+1}h_{i+3}\cdots h_{i+l-2}\cdot t_{i+l,i+l+1}^{-1}t_{i+l+2,i+l+3}\\
&&\cdot \underline{h_{i+l}\cdot t_{i+l+1,i+l+2}t_{i+l+3,i+l+4}^{-1}}h_{i+l+2}\underset{\rightarrow }{\underline{h_{i+l+1}}}\cdot h_{i+l+4}\cdots h_{j-2})\\
&\overset{\text{(2)(d)}}{\underset{}{=}}&t_{i+l+2,i+l+3}^{\frac{l+3}{2}}t_{i+1,i+2}^{-1}t_{i+3,i+4}^{-1}\cdots t_{i+l,i+l+1}^{-1}h_{i+l}t_{i+l+2,i+l+3}^{-1}\\
&&\cdot h_{i+l-2}\cdots h_{i+3}h_{i+1}t_{i,i+1}(h_{i+1}h_{i+3}\cdots h_{i+l-2}\cdot \underset{\leftarrow }{\underline{t_{i+l+2,i+l+3}t_{i+l+3,i+l+4}^{-1}}}\\
&&\cdot h_{i+l}h_{i+l+2}\cdot h_{i+l+4}\cdots h_{j-2})h_{i+l+1}\\
&\overset{\text{(1)(b)}}{\underset{\text{(1)(c)}}{=}}&t_{i+l+2,i+l+3}^{\frac{l+3}{2}}t_{i+1,i+2}^{-1}t_{i+3,i+4}^{-1}\cdots t_{i+l,i+l+1}^{-1}\cdot t_{i+l+3,i+l+4}^{-1}\\
&&\cdot h_{i+l}\cdots h_{i+3}h_{i+1}t_{i,i+1}(h_{i+1}h_{i+3}\cdots h_{j-2})h_{i+l+1}.
\end{eqnarray*}
Hence, by using this relation inductively, for odd $j-i\geq 7$, we have
\begin{eqnarray*}
&&h_{j-2}\cdots h_{i+4}h_{i+3}\cdot t_{i+3,i+4}\cdot t_{i+1,i+2}^{-1}t_{i+4,i+5}^{-1}h_{i+1}t_{i,i+1}(h_{i+1}h_{i+3}\cdots h_{j-2})h_{i+2}h_i\\
&=&h_{j-2}\cdots h_{i+6}h_{i+5}\cdot t_{i+5,i+6}^2t_{i+1,i+2}^{-1}t_{i+3,i+4}^{-1}t_{i+6,i+7}^{-1}h_{i+3}h_{i+1}t_{i,i+1}\\
&&\cdot (h_{i+1}h_{i+3}\cdots h_{j-2})h_{i+4}h_{i+2}h_i\\
&=&h_{j-2}\cdots h_{i+8}h_{i+7}\cdot t_{i+7,i+8}^3t_{i+1,i+2}^{-1}t_{i+3,i+4}^{-1}t_{i+5,i+6}^{-1}t_{i+8,i+9}^{-1}h_{i+5}h_{i+3}h_{i+1}t_{i,i+1}\\
&&\cdot (h_{i+1}h_{i+3}\cdots h_{j-2})h_{i+6}h_{i+4}h_{i+2}h_i\\
&\vdots &\\
&=&h_{j-2}\cdot t_{j-2,j-1}^{\frac{j-i-3}{2}}\cdot t_{i+1,i+2}^{-1}t_{i+3,i+4}^{-1}\cdots t_{j-4,j-3}^{-1}\cdot t_{j-1,j}^{-1}\\
&&\cdot h_{j-4}\cdots h_{i+3}h_{i+1}t_{i,i+1}(h_{i+1}h_{i+3}\cdots h_{j-2})h_{j-3}\cdots h_{i+2}h_i. 
\end{eqnarray*}
Thus, for odd $j-i\geq 5$, up to the relations~(1), (2), and $t_{2n+1,2n+2}=t_{1,2n}$, we have
\begin{eqnarray*}
&&h_{j-2}\cdots h_{i+1}h_i(h_{i+1}h_{i+3}\cdots h_{j-2})\\
&=&\underline{h_{j-2}\cdot t_{j-2,j-1}^{\frac{j-i-3}{2}}\cdot t_{i+1,i+2}^{-1}t_{i+3,i+4}^{-1}\cdots t_{j-4,j-3}^{-1}\cdot t_{j-1,j}^{-1}}\\
&&\cdot h_{j-4}\cdots h_{i+3}h_{i+1}t_{i,i+1}(h_{i+1}h_{i+3}\cdots h_{j-2})h_{j-3}\cdots h_{i+2}h_i\\
&\overset{\text{(2)(a)}}{\underset{\text{(1)(c)}}{=}}&t_{j-1,j}^{\frac{j-i-3}{2}}\cdot t_{i+1,i+2}^{-1}t_{i+3,i+4}^{-1}\cdots t_{j-4,j-3}^{-1}\cdot t_{j-2,j-1}^{-1}\\
&&\cdot h_{j-2}\cdots h_{i+3}h_{i+1}t_{i,i+1}h_{i+1}h_{i+3}\cdots h_{j-2}\cdot h_{j-3}\cdots h_{i+2}h_i. 
\end{eqnarray*}
Therefore, we have completed the proof of Lemma~\ref{tech_rel6}. 
\end{proof}

\begin{lem}\label{tech_rel7}
For $1\leq i<j\leq 2n+2$ in the case of $\LM $ and $1\leq i<j\leq 2n+1$ in the case of $\LMb $ such that $j-i\geq 3$ is odd, the relation
\begin{align*}
&(h_{j-4}\cdots h_{i+1}h_i)^{\frac{j-i-1}{2}}\\
\rightleftarrows &t_{i+1,i+2}^{-1}t_{i+3,i+4}^{-1}\cdots t_{j-2,j-1}^{-1}h_{j-2}\cdots h_{i+3}h_{i+1}t_{i,i+1}h_{i+1}h_{i+3}\cdots h_{j-2}
\end{align*}
is obtained from the relations~(1), (2), and $t_{2n+1,2n+2}=t_{1,2n}$. 
\end{lem}

\begin{proof}
In the case of $j-i=3$, the relation in Lemma~\ref{tech_rel7} coincides with the trivial relation. 
We assume that $j-i\geq 5$ is odd. 
Then we have
\begin{eqnarray*}
&&h_{j-2}h_{j-4}\cdots h_{i+1}t_{i,i+1}h_{i+1}\cdots \underline{h_{j-4}h_{j-2}\cdot (h_{j-4}}h_{j-5}\cdots h_i)\\
&\overset{\text{(2)(e)}}{\underset{}{=}}&h_{j-2}h_{j-4}\cdots h_{i+1}t_{i,i+1}h_{i+1}\cdots h_{j-8}h_{j-6}\cdot \underset{\leftarrow }{\underline{t_{j-3,j-2}t_{j-2,j-1}^{-1}h_{j-2}}}h_{j-4}h_{j-2}\\
&&\cdot h_{j-5}h_{j-6}\cdots h_i\\
&\overset{\text{(1)}}{\underset{}{=}}&\underline{h_{j-2}h_{j-4}\cdot t_{j-3,j-2}}\ \underline{t_{j-2,j-1}^{-1}h_{j-2}}\cdot h_{j-6}h_{j-8}\cdots h_{i+1}t_{i,i+1}h_{i+1}\cdots h_{j-4}h_{j-2}\\
&&\cdot h_{j-5}h_{j-6}\cdots h_i\\
&\overset{\text{(2)(a)}}{\underset{\text{(1)(c)}}{=}}&t_{j-4,j-3}\underline{h_{j-2}h_{j-4}h_{j-2}}t_{j-1,j}^{-1}\cdot h_{j-6}h_{j-8}\cdots h_{i+1}t_{i,i+1}h_{i+1}\cdots h_{j-4}h_{j-2}\\
&&\cdot h_{j-5}h_{j-6}\cdots h_i\\
&\overset{\text{(2)(e)}}{\underset{}{=}}&t_{j-4,j-3}t_{j-2,j-1}\underline{t_{j-3,j-2}^{-1}h_{j-4}}\ \underline{h_{j-2}h_{j-4}t_{j-1,j}^{-1}}\cdot h_{j-6}h_{j-8}\cdots h_{i+1}\\
&&\cdot t_{i,i+1}h_{i+1}\cdots h_{j-4}h_{j-2}\cdot h_{j-5}h_{j-6}\cdots h_i\\
&\overset{\text{(2)(a)}}{\underset{\text{(1)(c)}}{=}}&t_{j-4,j-3}t_{j-2,j-1}h_{j-4}t_{j-4,j-3}^{-1}t_{j-2,j-1}^{-1}\cdot h_{j-2}h_{j-4}\cdots h_{i+1}\\
&&\cdot t_{i,i+1}h_{i+1}\cdots h_{j-4}h_{j-2}\cdot \underset{\leftarrow }{\underline{h_{j-5}h_{j-6}\cdots h_{i+1}}}h_i\\
&\overset{\text{Lem.\ref{tech_rel4}}}{\underset{}{=}}&t_{j-4,j-3}t_{j-2,j-1}h_{j-4}t_{j-4,j-3}^{-1}t_{j-2,j-1}^{-1}\cdot h_{j-5}h_{j-6}\cdots h_{i+1}\\
&&\cdot h_{j-2}h_{j-4}\cdots h_{i+1}t_{i,i+1}h_{i+1}\cdots h_{j-4}h_{j-2}\cdot \underset{\leftarrow }{\underline{h_i}}\\
&\overset{\text{(1)(a)}}{\underset{}{=}}&t_{j-4,j-3}t_{j-2,j-1}h_{j-4}t_{j-4,j-3}^{-1}t_{j-2,j-1}^{-1}\cdot h_{j-5}h_{j-6}\cdots h_{i+1}\\
&&\cdot h_{j-2}h_{j-4}\cdots h_{i+1}\underline{t_{i,i+1}h_{i+1}h_i}\cdot h_{i+3}\cdots h_{j-4}h_{j-2}\\
&\overset{\text{(2)(d)}}{\underset{\text{(2)(b)}}{=}}&t_{j-4,j-3}t_{j-2,j-1}h_{j-4}t_{j-4,j-3}^{-1}t_{j-2,j-1}^{-1}\cdot h_{j-5}h_{j-6}\cdots h_{i+1}\\
&&\cdot h_{j-2}h_{j-4}\cdots \underline{h_{i+1}t_{i+2,i+3}}h_ih_{i+1}\cdot h_{i+3}\cdots h_{j-4}h_{j-2}\\
&\overset{\text{(2)(a)}}{\underset{}{=}}&t_{j-4,j-3}t_{j-2,j-1}h_{j-4}t_{j-4,j-3}^{-1}t_{j-2,j-1}^{-1}\cdot h_{j-5}h_{j-6}\cdots h_{i+1}\cdot h_{j-2}h_{j-4}\cdots h_{i+3}\\
&&\cdot t_{i+1,i+2}\underline{h_{i+1}h_i}h_{i+1}\cdot h_{i+3}\cdots h_{j-4}h_{j-2}\\
&\overset{\text{(2)(d)}}{\underset{}{=}}&t_{j-4,j-3}t_{j-2,j-1}h_{j-4}t_{j-4,j-3}^{-1}t_{j-2,j-1}^{-1}\cdot h_{j-5}h_{j-6}\cdots h_{i+1}\cdot h_{j-2}h_{j-4}\cdots h_{i+3}\\
&&\cdot t_{i+1,i+2}t_{i,i+1}^{-1}\underline{t_{i+2,i+3}h_ih_{i+1}}h_{i+1}\cdot h_{i+3}\cdots h_{j-4}h_{j-2}\\
&\overset{\text{(2)(b)}}{\underset{}{=}}&t_{j-4,j-3}t_{j-2,j-1}h_{j-4}t_{j-4,j-3}^{-1}t_{j-2,j-1}^{-1}\cdot h_{j-5}h_{j-6}\cdots h_{i+1}\cdot h_{j-2}h_{j-4}\cdots h_{i+3}\\
&&\cdot \underset{\leftarrow }{\underline{t_{i+1,i+2}t_{i,i+1}^{-1}h_i}}h_{i+1}t_{i,i+1}h_{i+1}\cdot h_{i+3}\cdots h_{j-4}h_{j-2}\\
&\overset{\text{(1)(a)}}{\underset{\text{(1)(c)}}{=}}&t_{j-4,j-3}t_{j-2,j-1}h_{j-4}\underline{t_{j-4,j-3}^{-1}t_{j-2,j-1}^{-1}}\cdot \underline{h_{j-5}h_{j-6}\cdots h_{i+1}\cdot t_{i+1,i+2}}\ \underline{t_{i,i+1}^{-1}h_i}\\
&&\cdot h_{j-2}h_{j-4}\cdots h_{i+1}t_{i,i+1}h_{i+1}\cdots h_{j-4}h_{j-2}\\
&\overset{\text{(1)(b)}}{\underset{\text{(2)(a)}}{=}}&t_{j-4,j-3}t_{j-2,j-1}h_{j-4}\underset{\rightarrow }{\underline{t_{j-2,j-1}^{-1}}}\cdot h_{j-5}h_{j-6}\cdots h_{i+1}h_it_{i+1,i+2}^{-1}\\
&&\cdot h_{j-2}h_{j-4}\cdots h_{i+1}t_{i,i+1}h_{i+1}\cdots h_{j-4}h_{j-2}\\
&\overset{\text{(1)(c)}}{\underset{}{=}}&t_{j-4,j-3}t_{j-2,j-1}h_{j-4}h_{j-5}\cdots h_it_{j-2,j-1}^{-1}t_{i+1,i+2}^{-1}\\
&&\cdot h_{j-2}h_{j-4}\cdots h_{i+1}t_{i,i+1}h_{i+1}\cdots h_{j-4}h_{j-2}.
\end{eqnarray*}
Thus we have
\begin{eqnarray*}
& &h_{j-2}h_{j-4}\cdots h_{i+1}t_{i,i+1}h_{i+1}\cdots h_{j-4}h_{j-2}(h_{j-4}h_{j-5}\cdots h_i)^{\frac{j-i-1}{2}}\\
&\overset{}{\underset{}{=}}&(t_{j-4,j-3}t_{j-2,j-1}h_{j-4}h_{j-5}\cdots h_it_{j-2,j-1}^{-1}t_{i+1,i+2}^{-1})^{\frac{j-i-1}{2}}\\
&&\cdot h_{j-2}h_{j-4}\cdots h_{i+1}t_{i,i+1}h_{i+1}\cdots h_{j-4}h_{j-2}.
\end{eqnarray*}
Hence it is enough for completing the proof of Lemma~\ref{tech_rel7} to prove that 
\begin{align*}
&(t_{j-4,j-3}t_{j-2,j-1}h_{j-4}h_{j-5}\cdots h_it_{j-2,j-1}^{-1}t_{i+1,i+2}^{-1})^{\frac{j-i-1}{2}}\\
=&t_{j-2,j-1}\cdots t_{i+3,i+4}t_{i+1,i+2}(h_{j-4}\cdots h_{i+1}h_i)^{\frac{j-i-1}{2}}t_{i+1,i+2}^{-1}t_{i+3,i+4}^{-1}\cdots t_{j-2,j-1}^{-1}.
\end{align*}
We have 
\begin{eqnarray*}
&&\underline{(t_{j-4,j-3}t_{j-2,j-1}h_{j-4}h_{j-5}\cdots h_it_{j-2,j-1}^{-1}t_{i+1,i+2}^{-1})^{\frac{j-i-1}{2}}}\\
&\overset{\text{(1)(b)}}{\underset{}{=}}&t_{j-2,j-1}\underline{(t_{j-4,j-3}h_{j-4}h_{j-5}\cdots h_it_{i+1,i+2}^{-1})^{\frac{j-i-1}{2}}}t_{j-2,j-1}^{-1}\\
&\overset{}{\underset{}{=}}&t_{j-2,j-1}t_{j-4,j-3}h_{j-4}h_{j-5}\cdots h_i\underline{t_{i+1,i+2}^{-1}(t_{j-4,j-3}h_{j-4}h_{j-5}\cdots h_it_{i+1,i+2}^{-1})^{\frac{j-i-5}{2}}}\\
&&\underline{t_{j-4,j-3}}h_{j-4}h_{j-5}\cdots h_it_{i+1,i+2}^{-1}t_{j-2,j-1}^{-1}\\
&\overset{\text{(1)(b)}}{\underset{}{=}}&t_{j-2,j-1}t_{j-4,j-3}\underline{h_{j-4}h_{j-5}\cdots h_it_{j-4,j-3}}(t_{i+1,i+2}^{-1}h_{j-4}h_{j-5}\cdots h_it_{j-4,j-3})^{\frac{j-i-5}{2}}\\
&&\cdot \underline{t_{i+1,i+2}^{-1}h_{j-4}h_{j-5}\cdots h_i}t_{i+1,i+2}^{-1}t_{j-2,j-1}^{-1}\\
&\overset{\text{(1)}}{\underset{\text{(2)(b)}}{=}}&t_{j-2,j-1}t_{j-4,j-3}t_{j-6,j-5}\cdot h_{j-4}h_{j-5}\cdots h_i(\underline{t_{i+1,i+2}^{-1}h_{j-4}h_{j-5}\cdots h_it_{j-4,j-3}})^{\frac{j-i-5}{2}}\\
&&\cdot h_{j-4}h_{j-5}\cdots h_i\cdot t_{i+1,i+2}^{-1}t_{i+3,i+4}^{-1}t_{j-2,j-1}^{-1}\\
&\overset{\text{(1)}}{\underset{\text{(2)(b)}}{=}}&t_{j-2,j-1}t_{j-4,j-3}t_{j-6,j-5}\cdot h_{j-4}h_{j-5}\cdots h_i\underline{(t_{j-6,j-5}h_{j-4}h_{j-5}\cdots h_it_{i+3,i+4}^{-1})^{\frac{j-i-5}{2}}}\\
&&\cdot h_{j-4}h_{j-5}\cdots h_i\cdot t_{i+1,i+2}^{-1}t_{i+3,i+4}^{-1}t_{j-2,j-1}^{-1}\\
&\overset{}{\underset{}{=}}&t_{j-2,j-1}t_{j-4,j-3}t_{j-6,j-5}\cdot \underline{h_{j-4}h_{j-5}\cdots h_i\cdot t_{j-6,j-5}}h_{j-4}h_{j-5}\cdots h_it_{i+3,i+4}^{-1}\\
&&(t_{j-6,j-5}h_{j-4}h_{j-5}\cdots h_it_{i+3,i+4}^{-1})^{\frac{j-i-9}{2}}\\
&&\cdot t_{j-6,j-5}h_{j-4}h_{j-5}\cdots h_i\underline{t_{i+3,i+4}^{-1}\cdot h_{j-4}h_{j-5}\cdots h_i}\cdot t_{i+1,i+2}^{-1}t_{i+3,i+4}^{-1}t_{j-2,j-1}^{-1}\\
&\overset{\text{(1)}}{\underset{\text{(2)(b)}}{=}}&t_{j-2,j-1}t_{j-4,j-3}t_{j-6,j-5}t_{j-8,j-7}(h_{j-4}h_{j-5}\cdots h_i)^2\underline{t_{i+3,i+4}^{-1}}\\
&&\underline{(t_{j-6,j-5}h_{j-4}h_{j-5}\cdots h_it_{i+3,i+4}^{-1})^{\frac{j-i-9}{2}} t_{j-6,j-5}}(h_{j-4}h_{j-5}\cdots h_i)^2\\
&&\cdot t_{i+1,i+2}^{-1}t_{i+3,i+4}^{-1}t_{i+5,i+6}^{-1}t_{j-2,j-1}^{-1}\\
&\overset{}{\underset{}{=}}&t_{j-2,j-1}t_{j-4,j-3}t_{j-6,j-5}t_{j-8,j-7}\underline{(h_{j-4}h_{j-5}\cdots h_i)^2t_{j-6,j-5}}\\
&&(t_{i+3,i+4}^{-1}h_{j-4}h_{j-5}\cdots h_it_{j-6,j-5})^{\frac{j-i-9}{2}} \underline{t_{i+3,i+4}^{-1}(h_{j-4}h_{j-5}\cdots h_i)^2}\\
&&\cdot t_{i+1,i+2}^{-1}t_{i+3,i+4}^{-1}t_{i+5,i+6}^{-1}t_{j-2,j-1}^{-1}\\
&\overset{\text{(1)}}{\underset{\text{(2)(b)}}{=}}&t_{j-2,j-1}t_{j-4,j-3}t_{j-6,j-5}t_{j-8,j-7}t_{j-10,j-9}(h_{j-4}h_{j-5}\cdots h_i)^2\\
&&(\underline{t_{i+3,i+4}^{-1}h_{j-4}h_{j-5}\cdots h_it_{j-6,j-5}})^{\frac{j-i-9}{2}} (h_{j-4}h_{j-5}\cdots h_i)^2\\
&&\cdot t_{i+1,i+2}^{-1}t_{i+3,i+4}^{-1}t_{i+5,i+6}^{-1}t_{i+7,i+8}^{-1}t_{j-2,j-1}^{-1}\\
&\overset{\text{(1)}}{\underset{\text{(2)(b)}}{=}}&t_{j-2,j-1}t_{j-4,j-3}t_{j-6,j-5}t_{j-8,j-7}t_{j-10,j-9}(h_{j-4}h_{j-5}\cdots h_i)^2\\
&&(t_{j-8,j-7}h_{j-4}h_{j-5}\cdots h_it_{i+5,i+6}^{-1})^{\frac{j-i-9}{2}} (h_{j-4}h_{j-5}\cdots h_i)^2\\
&&\cdot t_{i+1,i+2}^{-1}t_{i+3,i+4}^{-1}t_{i+5,i+6}^{-1}t_{i+7,i+8}^{-1}t_{j-2,j-1}^{-1}\\
&\vdots &\\
&\overset{\text{(1)}}{\underset{\text{(2)(b)}}{=}}&\left\{
\begin{array}{ll}
t_{j-2,j-1}\cdots t_{i+3,i+4}t_{i+1,i+2}(h_{j-4}h_{j-5}\cdots h_i)^{\frac{j-i-3}{4}}\\
\cdot \underline{t_{\frac{j+i-5}{2},\frac{j+i-3}{2}}h_{j-4}h_{j-5}\cdots h_i}t_{\frac{j+i-1}{2},\frac{j+i+1}{2}}^{-1} (h_{j-4}h_{j-5}\cdots h_i)^{\frac{j-i-3}{4}}\\
\cdot t_{i+1,i+2}^{-1}t_{i+3,i+4}^{-1}\cdots t_{j-2,j-1}^{-1}\quad \text{for }\frac{j-i-1}{2}\text{ is odd},\vspace{0.2cm}
\\
t_{j-2,j-1}\cdots t_{i+5,i+6}t_{i+3,i+4}(h_{j-4}h_{j-5}\cdots h_i)^{\frac{j-i-5}{4}}\\
\cdot \underline{(t_{\frac{j+i-3}{2},\frac{j+i-1}{2}}h_{j-4}h_{j-5}\cdots h_it_{\frac{j+i-3}{2},\frac{j+i-1}{2}}^{-1})^2} (h_{j-4}h_{j-5}\cdots h_i)^{\frac{j-i-5}{4}}\\
\cdot t_{i+1,i+2}^{-1}t_{i+3,i+4}^{-1}\cdots t_{j-6,j-5}^{-1}\cdot t_{j-2,j-1}^{-1}\\
=t_{j-2,j-1}\cdots t_{i+5,i+6}t_{i+3,i+4}\underline{(h_{j-4}h_{j-5}\cdots h_i)^{\frac{j-i-5}{4}}t_{\frac{j+i-3}{2},\frac{j+i-1}{2}}}\\
\cdot (h_{j-4}h_{j-5}\cdots h_i)^2\underline{t_{\frac{j+i-3}{2},\frac{j+i-1}{2}}^{-1} (h_{j-4}h_{j-5}\cdots h_i)^{\frac{j-i-5}{4}}}\\
\cdot t_{i+1,i+2}^{-1}t_{i+3,i+4}^{-1}\cdots t_{j-6,j-5}^{-1}\cdot t_{j-2,j-1}^{-1}\quad \text{for }\frac{j-i-1}{2}\text{ is even},\\
\end{array}
\right. \\
&\overset{\text{(1)}}{\underset{\text{(2)(b)}}{=}}&t_{j-2,j-1}\cdots t_{i+3,i+4}t_{i+1,i+2}(h_{j-4}\cdots h_{i+1}h_i)^{\frac{j-i-1}{2}}t_{i+1,i+2}^{-1}t_{i+3,i+4}^{-1}\cdots t_{j-2,j-1}^{-1}.
\end{eqnarray*}
Therefore we have completed the proof of Lemma~\ref{tech_rel7}. 
\end{proof}

\begin{lem}\label{tech_rel8}
For $1\leq i<j\leq 2n+2$ in the case of $\LM $ and $1\leq i<j\leq 2n+1$ in the case of $\LMb $ such that $j-i\geq 3$ is odd, the relation
\[
(h_{j-2}\cdots h_{i+1}h_i)^{\frac{j-i+1}{2}}=(h_ih_{i+1}\cdots h_{j-2})^{\frac{j-i+1}{2}}
\]
is obtained from the relations~(1), (2), $t_{2n+1,2n+2}=t_{1,2n}$, and relations in Corollary~\ref{cor_tech_rel4}, Lemmas~\ref{tech_rel5}, \ref{tech_rel6}, and \ref{tech_rel7}. 
\end{lem}

\begin{proof}
We proceed by induction on $\frac{j-i+1}{2}\geq 2$. 
In the case of $\frac{j-i+1}{2}=2$, i.e. $j-2=i+1$, we have
\begin{eqnarray*}
\underline{h_{i+1}h_i}\cdot \underline{h_{i+1}h_i}\overset{\text{(2)(d)}}{\underset{}{=}}h_ih_{i+1}\underline{t_{i,i+1}t_{i+2,i+3}^{-1}\cdot h_ih_{i+1}}t_{i,i+1}t_{i+2,i+3}^{-1}\overset{\text{(2)(a)}}{\underset{\text{(2)(b)}}{=}}h_ih_{i+1}h_ih_{i+1}.
\end{eqnarray*}
Thus Lemma~\ref{tech_rel8} is true for $\frac{j-i+1}{2}=2$. 

Assume $\frac{j-i+1}{2}\geq 3$. 
By the inductive hypothesis, the relation 
\[
(h_{j-4}\cdots h_{i+1}h_i)^{\frac{j-i-1}{2}}=(h_ih_{i+1}\cdots h_{j-4})^{\frac{j-i-1}{2}}
\]
holds in $\LMb $ for $j\leq 2n+1$ and $\LM $. 
Then we have
\begin{eqnarray*}
&&(h_{j-2}\cdots h_{i+1}h_i)^{\frac{j-i+1}{2}}=h_{j-2}\cdots h_{i+1}h_i\underline{(h_{j-2}\cdots h_{i+1}h_i)^{\frac{j-i-1}{2}}}\\
&\overset{\text{Lem.\ref{tech_rel5}}}{\underset{}{=}}&h_{j-2}\cdots h_{i+1}h_i\cdot h_{i+1}h_{i+3}\cdots h_{j-2}\underline{(h_{j-3}\cdots h_{i+1}h_i)^{\frac{j-i-1}{2}}}\\
&\overset{\text{Lem.\ref{tech_rel5}}}{\underset{}{=}}&\underline{h_{j-2}\cdots h_{i+1}h_i\cdot h_{i+1}h_{i+3}\cdots h_{j-2}}\cdot h_ih_{i+2}\cdots h_{j-3}(h_{j-4}\cdots h_{i+1}h_i)^{\frac{j-i-1}{2}}\\
&\overset{\text{Lem.\ref{tech_rel6}}}{\underset{}{=}}&t_{j-1,j}^{\frac{j-i-3}{2}}t_{i+1,i+2}^{-1}t_{i+3,i+4}^{-1}\cdots t_{j-2,j-1}^{-1}h_{j-2}\cdots h_{i+3}h_{i+1}t_{i,i+1}h_{i+1}h_{i+3}\cdots h_{j-2}\\
&&\underline{\cdot h_{j-3}\cdots h_{i+2}h_{i}\cdot h_ih_{i+2}\cdots h_{j-3}(h_{j-4}\cdots h_{i+1}h_i)^{\frac{j-i-1}{2}}}\\
&\overset{\text{Cor.\ref{cor_tech_rel4}}}{\underset{}{=}}&t_{j-1,j}^{\frac{j-i-3}{2}}\underline{t_{i+1,i+2}^{-1}t_{i+3,i+4}^{-1}\cdots t_{j-2,j-1}^{-1}h_{j-2}\cdots h_{i+3}h_{i+1}t_{i,i+1}h_{i+1}h_{i+3}\cdots h_{j-2}}\\
&&\underline{\cdot (h_{j-4}\cdots h_{i+1}h_i)^{\frac{j-i-1}{2}}}h_{j-3}\cdots h_{i+2}h_{i}\cdot h_ih_{i+2}\cdots h_{j-3}\\
&\overset{\text{Lem.\ref{tech_rel7}}}{\underset{}{=}}&\underline{t_{j-1,j}^{\frac{j-i-3}{2}}(h_{j-4}\cdots h_{i+1}h_i)^{\frac{j-i-1}{2}}}t_{i+1,i+2}^{-1}t_{i+3,i+4}^{-1}\cdots t_{j-2,j-1}^{-1}\\
&&\cdot h_{j-2}\cdots h_{i+3}h_{i+1}t_{i,i+1}h_{i+1}h_{i+3}\cdots h_{j-2}\\
&&\cdot h_{j-3}\cdots h_{i+2}h_{i}h_ih_{i+2}\cdots h_{j-3}\\
&\overset{\text{(1)(c)}}{\underset{}{=}}&(h_{j-4}\cdots h_{i+1}h_i)^{\frac{j-i-1}{2}}\underline{t_{j-1,j}^{\frac{j-i-3}{2}}t_{i+1,i+2}^{-1}t_{i+3,i+4}^{-1}\cdots t_{j-2,j-1}^{-1}}\\
&&\underline{\cdot h_{j-2}\cdots h_{i+3}h_{i+1}t_{i,i+1}h_{i+1}h_{i+3}\cdots h_{j-2}\cdot h_{j-3}\cdots h_{i+2}h_{i}}h_ih_{i+2}\cdots h_{j-3}\\
&\overset{\text{Lem.\ref{tech_rel6}}}{\underset{}{=}}&\underline{(h_{j-4}\cdots h_{i+1}h_i)^{\frac{j-i-1}{2}}}h_{j-2}\cdots h_{i+1}h_i\cdot h_{i+1}h_{i+3}\cdots h_{j-2}\cdot h_ih_{i+2}\cdots h_{j-3}\\
&=&(h_ih_{i+1}\cdots h_{j-4})^{\frac{j-i-1}{2}}h_{j-2}\cdots h_{i+1}h_i\cdot \underset{\rightarrow }{\underline{h_{i+1}h_{i+3}\cdots h_{j-2}}}\cdot h_ih_{i+2}\cdots h_{j-3}\\
&\overset{\text{(1)(a)}}{\underset{}{=}}&(h_ih_{i+1}\cdots h_{j-4})^{\frac{j-i-1}{2}}\underline{h_{j-2}\cdots h_{i+1}h_i}\cdot \underline{h_{i+1}h_i\cdot h_{i+3}h_{i+2}\cdots h_{j-2}h_{j-3}}\\
&\overset{\text{(2)(d)}}{\underset{}{=}}&(h_ih_{i+1}\cdots h_{j-4})^{\frac{j-i-1}{2}}\\
&&\cdot h_{j-3}h_{j-2}t_{j-3,j-2}\underset{\rightarrow }{\underline{t_{j-1,j}^{-1}}}\cdots h_{i+2}h_{i+3}t_{i+2,i+3}\underset{\rightarrow }{\underline{t_{i+4,i+5}^{-1}}}\cdot h_ih_{i+1}t_{i,i+1}t_{i+2,i+3}^{-1}\\
&&\cdot t_{i+2,i+3}t_{i,i+1}^{-1}h_ih_{i+1}\cdot \underset{\leftarrow }{\underline{t_{i+4,i+5}}}t_{i+2,i+3}^{-1}h_{i+2}h_{i+3}\cdots \underset{\leftarrow }{\underline{t_{j-1,j}}}t_{j-3,j-2}^{-1}h_{j-3}h_{j-2}\\
&\overset{\text{(1)(c)}}{\underset{}{=}}&(h_ih_{i+1}\cdots h_{j-4})^{\frac{j-i-1}{2}}\\
&&h_{j-3}h_{j-2}\underline{t_{j-3,j-2}h_{j-5}h_{j-4}}\cdots \underline{t_{i+4,i+5}h_{i+2}h_{i+3}}\ \underline{t_{i+2,i+3}\cdot h_ih_{i+1}}\\
&&\cdot \underline{h_ih_{i+1}\cdot t_{i+2,i+3}^{-1}}\ \underline{h_{i+2}h_{i+3}t_{i+4,i+5}^{-1}}\cdots  \underline{h_{j-5}h_{j-4}t_{j-3,j-2}^{-1}}h_{j-3}h_{j-2}\\
&\overset{\text{(2)(b)}}{\underset{}{=}}&(h_ih_{i+1}\cdots h_{j-4})^{\frac{j-i-1}{2}}\\
&&h_{j-3}h_{j-2}h_{j-5}h_{j-4}\underline{t_{j-5,j-4}h_{j-7}h_{j-6}} \cdots \underline{t_{i+4,i+5}h_{i+2}h_{i+3}}\ \underline{t_{i+2,i+3}h_ih_{i+1}}\\
&&\cdot \underline{h_ih_{i+1}t_{i+2,i+3}^{-1}}\ \underline{h_{i+2}h_{i+3}t_{i+4,i+5}^{-1}}\cdots \underline{h_{j-7}h_{j-6}t_{j-5,j-4}^{-1}}h_{j-5}h_{j-4}h_{j-3}h_{j-2}\\
&\overset{\text{(2)(b)}}{\underset{}{=}}&(h_ih_{i+1}\cdots h_{j-4})^{\frac{j-i-1}{2}}\\
&&h_{j-3}h_{j-2}\cdots h_{j-7}h_{j-6}\underline{t_{j-7,j-6}h_{j-9}h_{j-8}} \cdots \underline{t_{i+2,i+3}h_ih_{i+1}}\\
&&\cdot \underline{h_ih_{i+1}\cdot t_{i+2,i+3}^{-1}} \cdots \underline{h_{j-9}h_{j-8}t_{j-7,j-6}^{-1}}h_{j-7}h_{j-6}\cdots h_{j-3}h_{j-2}\\
&\vdots &\\
&\overset{\text{(2)(b)}}{\underset{}{=}}&(h_ih_{i+1}\cdots h_{j-4})^{\frac{j-i-1}{2}}h_{j-3}h_{j-2}\cdot \underset{\leftarrow }{\underline{h_{j-5}}}h_{j-4} \cdots \underset{\leftarrow }{\underline{h_{i+2}}}h_{i+3}\cdot \underset{\leftarrow }{\underline{h_i}}h_{i+1}\\
&&\cdot h_ih_{i+1}\cdot h_{i+2}h_{i+3}\cdots h_{j-5}h_{j-4}\cdot h_{j-3}h_{j-2}\\
&\overset{\text{(1)(a)}}{\underset{}{=}}&\underline{(h_ih_{i+1}\cdots h_{j-4})^{\frac{j-i-1}{2}}h_{j-3}h_{j-5}\cdots h_{i+2}h_i\cdot h_{j-2}h_{j-4} \cdots h_{i+3}h_{i+1}}\\
&&\cdot h_ih_{i+1}\cdots h_{j-2}\\
&\overset{\text{Lem.\ref{tech_rel5}}}{\underset{}{=}}&(h_ih_{i+1}\cdots h_{j-2})^{\frac{j-i+1}{2}}.
\end{eqnarray*}
Therefore, we have completed the proof of Lemma~\ref{tech_rel8}. 
\end{proof}

Lemma~\ref{tech_rel8} give the proof of Lemma~\ref{cor_tech_rel8} as follows. 

\begin{proof}[Proof of Lemma~\ref{cor_tech_rel8}]
Assume that $j-i\geq 3$ is odd for $1\leq i<j\leq 2n+1$. 
Then we have 
\begin{eqnarray*}
&&t_{j-1,j}^{-\frac{j-i-3}{2}}\cdots t_{i+2,i+3}^{-\frac{j-i-3}{2}}t_{i,i+1}^{-\frac{j-i-3}{2}}\underline{(h_{j-2}\cdots h_{i+1}h_i)^{\frac{j-i+1}{2}}}\\
&\overset{\text{Lem.\ref{cor_tech_rel8}}}{\underset{}{=}}&\underline{t_{j-1,j}^{-\frac{j-i-3}{2}}\cdots t_{i+2,i+3}^{-\frac{j-i-3}{2}}t_{i,i+1}^{-\frac{j-i-3}{2}}}(h_ih_{i+1}\cdots h_{j-2})^{\frac{j-i+1}{2}}\\
&\overset{\text{(1)(b)}}{\underset{}{=}}&t_{i,i+1}^{-\frac{j-i-3}{2}}t_{i+2,i+3}^{-\frac{j-i-3}{2}}\cdots t_{j-1,j}^{-\frac{j-i-3}{2}}(h_ih_{i+1}\cdots h_{j-2})^{\frac{j-i+1}{2}}
\end{eqnarray*}
Since $j\leq 2n+1$, $t_{2n+1,2n+2}$ does not appear in the relation above. 
Therefore the relation in~Lemma~\ref{tech_rel8} which is used above is obtained from the relations~(1) and (2), and we have completed the proof of Lemma~\ref{cor_tech_rel8}. 
\end{proof}

\subsection{Proof of Lemma~\ref{tech_rel19}}\label{section_tech_rel19}

In this section, we prove Lemma~\ref{tech_rel19}. 
To prove this lemma, we prepare the following six technical lemmas.

\begin{lem}\label{tech_rel14}
For $1\leq i<j\leq 2n$ such that $j-i\geq 4$ is even, the relation
\begin{align*}
&h_{j-2}(h_{j-3}\cdots h_{i+1}h_{i})\\
=&(h_{j-3}\cdots h_{i+1}h_{i})t_{i,i+1}^{-\frac{j-i-4}{2}}h_{i}^{-1}h_{i+2}^{-1}\cdots h_{j-4}^{-1}h_{j-2}h_{j-4}\cdots h_{i+2}h_{i}t_{j-1,j}^{-1}t_{i,i+1}^{\frac{j-i-2}{2}}
\end{align*}
is obtained from the relations~(1) and (2). 
\end{lem}

\begin{proof}
We have
\begin{eqnarray*}
&&\underline{h_{j-2}\cdot h_{j-3}}h_{j-4}\cdots h_{i}\\
&\overset{\text{(2)(d)}}{\underset{}{=}}&h_{j-3}h_{j-2}\underset{\rightarrow }{\underline{t_{j-1,j}^{-1}}}\ \underline{t_{j-3,j-2}\cdot h_{j-4}}h_{j-5}\cdots h_{i}\\
&\overset{\text{(2)(a)}}{\underset{\text{(1)}}{=}}&(h_{j-3}h_{j-4})h_{j-4}^{-1}h_{j-2}h_{j-4}\underline{t_{j-1,j}^{-1}t_{j-4,j-3}\cdot h_{j-5}}h_{j-6}\cdots h_{i}\\
&\overset{\text{(2)(a)}}{\underset{\text{(1)(c)}}{=}}&h_{j-3}h_{j-4}\cdot h_{j-4}^{-1}h_{j-2}\underline{h_{j-4}h_{j-5}}t_{j-1,j}^{-1}t_{j-5,j-4}\cdot h_{j-6}h_{j-7}\cdots h_{i}\\
&\overset{\text{(2)(d)}}{\underset{\text{(1)(c)}}{=}}&h_{j-3}h_{j-4}\cdot h_{j-4}^{-1}h_{j-2}\underset{\leftarrow }{\underline{h_{j-5}}}h_{j-4}t_{j-3,j-2}^{-1}\underline{t_{j-1,j}^{-1}t_{j-5,j-4}^2\cdot h_{j-6}}h_{j-7}\cdots h_{i}\\
&\overset{\text{(1)}}{\underset{\text{(2)(a)}}{=}}&h_{j-3}h_{j-4}\cdot \underline{h_{j-4}^{-1}h_{j-5}}h_{j-2}h_{j-4}t_{j-3,j-2}^{-1}h_{j-6}t_{j-1,j}^{-1}t_{j-6,j-5}^2\cdot h_{j-7}h_{j-8}\cdots h_{i}\\
&\overset{\text{(2)(d)}}{\underset{\text{(2)(a)}}{=}}&h_{j-3}h_{j-4}h_{j-5}\cdot t_{j-5,j-4}^{-1}h_{j-4}^{-1}\underline{t_{j-4,j-3}h_{j-2}h_{j-4}}t_{j-3,j-2}^{-1}h_{j-6}t_{j-1,j}^{-1}t_{j-6,j-5}^2\\
&&\cdot h_{j-7}h_{j-8}\cdots h_{i}\\
&\overset{\text{(1)(c)}}{\underset{\text{(2)(a)}}{=}}&h_{j-3}h_{j-4}h_{j-5}h_{j-6}\cdot \underline{h_{j-6}^{-1}t_{j-5,j-4}^{-1}}h_{j-4}^{-1}h_{j-2}h_{j-4}h_{j-6}t_{j-1,j}^{-1}t_{j-6,j-5}^2\\
&&\cdot h_{j-7}h_{j-8}\cdots h_{i}\\
&\overset{\text{(1)(c)}}{\underset{\text{(2)(a)}}{=}}&(h_{j-3}h_{j-4}h_{j-5}h_{j-6})t_{j-6,j-5}^{-1}h_{j-6}^{-1}h_{j-4}^{-1}h_{j-2}h_{j-4}h_{j-6}t_{j-1,j}^{-1}t_{j-6,j-5}^2\\
&&\cdot h_{j-7}h_{j-8}\cdots h_{i}.
\end{eqnarray*}
By an inductive argument, for a positive odd $3\leq l\leq j-i-7$, the relation 
\begin{align*}
&(h_{j-3}h_{j-4}\cdots h_{j-l-3})t_{j-l-3,j-l-2}^{-\frac{l-1}{2}}h_{j-l-3}^{-1}\cdots h_{j-6}^{-1}h_{j-4}^{-1}h_{j-2}h_{j-4}h_{j-6}\cdots h_{j-l-3}t_{j-1,j}^{-1}\\
&\cdot t_{j-l-3,j-l-2}^{\frac{l+1}{2}}h_{j-l-4}\cdots h_{i+1}h_i=(h_{j-3}h_{j-4}\cdots h_{j-l-5})t_{j-l-5,j-l-4}^{-\frac{l+1}{2}}h_{j-l-5}^{-1}\cdots h_{j-6}^{-1}h_{j-4}^{-1}\\
&\cdot h_{j-2}h_{j-4}h_{j-6}\cdots h_{j-l-5}t_{j-1,j}^{-1}\cdot t_{j-l-5,j-l-4}^{\frac{l+3}{2}}h_{j-l-6}\cdots h_{i+1}h_i
\end{align*}
is obtained from the relations~(1) and (2). 
Thus, up to the relations~(1) and (2), we have 
\begin{eqnarray*}
&&h_{j-2}(h_{j-3}h_{j-4}\cdots h_{i})\\
&=&(h_{j-3}h_{j-4}\cdots h_{i+2})t_{i+2,i+3}^{-\frac{j-i-6}{2}}h_{i+2}^{-1}\cdots h_{j-6}^{-1}h_{j-4}^{-1}h_{j-2}h_{j-4}h_{j-6}\cdots h_{i+2}\\
&&\cdot \underline{t_{j-1,j}^{-1}t_{i+2,i+3}^{\frac{j-i-4}{2}}h_{i+1}h_i}\\
&\overset{\text{(2)(a)}}{\underset{\text{(1)(c)}}{=}}&h_{j-3}h_{j-4}\cdots h_{i+2}\cdot t_{i+2,i+3}^{-\frac{j-i-6}{2}}h_{i+2}^{-1}\cdots h_{j-6}^{-1}h_{j-4}^{-1}h_{j-2}h_{j-4}h_{j-6}\cdots \underline{h_{i+2}\cdot h_{i+1}}h_i\\
&&\cdot t_{j-1,j}^{-1}t_{i,i+1}^{\frac{j-i-4}{2}}\\
&\overset{\text{(2)(d)}}{\underset{}{=}}&h_{j-3}h_{j-4}\cdots h_{i+2}\cdot t_{i+2,i+3}^{-\frac{j-i-6}{2}}h_{i+2}^{-1}\cdots h_{j-6}^{-1}h_{j-4}^{-1}h_{j-2}h_{j-4}h_{j-6}\cdots h_{i+4}\\
&&\cdot \underset{\leftarrow }{\underline{h_{i+1}}}h_{i+2}t_{i+3,i+4}^{-1}\underline{t_{i+1,i+2}h_i}t_{j-1,j}^{-1}t_{i,i+1}^{\frac{j-i-4}{2}}\\
&\overset{\text{(1)}}{\underset{\text{(2)(a)}}{=}}&h_{j-3}h_{j-4}\cdots h_{i+2}\cdot t_{i+2,i+3}^{-\frac{j-i-6}{2}}\underline{h_{i+2}^{-1}h_{i+1}}\\
&&\cdot h_{i+4}^{-1}\cdots h_{j-6}^{-1}h_{j-4}^{-1}h_{j-2}h_{j-4}h_{j-6}\cdots h_{i+4}h_{i+2}t_{i+3,i+4}^{-1}h_it_{j-1,j}^{-1}t_{i,i+1}^{\frac{j-i-2}{2}}\\
&\overset{\text{(2)(d)}}{\underset{}{=}}&h_{j-3}h_{j-4}\cdots h_{i+2}\cdot \underline{t_{i+2,i+3}^{-\frac{j-i-6}{2}}h_{i+1}}t_{i+1,i+2}^{-1}\underline{t_{i+3,i+4}h_{i+2}^{-1}}\\
&&\cdot h_{i+4}^{-1}\cdots h_{j-6}^{-1}h_{j-4}^{-1}h_{j-2}h_{j-4}h_{j-6}\cdots h_{i+4}h_{i+2}t_{i+3,i+4}^{-1}h_it_{j-1,j}^{-1}t_{i,i+1}^{\frac{j-i-2}{2}}\\
&\overset{\text{(2)(a)}}{\underset{}{=}}&h_{j-3}\cdots h_{i+2}h_{i+1}\cdot t_{i+1,i+2}^{-\frac{j-i-4}{2}}h_{i+2}^{-1}\underset{\rightarrow }{\underline{t_{i+2,i+3}}}\\
&&\cdot h_{i+4}^{-1}\cdots h_{j-6}^{-1}h_{j-4}^{-1}h_{j-2}h_{j-4}h_{j-6}\cdots h_{i+4}h_{i+2}t_{i+3,i+4}^{-1}h_it_{j-1,j}^{-1}t_{i,i+1}^{\frac{j-i-2}{2}}\\
&\overset{\text{(1)(a)}}{\underset{}{=}}&h_{j-3}\cdots h_{i+1}h_{i}\cdot \underline{h_{i}^{-1}t_{i+1,i+2}^{-\frac{j-i-4}{2}}}h_{i+2}^{-1}\\
&&\cdot h_{i+4}^{-1}\cdots h_{j-6}^{-1}h_{j-4}^{-1}h_{j-2}h_{j-4}h_{j-6}\cdots h_{i+4}\underline{t_{i+2,i+3}h_{i+2}}t_{i+3,i+4}^{-1}h_it_{j-1,j}^{-1}t_{i,i+1}^{\frac{j-i-2}{2}}\\
&\overset{\text{(2)(a)}}{\underset{}{=}}&(h_{j-3}\cdots h_{i+1}h_{i})t_{i,i+1}^{-\frac{j-i-4}{2}}h_{i}^{-1}h_{i+2}^{-1}\\
&&\cdot h_{i+4}^{-1}\cdots h_{j-6}^{-1}h_{j-4}^{-1}h_{j-2}h_{j-4}h_{j-6}\cdots h_{i+4}h_{i+2}h_it_{j-1,j}^{-1}t_{i,i+1}^{\frac{j-i-2}{2}}\\
\end{eqnarray*}
Therefore, we have completed the proof of Lemma~\ref{tech_rel14}. 
\end{proof}

\begin{lem}\label{tech_rel15}
For $1\leq i<j\leq 2n$ such that $j-i\geq 4$ is even, the relation
\begin{align*}
&t_{i,i+1}^{-\frac{j-i-4}{2}}h_{i}^{-1}h_{i+2}^{-1}\cdots h_{j-4}^{-1}h_{j-2}h_{j-4}\cdots h_{i+2}h_{i}t_{j-1,j}^{-1}t_{i,i+1}^{\frac{j-i-2}{2}}(h_{j-2}\cdots h_{i+1}h_{i})\\
=&(h_{j-2}\cdots h_{i+1}h_{i})t_{i+2,i+3}^{-\frac{j-i-4}{2}}t_{i,i+1}^{-\frac{j-i-2}{2}}h_{i}t_{i+2,i+3}^{\frac{j-i-2}{2}}t_{i,i+1}^{\frac{j-i-4}{2}}
\end{align*}
is obtained from the relations~(1) and (2). 
\end{lem}

\begin{proof}
We have
\begin{eqnarray*}
&&t_{i,i+1}^{-\frac{j-i-4}{2}}h_{i}^{-1}\cdots h_{j-6}^{-1}h_{j-4}^{-1}h_{j-2}h_{j-4}h_{j-6}\cdots h_{i}\underline{t_{j-1,j}^{-1}t_{i,i+1}^{\frac{j-i-2}{2}}(h_{j-2}h_{j-3}\cdots h_{i})}\\
&\overset{\text{(1)}}{\underset{\text{(2)}}{=}}&t_{i,i+1}^{-\frac{j-i-4}{2}}h_{i}^{-1}\cdots h_{j-6}^{-1}h_{j-4}^{-1}h_{j-2}h_{j-4}h_{j-6}\cdots h_{i}(\underset{\leftarrow }{\underline{h_{j-2}}}h_{j-3}\cdots h_{i})t_{i,i+1}^{-1}t_{i+2,i+3}^{\frac{j-i-2}{2}}\\
&\overset{\text{(1)(a)}}{\underset{}{=}}&t_{i,i+1}^{-\frac{j-i-4}{2}}h_{i}^{-1}\cdots h_{j-6}^{-1}h_{j-4}^{-1}\underline{h_{j-2}h_{j-4}h_{j-2}}\cdot h_{j-6}h_{j-8}\cdots h_{i}\cdot h_{j-3}h_{j-4}\cdots h_{i}\\
&&\cdot t_{i,i+1}^{-1}t_{i+2,i+3}^{\frac{j-i-2}{2}}\\
&\overset{\text{(2)(e)}}{\underset{}{=}}&t_{i,i+1}^{-\frac{j-i-4}{2}}h_{i}^{-1}\cdots h_{j-8}^{-1}h_{j-6}^{-1}\cdot \underset{\leftarrow }{\underline{h_{j-2}}}h_{j-4}\underset{\rightarrow }{\underline{t_{j-3,j-2}^{-1}t_{j-2,j-1}}}\cdot h_{j-6}h_{j-8}\cdots h_{i}\\
&&\cdot h_{j-3}h_{j-4}\cdots h_{i}t_{i,i+1}^{-1}t_{i+2,i+3}^{\frac{j-i-2}{2}}\\
&\overset{\text{(1)(a)}}{\underset{\text{(1)(c)}}{=}}&h_{j-2}t_{i,i+1}^{-\frac{j-i-4}{2}}h_{i}^{-1}\cdots h_{j-8}^{-1}h_{j-6}^{-1}\cdot h_{j-4}\cdot h_{j-6}h_{j-8}\cdots h_{i}\\
&&\cdot t_{j-3,j-2}^{-1}\underline{t_{j-2,j-1}h_{j-3}h_{j-4}\cdots h_{i}}t_{i,i+1}^{-1}t_{i+2,i+3}^{\frac{j-i-2}{2}}\\
&\overset{\text{(2)(e)}}{\underset{}{=}}&h_{j-2}t_{i,i+1}^{-\frac{j-i-4}{2}}h_{i}^{-1}\cdots h_{j-8}^{-1}h_{j-6}^{-1}\cdot h_{j-4}\cdot h_{j-6}h_{j-8}\cdots h_{i}\\
&&\cdot \underline{t_{j-3,j-2}^{-1}h_{j-3}}h_{j-4}\cdots h_{i}t_{i+2,i+3}^{\frac{j-i-2}{2}}\\
&\overset{\text{(2)(a)}}{\underset{}{=}}&h_{j-2}t_{i,i+1}^{-\frac{j-i-4}{2}}h_{i}^{-1}\cdots h_{j-8}^{-1}h_{j-6}^{-1}\cdot h_{j-4}\cdot h_{j-6}h_{j-8}\cdots h_{i}\\
&&\cdot \underset{\leftarrow }{\underline{h_{j-3}}}t_{j-2,j-1}^{-1}h_{j-4}h_{j-5}\cdots h_{i}t_{i+2,i+3}^{\frac{j-i-2}{2}}\\
&\overset{\text{(1)(a)}}{\underset{}{=}}&h_{j-2}t_{i,i+1}^{-\frac{j-i-4}{2}}h_{i}^{-1}\cdots h_{j-8}^{-1}h_{j-6}^{-1}\cdot \underline{h_{j-4}h_{j-3}}\cdot h_{j-6}h_{j-8}\cdots h_{i}\\
&&\cdot t_{j-2,j-1}^{-1}h_{j-4}h_{j-5}\cdots h_{i}t_{i+2,i+3}^{\frac{j-i-2}{2}}\\
&\overset{\text{(2)(d)}}{\underset{\text{(1),(2)}}{=}}&h_{j-2}t_{i,i+1}^{-\frac{j-i-4}{2}}h_{i}^{-1}\cdots h_{j-8}^{-1}h_{j-6}^{-1}\cdot \underset{\leftarrow }{\underline{h_{j-3}t_{j-3,j-2}^{-1}}}h_{j-4}\cdot h_{j-6}h_{j-8}\cdots h_{i}\\
&&\cdot h_{j-4}h_{j-5}\cdots h_{i}t_{i+2,i+3}^{\frac{j-i-2}{2}}\\
&\overset{\text{(2)(d)}}{\underset{\text{(1)(c)}}{=}}&(h_{j-2}h_{j-3})t_{i,i+1}^{-\frac{j-i-4}{2}}t_{j-3,j-2}^{-1}h_{i}^{-1}\cdots h_{j-8}^{-1}h_{j-6}^{-1}h_{j-4}h_{j-6}h_{j-8}\cdots h_{i}\\
&&\cdot (h_{j-4}h_{j-5}\cdots h_{i})t_{i+2,i+3}^{\frac{j-i-2}{2}}
\end{eqnarray*}
By an inductive argument, for a positive odd $3\leq l\leq j-i-5$, the relation 
\begin{align*}
&(h_{j-2}h_{j-3}\cdots h_{j-l})t_{i,i+1}^{-\frac{j-i-4}{2}}t_{j-l,j-l+1}^{-\frac{l-1}{2}}h_{i}^{-1}\cdots h_{j-l-5}^{-1}h_{j-l-3}^{-1}h_{j-l-1}h_{j-l-3}h_{j-l-5}\cdots h_{i}\\
&\cdot (h_{j-l-1}h_{j-l-2}\cdots h_{i})t_{i,i+1}^{\frac{l-3}{2}}t_{i+2,i+3}^{\frac{j-i-2}{2}}=(h_{j-2}h_{j-3}\cdots h_{j-l-2})t_{i,i+1}^{-\frac{j-i-4}{2}}t_{j-l-2,j-l-1}^{-\frac{l+1}{2}}\\
&\cdot h_{i}^{-1}\cdots h_{j-l-7}^{-1}h_{j-l-5}^{-1}h_{j-l-3}h_{j-l-5}h_{j-l-7}\cdots h_{i}(h_{j-l-3}h_{j-l-4}\cdots h_{i})t_{i,i+1}^{\frac{l-1}{2}}t_{i+2,i+3}^{\frac{j-i-2}{2}}
\end{align*}
is obtained from the relations~(1) and (2). 
Thus, up to the relations~(1) and (2), we have 
\begin{eqnarray*}
&&t_{i,i+1}^{-\frac{j-i-4}{2}}h_{i}^{-1}\cdots h_{j-6}^{-1}h_{j-4}^{-1}h_{j-2}h_{j-4}h_{j-6}\cdots h_{i}t_{j-1,j}^{-1}t_{i,i+1}^{\frac{j-i-2}{2}}\cdot h_{j-2}h_{j-3}\cdots h_{i}\\
&\overset{}{\underset{}{=}}&h_{j-2}h_{j-3}\cdots h_{i+3}t_{i,i+1}^{-\frac{j-i-4}{2}}t_{i+3,i+4}^{-\frac{j-i-4}{2}}h_{i}^{-1}\underline{h_{i+2}h_{i}\cdot h_{i+2}}h_{i+1}h_{i}t_{i,i+1}^{\frac{j-i-6}{2}}t_{i+2,i+3}^{\frac{j-i-2}{2}}\\
&\overset{\text{(2)(e)}}{\underset{\text{(2)(a)}}{=}}&h_{j-2}h_{j-3}\cdots h_{i+3}t_{i,i+1}^{-\frac{j-i-4}{2}}\underline{t_{i+3,i+4}^{-\frac{j-i-4}{2}}h_{i+2}t_{i,i+1}^{-1}}h_{i}\underline{t_{i+2,i+3}h_{i+1}h_{i}}t_{i,i+1}^{\frac{j-i-6}{2}}t_{i+2,i+3}^{\frac{j-i-2}{2}}\\
&\overset{\text{(2)(a)}}{\underset{\text{(1)}}{=}}&h_{j-2}h_{j-3}\cdots h_{i+3}t_{i,i+1}^{-\frac{j-i-4}{2}}\underset{\leftarrow }{\underline{h_{i+2}}}t_{i,i+1}^{-1}t_{i+2,i+3}^{-\frac{j-i-4}{2}}\underline{h_{i}h_{i+1}}h_{i}t_{i,i+1}^{\frac{j-i-4}{2}}t_{i+2,i+3}^{\frac{j-i-2}{2}}\\
&\overset{\text{(2)(d)}}{\underset{\text{(1)}}{=}}&h_{j-2}h_{j-3}\cdots h_{i+2}\underline{t_{i,i+1}^{-\frac{j-i-2}{2}}t_{i+2,i+3}^{-\frac{j-i-4}{2}}h_{i+1}h_{i}}t_{i+2,i+3}t_{i,i+1}^{-1}h_{i}t_{i,i+1}^{\frac{j-i-4}{2}}t_{i+2,i+3}^{\frac{j-i-2}{2}}\\
&\overset{\text{(2)(a)}}{\underset{\text{(2)(b)}}{=}}&h_{j-2}h_{j-3}\cdots h_{i}\underline{t_{i+2,i+3}^{-\frac{j-i-2}{2}}t_{i,i+1}^{-\frac{j-i-4}{2}}t_{i+2,i+3}t_{i,i+1}^{-1}}h_{i}\underline{t_{i,i+1}^{\frac{j-i-4}{2}}t_{i+2,i+3}^{\frac{j-i-2}{2}}}\\
&\overset{\text{(1)(b)}}{\underset{}{=}}&h_{j-2}h_{j-3}\cdots h_{i}t_{i,i+1}^{-\frac{j-i-2}{2}}t_{i+2,i+3}^{-\frac{j-i-4}{2}}h_{i}t_{i+2,i+3}^{\frac{j-i-2}{2}}t_{i,i+1}^{\frac{j-i-4}{2}}
\end{eqnarray*}
Therefore, we have completed the proof of Lemma~\ref{tech_rel15}. 
\end{proof}

\begin{lem}\label{tech_rel16}
For $1\leq i<j\leq 2n$ such that $j-i\geq 4$ is even, the relation
\begin{align*}
&(t_{i+2,i+3}^{-\frac{j-i-4}{2}}t_{i,i+1}^{-\frac{j-i-2}{2}}h_{i}t_{i+2,i+3}^{\frac{j-i-2}{2}}t_{i,i+1}^{\frac{j-i-4}{2}}h_{j-2}\cdots h_{i+1}h_{i})^{\frac{j-i-4}{2}}=(h_{j-2}\cdots h_{i+1}h_{i})^{\frac{j-i-4}{2}}\\
&\cdot t_{j-2,j-1}^{-\frac{j-i-4}{2}}t_{j-4,j-3}^{-\frac{j-i-2}{2}}\cdots t_{i+4,i+5}^{-\frac{j-i-2}{2}}t_{i+2,i+3}^{-\frac{j-i-2}{2}}h_{j-4}\cdots h_{i+4}h_{i+2}t_{i+2,i+3}^{\frac{j-i-4}{2}}t_{i+4,i+5}^{\frac{j-i-2}{2}}t_{i+6,i+7}^{\frac{j-i-2}{2}}\cdots t_{j-2,j-1}^{\frac{j-i-2}{2}}
\end{align*}
is obtained from the relations~(1) and (2). 
\end{lem}

\begin{proof}
We have
\begin{eqnarray*}
&&(t_{i+2,i+3}^{-\frac{j-i-4}{2}}t_{i,i+1}^{-\frac{j-i-2}{2}}h_{i}t_{i+2,i+3}^{\frac{j-i-2}{2}}t_{i,i+1}^{\frac{j-i-4}{2}}\underset{\leftarrow }{\underline{h_{j-2}\cdots h_{i+2}}}h_{i+1}h_{i})^{\frac{j-i-4}{2}}\\
&\overset{\text{(1)(a)}}{\underset{\text{(1)(c)}}{=}}&(h_{j-2}\cdots h_{i+5}h_{i+4}t_{i+2,i+3}^{-\frac{j-i-4}{2}}t_{i,i+1}^{-\frac{j-i-2}{2}}h_{i}\underline{t_{i+2,i+3}^{\frac{j-i-2}{2}}h_{i+3}h_{i+2}}\ \underline{t_{i,i+1}^{\frac{j-i-4}{2}}h_{i+1}h_{i}})^{\frac{j-i-4}{2}}\\
&\overset{\text{(2)(b)}}{\underset{}{=}}&(h_{j-2}\cdots h_{i+5}h_{i+4}t_{i+2,i+3}^{-\frac{j-i-4}{2}}t_{i,i+1}^{-\frac{j-i-2}{2}}h_{i}\underset{\leftarrow }{\underline{h_{i+3}}}h_{i+2}t_{i+4,i+5}^{\frac{j-i-2}{2}}\underset{\leftarrow }{\underline{h_{i+1}h_{i}}}t_{i+2,i+3}^{\frac{j-i-4}{2}})^{\frac{j-i-4}{2}}\\
&\overset{\text{(1)(a)}}{\underset{\text{(1)(c)}}{=}}&(h_{j-2}\cdots h_{i+5}h_{i+4}t_{i+2,i+3}^{-\frac{j-i-4}{2}}t_{i,i+1}^{-\frac{j-i-2}{2}}h_{i+3}\underline{h_{i}h_{i+2}h_{i+1}h_{i}}t_{i+4,i+5}^{\frac{j-i-2}{2}}t_{i+2,i+3}^{\frac{j-i-4}{2}})^{\frac{j-i-4}{2}}\\
&\overset{\text{(2)(c)}}{\underset{}{=}}&(h_{j-2}\cdots h_{i+5}h_{i+4}t_{i+2,i+3}^{-\frac{j-i-4}{2}}t_{i,i+1}^{-\frac{j-i-2}{2}}\underset{\leftarrow }{\underline{h_{i+3}h_{i+2}}}h_{i+1}h_{i}h_{i+2}t_{i+4,i+5}^{\frac{j-i-2}{2}}t_{i+2,i+3}^{\frac{j-i-4}{2}})^{\frac{j-i-4}{2}}\\
&\overset{\text{(1)(c)}}{\underset{}{=}}&(h_{j-2}\cdots h_{i+5}h_{i+4}\underline{t_{i+2,i+3}^{-\frac{j-i-4}{2}}h_{i+3}h_{i+2}}\ \underline{t_{i,i+1}^{-\frac{j-i-2}{2}}h_{i+1}h_{i}}h_{i+2}t_{i+4,i+5}^{\frac{j-i-2}{2}}t_{i+2,i+3}^{\frac{j-i-4}{2}})^{\frac{j-i-4}{2}}\\
&\overset{\text{(2)(b)}}{\underset{}{=}}&(h_{j-2}\cdots h_{i+3}h_{i+2}t_{i+4,i+5}^{-\frac{j-i-4}{2}}\underset{\leftarrow }{\underline{h_{i+1}h_{i}}}t_{i+2,i+3}^{-\frac{j-i-2}{2}}h_{i+2}t_{i+4,i+5}^{\frac{j-i-2}{2}}t_{i+2,i+3}^{\frac{j-i-4}{2}})^{\frac{j-i-4}{2}}\\
&\overset{\text{(1)(c)}}{\underset{}{=}}&(h_{j-2}\cdots h_{i+1}h_{i}t_{i+4,i+5}^{-\frac{j-i-4}{2}}t_{i+2,i+3}^{-\frac{j-i-2}{2}}h_{i+2}t_{i+4,i+5}^{\frac{j-i-2}{2}}t_{i+2,i+3}^{\frac{j-i-4}{2}})^{\frac{j-i-4}{2}}\\
&\overset{}{\underset{}{=}}&h_{j-2}\cdots h_{i+1}h_{i}(t_{i+4,i+5}^{-\frac{j-i-4}{2}}t_{i+2,i+3}^{-\frac{j-i-2}{2}}h_{i+2}t_{i+4,i+5}^{\frac{j-i-2}{2}}t_{i+2,i+3}^{\frac{j-i-4}{2}}\underset{\leftarrow }{\underline{h_{j-2}\cdots h_{i+4}}}h_{i+3}h_{i+2}\\
&&\cdot h_{i+1}h_{i})^{\frac{j-i-6}{2}}t_{i+4,i+5}^{-\frac{j-i-4}{2}}t_{i+2,i+3}^{-\frac{j-i-2}{2}}h_{i+2}t_{i+4,i+5}^{\frac{j-i-2}{2}}t_{i+2,i+3}^{\frac{j-i-4}{2}}\\
&\overset{\text{(1)(a)}}{\underset{\text{(1)(c)}}{=}}&h_{j-2}\cdots h_{i+1}h_{i}(h_{j-2}\cdots h_{i+7}h_{i+6}t_{i+4,i+5}^{-\frac{j-i-4}{2}}t_{i+2,i+3}^{-\frac{j-i-2}{2}}h_{i+2}\underline{t_{i+4,i+5}^{\frac{j-i-2}{2}}h_{i+5}h_{i+4}}\\
&&\cdot \underline{t_{i+2,i+3}^{\frac{j-i-4}{2}}h_{i+3}h_{i+2}}h_{i+1}h_{i})^{\frac{j-i-6}{2}}t_{i+4,i+5}^{-\frac{j-i-4}{2}}t_{i+2,i+3}^{-\frac{j-i-2}{2}}h_{i+2}t_{i+4,i+5}^{\frac{j-i-2}{2}}t_{i+2,i+3}^{\frac{j-i-4}{2}}\\
&\overset{\text{(2)(b)}}{\underset{}{=}}&h_{j-2}\cdots h_{i+1}h_{i}(h_{j-2}\cdots h_{i+7}h_{i+6}t_{i+4,i+5}^{-\frac{j-i-4}{2}}t_{i+2,i+3}^{-\frac{j-i-2}{2}}h_{i+2}h_{i+5}h_{i+4}t_{i+6,i+7}^{\frac{j-i-2}{2}}\\
&&\cdot \underset{\leftarrow }{\underline{h_{i+3}h_{i+2}}}t_{i+4,i+5}^{\frac{j-i-4}{2}}\underset{\leftarrow }{\underline{h_{i+1}h_{i}}})^{\frac{j-i-6}{2}}t_{i+4,i+5}^{-\frac{j-i-4}{2}}t_{i+2,i+3}^{-\frac{j-i-2}{2}}h_{i+2}t_{i+4,i+5}^{\frac{j-i-2}{2}}t_{i+2,i+3}^{\frac{j-i-4}{2}}\\
&\overset{\text{(1)(c)}}{\underset{}{=}}&h_{j-2}\cdots h_{i+1}h_{i}(h_{j-2}\cdots h_{i+7}h_{i+6}t_{i+4,i+5}^{-\frac{j-i-4}{2}}t_{i+2,i+3}^{-\frac{j-i-2}{2}}h_{i+2}\underset{\leftarrow }{\underline{h_{i+5}}}h_{i+4}h_{i+3}h_{i+2}\\
&&\cdot h_{i+1}h_{i}t_{i+6,i+7}^{\frac{j-i-2}{2}}t_{i+4,i+5}^{\frac{j-i-4}{2}})^{\frac{j-i-6}{2}}t_{i+4,i+5}^{-\frac{j-i-4}{2}}t_{i+2,i+3}^{-\frac{j-i-2}{2}}h_{i+2}t_{i+4,i+5}^{\frac{j-i-2}{2}}t_{i+2,i+3}^{\frac{j-i-4}{2}}\\
&\overset{\text{(1)(a)}}{\underset{}{=}}&h_{j-2}\cdots h_{i+1}h_{i}(h_{j-2}\cdots h_{i+7}h_{i+6}t_{i+4,i+5}^{-\frac{j-i-4}{2}}t_{i+2,i+3}^{-\frac{j-i-2}{2}}h_{i+5}\underline{h_{i+2}h_{i+4}h_{i+3}h_{i+2}}\\
&&\cdot h_{i+1}h_{i}t_{i+6,i+7}^{\frac{j-i-2}{2}}t_{i+4,i+5}^{\frac{j-i-4}{2}})^{\frac{j-i-6}{2}}t_{i+4,i+5}^{-\frac{j-i-4}{2}}t_{i+2,i+3}^{-\frac{j-i-2}{2}}h_{i+2}t_{i+4,i+5}^{\frac{j-i-2}{2}}t_{i+2,i+3}^{\frac{j-i-4}{2}}\\
&\overset{\text{(2)(c)}}{\underset{}{=}}&h_{j-2}\cdots h_{i+1}h_{i}(h_{j-2}\cdots h_{i+7}h_{i+6}t_{i+4,i+5}^{-\frac{j-i-4}{2}}t_{i+2,i+3}^{-\frac{j-i-2}{2}}\underset{\leftarrow }{\underline{h_{i+5}h_{i+4}}}h_{i+3}h_{i+2}h_{i+4}\\
&&\cdot \underset{\leftarrow }{\underline{h_{i+1}h_{i}}}t_{i+6,i+7}^{\frac{j-i-2}{2}}t_{i+4,i+5}^{\frac{j-i-4}{2}})^{\frac{j-i-6}{2}}t_{i+4,i+5}^{-\frac{j-i-4}{2}}t_{i+2,i+3}^{-\frac{j-i-2}{2}}h_{i+2}t_{i+4,i+5}^{\frac{j-i-2}{2}}t_{i+2,i+3}^{\frac{j-i-4}{2}}\\
&\overset{\text{(1)(a)}}{\underset{\text{(1)(c)}}{=}}&h_{j-2}\cdots h_{i+1}h_{i}(h_{j-2}\cdots h_{i+7}h_{i+6}\underline{t_{i+4,i+5}^{-\frac{j-i-4}{2}}h_{i+5}h_{i+4}}\ \underline{t_{i+2,i+3}^{-\frac{j-i-2}{2}}h_{i+3}h_{i+2}}h_{i+1}h_{i}\\
&&\cdot h_{i+4}t_{i+6,i+7}^{\frac{j-i-2}{2}}t_{i+4,i+5}^{\frac{j-i-4}{2}})^{\frac{j-i-6}{2}}t_{i+4,i+5}^{-\frac{j-i-4}{2}}t_{i+2,i+3}^{-\frac{j-i-2}{2}}h_{i+2}t_{i+4,i+5}^{\frac{j-i-2}{2}}t_{i+2,i+3}^{\frac{j-i-4}{2}}\\
&\overset{\text{(2)(b)}}{\underset{}{=}}&h_{j-2}\cdots h_{i+1}h_{i}(h_{j-2}\cdots h_{i+5}h_{i+4}t_{i+6,i+7}^{-\frac{j-i-4}{2}}\underset{\leftarrow }{\underline{h_{i+3}h_{i+2}}}t_{i+4,i+5}^{-\frac{j-i-2}{2}}\underset{\leftarrow }{\underline{h_{i+1}h_{i}}}\\
&&\cdot h_{i+4}t_{i+6,i+7}^{\frac{j-i-2}{2}}t_{i+4,i+5}^{\frac{j-i-4}{2}})^{\frac{j-i-6}{2}}t_{i+4,i+5}^{-\frac{j-i-4}{2}}t_{i+2,i+3}^{-\frac{j-i-2}{2}}h_{i+2}t_{i+4,i+5}^{\frac{j-i-2}{2}}t_{i+2,i+3}^{\frac{j-i-4}{2}}\\
&\overset{\text{(1)(c)}}{\underset{}{=}}&h_{j-2}\cdots h_{i+1}h_{i}(h_{j-2}\cdots h_{i+1}h_{i}t_{i+6,i+7}^{-\frac{j-i-4}{2}}t_{i+4,i+5}^{-\frac{j-i-2}{2}}h_{i+4}t_{i+6,i+7}^{\frac{j-i-2}{2}}t_{i+4,i+5}^{\frac{j-i-4}{2}})^{\frac{j-i-6}{2}}\\
&&\cdot t_{i+4,i+5}^{-\frac{j-i-4}{2}}t_{i+2,i+3}^{-\frac{j-i-2}{2}}h_{i+2}t_{i+4,i+5}^{\frac{j-i-2}{2}}t_{i+2,i+3}^{\frac{j-i-4}{2}}\\
&\overset{\text{(1)(c)}}{\underset{}{=}}&(h_{j-2}\cdots h_{i+1}h_{i})^2(t_{i+6,i+7}^{-\frac{j-i-4}{2}}t_{i+4,i+5}^{-\frac{j-i-2}{2}}h_{i+4}t_{i+6,i+7}^{\frac{j-i-2}{2}}t_{i+4,i+5}^{\frac{j-i-4}{2}}h_{j-2}\cdots h_{i+1}h_{i})^{\frac{j-i-8}{2}}\\
&&\cdot t_{i+6,i+7}^{-\frac{j-i-4}{2}}t_{i+4,i+5}^{-\frac{j-i-2}{2}}h_{i+4}t_{i+6,i+7}^{\frac{j-i-2}{2}}t_{i+2,i+3}^{-\frac{j-i-2}{2}}h_{i+2}t_{i+4,i+5}^{\frac{j-i-2}{2}}t_{i+2,i+3}^{\frac{j-i-4}{2}}\\
&\vdots &\\
&=&(h_{j-2}\cdots h_{i+1}h_{i})^{\frac{j-i-4}{2}}t_{j-2,j-1}^{-\frac{j-i-4}{2}}t_{j-4,j-3}^{-\frac{j-i-2}{2}}h_{j-4}\underset{\rightarrow }{\underline{t_{j-2,j-1}^{\frac{j-i-2}{2}}}}\ \underset{\leftarrow }{\underline{t_{j-6,j-5}^{-\frac{j-i-2}{2}}}}h_{j-6}\\
&&\cdots \underset{\rightarrow }{\underline{t_{i+8,i+9}^{\frac{j-i-4}{2}}}}\ \underset{\leftarrow }{\underline{t_{i+4,i+5}^{-\frac{j-i-2}{2}}}}h_{i+4}\underset{\rightarrow }{\underline{t_{i+6,i+7}^{\frac{j-i-2}{2}}}}\ \underset{\leftarrow }{\underline{t_{i+2,i+3}^{-\frac{j-i-2}{2}}}}h_{i+2}t_{i+4,i+5}^{\frac{j-i-2}{2}}\underset{\leftarrow }{\underline{t_{i+2,i+3}^{\frac{j-i-4}{2}}}}\\
&\overset{\text{(1)(b)}}{\underset{\text{(1)(c)}}{=}}&(h_{j-2}\cdots h_{i+1}h_{i})^{\frac{j-i-4}{2}}t_{j-2,j-1}^{-\frac{j-i-4}{2}}t_{j-4,j-3}^{-\frac{j-i-2}{2}}\cdots t_{i+4,i+5}^{-\frac{j-i-2}{2}}t_{i+2,i+3}^{-\frac{j-i-2}{2}}\\
&&\cdot h_{j-4}\cdots h_{i+4}h_{i+2}t_{i+2,i+3}^{\frac{j-i-4}{2}}t_{i+4,i+5}^{\frac{j-i-2}{2}}t_{i+6,i+7}^{\frac{j-i-2}{2}}\cdots t_{j-2,j-1}^{\frac{j-i-2}{2}}.
\end{eqnarray*}
Therefore, we have completed the proof of Lemma~\ref{tech_rel16}. 
\end{proof}

\begin{lem}\label{tech_rel17}
For $1\leq i<j\leq 2n$ such that $j-i\geq 4$ is even, the relation
\begin{align*}
&h_{i}^{-1}h_{i+2}^{-1}\cdots h_{j-4}^{-1}h_{j-2}h_{j-4}\cdots h_{i+2}h_{i}\\
=&h_{j-2}h_{j-4}\cdots h_{i+2}h_{i}t_{i+1,i+2}^{-1}h_{i+2}^{-1}\cdots h_{j-4}^{-1}h_{j-2}^{-1}t_{j-1,j}
\end{align*}
is obtained from the relations~(1) and (2). 
\end{lem}

\begin{proof}
We have
\begin{eqnarray*}
&&h_{i}^{-1}\cdots h_{j-6}^{-1}\underline{h_{j-4}^{-1}h_{j-2}h_{j-4}}h_{j-6}\cdots h_{i}\\
&\overset{\text{(2)(e)}}{\underset{\text{(2)(a)}}{=}}&h_{i}^{-1}\cdots h_{j-8}^{-1}h_{j-6}^{-1}\cdot \underset{\leftarrow }{\underline{h_{j-2}}}h_{j-4}\underset{\rightarrow }{\underline{t_{j-3,j-2}^{-1}h_{j-2}^{-1}t_{j-1,j}}}\cdot h_{j-6}h_{j-8}\cdots h_{i}\\
&\overset{\text{(1)(a)}}{\underset{\text{(1)(c)}}{=}}&h_{j-2}\cdot h_{i}^{-1}\cdots h_{j-8}^{-1}\underline{h_{j-6}^{-1}\cdot h_{j-4}\cdot h_{j-6}}h_{j-8}\cdots h_{i}\cdot t_{j-3,j-2}^{-1}h_{j-2}^{-1}t_{j-1,j}\\
&\overset{\text{(2)(e)}}{\underset{\text{(2)(a)}}{=}}&h_{j-2}\cdot h_{i}^{-1}\cdots h_{j-10}^{-1}h_{j-8}^{-1}\cdot \underset{\leftarrow }{\underline{h_{j-4}}}h_{j-6}\underset{\rightarrow }{\underline{t_{j-5,j-4}^{-1}h_{j-4}^{-1}t_{j-3,j-2}}}\cdot h_{j-8}h_{j-10}\cdots h_{i}\\
&&\cdot t_{j-3,j-2}^{-1}h_{j-2}^{-1}t_{j-1,j}\\
&\overset{\text{(1)(a)}}{\underset{\text{(1)(c)}}{=}}&h_{j-2}h_{j-4}\cdot h_{i}^{-1}\cdots h_{j-10}^{-1}h_{j-8}^{-1}\cdot h_{j-6}\cdot h_{j-8}h_{j-10}\cdots h_{i}\cdot t_{j-5,j-4}^{-1}h_{j-4}^{-1}h_{j-2}^{-1}t_{j-1,j}\\
&\vdots &\\
&\overset{\text{(1)(a)}}{\underset{\text{(1)(c)}}{=}}&h_{j-2}h_{j-4}\cdots h_{i+2}\cdot h_{i}\cdot t_{i+1,i+2}^{-1}h_{i+2}^{-1}\cdots h_{j-4}^{-1}h_{j-2}^{-1}t_{j-1,j}
\end{eqnarray*}
Therefore, we have completed the proof of Lemma~\ref{tech_rel17}. 
\end{proof}

\begin{lem}\label{tech_rel18}
For $1\leq i<j\leq 2n$ such that $j-i\geq 4$ is even, the relation
\begin{align*}
&(h_{j-2}h_{j-2}h_{j-3}\cdots h_{i+1}h_{i})^{\frac{j-i}{2}}\\
=&(h_{j-2}h_{j-3}\cdots h_{i+1}h_{i})^{\frac{j-i}{2}}t_{j-2,j-1}^{-\frac{j-i-4}{2}}t_{j-4,j-3}^{-\frac{j-i-2}{2}}\cdots t_{i+2,i+3}^{-\frac{j-i-2}{2}}t_{i,i+1}^{-\frac{j-i-2}{2}}h_{j-4}\cdots h_{i+2}h_{i}\\
&\cdot h_{j-2}\cdots h_{i+4}h_{i+2}h_it_{i+1,i+2}^{-1}h_{i+2}^{-1}h_{i+4}^{-1}\cdots h_{j-2}^{-1}t_{i,i+1}^{\frac{j-i-2}{2}}t_{i+2,i+3}^{\frac{j-i-2}{2}}\cdots t_{j-2,j-1}^{\frac{j-i-2}{2}}
\end{align*}
is obtained from the relations~(1) and (2). 
\end{lem}

\begin{proof}
Recall that the relations in Lemmas~\ref{tech_rel14}, \ref{tech_rel15}, \ref{tech_rel16}, and \ref{tech_rel17} are obtained from the relations~(1) and (2). 
Then we have
\begin{eqnarray*}
&&(h_{j-2}\underline{h_{j-2}h_{j-3}\cdots h_{i+1}h_{i}})^{\frac{j-i}{2}}\\
&\overset{\text{Lem.\ref{tech_rel14}}}{\underset{}{=}}&(h_{j-2}h_{j-3}\cdots h_{i+1}h_{i}\cdot t_{i,i+1}^{-\frac{j-i-4}{2}}h_{i}^{-1}h_{i+2}^{-1}\cdots h_{j-4}^{-1}h_{j-2}h_{j-4}\cdots h_{i+2}h_{i}t_{j-1,j}^{-1}t_{i,i+1}^{\frac{j-i-2}{2}})^{\frac{j-i}{2}}\\
&\overset{}{\underset{}{=}}&h_{j-2}\cdots h_{i+1}h_{i}\\
&&\cdot (\underline{t_{i,i+1}^{-\frac{j-i-4}{2}}h_{i}^{-1}h_{i+2}^{-1}\cdots h_{j-4}^{-1}h_{j-2}h_{j-4}\cdots h_{i+2}h_{i}t_{j-1,j}^{-1}t_{i,i+1}^{\frac{j-i-2}{2}}h_{j-2}\cdots h_{i+1}h_{i}})^{\frac{j-i-2}{2}}\\
&&\cdot t_{i,i+1}^{-\frac{j-i-4}{2}}h_{i}^{-1}h_{i+2}^{-1}\cdots h_{j-4}^{-1}h_{j-2}h_{j-4}\cdots h_{i+2}h_{i}t_{j-1,j}^{-1}t_{i,i+1}^{\frac{j-i-2}{2}}\\
&\overset{\text{Lem.\ref{tech_rel15}}}{\underset{}{=}}&h_{j-2}\cdots h_{i+1}h_{i}(h_{j-2}\cdots h_{i+1}h_{i}\cdot t_{i+2,i+3}^{-\frac{j-i-4}{2}}t_{i,i+1}^{-\frac{j-i-2}{2}}h_{i}t_{i+2,i+3}^{\frac{j-i-2}{2}}t_{i,i+1}^{\frac{j-i-4}{2}})^{\frac{j-i-2}{2}}\\
&&\cdot t_{i,i+1}^{-\frac{j-i-4}{2}}h_{i}^{-1}h_{i+2}^{-1}\cdots h_{j-4}^{-1}h_{j-2}h_{j-4}\cdots h_{i+2}h_{i}t_{j-1,j}^{-1}t_{i,i+1}^{\frac{j-i-2}{2}}\\
&\overset{}{\underset{}{=}}&(h_{j-2}\cdots h_{i+1}h_{i})^2\underline{(t_{i+2,i+3}^{-\frac{j-i-4}{2}}t_{i,i+1}^{-\frac{j-i-2}{2}}h_{i}t_{i+2,i+3}^{\frac{j-i-2}{2}}t_{i,i+1}^{\frac{j-i-4}{2}}h_{j-2}\cdots h_{i+1}h_{i})^{\frac{j-i-4}{2}}}\\
&&\cdot t_{i,i+1}^{-\frac{j-i-2}{2}}t_{i+2,i+3}^{-\frac{j-i-4}{2}}h_{i}t_{i+2,i+3}^{\frac{j-i-2}{2}}h_{i}^{-1}h_{i+2}^{-1}\cdots h_{j-4}^{-1}h_{j-2}h_{j-4}\cdots h_{i+2}h_{i}t_{j-1,j}^{-1}t_{i,i+1}^{\frac{j-i-2}{2}}\\
&\overset{\text{Lem.\ref{tech_rel16}}}{\underset{}{=}}&(h_{j-2}\cdots h_{i+1}h_{i})^{\frac{j-i}{2}}t_{j-2,j-1}^{-\frac{j-i-4}{2}}t_{j-4,j-3}^{-\frac{j-i-2}{2}}\cdots t_{i+4,i+5}^{-\frac{j-i-2}{2}}t_{i+2,i+3}^{-\frac{j-i-2}{2}}h_{j-4}\cdots h_{i+4}h_{i+2}\\
&&\cdot t_{i+2,i+3}^{\frac{j-i-4}{2}}t_{i+4,i+5}^{\frac{j-i-2}{2}}t_{i+6,i+7}^{\frac{j-i-2}{2}}\cdots t_{j-2,j-1}^{\frac{j-i-2}{2}}\\
&&\cdot \underset{\leftarrow }{\underline{t_{i,i+1}^{-\frac{j-i-2}{2}}t_{i+2,i+3}^{-\frac{j-i-4}{2}}h_{i}t_{i+2,i+3}^{\frac{j-i-2}{2}}}}h_{i}^{-1}h_{i+2}^{-1}\cdots h_{j-4}^{-1}h_{j-2}h_{j-4}\cdots h_{i+2}h_{i}t_{j-1,j}^{-1}t_{i,i+1}^{\frac{j-i-2}{2}}\\
&\overset{\text{(1)(b)}}{\underset{\text{(1)(c)}}{=}}&(h_{j-2}\cdots h_{i+1}h_{i})^{\frac{j-i}{2}}t_{j-2,j-1}^{-\frac{j-i-4}{2}}t_{j-4,j-3}^{-\frac{j-i-2}{2}}\cdots t_{i+2,i+3}^{-\frac{j-i-2}{2}}t_{i,i+1}^{-\frac{j-i-2}{2}}h_{j-4}\cdots h_{i+2}h_{i}\\
&&\cdot t_{i+2,i+3}^{\frac{j-i-2}{2}}t_{i+4,i+5}^{\frac{j-i-2}{2}}\cdots t_{j-2,j-1}^{\frac{j-i-2}{2}}\underline{h_{i}^{-1}h_{i+2}^{-1}\cdots h_{j-4}^{-1}h_{j-2}h_{j-4}\cdots h_{i+2}h_{i}}t_{j-1,j}^{-1}t_{i,i+1}^{\frac{j-i-2}{2}}\\
&\overset{\text{Lem.\ref{tech_rel17}}}{\underset{}{=}}&(h_{j-2}\cdots h_{i+1}h_{i})^{\frac{j-i}{2}}t_{j-2,j-1}^{-\frac{j-i-4}{2}}t_{j-4,j-3}^{-\frac{j-i-2}{2}}\cdots t_{i+2,i+3}^{-\frac{j-i-2}{2}}t_{i,i+1}^{-\frac{j-i-2}{2}}h_{j-4}\cdots h_{i+2}h_{i}\\
&&\cdot \underline{t_{i+2,i+3}^{\frac{j-i-2}{2}}t_{i+4,i+5}^{\frac{j-i-2}{2}}\cdots t_{j-2,j-1}^{\frac{j-i-2}{2}}h_{j-2}h_{j-4}\cdots h_{i+2}h_{i}t_{i+1,i+2}^{-1}h_{i+2}^{-1}\cdots h_{j-4}^{-1}h_{j-2}^{-1}}t_{i,i+1}^{\frac{j-i-2}{2}}\\
&\overset{\text{(1)}}{\underset{\text{(2)}}{=}}&(h_{j-2}\cdots h_{i+1}h_{i})^{\frac{j-i}{2}}t_{j-2,j-1}^{-\frac{j-i-4}{2}}t_{j-4,j-3}^{-\frac{j-i-2}{2}}\cdots t_{i+2,i+3}^{-\frac{j-i-2}{2}}t_{i,i+1}^{-\frac{j-i-2}{2}}h_{j-4}\cdots h_{i+2}h_{i}\\
&&\cdot h_{j-2}h_{j-4}\cdots h_{i+2}h_{i}t_{i+1,i+2}^{-1}h_{i+2}^{-1}\cdots h_{j-4}^{-1}h_{j-2}^{-1}t_{i,i+1}^{\frac{j-i-2}{2}}t_{i+2,i+3}^{\frac{j-i-2}{2}}\cdots t_{j-2,j-1}^{\frac{j-i-2}{2}}\\
\end{eqnarray*}
Therefore, we have completed the proof of Lemma~\ref{tech_rel18}. 
\end{proof}

The next lemma follows from the relations~(1)~(a) and (2)~(c).  

\begin{lem}\label{tech_rel13}
For $1\leq i<j\leq 2n$ such that $j-i\geq 4$ is even, the relation
\begin{align*}
h_{j-2}\cdots h_{i+2}h_{i}(h_{j-2}\cdots h_{i+1}h_{i})^{\frac{j-i}{2}}=(h_{j-2}^2h_{j-3}\cdots h_{i+1}h_{i})^{\frac{j-i}{2}}
\end{align*}
is obtained from the relations~(1) and (2). 
\end{lem}

We prove Lemma~\ref{tech_rel19} as follows. 

\begin{proof}[Proof of Lemma~\ref{tech_rel19}]
Recall that the relations in Lemmas~\ref{tech_rel2}, \ref{tech_rel5}, \ref{tech_rel13}, and \ref{tech_rel18}, and Corollary~\ref{cor_tech_rel3} are obtained from the relations~(1) and (2). 
Then we have
\begin{eqnarray*}
&&\underset{\rightarrow }{\underline{t_{j-2,j-1}^{-\frac{j-i-2}{2}}\cdots t_{i+2,i+3}^{-\frac{j-i-2}{2}}t_{i,i+1}^{-\frac{j-i-2}{2}}}}h_{j-2}\cdots h_{i+2}h_{i}h_{i}h_{i+2}\cdots h_{j-2}(h_{j-3}\cdots h_{i+1}h_{i})^{\frac{j-i}{2}}\\
&\overset{\text{Cor.\ref{cor_tech_rel3}}}{\underset{\text{Lem.\ref{tech_rel2}}}{=}}&h_{j-2}\cdots h_{i+2}h_{i}\underline{h_{i}h_{i+2}\cdots h_{j-2}(h_{j-3}\cdots h_{i+1}h_{i})^{\frac{j-i}{2}}}t_{j-2,j-1}^{-\frac{j-i-2}{2}}\cdots t_{i+2,i+3}^{-\frac{j-i-2}{2}}t_{i,i+1}^{-\frac{j-i-2}{2}}\\
&\overset{\text{Lem.\ref{tech_rel5}}}{\underset{}{=}}&\underline{h_{j-2}\cdots h_{i+2}h_{i}(h_{j-2}\cdots h_{i+1}h_{i})^{\frac{j-i}{2}}}t_{j-2,j-1}^{-\frac{j-i-2}{2}}\cdots t_{i+2,i+3}^{-\frac{j-i-2}{2}}t_{i,i+1}^{-\frac{j-i-2}{2}}\\
&\overset{\text{Lem.\ref{tech_rel13}}}{\underset{}{=}}&\underline{(h_{j-2}h_{j-2}h_{j-3}\cdots h_{i+1}h_{i})^{\frac{j-i}{2}}}t_{j-2,j-1}^{-\frac{j-i-2}{2}}\cdots t_{i+2,i+3}^{-\frac{j-i-2}{2}}t_{i,i+1}^{-\frac{j-i-2}{2}}\\
&\overset{\text{Lem.\ref{tech_rel18}}}{\underset{}{=}}&(h_{j-2}h_{j-3}\cdots h_{i+1}h_{i})^{\frac{j-i}{2}}t_{j-2,j-1}^{-\frac{j-i-4}{2}}t_{j-4,j-3}^{-\frac{j-i-2}{2}}\cdots t_{i+2,i+3}^{-\frac{j-i-2}{2}}t_{i,i+1}^{-\frac{j-i-2}{2}}\\
&&\cdot h_{j-4}\cdots h_{i+2}h_{i}\cdot \underset{\leftarrow }{\underline{h_{j-2}\cdots h_{i+4}}}h_{i+2}h_it_{i+1,i+2}^{-1}\cdot h_{i+2}^{-1}h_{i+4}^{-1}\cdots h_{j-2}^{-1}\\
&\overset{\text{(1)(a)}}{\underset{}{=}}&(h_{j-2}h_{j-3}\cdots h_{i+1}h_{i})^{\frac{j-i}{2}}t_{j-2,j-1}^{-\frac{j-i-4}{2}}t_{j-4,j-3}^{-\frac{j-i-2}{2}}\cdots t_{i+2,i+3}^{-\frac{j-i-2}{2}}t_{i,i+1}^{-\frac{j-i-2}{2}}\\
&&\cdot h_{j-4}h_{j-2}\cdots h_{i+4}h_{i+6}h_{i+2}h_{i+4}\cdot \underline{h_{i}h_{i+2}\cdot h_it_{i+1,i+2}^{-1}}\\
&&\cdot h_{i+2}^{-1}h_{i+4}^{-1}\cdots h_{j-2}^{-1}\\
&\overset{\text{(2)(e)}}{\underset{}{=}}&(h_{j-2}h_{j-3}\cdots h_{i+1}h_{i})^{\frac{j-i}{2}}t_{j-2,j-1}^{-\frac{j-i-4}{2}}t_{j-4,j-3}^{-\frac{j-i-2}{2}}\cdots t_{i+2,i+3}^{-\frac{j-i-2}{2}}t_{i,i+1}^{-\frac{j-i-2}{2}}\\
&&\cdot h_{j-4}h_{j-2}\cdots h_{i+4}h_{i+6}\underline{h_{i+2}h_{i+4}\cdot h_{i+2}}h_{i}h_{i+2}t_{i+2,i+3}^{-1}\\
&&\cdot h_{i+2}^{-1}h_{i+4}^{-1}\cdots h_{j-2}^{-1}\\
&\overset{\text{(2)(e)}}{\underset{}{=}}&(h_{j-2}h_{j-3}\cdots h_{i+1}h_{i})^{\frac{j-i}{2}}t_{j-2,j-1}^{-\frac{j-i-4}{2}}t_{j-4,j-3}^{-\frac{j-i-2}{2}}\cdots t_{i+2,i+3}^{-\frac{j-i-2}{2}}t_{i,i+1}^{-\frac{j-i-2}{2}}\\
&&\cdot h_{j-4}h_{j-2}\cdots \underline{h_{i+4}h_{i+6}\cdot h_{i+4}}h_{i+2}h_{i+4}t_{i+4,i+5}^{-1}t_{i+3,i+4}\cdot h_{i}h_{i+2}t_{i+2,i+3}^{-1}\\
&&\cdot h_{i+2}^{-1}h_{i+4}^{-1}\cdots h_{j-2}^{-1}\\
&\overset{\text{(2)(e)}}{\underset{}{=}}&(h_{j-2}h_{j-3}\cdots h_{i+1}h_{i})^{\frac{j-i}{2}}t_{j-2,j-1}^{-\frac{j-i-4}{2}}t_{j-4,j-3}^{-\frac{j-i-2}{2}}\cdots t_{i+2,i+3}^{-\frac{j-i-2}{2}}t_{i,i+1}^{-\frac{j-i-2}{2}}\\
&&\cdot h_{j-4}h_{j-2}\cdots \underline{h_{i+6}h_{i+8}\cdot h_{i+6}}h_{i+4}h_{i+6}t_{i+6,i+7}^{-1}t_{i+5,i+6}\\
&&\cdot h_{i+2}h_{i+4}t_{i+4,i+5}^{-1}t_{i+3,i+4}\cdot h_{i}h_{i+2}t_{i+2,i+3}^{-1}\\
&&\cdot h_{i+2}^{-1}h_{i+4}^{-1}\cdots h_{j-2}^{-1}\\
&\overset{\text{(2)(e)}}{\underset{}{=}}&(h_{j-2}h_{j-3}\cdots h_{i+1}h_{i})^{\frac{j-i}{2}}t_{j-2,j-1}^{-\frac{j-i-4}{2}}t_{j-4,j-3}^{-\frac{j-i-2}{2}}\cdots t_{i+2,i+3}^{-\frac{j-i-2}{2}}t_{i,i+1}^{-\frac{j-i-2}{2}}\\
&&\cdot h_{j-4}h_{j-2}\cdots h_{i+8}h_{i+10}\cdot h_{i+8}h_{i+6}h_{i+8}t_{i+8,i+9}^{-1}t_{i+7,i+8}\\
&&\cdot h_{i+4}h_{i+6}t_{i+6,i+7}^{-1}t_{i+5,i+6}\cdot h_{i+2}h_{i+4}t_{i+4,i+5}^{-1}t_{i+3,i+4}\cdot h_{i}h_{i+2}t_{i+2,i+3}^{-1}\\
&&\cdot h_{i+2}^{-1}h_{i+4}^{-1}\cdots h_{j-2}^{-1}\\
&\vdots &\\
&\overset{\text{(2)(e)}}{\underset{}{=}}&(h_{j-2}h_{j-3}\cdots h_{i+1}h_{i})^{\frac{j-i}{2}}t_{j-2,j-1}^{-\frac{j-i-4}{2}}t_{j-4,j-3}^{-\frac{j-i-2}{2}}\cdots t_{i+2,i+3}^{-\frac{j-i-2}{2}}t_{i,i+1}^{-\frac{j-i-2}{2}}\\
&&\cdot h_{j-2}\cdot h_{j-4}h_{j-2}t_{j-2,j-1}^{-1}t_{j-3,j-2}\cdot \underset{\leftarrow }{\underline{h_{j-6}}}h_{j-4}t_{j-4,j-3}^{-1}t_{j-5,j-4}\cdots \\
&&\cdot \underset{\leftarrow }{\underline{h_{i+2}}}h_{i+4}t_{i+4,i+5}^{-1}t_{i+3,i+4}\cdot \underset{\leftarrow }{\underline{h_{i}}}h_{i+2}t_{i+2,i+3}^{-1}\cdot h_{i+2}^{-1}h_{i+4}^{-1}\cdots h_{j-2}^{-1}\\
&\overset{\text{(1)(a)}}{\underset{\text{(1)(c)}}{=}}&(h_{j-2}h_{j-3}\cdots h_{i+1}h_{i})^{\frac{j-i}{2}}\\
&&\underline{t_{j-2,j-1}^{-\frac{j-i-4}{2}}t_{j-4,j-3}^{-\frac{j-i-2}{2}}\cdots t_{i+2,i+3}^{-\frac{j-i-2}{2}}t_{i,i+1}^{-\frac{j-i-2}{2}}h_{j-2}\cdots h_{i+2}h_{i}}\\
&&\cdot \underline{h_{j-2}t_{j-2,j-1}^{-1}}t_{j-3,j-2}\cdot \underline{h_{j-4}t_{j-4,j-3}^{-1}}t_{j-5,j-4}\cdots \underline{h_{i+4}t_{i+4,i+5}^{-1}}t_{i+3,i+4}\\
&&\cdot \underline{h_{i+2}t_{i+2,i+3}^{-1}}\cdot h_{i+2}^{-1}h_{i+4}^{-1}\cdots h_{j-2}^{-1}\\
&\overset{\text{(1)(b)}}{\underset{}{=}}&\underline{(h_{j-2}h_{j-3}\cdots h_{i+1}h_{i})^{\frac{j-i}{2}}}h_{j-2}\cdots h_{i+2}h_{i}t_{i+1,i+2}^{-\frac{j-i-2}{2}}t_{i+3,i+4}^{-\frac{j-i-2}{2}}\cdots t_{j-1,j}^{-\frac{j-i-2}{2}}\\
&\overset{\text{Lem.\ref{tech_rel5-1}}}{\underset{}{=}}&\underline{(h_{j-2}h_{j-3}\cdots h_{i+1})^{\frac{j-i}{2}}}h_{i}h_{i+2}\cdots h_{j-2}h_{j-2}\cdots h_{i+2}h_{i}\\
&&\cdot t_{i+1,i+2}^{-\frac{j-i-2}{2}}t_{i+3,i+4}^{-\frac{j-i-2}{2}}\cdots t_{j-1,j}^{-\frac{j-i-2}{2}}\\
&\overset{\text{Lem.\ref{tech_rel8}}}{\underset{}{=}}&\underset{\rightarrow }{\underline{(h_{i+1}h_{i+2}\cdots h_{j-2})^{\frac{j-i}{2}}}}\ \underset{\rightarrow }{\underline{h_{i}h_{i+2}\cdots h_{j-2}h_{j-2}\cdots h_{i+2}h_{i}}}\\
&&\cdot t_{i+1,i+2}^{-\frac{j-i-2}{2}}t_{i+3,i+4}^{-\frac{j-i-2}{2}}\cdots t_{j-1,j}^{-\frac{j-i-2}{2}}\\
&\overset{}{\underset{}{=}}&t_{i+1,i+2}^{-\frac{j-i-2}{2}}t_{i+3,i+4}^{-\frac{j-i-2}{2}}\cdots t_{j-1,j}^{-\frac{j-i-2}{2}}h_{i}h_{i+2}\cdots h_{j-2}h_{j-2}\cdots h_{i+2}h_{i}\\
&&\cdot (h_{i+1}h_{i+2}\cdots h_{j-2})^{\frac{j-i}{2}}.
\end{eqnarray*}
The last commutative relations are analogies of the relations in Corollaries~\ref{cor_tech_rel3}, \ref{cor_tech_rel4}, and Lemma~\ref{tech_rel3}, and also obtained from the relations~(1) and (2). 
Therefore, we have completed the proof of Lemma~\ref{tech_rel19}. 
\end{proof}

\subsection{Proofs of Lemmas~\ref{lem_t_ij_pres1}, \ref{lem_L_ij_k=j-1}, \ref{lem_L_ij_k=i-2}, and \ref{lem_L_ij_k=i-1}}\label{section_proof_t_ij_pres}

In this section, we prove Lemmas~\ref{lem_t_ij_pres1}, \ref{lem_L_ij_k=j-1}, \ref{lem_L_ij_k=i-2}, and \ref{lem_L_ij_k=i-1}. 
To prove these lemmas, we prepare the following four technical lemmas. 

\begin{lem}\label{tech_rel9}
For $1\leq i<j\leq 2n+1$ in the case of $\LM $ and $1\leq i<j\leq 2n$ in the case of $\LMb $ such that $j-i\geq 3$ is odd, the relation
\begin{align*}
&h_{i}h_{i+1}\cdots h_{j-2}h_{j-1}\\
=&t_{i,i+1}^{\frac{j-i+1}{2}}(h_ih_{i+2}\cdots h_{j-1})t_{i+2,i+3}^{-1}t_{i+4,i+5}^{-1}\cdots t_{j-1,j}^{-1}(h_{i+1}h_{i+3}\cdots h_{j-2}).
\end{align*}
is obtained from the relations~(1), (2), and $t_{2n+1,2n+2}=t_{1,2n}$. 
\end{lem}

\begin{proof}
We have
\begin{eqnarray*}
&&h_{i}h_{i+1}\cdots h_{j-3}\underline{h_{j-2}h_{j-1}}\\
&\overset{\text{(2)(d)}}{\underset{}{=}}&h_{i}h_{i+1}\cdots h_{j-5}\underline{h_{j-4}h_{j-3}\cdot t_{j-2,j-1}}t_{j,j+1}^{-1}h_{j-1}h_{j-2}\\
&\overset{\text{(2)(a)}}{\underset{}{=}}&h_{i}h_{i+1}\cdots h_{j-5}\cdot t_{j-4,j-3}\underline{h_{j-4}h_{j-3}}\cdot t_{j,j+1}^{-1}h_{j-1}h_{j-2}\\
&\overset{\text{(2)(d)}}{\underset{}{=}}&h_{i}h_{i+1}\cdots h_{j-5}\cdot t_{j-4,j-3}^{2}t_{j-2,j-1}^{-1}\underset{\rightarrow }{\underline{h_{j-3}h_{j-4}}}t_{j,j+1}^{-1}h_{j-1}h_{j-2}\\
&\overset{\text{(1)(c)}}{\underset{\text{(1)(a)}}{=}}&h_{i}h_{i+1}\cdots h_{j-7}\underline{h_{j-6}h_{j-5}\cdot t_{j-4,j-3}^{2}}t_{j-2,j-1}^{-1}t_{j,j+1}^{-1}\cdot h_{j-3}h_{j-1}\cdot h_{j-4}h_{j-2}\\
&\overset{\text{(2)(a)}}{\underset{}{=}}&h_{i}h_{i+1}\cdots h_{j-7}\cdot t_{j-6,j-5}^{2}\underline{h_{j-6}h_{j-5}}\cdot t_{j-2,j-1}^{-1}t_{j,j+1}^{-1}\cdot h_{j-3}h_{j-1}\cdot h_{j-4}h_{j-2}\\
&\overset{\text{(2)(d)}}{\underset{}{=}}&h_{i}h_{i+1}\cdots h_{j-7}\cdot t_{j-6,j-5}^{3}t_{j-4,j-3}^{-1}\underset{\rightarrow }{\underline{h_{j-5}h_{j-6}}}\cdot t_{j-2,j-1}^{-1}t_{j,j+1}^{-1}\cdot h_{j-3}h_{j-1}\\
&&\cdot h_{j-4}h_{j-2}\\
&\overset{\text{(1)(c)}}{\underset{\text{(1)(a)}}{=}}&h_{i}h_{i+1}\cdots h_{j-7}\cdot t_{j-6,j-5}^{3}t_{j-4,j-3}^{-1}t_{j-2,j-1}^{-1}t_{j,j+1}^{-1}\cdot h_{j-5}h_{j-3}h_{j-1}\\
&&\cdot h_{j-6}h_{j-4}h_{j-2}\\
&\vdots&\\
&\overset{}{\underset{}{=}}&h_it_{i+1,i+2}^{\frac{j-i+1}{2}}\underline{t_{i+3,i+4}^{-1}\cdots t_{j-2,j-1}^{-1}t_{j,j+1}^{-1}\cdot h_{i+2}\cdots h_{j-3}h_{j-1}}\cdot h_{i+1}\cdots h_{j-4}h_{j-2}\\
&\overset{\text{(2)(a)}}{\underset{}{=}}&t_{i,i+1}^{\frac{j-i+1}{2}}(h_{i}\cdots h_{j-3}h_{j-1})t_{i+2,i+3}^{-1}\cdots t_{j-3,j-2}^{-1}t_{j-1,j}^{-1}(h_{i+1}\cdots h_{j-4}h_{j-2}).
\end{eqnarray*}
Therefore, we have completed the proof of Lemma~\ref{tech_rel9}. 
\end{proof}

\begin{lem}\label{tech_rel10}
For $1\leq i<j\leq 2n+2$ in the case of $\LM $ and $1\leq i<j\leq 2n+1$ in the case of $\LMb $ such that $j-i\geq 3$ is odd, the relation
\begin{align*}
&(h_{j-1}\cdots h_{i+2}h_ih_{i+2}\cdots h_{j-1})t_{i+2,i+3}^{-1}t_{i+4,i+5}^{-1}\cdots t_{j-1,j}^{-1}\\
=&(h_{i}\cdots h_{j-3}h_{j-1}h_{j-3}\cdots h_{i})t_{i+1,i+2}^{-1}t_{i+3,i+4}^{-1}\cdots t_{j-2,j-1}^{-1}
\end{align*}
is obtained from the relations~(1), (2), and $t_{2n+1,2n+2}=t_{1,2n}$. 
\end{lem}

\begin{proof}
When $j-i=3$, the relation in Lemma~\ref{tech_rel10} is equivalent to the relation~(2) (e). 
Assume $j-i\geq 5$. 
We have
\begin{eqnarray*}
&&(h_{j-1}\cdots h_{i+4}\underline{h_{i+2}h_ih_{i+2}}h_{i+4}\cdots h_{j-1})t_{i+2,i+3}^{-1}t_{i+4,i+5}^{-1}\cdots t_{j-1,j}^{-1}\\
&\overset{\text{(2)(e)}}{\underset{}{=}}&h_{j-1}\cdots h_{i+6}h_{i+4}\cdot \underset{\leftarrow }{\underline{h_i}}h_{i+2}\underset{\rightarrow }{\underline{h_{i}t_{i+1,i+2}^{-1}t_{i+2,i+3}}}\cdot h_{i+4}h_{i+6}\cdots h_{j-1}\\
&&\cdot t_{i+2,i+3}^{-1}t_{i+4,i+5}^{-1}\cdots t_{j-1,j}^{-1}\\
&\overset{\text{(1)(a)}}{\underset{\text{(1)(c)}}{=}}&h_i(h_{j-1}\cdots h_{i+6}\underline{h_{i+4}h_{i+2}h_{i+4}}h_{i+6}\cdots h_{j-1})h_{i}t_{i+1,i+2}^{-1}\\
&&\cdot t_{i+4,i+5}^{-1}t_{i+6,i+7}^{-1}\cdots t_{j-1,j}^{-1}\\
&\overset{\text{(2)(e)}}{\underset{}{=}}&h_i\cdot h_{j-1}\cdots h_{i+8}h_{i+6}\cdot \underset{\leftarrow }{\underline{h_{i+2}}}h_{i+4}\underset{\rightarrow }{\underline{h_{i+2}t_{i+3,i+4}^{-1}t_{i+4,i+5}}}\cdot h_{i+6}h_{i+8}\cdots h_{j-1}\\
&&\cdot h_{i}t_{i+1,i+2}^{-1}\cdot t_{i+4,i+5}^{-1}t_{i+6,i+7}^{-1}\cdots t_{j-1,j}^{-1}\\
&\overset{\text{(1)}}{\underset{}{=}}&h_ih_{i+2}(h_{j-1}\cdots h_{i+8}h_{i+6}h_{i+4}h_{i+6}h_{i+8}\cdots h_{j-1})h_{i+2}h_{i}t_{i+1,i+2}^{-1}t_{i+3,i+4}^{-1}\\
&&\cdot t_{i+6,i+7}^{-1}t_{i+8,i+9}^{-1}\cdots t_{j-1,j}^{-1}\\
&\vdots &\\
&=&(h_{i}\cdots h_{j-3}h_{j-1}h_{j-3}\cdots h_{i})t_{i+1,i+2}^{-1}t_{i+3,i+4}^{-1}\cdots t_{j-2,j-1}^{-1}.
\end{eqnarray*}
Therefore, we have completed the proof of Lemma~\ref{tech_rel10}. 
\end{proof}

\begin{lem}\label{tech_rel11}
For $1\leq i<j\leq 2n+2$ in the case of $\LM $ and $1\leq i<j\leq 2n+1$ in the case of $\LMb $ such that $j-i\geq 3$ is odd, the relation
\[
h_{j-3}^{-1}\cdots h_{i+2}^{-1}h_{i}^{-1}(h_{j-2}\cdots h_{i+3}h_{i+1})h_ih_{i+2}\cdots h_{j-3}=(h_{j-2}\cdots h_{i+3}h_{i+1})t_{i,i+1}t_{j-1,j}^{-1} 
\]
is obtained from the relations~(1), (2), and $t_{2n+1,2n+2}=t_{1,2n}$. 
\end{lem}

\begin{proof}
When $j-i=3$, the relation in Lemma~\ref{tech_rel11} is equivalent to the relation~(2)~(d). 
Assume $j-i\geq 5$. 
We have
\begin{eqnarray*}
&&h_{j-3}^{-1}\cdots h_{i+2}^{-1}h_{i}^{-1}(h_{j-2}\cdots h_{i+3}\underline{h_{i+1})h_i}h_{i+2}\cdots h_{j-3}\\
&\overset{\text{(2)(d)}}{\underset{}{=}}&h_{j-3}^{-1}\cdots h_{i+2}^{-1}h_{i}^{-1}\cdot h_{j-2}\cdots h_{i+5}h_{i+3}\cdot \underset{\leftarrow }{\underline{h_i}}h_{i+1}t_{i,i+1}t_{i+2,i+3}^{-1}\cdot h_{i+2}h_{i+4}\cdots h_{j-3}\\
&\overset{\text{(1)(a)}}{\underset{}{=}}&h_{j-3}^{-1}\cdots h_{i+4}^{-1}h_{i+2}^{-1} (h_{j-2}\cdots h_{i+5}h_{i+3}\cdot h_{i+1})\underline{t_{i,i+1}t_{i+2,i+3}^{-1}\cdot h_{i+2}}h_{i+4}\cdots h_{j-3}\\
&\overset{\text{(2)(a)}}{\underset{\text{(1)(c)}}{=}}&h_{j-3}^{-1}\cdots h_{i+4}^{-1}h_{i+2}^{-1}\cdot h_{j-2}\cdots h_{i+5}h_{i+3}\cdot \underline{h_{i+1}h_{i+2}}t_{i,i+1}t_{i+3,i+4}^{-1}\\
&&\cdot h_{i+4}h_{i+6}\cdots h_{j-3}\\
&\overset{\text{(2)(d)}}{\underset{\text{(1)(b)}}{=}}&h_{j-3}^{-1}\cdots h_{i+4}^{-1}h_{i+2}^{-1}\cdot h_{j-2}\cdots h_{i+5}\underline{h_{i+3}\cdot h_{i+2}}h_{i+1}t_{i+1,i+2}^{-1}t_{i,i+1}\\
&&\cdot h_{i+4}h_{i+6}\cdots h_{j-3}\\
&\overset{\text{(2)(d)}}{\underset{\text{(1)(b)}}{=}}&h_{j-3}^{-1}\cdots h_{i+4}^{-1}h_{i+2}^{-1}\cdot h_{j-2}\cdots h_{i+7}h_{i+5}\cdot \underset{\leftarrow }{\underline{h_{i+2}}}h_{i+3}t_{i+2,i+3}t_{i+4,i+5}^{-1}\\
&&\cdot h_{i+1}t_{i+1,i+2}^{-1}t_{i,i+1}\cdot h_{i+4}h_{i+6}\cdots h_{j-3}\\
&\overset{\text{(1)(a)}}{\underset{}{=}}&h_{j-3}^{-1}\cdots h_{i+6}^{-1}h_{i+4}^{-1}\cdot h_{j-2}\cdots h_{i+7}h_{i+5}\cdot h_{i+3}\underline{t_{i+2,i+3}t_{i+4,i+5}^{-1}h_{i+1}}t_{i+1,i+2}^{-1}t_{i,i+1}\\
&&\cdot h_{i+4}h_{i+6}\cdots h_{j-3}\\
&\overset{\text{(1)(c)}}{\underset{\text{(2)(a)}}{=}}&h_{j-3}^{-1}\cdots h_{i+6}^{-1}h_{i+4}^{-1}\cdot h_{j-2}\cdots h_{i+7}h_{i+5}\cdot h_{i+3}h_{i+1}\underset{\rightarrow }{\underline{t_{i+1,i+2}}}t_{i+4,i+5}^{-1}t_{i+1,i+2}^{-1}t_{i,i+1}\\
&&\cdot h_{i+4}h_{i+6}\cdots h_{j-3}\\
&\overset{\text{(1)(b)}}{\underset{}{=}}&h_{j-3}^{-1}\cdots h_{i+6}^{-1}h_{i+4}^{-1}(h_{j-2}\cdots h_{i+7}h_{i+5}\cdot h_{i+3}h_{i+1})t_{i,i+1}t_{i+4,i+5}^{-1}\\
&&\cdot h_{i+4}h_{i+6}\cdots h_{j-3}.
\end{eqnarray*}
By an inductive argument, for a positive even $2\leq l\leq j-i-11$, the relation 
\begin{align*}
&h_{j-3}^{-1}\cdots h_{i+l+6}^{-1}h_{i+l+4}^{-1}(h_{j-2}\cdots h_{i+3}h_{i+1})t_{i,i+1}t_{i+l+4,i+l+5}^{-1}\cdot h_{i+l+4}h_{i+l+6}\cdots h_{j-3}\\
&=h_{j-3}^{-1}\cdots h_{i+l+8}^{-1}h_{i+l+6}^{-1}(h_{j-2}\cdots h_{i+3}h_{i+1})t_{i,i+1}t_{i+l+6,i+l+7}^{-1}\cdot h_{i+l+6}h_{i+l+8}\cdots h_{j-3}
\end{align*}
is obtained from the relations~(1), (2), and $t_{2n+1,2n+2}=t_{1,2n}$. 
Thus, up to the relations~(1), (2), and $t_{2n+1,2n+2}=t_{1,2n}$, we have 
\begin{eqnarray*}
&&h_{j-3}^{-1}\cdots h_{i+2}^{-1}h_{i}^{-1}(h_{j-2}\cdots h_{i+3}h_{i+1})h_ih_{i+2}\cdots h_{j-3}\\
&=&h_{j-3}^{-1}(h_{j-2}\cdots h_{i+3}h_{i+1})t_{i,i+1}\underline{t_{j-3,j-2}^{-1}\cdot h_{j-3}}\\
&\overset{\text{(2)(a)}}{\underset{}{=}}&h_{j-3}^{-1}(h_{j-2}\cdots h_{i+3}h_{i+1})t_{i,i+1}\underset{\leftarrow }{\underline{h_{j-3}}}t_{j-2,j-1}^{-1}\\
&\overset{\text{(1)(c)}}{\underset{\text{(1)(a)}}{=}}&h_{j-3}^{-1}\cdot h_{j-2}\underline{h_{j-4}\cdot h_{j-3}}\cdot h_{j-6}\cdots h_{i+3}h_{i+1}t_{i,i+1}t_{j-2,j-1}^{-1}\\
&\overset{\text{(2)(d)}}{\underset{\text{(1)(c)}}{=}}&h_{j-3}^{-1}\cdot \underline{h_{j-2}h_{j-3}}h_{j-4}t_{j-4,j-3}^{-1}\underset{\rightarrow }{\underline{t_{j-2,j-1}}}\cdot h_{j-6}\cdots h_{i+3}h_{i+1}t_{i,i+1}t_{j-2,j-1}^{-1}\\
&\overset{\text{(2)(d)}}{\underset{}{=}}&h_{j-2}\underline{t_{j-3,j-2}t_{j-1,j}^{-1}h_{j-4}}t_{j-4,j-3}^{-1}\cdot h_{j-6}\cdots h_{i+3}h_{i+1}t_{i,i+1}\\
&\overset{\text{(1)(c)}}{\underset{\text{(2)(a)}}{=}}&h_{j-2}h_{j-4}\underset{\rightarrow }{\underline{t_{j-4,j-3}t_{j-1,j}^{-1}}}t_{j-4,j-3}^{-1}\cdot h_{j-6}\cdots h_{i+3}h_{i+1}t_{i,i+1}\\
&\overset{\text{(1)(a)}}{\underset{\text{(1)(c)}}{=}}&(h_{j-2}h_{j-4}h_{j-6}\cdots h_{i+3}h_{i+1})t_{i,i+1}t_{j-1,j}^{-1}.
\end{eqnarray*}
Therefore, we have completed the proof of Lemma~\ref{tech_rel11}. 
\end{proof}

\begin{lem}\label{tech_rel12}
For $1\leq i<j\leq 2n$ in the case of $\LM $ and $1\leq i<j\leq 2n-1$ in the case of $\LMb $ such that $j-i\geq 3$ is odd, the relation
\[
h_{j}\cdots h_{i+3}h_{i+1}\cdot h_{j-1}\cdots h_{i+2}h_i=t_{j+1,j+2}^{-\frac{j-i-1}{2}}(h_{j}\cdots h_{i+1}h_i)t_{i+3,i+4}t_{i+5,i+6}\cdots t_{j,j+1} 
\]
is obtained from the relations~(1), (2), and $t_{2n+1,2n+2}=t_{1,2n}$. 
\end{lem}

\begin{proof}
We have
\begin{eqnarray*}
&&h_{j}\cdots h_{i+3}\underset{\rightarrow }{\underline{h_{i+1}}}\cdot h_{j-1}\cdots h_{i+2}h_i\\
&\overset{\text{(1)(a)}}{\underset{}{=}}&h_{j}\cdots h_{i+5}h_{i+3}\cdot h_{j-1}\cdots h_{i+6}h_{i+4}\cdot \underline{h_{i+1}h_{i+2}}h_i\\
&\overset{\text{(2)(d)}}{\underset{\text{(2)(a)}}{=}}&h_{j}\cdots h_{i+5}h_{i+3}\cdot h_{j-1}\cdots h_{i+6}h_{i+4}\cdot t_{i+3,i+4}^{-1}h_{i+2}h_{i+1}\underset{\rightarrow }{\underline{t_{i+3,i+4}}}h_i\\
&\overset{\text{(1)(c)}}{\underset{}{=}}&h_{j}\cdots h_{i+5}\underset{\rightarrow }{\underline{h_{i+3}}}\cdot h_{j-1}\cdots h_{i+6}h_{i+4}\cdot t_{i+3,i+4}^{-1}(h_{i+2}h_{i+1}h_i)t_{i+3,i+4}\\
&\overset{\text{(1)(a)}}{\underset{}{=}}&h_{j}\cdots h_{i+7}h_{i+5}\cdot h_{j-1}\cdots h_{i+8}h_{i+6}\cdot \underline{h_{i+3}h_{i+4}}\cdot t_{i+3,i+4}^{-1}h_{i+2}h_{i+1}h_it_{i+3,i+4}\\
&\overset{\text{(2)(d)}}{\underset{\text{(2)(a)}}{=}}&h_{j}\cdots h_{i+7}h_{i+5}\cdot h_{j-1}\cdots h_{i+8}h_{i+6}\cdot t_{i+5,i+6}^{-1}h_{i+4}h_{i+3}\underset{\rightarrow }{\underline{t_{i+5,i+6}}}\\
&&\cdot t_{i+3,i+4}^{-1}h_{i+2}h_{i+1}h_it_{i+3,i+4}\\
&\overset{\text{(1)(a)}}{\underset{\text{(1)(c)}}{=}}&h_{j}\cdots h_{i+7}h_{i+5}\cdot h_{j-1}\cdots h_{i+8}h_{i+6}\\
&&\cdot t_{i+5,i+6}^{-1}\underline{h_{i+4}h_{i+3}\cdot t_{i+3,i+4}^{-1}}h_{i+2}h_{i+1}h_it_{i+3,i+4}t_{i+5,i+6}\\
&\overset{\text{(2)(a)}}{\underset{\text{(1)(c)}}{=}}&h_{j}\cdots h_{i+7}h_{i+5}\cdot h_{j-1}\cdots h_{i+8}h_{i+6}\\
&&\cdot t_{i+5,i+6}^{-2}(h_{i+4}h_{i+3}h_{i+2}h_{i+1}h_i)t_{i+3,i+4}t_{i+5,i+6}\\
&\vdots &\\
&=&\underline{h_j\cdot t_{j,j+1}^{-\frac{j-i-1}{2}}}(h_{j-1}\cdots h_{i+1}h_i)\cdot t_{i+3,i+4}t_{i+5,i+6}\cdots t_{j,j+1}.\\
&\overset{\text{(2)(a)}}{\underset{}{=}}&t_{j+1,j+2}^{-\frac{j-i-1}{2}}(h_jh_{j-1}\cdots h_{i+1}h_i)t_{i+3,i+4}t_{i+5,i+6}\cdots t_{j,j+1}.\\
\end{eqnarray*}
Therefore, we have completed the proof of Lemma~\ref{tech_rel12}. 
\end{proof}

\begin{proof}[Proof of Lemma~\ref{lem_t_ij_pres1}]
First, we remark that the relation~$(L_{i,j}^{k=j})$ is equivalent to the following relation:
\begin{eqnarray*}
&&h_jt_{i,j}h_j^{-1}=\underline{t_{i,j-1}t_{j,j+1}}t_{i,j+2}\underline{t_{i,j+1}^{-1}t_{j,j+2}^{-1}}\\
&\overset{\text{CONJ}}{\underset{}{\Longleftrightarrow}}&t_{i,j+2}=t_{j,j+1}^{-1}t_{i,j-1}^{-1}h_jt_{i,j}h_j^{-1}t_{j,j+2}t_{i,j+1}.
\end{eqnarray*}
In the case of $j-i=1$, by the relation $t_{i,i}=1$, we have 
\begin{eqnarray*}
t_{i,i+3}&=&t_{i+1,i+2}^{-1}h_{i+1}t_{i,i+1}h_{i+1}^{-1}\underline{t_{i+1,i+3}}\ \underline{t_{i,i+2}}\\
&\overset{(T_{i,i+2})}{\underset{(T_{i+1,i+3})}{=}}&t_{i+1,i+2}^{-1}h_{i+1}\underline{t_{i,i+1}h_{i+1}h_i}h_i
\overset{\text{(2)(d)}}{\underset{\text{(2)(a)}}{=}}t_{i+1,i+2}^{-1}\underline{h_{i+1}t_{i+2,i+3}}h_ih_{i+1}h_i\\
&\overset{\text{(2)(a)}}{\underset{}{=}}&h_{i+1}h_ih_{i+1}h_i.
\end{eqnarray*}
Thus the relation~$(L_{i,i+1}^{k=j})$ is equivalent to the relation~$(T_{i,i+3})$ up to the relations~(2)~(a), (2) (d), $(T_{i,i+2})$, and $(T_{i+1,i+3})$. 
In the case of $j-i=2$, the relation~$(L_{i,i+2}^{k=j})$ is equivalent to the following relation:
\begin{eqnarray*}
t_{i,i+4}&=&t_{i+2,i+3}^{-1}t_{i,i+1}^{-1}h_{i+2}t_{i,i+2}h_{i+2}^{-1}t_{i+2,i+4}\underline{t_{i,i+3}}\\
&\overset{(T_{i,i+3})}{\underset{}{=}}&t_{i+2,i+3}^{-1}t_{i,i+1}^{-1}h_{i+2}\underline{t_{i,i+2}}h_{i+2}^{-1}\underline{t_{i+2,i+4}}(h_{i+1}h_i)^2\\
&\overset{(T_{i,i+2})}{\underset{(T_{i+2,i+4})}{=}}&t_{i+2,i+3}^{-1}t_{i,i+1}^{-1}h_{i+2}h_i^2\underline{h_{i+2}^{-1}h_{i+2}^2}(h_{i+1}h_i)^2\\
&\overset{}{\underset{}{=}}&t_{i+2,i+3}^{-1}t_{i,i+1}^{-1}h_{i+2}h_ih_ih_{i+2}(h_{i+1}h_i)^2.
\end{eqnarray*}
Thus the relation~$(L_{i,i+2}^{k=j})$ is equivalent to the relation~$(T_{i,i+4})$ up to the relations~$(T_{i^\prime ,j^\prime })$ for $i\leq i^\prime <j^\prime \leq j+2$ and $2>j^\prime -i^\prime $. 

Assume that $j-i\geq 6$ is even. 
Then the relation~$(L_{i,j}^{k=j})$ is equivalent to the following relation:
\begin{eqnarray*}
&&t_{i,j+2}\\
&=&t_{j,j+1}^{-1}\underline{t_{i,j-1}^{-1}}h_jt_{i,j}h_j^{-1}t_{j,j+2}\underline{t_{i,j+1}}\\
&\overset{(T_{i,j-1})}{\underset{(T_{i,j+1})}{=}}&t_{j,j+1}^{-1}\cdot (h_{j-3}\cdots h_{i+1}h_{i})^{-\frac{j-i}{2}}\underset{\leftarrow }{\underline{
t_{i,i+1}^{\frac{j-i-4}{2}}t_{i+2,i+3}^{\frac{j-i-4}{2}}\cdots t_{j-2,j-1}^{\frac{j-i-4}{2}}}}\cdot h_jt_{i,j}h_j^{-1}t_{j,j+2}\\
&&\cdot \underset{\rightarrow }{\underline{t_{j,j+1}^{-\frac{j-i-2}{2}}\cdots 
t_{i+2,i+3}^{-\frac{j-i-2}{2}}t_{i,i+1}^{-\frac{j-i-2}{2}}}}(h_{j-1}\cdots h_{i+1}h_{i})^{\frac{j-i+2}{2}}\\
&\overset{\text{Lem.\ref{tech_rel2}}}{\underset{\text{(1)(b)}}{=}}&t_{j,j+1}^{-1}t_{j-2,j-1}^{\frac{j-i-4}{2}}\cdots t_{i+2,i+3}^{\frac{j-i-4}{2}}t_{i,i+1}^{\frac{j-i-4}{2}}(h_{j-3}\cdots h_{i+1}h_{i})^{-\frac{j-i}{2}}\cdot h_j\underline{t_{i,j}}h_j^{-1}\underline{t_{j,j+2}}\\
&&\cdot (h_{j-1}\cdots h_{i+1}h_{i})^{\frac{j-i+2}{2}}t_{j,j+1}^{-\frac{j-i-2}{2}}\cdots t_{i+2,i+3}^{-\frac{j-i-2}{2}}t_{i,i+1}^{-\frac{j-i-2}{2}}\\
&\overset{(T_{i,j})}{\underset{(T_{j,j+2})}{=}}&t_{j,j+1}^{-1}t_{j-2,j-1}^{\frac{j-i-4}{2}}\cdots t_{i+2,i+3}^{\frac{j-i-4}{2}}t_{i,i+1}^{\frac{j-i-4}{2}}(h_{j-3}\cdots h_{i+1}h_{i})^{-\frac{j-i}{2}}\underset{\leftarrow }{\underline{h_j}}\\
&&\cdot \underset{\leftarrow }{\underline{t_{j-2,j-1}^{-\frac{j-i-2}{2}}\cdots t_{i+2,i+3}^{-\frac{j-i-2}{2}}t_{i,i+1}^{-\frac{j-i-2}{2}}}}h_{j-2}\cdots h_{i+2}h_{i}h_{i}h_{i+2}\cdots h_{j-2}\\ 
&&\cdot (h_{j-3}\cdots h_{i+1}h_{i})^{\frac{j-i}{2}}h_j(h_{j-1}\cdots h_{i+1}h_{i})^{\frac{j-i+2}{2}}t_{j,j+1}^{-\frac{j-i-2}{2}}\cdots t_{i+2,i+3}^{-\frac{j-i-2}{2}}t_{i,i+1}^{-\frac{j-i-2}{2}}\\
&\overset{\text{(1)(c)}}{\underset{\text{Lem.\ref{tech_rel2}}}{=}}&t_{j,j+1}^{-1}t_{j-2,j-1}^{\frac{j-i-4}{2}}\cdots t_{i+2,i+3}^{\frac{j-i-4}{2}}t_{i,i+1}^{\frac{j-i-4}{2}}h_j\underset{\leftarrow }{\underline{t_{j-2,j-1}^{-\frac{j-i-2}{2}}\cdots t_{i+2,i+3}^{-\frac{j-i-2}{2}}t_{i,i+1}^{-\frac{j-i-2}{2}}}}\\
&&\cdot (h_{j-3}\cdots h_{i+1}h_{i})^{-\frac{j-i}{2}}\underset{\leftarrow }{\underline{h_{j-2}\cdots h_{i+2}h_{i}h_{i}h_{i+2}\cdots h_{j-2}}}\\ 
&&\cdot (h_{j-3}\cdots h_{i+1}h_{i})^{\frac{j-i}{2}}h_j(h_{j-1}\cdots h_{i+1}h_{i})^{\frac{j-i+2}{2}}t_{j,j+1}^{-\frac{j-i-2}{2}}\cdots t_{i+2,i+3}^{-\frac{j-i-2}{2}}t_{i,i+1}^{-\frac{j-i-2}{2}}\\
&\overset{\text{(1)(c)}}{\underset{\text{Cor.\ref{cor_tech_rel4}}}{=}}&t_{j,j+1}^{-1}t_{j-2,j-1}^{-1}\cdots t_{i+2,i+3}^{-1}t_{i,i+1}^{-1}h_jh_{j-2}\cdots h_{i+2}h_{i}h_{i}h_{i+2}\cdots h_{j-2}h_j\\
&&\cdot (h_{j-1}\cdots h_{i+1}h_{i})^{\frac{j-i+2}{2}}\underset{\leftarrow }{\underline{t_{j,j+1}^{-\frac{j-i-2}{2}}\cdots t_{i+2,i+3}^{-\frac{j-i-2}{2}}t_{i,i+1}^{-\frac{j-i-2}{2}}}}\\
&\overset{\text{Lem.\ref{tech_rel2}}}{\underset{\text{Cor.\ref{cor_tech_rel3}}}{=}}&t_{j,j+1}^{-\frac{j-i}{2}}\cdots t_{i+2,i+3}^{-\frac{j-i}{2}}t_{i,i+1}^{-\frac{j-i}{2}}h_j\cdots h_{i+2}h_{i}h_{i}h_{i+2}\cdots h_j(h_{j-1}\cdots h_{i+1}h_{i})^{\frac{j-i+2}{2}}
\end{eqnarray*}
Recall that the relations in Lemma~\ref{tech_rel2}, Corollaries~\ref{cor_tech_rel3}, and \ref{cor_tech_rel4} for $j\leq 2n+1$ are obtained from the relations~(1) and (2). 
Thus the relation~$(L_{i,j}^{k=j})$ for even $j-i\geq 6$ is equivalent to the relation~$(T_{i,j+2})$ up to the relations~(1), (2), and $(T_{i^\prime ,j^\prime })$ for $i\leq i^\prime <j^\prime \leq j+2$ and $j-i+2>j^\prime -i^\prime $. 

Assume that $j-i\geq 5$ is odd. 
Then the relation~$(L_{i,j}^{k=j})$ is equivalent to the following relation:
\begin{eqnarray*}
&&t_{i,j+2}\\
&=&t_{j,j+1}^{-1}\underline{t_{i,j-1}^{-1}}h_jt_{i,j}h_j^{-1}t_{j,j+2}\underline{t_{i,j+1}}\\
&\overset{(T_{i,j-1})}{\underset{(T_{i,j+1})}{=}}&t_{j,j+1}^{-1}\cdot (h_{j-4}\cdots h_{i+1}h_{i})^{-\frac{j-i-1}{2}}h_{j-3}^{-1}\cdots h_{i+2}^{-1}h_{i}^{-1}h_{i}^{-1}h_{i+2}^{-1}\cdots h_{j-3}^{-1}\\
&&\cdot \underset{\leftarrow }{\underline{t_{i,i+1}^{\frac{j-i-3}{2}}t_{i+2,i+3}^{\frac{j-i-3}{2}}\cdots t_{j-3,j-2}^{\frac{j-i-3}{2}}}}\cdot h_jt_{i,j}h_j^{-1}t_{j,j+2}\cdot \underset{\rightarrow }{\underline{t_{j-1,j}^{-\frac{j-i-1}{2}}\cdots t_{i+2,i+3}^{-\frac{j-i-1}{2}}t_{i,i+1}^{-\frac{j-i-1}{2}}}}\\
&&\cdot h_{j-1}\cdots h_{i+2}h_{i}h_{i}h_{i+2}\cdots h_{j-1}(h_{j-2}\cdots h_{i+1}h_{i})^{\frac{j-i+1}{2}}\\
&\overset{\text{Lem.\ref{tech_rel2}}}{\underset{\text{Cor.\ref{cor_tech_rel3}}}{=}}&\underset{\rightarrow }{\underline{t_{j,j+1}^{-1}}}t_{i,i+1}^{\frac{j-i-3}{2}}t_{i+2,i+3}^{\frac{j-i-3}{2}}\cdots t_{j-3,j-2}^{\frac{j-i-3}{2}}\cdot (h_{j-4}\cdots h_{i+1}h_{i})^{-\frac{j-i-1}{2}}\\
&&\cdot h_{j-3}^{-1}\cdots h_{i+2}^{-1}h_{i}^{-1}h_{i}^{-1}h_{i+2}^{-1}\cdots h_{j-3}^{-1}\cdot h_jt_{i,j}h_j^{-1}t_{j,j+2}\\
&&\cdot h_{j-1}\cdots h_{i+2}h_{i}h_{i}h_{i+2}\cdots h_{j-1}(h_{j-2}\cdots h_{i+1}h_{i})^{\frac{j-i+1}{2}}\underset{\rightarrow }{\underline{t_{j-1,j}^{-\frac{j-i-1}{2}}\cdots t_{i+2,i+3}^{-\frac{j-i-1}{2}}t_{i,i+1}^{-\frac{j-i-1}{2}}}}\\
&\overset{\text{(1)}}{\underset{\text{CONJ}}{=}}&t_{i,i+1}^{-1}t_{i+2,i+3}^{-1}\cdots t_{j-3,j-2}^{-1}t_{j-1,j}^{-\frac{j-i-1}{2}}\cdot (h_{j-4}\cdots h_{i+1}h_{i})^{-\frac{j-i-1}{2}}\\
&&\cdot h_{j-3}^{-1}\cdots h_{i+2}^{-1}h_{i}^{-1}h_{i}^{-1}h_{i+2}^{-1}\cdots h_{j-3}^{-1}\cdot t_{j,j+1}^{-1}\cdot h_j\underline{t_{i,j}}h_j^{-1}t_{j,j+2}\\
&&\cdot h_{j-1}\cdots h_{i+2}h_{i}\underline{h_{i}h_{i+2}\cdots h_{j-1}(h_{j-2}\cdots h_{i+1}h_{i})^{\frac{j-i+1}{2}}}\\
&\overset{(T_{i,j})}{\underset{\text{Lem.\ref{tech_rel5}}}{=}}&t_{i,i+1}^{-1}t_{i+2,i+3}^{-1}\cdots t_{j-3,j-2}^{-1}t_{j-1,j}^{-\frac{j-i-1}{2}}\cdot (h_{j-4}\cdots h_{i+1}h_{i})^{-\frac{j-i-1}{2}}\\
&&\cdot h_{j-3}^{-1}\cdots h_{i+2}^{-1}h_{i}^{-1}h_{i}^{-1}h_{i+2}^{-1}\cdots h_{j-3}^{-1}\cdot \underline{t_{j,j+1}^{-1}\cdot h_j}t_{j-1,j}^{-\frac{j-i-3}{2}}\cdots t_{i+2,i+3}^{-\frac{j-i-3}{2}}t_{i,i+1}^{-\frac{j-i-3}{2}}\\
&&\cdot (h_{j-2}\cdots h_{i+1}h_{i})^{\frac{j-i+1}{2}}h_j^{-1}\underline{t_{j,j+2}}\cdot h_{j-1}\cdots h_{i+2}h_{i}(h_{j-1}\cdots h_{i+1}h_{i})^{\frac{j-i+1}{2}}\\
&\overset{(T_{j,j+2})}{\underset{\text{(2)(a)}}{=}}&t_{i,i+1}^{-1}t_{i+2,i+3}^{-1}\cdots t_{j-3,j-2}^{-1}t_{j-1,j}^{-\frac{j-i-1}{2}}\cdot \underline{(h_{j-4}\cdots h_{i+1}h_{i})^{-\frac{j-i-1}{2}}}\\
&&\underline{\cdot h_{j-3}^{-1}\cdots h_{i+2}^{-1}h_{i}^{-1}h_{i}^{-1}h_{i+2}^{-1}\cdots h_{j-3}^{-1}}\cdot h_jt_{j+1,j+2}^{-1}t_{j-1,j}^{-\frac{j-i-3}{2}}\cdots t_{i+2,i+3}^{-\frac{j-i-3}{2}}t_{i,i+1}^{-\frac{j-i-3}{2}}\\
&&\cdot (h_{j-2}\cdots h_{i+1}h_{i})^{\frac{j-i+1}{2}}\underline{h_j\cdot h_{j-1}}\cdots h_{i+2}h_{i}(h_{j-1}\cdots h_{i+1}h_{i})^{\frac{j-i+1}{2}}\\
&\overset{\text{Cor.\ref{cor_tech_rel4}}}{\underset{\text{(2)(d)}}{=}}&t_{i,i+1}^{-1}t_{i+2,i+3}^{-1}\cdots t_{j-3,j-2}^{-1}t_{j-1,j}^{-\frac{j-i-1}{2}}h_{j-3}^{-1}\cdots h_{i+2}^{-1}h_{i}^{-1}h_{i}^{-1}h_{i+2}^{-1}\cdots h_{j-3}^{-1}\\
&&\cdot (h_{j-4}\cdots h_{i+1}h_{i})^{-\frac{j-i-1}{2}}\underset{\leftarrow }{\underline{h_j}}t_{j+1,j+2}^{-1}t_{j-1,j}^{-\frac{j-i-3}{2}}\underset{\leftarrow }{\underline{t_{j-3,j-2}^{-\frac{j-i-3}{2}}\cdots t_{i+2,i+3}^{-\frac{j-i-3}{2}}t_{i,i+1}^{-\frac{j-i-3}{2}}}}\\
&&\cdot (h_{j-2}\cdots h_{i+1}h_{i})^{\frac{j-i+1}{2}}\underset{\leftarrow }{\underline{t_{j-1,j}^{-1}t_{j+1,j+2}}}h_{j-1}\underset{\rightarrow }{\underline{h_j}}\cdot h_{j-3}\cdots h_{i+2}h_{i}(h_{j-1}\cdots h_{i+1}h_{i})^{\frac{j-i+1}{2}}\\
&\overset{\text{(1)}}{\underset{\text{Lem.\ref{tech_rel2}}}{=}}&t_{i,i+1}^{-1}t_{i+2,i+3}^{-1}\cdots t_{j-3,j-2}^{-1}t_{j-1,j}^{-\frac{j-i-1}{2}}h_j\cdot h_{j-3}^{-1}\cdots h_{i+2}^{-1}h_{i}^{-1}h_{i}^{-1}h_{i+2}^{-1}\cdots h_{j-3}^{-1}\\
&&\cdot \underset{\leftarrow }{\underline{t_{j-3,j-2}^{-\frac{j-i-3}{2}}\cdots t_{i+2,i+3}^{-\frac{j-i-3}{2}}t_{i,i+1}^{-\frac{j-i-3}{2}}}}(h_{j-4}\cdots h_{i+1}h_{i})^{-\frac{j-i-1}{2}}t_{j-1,j}^{-\frac{j-i-1}{2}}\\
&&\cdot (h_{j-2}\cdots h_{i+1}h_{i})^{\frac{j-i+1}{2}}h_{j-1}h_{j-3}\cdots h_{i+2}h_{i}\cdot h_j(h_{j-1}\cdots h_{i+1}h_{i})^{\frac{j-i+1}{2}}\\
&\overset{\text{Cor.\ref{cor_tech_rel3}}}{\underset{\text{(1)}}{=}}&\underline{t_{i,i+1}^{-\frac{j-i-1}{2}}t_{i+2,i+3}^{-\frac{j-i-1}{2}}\cdots t_{j-3,j-2}^{-\frac{j-i-1}{2}}t_{j-1,j}^{-\frac{j-i-1}{2}}}h_j\cdot h_{j-3}^{-1}\cdots h_{i+2}^{-1}h_{i}^{-1}h_{i}^{-1}h_{i+2}^{-1}\cdots h_{j-3}^{-1}\\
&&\underline{\cdot (h_{j-4}\cdots h_{i+1}h_{i})^{-\frac{j-i-1}{2}}}t_{j-1,j}^{-\frac{j-i-1}{2}}\underline{(h_{j-2}\cdots h_{i+1}h_{i})^{\frac{j-i+1}{2}}}\\
&&\cdot h_{j-1}h_{j-3}\cdots h_{i+2}h_{i}\cdot h_j(h_{j-1}\cdots h_{i+1}h_{i})^{\frac{j-i+1}{2}}\\
&\overset{\text{Lem.\ref{tech_rel8}}}{\underset{\text{(1)(b)}}{=}}&t_{j-1,j}^{-\frac{j-i-1}{2}}\cdots t_{i+2,i+3}^{-\frac{j-i-1}{2}}t_{i,i+1}^{-\frac{j-i-1}{2}}h_j\cdot h_{j-3}^{-1}\cdots h_{i+2}^{-1}h_{i}^{-1}\underline{h_{i}^{-1}h_{i+2}^{-1}\cdots h_{j-3}^{-1}}\\
&&\underline{\cdot (h_{i}h_{i+1}\cdots h_{j-4})^{-\frac{j-i-1}{2}}}t_{j-1,j}^{-\frac{j-i-1}{2}}(h_{i}h_{i+1}\cdots h_{j-2})^{\frac{j-i-1}{2}}\\
&&\cdot \underline{h_{i}h_{i+1}\cdots h_{j-2}\cdot h_{j-1}h_{j-3}\cdots h_{i+2}h_{i}}\cdot h_j(h_{j-1}\cdots h_{i+1}h_{i})^{\frac{j-i+1}{2}}\\
&\overset{\text{Lem.\ref{tech_rel5}}}{\underset{\text{(2)(c)}}{=}}&t_{j-1,j}^{-\frac{j-i-1}{2}}\cdots t_{i+2,i+3}^{-\frac{j-i-1}{2}}t_{i,i+1}^{-\frac{j-i-1}{2}}h_j\cdot h_{j-3}^{-1}\cdots h_{i+2}^{-1}h_{i}^{-1}(h_{i}h_{i+1}\cdots h_{j-3})^{-\frac{j-i-1}{2}}\\
&&\cdot t_{j-1,j}^{-\frac{j-i-1}{2}}(h_{i}h_{i+1}\cdots h_{j-2})^{\frac{j-i-1}{2}}h_{j-1}\cdots h_{i+4}h_{i+2}\cdot \underline{h_{i}h_{i+1}\cdots h_{j-2}h_{j-1}}\\
&&\cdot h_j(h_{j-1}\cdots h_{i+1}h_{i})^{\frac{j-i+1}{2}}\\
&\overset{\text{Lem.\ref{tech_rel9}}}{\underset{}{=}}&t_{j-1,j}^{-\frac{j-i-1}{2}}\cdots t_{i+2,i+3}^{-\frac{j-i-1}{2}}t_{i,i+1}^{-\frac{j-i-1}{2}}h_j\cdot h_{j-3}^{-1}\cdots h_{i+2}^{-1}h_{i}^{-1}(h_{i}h_{i+1}\cdots h_{j-3})^{-\frac{j-i-1}{2}}\\
&&\cdot t_{j-1,j}^{-\frac{j-i-1}{2}}(h_{i}h_{i+1}\cdots h_{j-2})^{\frac{j-i-1}{2}}h_{j-1}\cdots h_{i+4}h_{i+2}\cdot \underset{\leftarrow }{\underline{t_{i,i+1}^{\frac{j-i+1}{2}}}}h_ih_{i+2}\cdots h_{j-1}\\
&&\cdot t_{i+2,i+3}^{-1}t_{i+4,i+5}^{-1}\cdots t_{j-1,j}^{-1}\underline{h_{i+1}h_{i+3}\cdots h_{j-2}\cdot h_j(h_{j-1}\cdots h_{i+1}h_{i})^{\frac{j-i+1}{2}}}\\
&\overset{\text{(1)(c)}}{\underset{\text{Lem.\ref{tech_rel5}}}{=}}&t_{j-1,j}^{-\frac{j-i-1}{2}}\cdots t_{i+2,i+3}^{-\frac{j-i-1}{2}}t_{i,i+1}^{-\frac{j-i-1}{2}}h_j\cdot h_{j-3}^{-1}\cdots h_{i+2}^{-1}h_{i}^{-1}(h_{i}h_{i+1}\cdots h_{j-3})^{-\frac{j-i-1}{2}}\\
&&\cdot t_{j-1,j}^{-\frac{j-i-1}{2}}\underline{(h_{i}h_{i+1}\cdots h_{j-2})^{\frac{j-i-1}{2}}t_{i,i+1}^{\frac{j-i+1}{2}}}h_{j-1}\cdots h_{i+4}h_{i+2}\cdot h_ih_{i+2}\cdots h_{j-1}\\
&&\cdot t_{i+2,i+3}^{-1}t_{i+4,i+5}^{-1}\cdots t_{j-1,j}^{-1}(h_{j}\cdots h_{i+1}h_{i})^{\frac{j-i+1}{2}}\\
&\overset{\text{(2)(b)}}{\underset{}{=}}&t_{j-1,j}^{-\frac{j-i-1}{2}}\cdots t_{i+2,i+3}^{-\frac{j-i-1}{2}}t_{i,i+1}^{-\frac{j-i-1}{2}}h_j\cdot h_{j-3}^{-1}\cdots h_{i+2}^{-1}h_{i}^{-1}(h_{i}h_{i+1}\cdots h_{j-3})^{-\frac{j-i-1}{2}}\\
&&\cdot \underline{(h_{i}h_{i+1}\cdots h_{j-2})^{\frac{j-i-1}{2}}}h_{j-1}\cdots h_{i+4}h_{i+2}\cdot h_ih_{i+2}\cdots h_{j-1}\\
&&\cdot t_{i+2,i+3}^{-1}t_{i+4,i+5}^{-1}\cdots t_{j-1,j}^{-1}(h_{j}\cdots h_{i+1}h_{i})^{\frac{j-i+1}{2}}\\
&\overset{\text{Lem.\ref{tech_rel5}}}{\underset{}{=}}&t_{j-1,j}^{-\frac{j-i-1}{2}}\cdots t_{i+2,i+3}^{-\frac{j-i-1}{2}}t_{i,i+1}^{-\frac{j-i-1}{2}}h_j\cdot h_{j-3}^{-1}\cdots h_{i+2}^{-1}h_{i}^{-1}\cdot h_{j-2}\cdots h_{i+3}h_{i+1}\\
&&\cdot \underline{h_{j-1}\cdots h_{i+4}h_{i+2}\cdot h_ih_{i+2}\cdots h_{j-1}\cdot t_{i+2,i+3}^{-1}t_{i+4,i+5}^{-1}\cdots t_{j-1,j}^{-1}}\\
&&\cdot (h_{j}\cdots h_{i+1}h_{i})^{\frac{j-i+1}{2}}\\
&\overset{\text{Lem.\ref{tech_rel10}}}{\underset{}{=}}&t_{j-1,j}^{-\frac{j-i-1}{2}}\cdots t_{i+2,i+3}^{-\frac{j-i-1}{2}}t_{i,i+1}^{-\frac{j-i-1}{2}}h_j\cdot \underline{h_{j-3}^{-1}\cdots h_{i+2}^{-1}h_{i}^{-1}\cdot h_{j-2}\cdots h_{i+3}h_{i+1}}\\
&&\underline{\cdot h_{i}\cdots h_{j-3}}h_{j-1}h_{j-3}\cdots h_{i}t_{i+1,i+2}^{-1}t_{i+3,i+4}^{-1}\cdots t_{j-2,j-1}^{-1}\\
&&\cdot (h_{j}\cdots h_{i+1}h_{i})^{\frac{j-i+1}{2}}\\
&\overset{\text{Lem.\ref{tech_rel11}}}{\underset{}{=}}&t_{j-1,j}^{-\frac{j-i-1}{2}}\cdots t_{i+2,i+3}^{-\frac{j-i-1}{2}}t_{i,i+1}^{-\frac{j-i-1}{2}}h_j\cdot h_{j-2}\cdots h_{i+3}h_{i+1}\underline{t_{i,i+1}t_{j-1,j}^{-1}}\\
&&\underline{\cdot h_{j-1}h_{j-3}\cdots h_{i}}t_{i+1,i+2}^{-1}t_{i+3,i+4}^{-1}\cdots t_{j-2,j-1}^{-1}(h_{j}\cdots h_{i+1}h_{i})^{\frac{j-i+1}{2}}\\
&\overset{\text{(1)}}{\underset{\text{(2)(a)}}{=}}&t_{j-1,j}^{-\frac{j-i-1}{2}}\cdots t_{i+2,i+3}^{-\frac{j-i-1}{2}}t_{i,i+1}^{-\frac{j-i-1}{2}}\underline{h_j\cdot h_{j-2}\cdots h_{i+3}h_{i+1}\cdot h_{j-1}h_{j-3}\cdots h_{i}}\\
&&\cdot t_{i+3,i+4}^{-1}\cdots t_{j-2,j-1}^{-1}t_{j,j+1}^{-1}(h_{j}\cdots h_{i+1}h_{i})^{\frac{j-i+1}{2}}\\
&\overset{\text{Lem.\ref{tech_rel12}}}{\underset{}{=}}&t_{j-1,j}^{-\frac{j-i-1}{2}}\cdots t_{i+2,i+3}^{-\frac{j-i-1}{2}}t_{i,i+1}^{-\frac{j-i-1}{2}}\cdot \underset{\leftarrow }{\underline{t_{j+1,j+2}^{-\frac{j-i-1}{2}}}}h_{j}\cdots h_{i+1}h_i\underline{t_{i+3,i+4}t_{i+5,i+6}\cdots t_{j,j+1}}\\
&&\underline{\cdot t_{i+3,i+4}^{-1}\cdots t_{j-2,j-1}^{-1}t_{j,j+1}^{-1}}(h_{j}\cdots h_{i+1}h_{i})^{\frac{j-i+1}{2}}\\
&\overset{\text{(1)(a)}}{\underset{}{=}}&t_{j+1,j+2}^{-\frac{j-i-1}{2}}\cdots t_{i+2,i+3}^{-\frac{j-i-1}{2}}t_{i,i+1}^{-\frac{j-i-1}{2}}(h_{j}\cdots h_{i+1}h_{i})^{\frac{j-i+3}{2}}.
\end{eqnarray*}
Recall that the relations in Lemma~\ref{tech_rel5}, \ref{tech_rel10}, \ref{tech_rel11}, and \ref{tech_rel12} for $j\leq 2n+1$ are also obtained from the relations~(1) and (2). 
Therefore the relation~$(L_{i,j}^{k=j})$ for even $j-i\geq 6$ is equivalent to the relation~$(T_{i,j+2})$ up to the relations~(1), (2), and $(T_{i^\prime ,j^\prime })$ for $i\leq i^\prime <j^\prime \leq j+2$ and $j-i+2>j^\prime -i^\prime $, and we have completed the proof of Lemma~\ref{lem_t_ij_pres1}. 
\end{proof}

By an argument similar to the proof of Lemma~\ref{lem_t_ij_pres1} of the odd $j-i$ case and Lemma~\ref{tech_rel1}, we have the following lemma: 

\begin{proof}[Proof of Lemma~\ref{lem_L_ij_k=j-1}]
First, we remark that the relation~$(L_{i,j-1}^{k=j})$ is equivalent to the following relation:
\begin{eqnarray*}
&&h_{j-1}t_{i,j-1}h_{j-1}^{-1}=\underline{t_{i,j-2}}t_{j-1,j}\underline{t_{i,j+1}t_{i,j}^{-1}t_{j-1,j+1}^{-1}}\\
&\overset{\text{CONJ}}{\underset{}{\Longleftrightarrow}}&t_{i,j-2}^{-1}h_{j-1}t_{i,j-1}h_{j-1}^{-1}\underline{t_{j-1,j+1}}t_{i,j}t_{i,j+1}^{-1}=t_{j-1,j}\\
&\overset{(T_{j-1,j+1})}{\underset{}{\Longleftrightarrow}}&t_{i,j-2}^{-1}h_{j-1}t_{i,j-1}h_{j-1}t_{i,j}t_{i,j+1}^{-1}=t_{j-1,j}.
\end{eqnarray*}
Then we have
\begin{eqnarray*}
(L_{i,j}^{k=j-1})&\Longleftrightarrow&h_{j-1}t_{i,j}h_{j-1}^{-1}=t_{i,j-2}t_{j,j+1}\underset{\leftarrow }{\underline{t_{i,j+1}t_{i,j-1}^{-1}}}t_{j-1,j+1}^{-1}\\
&\overset{\text{(1)(b)}}{\underset{}{\Longleftrightarrow}}&h_{j-1}t_{i,j}h_{j-1}^{-1}=\underline{t_{i,j-2}t_{i,j+1}t_{i,j-1}^{-1}}t_{j,j+1}\underline{t_{j-1,j+1}^{-1}}\\
&\overset{\text{CONJ}}{\underset{}{\Longleftrightarrow}}&\underset{\rightarrow }{\underline{t_{i,j-1}}}\ \underset{\rightarrow }{\underline{t_{i,j+1}^{-1}}}t_{i,j-2}^{-1}h_{j-1}t_{i,j}h_{j-1}^{-1}\underline{t_{j-1,j+1}}=t_{j,j+1}\\
&\overset{\text{(1)}}{\underset{(T_{j-1,j+1})}{\Longleftrightarrow}}&t_{i,j-2}^{-1}t_{i,j-1}h_{j-1}t_{i,j}t_{i,j+1}^{-1}h_{j-1}=t_{j,j+1}\\
&\overset{}{\underset{}{\Longleftrightarrow}}&h_{j-1}^{-1}\underset{\rightarrow }{\underline{h_{j-1}}}t_{i,j-2}^{-1}t_{i,j-1}h_{j-1}t_{i,j}t_{i,j+1}^{-1}h_{j-1}=t_{j,j+1}\\
&\overset{\text{(1)(c)}}{\underset{}{\Longleftrightarrow}}&h_{j-1}^{-1}\underline{t_{i,j-2}^{-1}h_{j-1}t_{i,j-1}h_{j-1}t_{i,j}t_{i,j+1}^{-1}}h_{j-1}=t_{j,j+1}\\
&\overset{(L_{i,j-1}^{k=j})}{\underset{}{\Longleftrightarrow}}&h_{j-1}^{-1}t_{j-1,j}h_{j-1}=t_{j,j+1}\\
&\overset{}{\underset{}{\Longleftrightarrow}}&\text{(2) (a)}.
\end{eqnarray*}
Therefore we have completed the proof of Lemma~\ref{lem_L_ij_k=j-1}. 
\end{proof}

By an argument similar to the proof of Lemma~\ref{lem_L_ij_k=j-1} and Lemma~\ref{tech_rel1}, we have the following lemma: 

Let $D_{i,j}$ for $1\leq i<j\leq 2n+2$ be the disk in $\Sigma _0$ with boundary $\gamma _{i,j}$ which contains the points $p_{i},\ p_{i+1},\ \dots ,\ p_{j}$ and $h_{i,j}$ a self-homeomorphism on $(\Sigma _0, \B )$ with support $D_{i,j}$ which is described as the result of half-rotation of $l_i\cup l_{i+1}\cup \cdots \cup l_{j-1}$ as in Figure~\ref{fig_h_ij}. 
We remark that $h_{i,j}$ is liftable for even $j-i$ and $h_{1,2n+2}$ is isotopic to $r$. 
Since we have the relations $h_{i,j}^{\pm 1}h_lh_{i,j}^{\mp 1}=h_{j+i-l-2}$ for $i\leq l\leq j-1$ and $h_{i,j}^{\pm 1}t_{k,l}h_{i,j}^{\mp 1}=t_{j+i-l,j+i-k}$ for $i\leq k<l\leq j$, by a similar argument in the proof of Lemma~\ref{tech_rel1}, the conjugations by $h_{i,j}^{\pm 1}$ of the relations~(1) and (2) among generators supported on $D_{i,j}$ are obtained from the relations~(1) and (2) among generators supported on $D_{i,j}$. 
Then we have the following lemma.

\begin{figure}[h]
\includegraphics[scale=0.95]{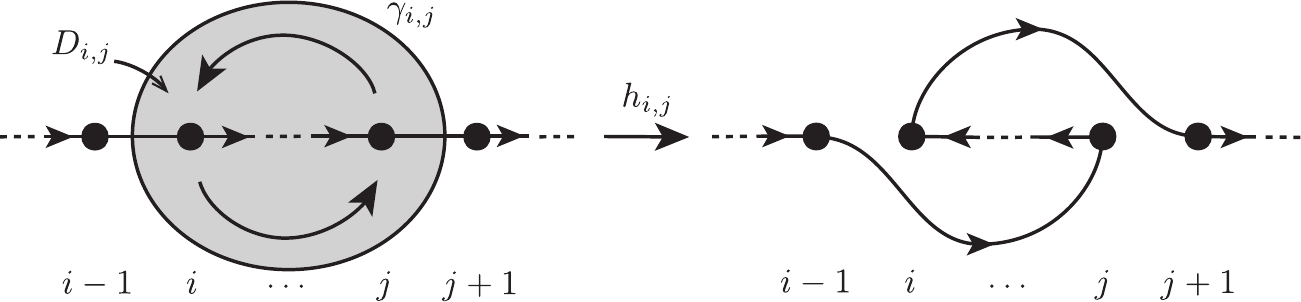}
\caption{The mapping class $h_{i,j}$.}\label{fig_h_ij}
\end{figure}

\begin{proof}[Proof of Lemma~\ref{lem_L_ij_k=i-2}]
By the relations $h_{i,j}^{\pm 1}h_lh_{i,j}^{\mp 1}=h_{j+i-l-2}$ for $i\leq l\leq j-1$ and $h_{i,j}^{\pm 1}t_{k,l}h_{i,j}^{\mp 1}=t_{j+i-l,j+i-k}$ for $i\leq k<l\leq j$, the conjugation of the relation~$(L_{i,j}^{k=i-2})$ by $h_{i-2,j}$ coincides with the relation~$(L_{i-2,j-2}^{k=j})$.  
By Lemma~\ref{lem_t_ij_pres1}, the relation~$(L_{i-2,j-2}^{k=j})$ is equivalent to the relation~$(T_{i-2,j})$ up to the relations~(1), (2), $(T_{i^\prime ,j^\prime })$ for $i\leq i^\prime <j^\prime \leq j$ and $j-i>j^\prime -i^\prime $, and $t_{i,i}=1$ for $1\leq i\leq 2n-1$. 
By an argument in the proof of Lemma~\ref{lem_t_ij_pres1}, these using relations~(1), (2), $(T_{i^\prime ,j^\prime })$, and $t_{i,i}=1$ are supported on $D_{i-2,j}$. 
Thus the conjugations of these relations~(1) and (2) by $h_{i-2,j}^{-1}$ are obtained from the relations~(1) and (2). 

By Lemma~\ref{cor_tech_rel8} and Lemma~\ref{tech_rel19}, the conjugation of the relation~$(T_{i-2,j})$ by $h_{i-2,j}^{-1}$ is equivalent to the relation~$(T_{i-2,j})$ up to the relations~(1) and (2), and the conjugation of the relation $(T_{i^{\prime },j^{\prime }})$ for $i\leq i^{\prime }<j^{\prime }\leq j$ by $h_{i-2,j}^{-1}$ is also equivalent to a relation~$(T_{i^{\prime \prime },j^{\prime \prime }})$ for some $i\leq i^{\prime \prime }<j^{\prime \prime }\leq j$ up to the relations~(1) and (2). 
Thus, the relation~$(L_{i,j}^{k=i-2})$ is equivalent to the relation~$(T_{i-2,j})$ up to the relations~(1) and (2) in Theorem, $(T_{i^\prime ,j^\prime })$ for $i-2\leq i^\prime <j^\prime \leq j$ and $j-i>j^\prime -i^\prime $, and $t_{i,i}=1$ for $1\leq i\leq 2n-1$. 
Therefore, we have completed the proof of Lemma~\ref{lem_L_ij_k=i-2}. 
\end{proof}

\begin{proof}[Proof of Lemma~\ref{lem_L_ij_k=i-1}]
By the relations $h_{i,j}^{\pm 1}h_lh_{i,j}^{\mp 1}=h_{j+i-l-2}$ for $i\leq l\leq j-1$ and $h_{i,j}^{\pm 1}t_{k,l}h_{i,j}^{\mp 1}=t_{j+i-l,j+i-k}$ for $i\leq k<l\leq j$, the conjugation of the relation~$(L_{i,j}^{k=i-1})$ by $h_{i-1,j}$ coincides with the relation~$(L_{i-1,j-1}^{k=j-1})$.  
By Lemma~\ref{lem_L_ij_k=j-1} and its proof, the relation~$(L_{i-1,j-1}^{k=j-1})$ is obtained from the relations~(1) and (2) among generators supported on $D_{i-1,j}$, and the relations $(T_{j-2,j})$ and $(L_{i-1,j-2}^{k=j})$. 
The conjugations by $h_{i-1,j}^{-1}$ of these relations~(1) and (2) are obtained from the relations~(1) and (2), and the conjugations by $h_{i-1,j}^{-1}$ of the relations~$(T_{j-2,j})$ and $(L_{i-1,j-2}^{k=j})$ coincide with the relations~$(T_{i-1,i+1})$ and $(L_{i+1,j}^{k=i-2})$, respectively. 
Therefore, the relation~$(L_{i,j}^{k=i-1})$ is obtained from the relations~(1), (2), $(T_{i-1,i+1})$, and $(L_{i+1,j}^{k=i-2})$ and we have completed the proof of Lemma~\ref{lem_L_ij_k=i-1}. 
\end{proof}

\section{Presentations for the balanced superelliptic mapping class groups}\label{section_smod}

Throughout this section, we assume that $g=n(k-1)$ for $n\geq 1$ and $k\geq 3$.  

\subsection{Explicit lifts of generators for the liftable mapping class groups}\label{section_lifts}
In this section, we give explicit lifts of generators for the liftable mapping class groups in Theorems~\ref{thm_pres_lmodb}, \ref{thm_pres_lmodp}, and \ref{thm_pres_lmod} with respect to the balanced superelliptic covering map $p=p_{g,k}\colon \Sigma _g\to \Sigma _0$. 
First, we consider lifts of the half-twists $a_i$ and $b_i$ for $1\leq i\leq n$ which are reviewed in Section~\ref{section_liftable-element}. 

Let $a$ be a simple arc on $\Sigma _0$ whose boundary lies in either $\B _o=\{ p_1, p_3, \dots , p_{2n+1}\}$ or $\B _e=\{ p_2, p_4, \dots ,p_{2n+2}\}$ and interior does not intersect with $\B $. 
We denote $\partial a=\{ p_i, p_j\}$ for $i<j$ (i.e. $j-i$ is even), $\bar{l}_m=l_m$ for $m=i, j$ if $i$ and $j$ are odd, and $\bar{l}_m=l_{m-1}$ for $m=i, j$ if $i$ and $j$ are even (see Figure~\ref{fig_proof2_lift_half-twist}). 
Let $\mathcal{N}$ be a regular neighborhood of $a$ in $(\Sigma _0-\B )\cup \{ p_i, p_j\}$ whose intersection with $\bar{l}_m$ for each $m=i, j$ is an arc which connects $\partial \mathcal{N}$ and $p_m$ as in Figure~\ref{fig_proof2_lift_half-twist}. 
When the point $p_{2n+2}$ does not lies in $\partial a$, suppose that $\mathcal{N}$ is disjoint from the disk $D$.  

\begin{figure}[h]
\includegraphics[scale=0.92]{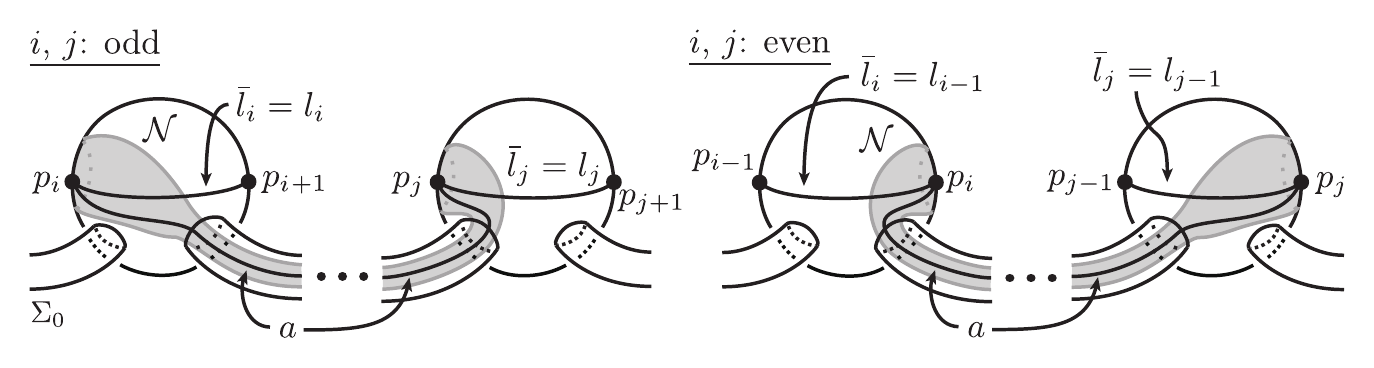}
\caption{A regular neighborhood $\mathcal{N}$ of a simple arc $a$ on $\Sigma _0$ with endpoints $p_i$ and $p_j$ for even $j-i$, $\bar{l}_m=l_m$ for $m=i, j$ if $i$ and $j$ are odd, and $\bar{l}_m=l_{m-1}$ for $m=i, j$ if $i$ and $j$ are even.}\label{fig_proof2_lift_half-twist}
\end{figure}

The neighborhood $\mathcal{N}$ is constructed by connecting small disk neighborhoods of $p_i$ and $p_j$ by a thin band in $\Sigma _0$ whose core is a subset of $a$ as on the lower side in Figure~\ref{fig_lift_half-twist1}. 
Let $c$ be an oriented simple proper arc in $\mathcal{N}$ which transversely intersects with $a$ at one point as on the lower left-hand side in Figure~\ref{fig_lift_half-twist1}. 
Remark that the image $\sigma [a](c)$ is a simple arc as on the lower right-hand side in Figure~\ref{fig_lift_half-twist1} and the isotopy class of $\sigma [a]$ relative to $\partial \mathcal{N}$ is determined by the isotopy class of $\sigma [a](c)$ relative to $\partial \mathcal{N}$. 

The total space $\Sigma _g$ of the balanced superelliptic covering map $p\colon \Sigma _g \to \Sigma _0$ is constructed by cutting $\Sigma _0$ along $l_1\cup l_3\cup \cdots l_{2n+1}$ and pasting its $k$ copies along the cut off $2k$ copies of $l_1\cup l_3\cup \cdots l_{2n+1}$. 
Remark that when $i$ and $j$ are odd (resp. even), the balanced superelliptic rotation $\zeta =\zeta _{g,k}$ acts on the small disk neighborhoods of $\widetilde{p}_{i}$ and $\widetilde{p}_{j}$ by the counter‐clockwise (resp. clockwise) $\frac{2\pi }{k}$-rotation. 
Since $\partial l_{2s-1}=\{ p_{2s-1}, p_{2s}\}$ for $1\leq s\leq n+1$, for $m=i, j$, $\bar{l}_m=l_m$ lies in $l_1\cup l_3\cup \cdots l_{2n+1}$ if $i$ and $j$ are odd, and $\bar{l}_m=l_{m-1}$ lies in $l_1\cup l_3\cup \cdots l_{2n+1}$ if $i$ and $j$ are even (see Figure~\ref{fig_proof2_lift_half-twist}). 
Thus the preimage $\widetilde{\mathcal{N}}=p^{-1}(\mathcal{N})\subset \Sigma _g$ is topologically constructed by cutting $\mathcal{N}$ along $\bar{l}_{i}\cap \mathcal{N}$ and $\bar{l}_{j}\cap \mathcal{N}$ and pasting these $k$ copies along the cut off $2k$ copies of $\bar{l}_{i}\cap \mathcal{N}$ and $\bar{l}_{j}\cap \mathcal{N}$ as on the upper and middle left-hand side in Figure~\ref{fig_lift_half-twist1}. 

Let $\widetilde{a}^l$ for $1\leq l\leq k$ be a lift of $a$ with respect to $p$ such that $\zeta ^l(\widetilde{a}^1)=\widetilde{a}^{l+1}$ for $1\leq l\leq k-1$. 
Then we denote by $\gamma ^1$ a simple closed curve on $\widetilde{\mathcal{N}}$ which is isotopic to $\widetilde{a}^1\cup \widetilde{a}^{2}$ and $\gamma ^l=\zeta ^{l-1}(l^1)$ for $2\leq l\leq k$ (see the upper and middle left-hand side in Figure~\ref{fig_lift_half-twist1}). 
Remark that $\gamma ^l$ for $2\leq l\leq k-1$ (resp. $\gamma ^k$) is also isotopic to $\widetilde{a}^{l}\cup \widetilde{a}^{l+1}$ (resp. $\widetilde{a}^{k}\cup \widetilde{a}^{1}$). 
Under the situation above, we have the following lemma that is a specialized version of Proposition~5.3 in~\cite{Ghaswala-McLeay}. 

\begin{figure}[h]
\includegraphics[scale=0.83]{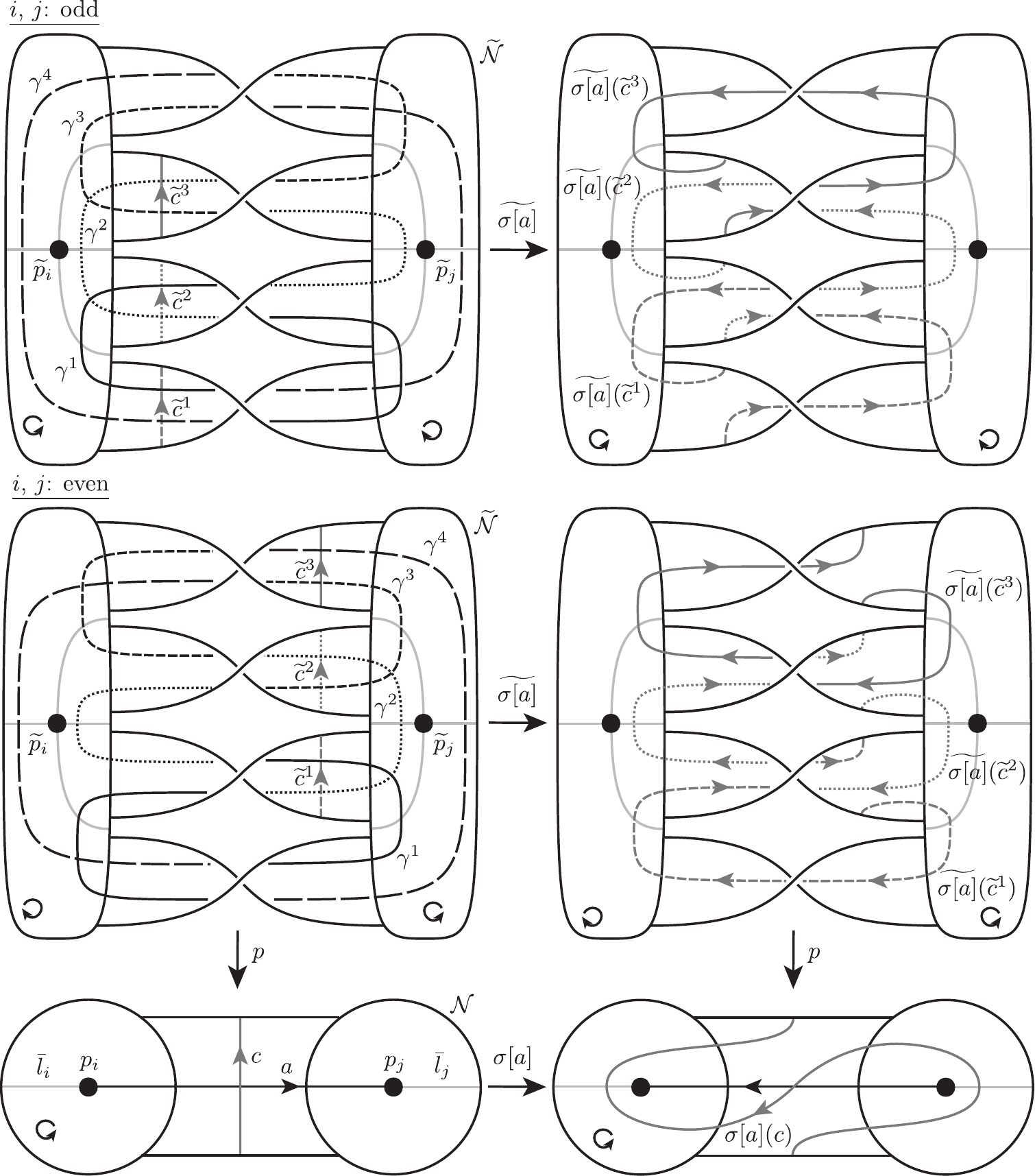}
\caption{A construction of the preimage $\widetilde{\mathcal{N}}$ of $N$ with respect to $p_{g,k}$ when $k=4$ and the images of $\widetilde{c}^i$ by $\widetilde{\sigma [a]}$ for $i=1,2,3$.}\label{fig_lift_half-twist1}
\end{figure}

\begin{lem}\label{lift_half-twist}
The relations
\[
\widetilde{\sigma [a]}=\left\{
		\begin{array}{ll}
		t_{\gamma ^1}t_{\gamma ^2}\cdots t_{\gamma ^{k-1}} & \text{for }i \text{ and }j\text{ are odd},\\
		t_{\gamma ^{k-1}}t_{\gamma ^{k-2}}\cdots t_{\gamma ^1} & \text{for }i \text{ and }j\text{ are even}\\
		\end{array}
		\right.\\
\]
hold relative to $\partial \widetilde{\mathcal{N}}$ for some lift $\widetilde{\sigma [a]}$ of $\sigma [a]$ with respect to $p$. 
\end{lem}

\begin{proof}
Let $\widetilde{c}^l$ for $1\leq l\leq k$ be the lift of $c$ with respect to $p$ such that $\widetilde{c}^l$ transversely intersects with $\widetilde{a}^l$ (resp. $\widetilde{a}^{l+1}$) at one point when $i$ and $j$ are odd (resp. even).  
Remark that $\zeta ^{l}(\widetilde{c}^1)=\widetilde{c}^{l+1}$ for $1\leq l\leq k-1$. 
Since the surface which is obtained from $\widetilde{\mathcal{N}}$ by cutting along $\widetilde{c}^1\sqcup \widetilde{c}^2\sqcup \cdots \sqcup \widetilde{c}^{k-1}$ is a disk, the isotopy class of $\widetilde{\sigma [a]}$ relative to $\partial \widetilde{\mathcal{N}}$ is determined by the isotopy class of the image $\widetilde{\sigma [a]}(\widetilde{c}^1\sqcup \widetilde{c}^2\sqcup \cdots \sqcup \widetilde{c}^{k-1})$ relative to $\partial \widetilde{\mathcal{N}}$. 

The image $\widetilde{\sigma [a]}(\widetilde{c}^1\sqcup \widetilde{c}^2\sqcup \cdots \sqcup \widetilde{c}^{k-1})$ in $\widetilde{\sigma [a]}(\widetilde{\mathcal{N}})=\widetilde{\mathcal{N}}$ for some lift $\widetilde{\sigma [a]}$ is obtained by cutting and pasting of $k$ copies of $\sigma [a](\mathcal{N})$ as in the lower right-hand side in Figure~\ref{fig_lift_half-twist1} as the same way to construct $\widetilde{\mathcal{N}}$ from $\mathcal{N}$ (see the upper and middle right-hand side in Figure~\ref{fig_lift_half-twist1}). 
By the construction, when $i$ and $j$ are odd (resp. even), $\widetilde{c}^l$ for $1\leq l\leq k-1$ is disjoint from $\gamma ^s$ for $l+1\leq s\leq k-1$ (resp. $1\leq s\leq l-1$) and transversely intersects with $\gamma ^l$ at one point, and $\widetilde{\sigma [a]}(\widetilde{c}^{l})$ is disjoint from $\gamma ^s$ for $1\leq s\leq l-1$ (resp. $l+1\leq s\leq k-1$). 
Since we can check that $t_{\gamma ^l}(\widetilde{c}^{l})=\widetilde{\sigma [a]}(\widetilde{c}^{l})$ for $1\leq l\leq k-1$, when $i$ and $j$ are odd, we have 
\begin{eqnarray*}
&&t_{\gamma ^1}\cdots t_{\gamma ^{k-2}}t_{\gamma ^{k-1}}(\widetilde{c}^1\sqcup \cdots \sqcup \widetilde{c}^{k-2}\sqcup \widetilde{c}^{k-1})\\
&=&t_{\gamma ^1}\cdots t_{\gamma ^{k-3}}t_{\gamma ^{k-2}}(t_{\gamma ^{k-1}}(\widetilde{c}^1)\sqcup \cdots \sqcup t_{\gamma ^{k-1}}(\widetilde{c}^{k-2})\sqcup t_{\gamma ^{k-1}}(\widetilde{c}^{k-1}))\\
&=&t_{\gamma ^1}\cdots t_{\gamma ^{k-3}}t_{\gamma ^{k-2}}(\widetilde{c}^1\sqcup \cdots \sqcup \widetilde{c}^{k-2}\sqcup \widetilde{\sigma [a]}(\widetilde{c}^{k-1}))\\
&=&t_{\gamma ^1}\cdots t_{\gamma ^{k-3}}(t_{\gamma ^{k-2}}(\widetilde{c}^1)\sqcup \cdots \sqcup t_{\gamma ^{k-2}}(\widetilde{c}^{k-3})\sqcup t_{\gamma ^{k-2}}(\widetilde{c}^{k-2})\sqcup t_{\gamma ^{k-2}}(\widetilde{\sigma [a]}(\widetilde{c}^{k-1})))\\
&=&t_{\gamma ^1}\cdots t_{\gamma ^{k-3}}(\widetilde{c}^1\sqcup \cdots \sqcup \widetilde{c}^{k-3}\sqcup \widetilde{\sigma [a]}(\widetilde{c}^{k-2})\sqcup \widetilde{\sigma [a]}(\widetilde{c}^{k-1}))\\
&\vdots &\\
&=&\widetilde{\sigma [a]}(\widetilde{c}^1)\sqcup \cdots \sqcup \widetilde{\sigma [a]}(\widetilde{c}^{k-2})\sqcup \widetilde{\sigma [a]}(\widetilde{c}^{k-1})\\
&=&\widetilde{\sigma [a]}(\widetilde{c}^1\sqcup \cdots \sqcup \widetilde{c}^{k-2}\sqcup \widetilde{c}^{k-1}). 
\end{eqnarray*}
Similarly, we also have $t_{\gamma ^{k-1}}t_{\gamma ^{k-2}}\cdots t_{\gamma ^1}(\widetilde{c}^1\sqcup \cdots \sqcup \widetilde{c}^{k-2}\sqcup \widetilde{c}^{k-1})=\widetilde{\sigma [a]}(\widetilde{c}^1\sqcup \cdots \sqcup \widetilde{c}^{k-2}\sqcup \widetilde{c}^{k-1})$ when $i$ and $j$ are even. 
Therefore, we have $t_{\gamma ^1}t_{\gamma ^{2}}\cdots t_{\gamma ^{k-1}}=\widetilde{\sigma [a]}$ when $i$ and $j$ are odd and $t_{\gamma ^{k-1}}t_{\gamma ^{k-2}}\cdots t_{\gamma ^1}=\widetilde{\sigma [a]}$ when $i$ and $j$ are even, and we have completed the proof of Lemma~\ref{lift_half-twist}. 
\end{proof} 

Under the situation above and for odd $i$ and $j$, by conjugation relations, we have
\begin{eqnarray*}
\zeta t_{\gamma ^1}t_{\gamma ^2}\cdots t_{\gamma ^{k-1}}\zeta ^{-1}&=&t_{\gamma ^2}t_{\gamma ^3}\cdots t_{\gamma ^{k}}\\
&=&(t_{\gamma ^2}t_{\gamma ^3}\cdots t_{\gamma ^{k-1}})t_{\gamma ^{k}}(t_{\gamma ^2}t_{\gamma ^3}\cdots t_{\gamma ^{k-1}})^{-1}\cdot t_{\gamma ^2}t_{\gamma ^3}\cdots t_{\gamma ^{k-1}}\\
&=&t_{\gamma ^1}t_{\gamma ^2}\cdots t_{\gamma ^{k-1}}
\end{eqnarray*}
relative to $\partial \widetilde{\mathcal{N}}$. 
Similarly, for even $i$ and $j$, we also have  
\[
\zeta t_{\gamma ^{k-1}}t_{\gamma ^{k-2}}\cdots t_{\gamma ^1}\zeta ^{-1}=t_{\gamma ^{k}}t_{\gamma ^{k-1}}\cdots t_{\gamma ^2}=t_{\gamma ^{k-1}}t_{\gamma ^{k-2}}\cdots t_{\gamma ^1}
\]
relative to $\partial \widetilde{\mathcal{N}}$. 
Hence $\widetilde{\sigma [a]}$ commutes with $\zeta $ relative to $\partial \widetilde{\mathcal{N}}$. 

When $i<j\leq 2n-1$, by an assumption, since $\widetilde{\mathcal{N}}$ does not intersect with the disk $\widetilde{D}$ (i.e. $\widetilde{\mathcal{N}}$ does not include the point $\widetilde{p}_{2n+2}$), the relation in Lemma~\ref{lift_half-twist} is also holds relative to $\partial \widetilde{\mathcal{N}}\cup \widetilde{D}$, namely, $\widetilde{\sigma [a]}$ lifts to an element of $\SMp $ and $\SMb $ and the relation in Lemma~\ref{lift_half-twist} is also holds in $\SMp $ and $\SMb $. 
However, when $j=2n+2$, the relation $t_{\gamma ^{k}}t_{\gamma ^{k-1}}\cdots t_{\gamma ^2}=t_{\gamma ^{k-1}}t_{\gamma ^{k-2}}\cdots t_{\gamma ^1}$ does not hold in $\SMp $ and $\SMb $ and $t_{\gamma ^1}t_{\gamma ^2}\cdots t_{\gamma ^{k-1}}$ does not commute with $\zeta $. 
Recall that $\zeta ^\prime $ is a self-homeomorphism on $\Sigma _g$ which is described as a result of a $(-\frac{2\pi }{k})$-rotation of $\Sigma _g^1\subset \Sigma _g$ fixing the disk $\widetilde{D}$ pointwise (see Figure~\ref{fig_lift_t_partial-d}). 
As a corollary of Lemma~\ref{lift_half-twist}, we have the following corollary.  

\begin{cor}\label{cor_lift_half-twist}
Let $a$ be a simple arc on $\Sigma _0$ such that $\partial a=\{ p_i, p_j\} \subset \B$ with $j-i>0$ is even and the interior of the arc $a$ does not intersect with $\B $. 
Then, for $j\leq 2n+1$, a lift $\widetilde{\sigma [a]}$ of $\sigma [a]$ with respect to $p$ commutes with $\zeta $ (resp. $\zeta ^\prime $) in $\SMp $ (resp. $\SMb $), and for $j=2n+2$, $\widetilde{\sigma [a]}$ commutes with $\zeta $ in $\SM $.  
\end{cor}

Let $\widetilde{l}_i^l$ for $1\leq i\leq 2n+1$ and $1\leq l\leq k$ be a lift of $l_i$ with respect to $p$ such that $\zeta (\widetilde{l}_i^l)=\widetilde{l}_i^{l+1}$ for  $1\leq l\leq k-1$ and $\zeta (\widetilde{l}_i^k)=\widetilde{l}_i^{1}$. 
We consider a homeomorphism of $\Sigma _g$ as in Figure~\ref{fig_isotopy_surface_3-handles} and identify $\Sigma _g$ with the surface as on the lower side in Figure~\ref{fig_isotopy_surface_3-handles}. 
Let $\alpha _{i}^{l}$ and $\beta _{i}^l$ for $1\leq i\leq n$ and $1\leq l\leq k-1$ be simple closed curves on $\Sigma _g$ as in Figures~\ref{fig_scc_a_il} and \ref{fig_scc_b_il}. 
Remark that $\alpha _{i}^{l}$ (resp. $\beta _{i}^{l}$) intersects with $\alpha _{i}^{l\pm 1}$, $\alpha _{i\pm 1}^{l}$, and $\alpha _{i\pm 1}^{l\mp 1}$ (resp. $\beta _{i}^{l\pm 1}$, $\beta _{i\pm 1}^{l}$, and $\beta _{i\pm 1}^{l\pm 1}$) at one point. 
By applying an argument before Lemma~\ref{lift_half-twist} to the arc $a=\alpha _i$ or $a=\beta _i$ on $\Sigma _0$ and Lemma~\ref{lift_half-twist} (see also Figures~\ref{fig_scc_a_il} and \ref{fig_scc_b_il}), we have the following lemma.

\begin{figure}[h]
\includegraphics[scale=0.9]{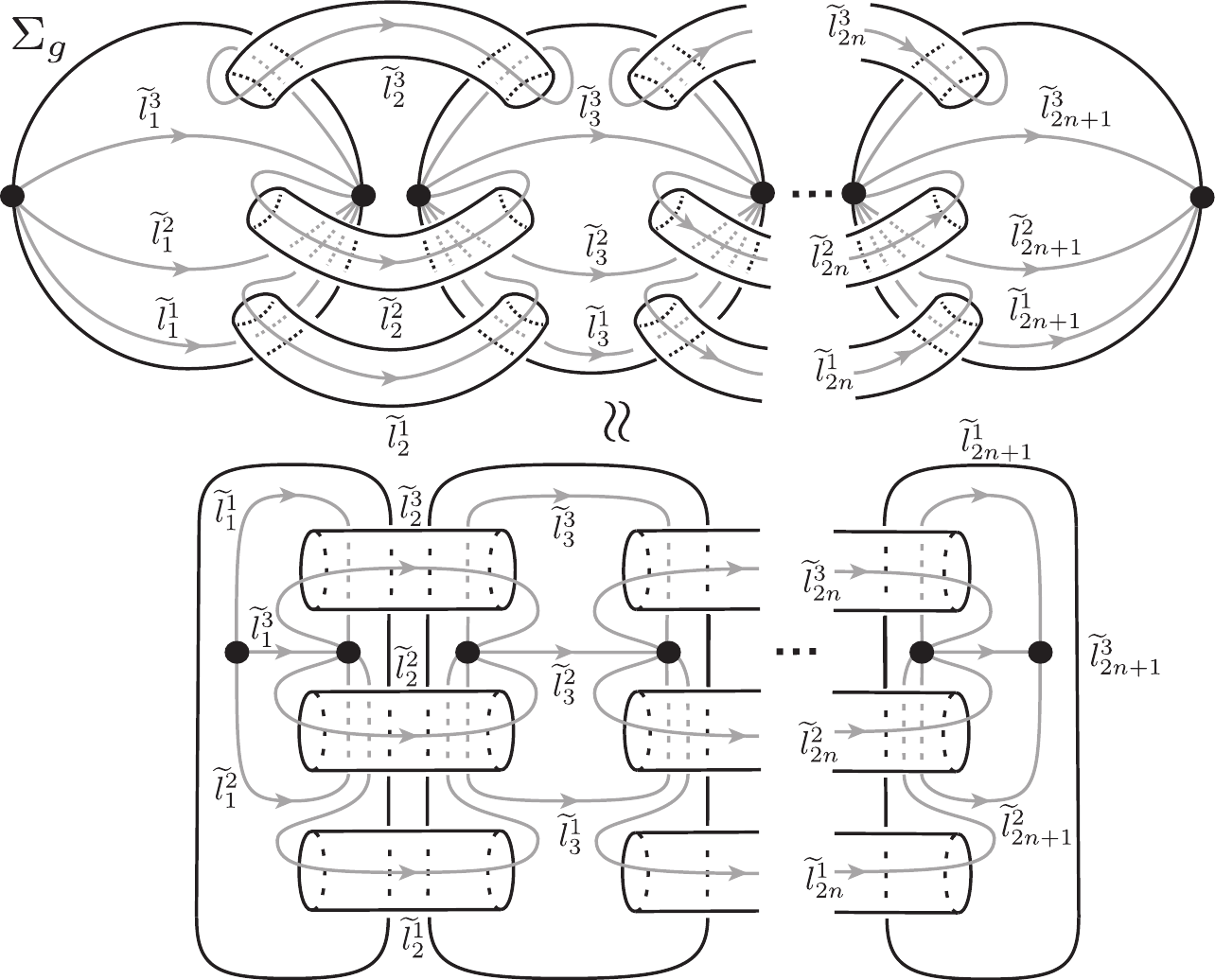}
\caption{A natural homeomorphism of $\Sigma _g$ when $k=3$.}\label{fig_isotopy_surface_3-handles}
\end{figure}

\begin{figure}[h]
\includegraphics[scale=0.90]{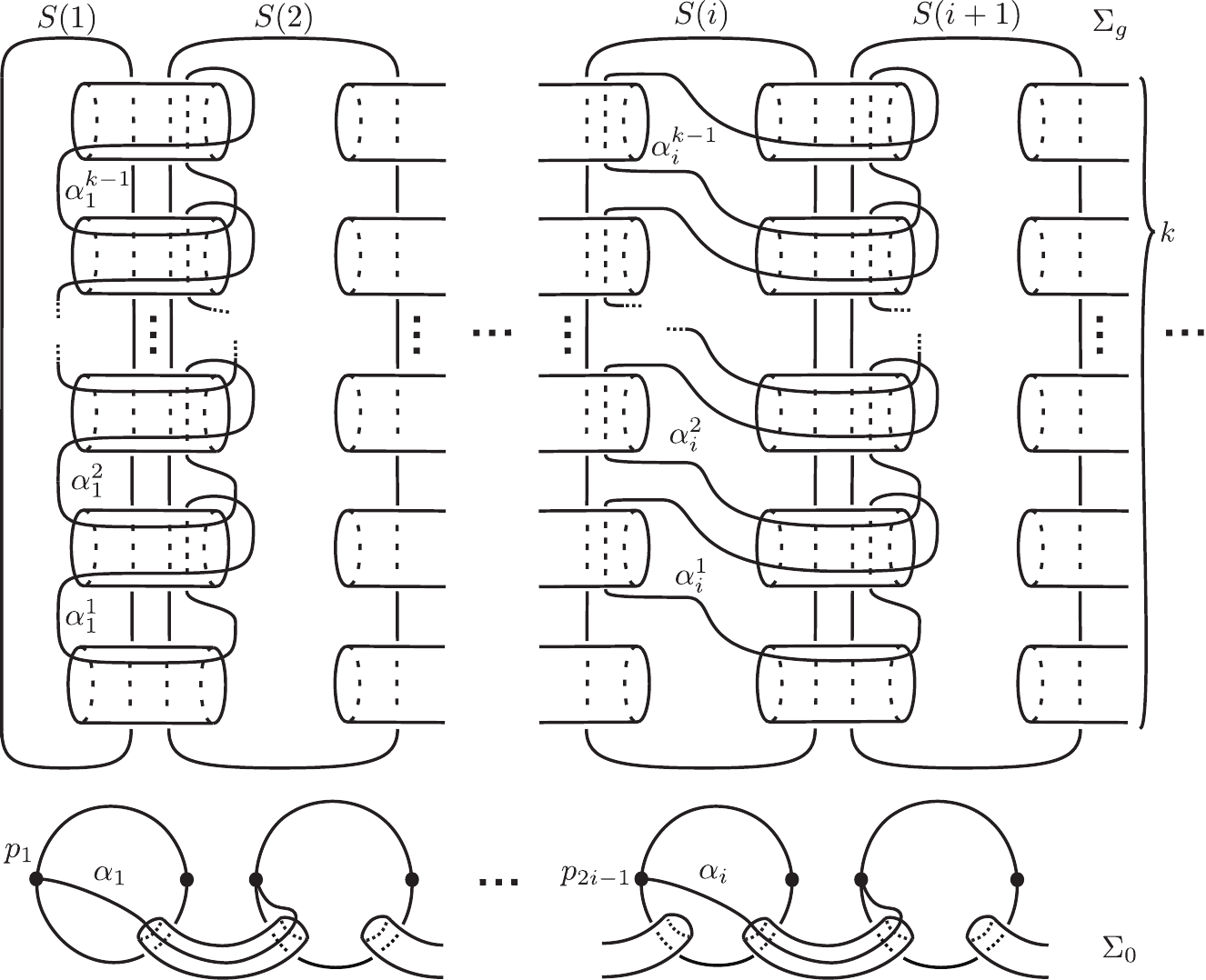}
\caption{Simple closed curves $\alpha _i^l$ on $\Sigma _g$ for $1\leq i\leq n$ and $1\leq l\leq k-1$.}\label{fig_scc_a_il}
\end{figure}

\begin{figure}[h]
\includegraphics[scale=0.90]{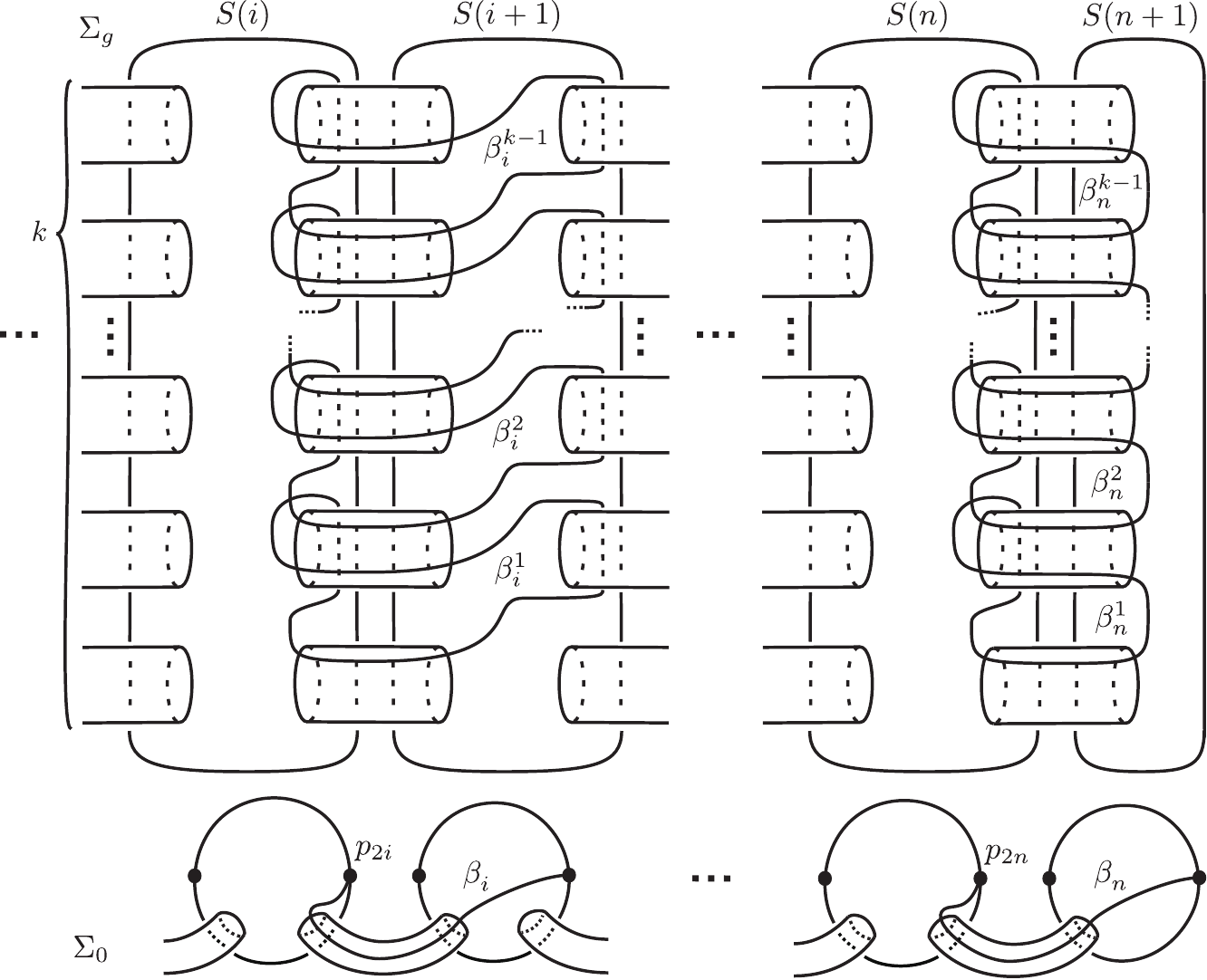}
\caption{Simple closed curves $\beta _i^l$ on $\Sigma _g$ for $1\leq i\leq n$ and $1\leq l\leq k-1$.}\label{fig_scc_b_il}
\end{figure}

\begin{lem}\label{lift-a_i}
The relations 
\begin{enumerate}
\item $\widetilde{a}_i=t_{\alpha _i^1}t_{\alpha _i^2}\cdots t_{\alpha _i^{k-1}}$\quad for $1\leq i\leq n$,
\item $\widetilde{b}_i=t_{\beta _i^{k-1}}t_{\beta _i^{k-2}}\cdots t_{\beta _i^1}$ \quad for $1\leq i\leq n-1$
\end{enumerate}
hold relative to $\partial \widetilde{\mathcal{N}}\cup \widetilde{D}$ for some lifts $\widetilde{a}_i$ and $\widetilde{b}_i$ of $a_i$ and $b_i$ with respect to $p$, respectively, and the relation
\[
\widetilde{b}_n=t_{\beta _n^{k-1}}t_{\beta _n^{k-2}}\cdots t_{\beta _n^1}
\] 
holds relative to $\partial \widetilde{\mathcal{N}}$ for some lift $\widetilde{b}_n$ of $b_n$ with respect to $p$.
\end{lem}

By Lemma~\ref{lift-a_i}, we regard the relations $\widetilde{a}_i=t_{\alpha _i^1}t_{\alpha _i^2}\cdots t_{\alpha _i^{k-1}}$ for $1\leq i\leq n$ and $\widetilde{b}_i=t_{\beta _i^{k-1}}t_{\beta _i^{k-2}}\cdots t_{\beta _i^1}$ for $1\leq i\leq n-1$ as relations in $\SMp $ and $\SMb $, and also regard the relation $\widetilde{b}_n=t_{\beta _n^{k-1}}t_{\beta _n^{k-2}}\cdots t_{\beta _n^1}$ as a relation in $\SM $. 

A simple closed curve $\gamma $ on $\Sigma _0-\B $ \textit{lifts} with respect to $p$ if there exists a simple closed curve $\widetilde{\gamma }$ on $\Sigma _{g}-p^{-1}(\B )$ such that the restriction $p|_{\widetilde{\gamma }}\colon \widetilde{\gamma } \to \gamma $ is bijective. 
By Lemma~3.3 in \cite{Ghaswala-Winarski1}, a simple closed curve $\gamma $ on $\Sigma _0-\B $ lifts with respect to $p=p_{g,k}$ if and only if the algebraic intersection number of $\gamma $ and $l_1\cup l_3\cup \cdots \cup l_{2n+1}$ is zero mod $k$. 
Hence the simple closed curve $\gamma _{i,j}$ on $\Sigma _0$ (see Figure~\ref{fig_path_l}) for odd $j-i$ lifts with respect to $p$. 
For instance, $\gamma _{i,i+1}$ for $2\leq i\leq 2n$ transversely intersects with $l_{i-1}$ and $l_{i+1}$ at one point, respectively, is disjoint from other $l_j$ $(j\not= i-1, i+1)$, and the algebraic intersection number of $\gamma _{i,i+1}$ and $l_{i-1}\cup l_{i+1}$ is zero. 
Thus, a lift of $\gamma _{i,i+1}$ with respect to $p$ is the simple closed curve $\gamma _{i}^l$ on $\Sigma _g$ for $1\leq l\leq k$ as in Figure~\ref{fig_scc_c_il}. 
Similarly, since the simple closed curves $\gamma _{1,2}$ and $\gamma _{2n+1,2n+2}$ intersect with $l_2$ and $l_{2n}$, respectively, lifts of $\gamma _{1,2}$ and $\gamma _{2n+1,2n+2}$ with respect to $p$ are the simple closed curves $\gamma _{1}^l$ and $\gamma _{2n+1}^l$ on $\Sigma _g$ for $1\leq l\leq k$ as in Figure~\ref{fig_scc_c_il}. 
Remark that $\gamma _i^l$ for odd (resp. even) $i$ intersects with $\gamma _{i\pm 1}^{l+\varepsilon }$ for $\varepsilon =-1,\ 0$ (resp. $\varepsilon =0,\ 1$) at one point. 
We can take a lift of the Dehn twist along a liftable simple closed curve on $\Sigma _0$ by the product of the Dehn twists along all lifts of the simple closed curve. 
Thus, we have the following lemma. 

\begin{figure}[h]
\includegraphics[scale=0.90]{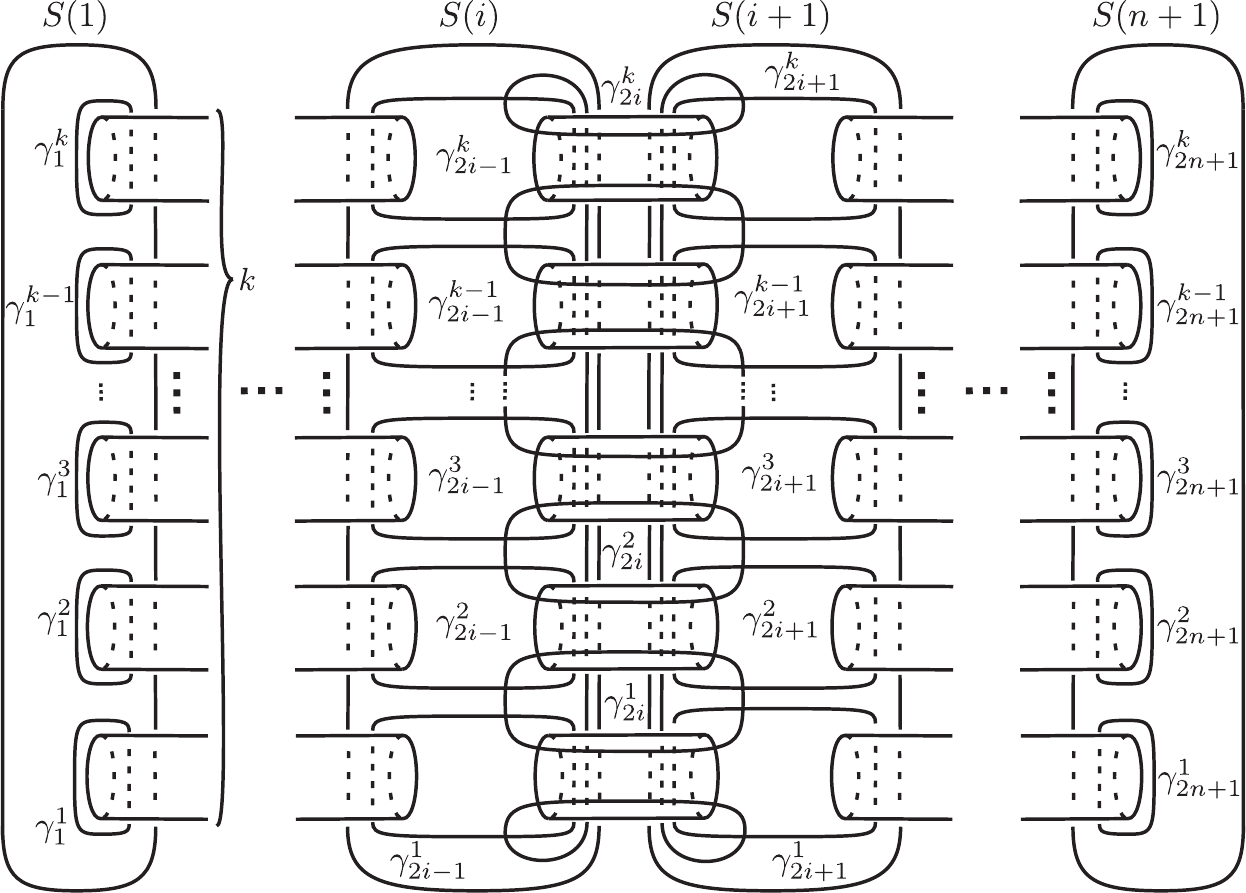}
\caption{Simple closed curves $\gamma _i^l$ on $\Sigma _g$ for $1\leq i\leq 2n+1$ and $1\leq l\leq k$.}\label{fig_scc_c_il}
\end{figure}

\begin{lem}\label{lift-t_{i,i+1}}
For $1\leq i\leq 2n+1$, the relation
\[
\widetilde{t}_{i,i+1}=t_{\gamma _i^1}t_{\gamma _i^2}\cdots t_{\gamma _i^{k}}
\] 
holds relative to $\widetilde{D}$ for some lift $\widetilde{t}_{i,i+1}$ of $t_{i,i+1}$ with respect to $p$.
\end{lem}

By Lemma~\ref{lift-t_{i,i+1}}, we regard the relation $\widetilde{t}_{i,i+1}=t_{\gamma _i^1}t_{\gamma _i^2}\cdots t_{\gamma _i^{k}}$ for $1\leq i\leq 2n+1$ as a relation in $\SMp \subset \SM $ and $\SMb $. 
Since the homeomorphisms $\zeta $ and $\zeta ^\prime $ preserve the set $\gamma _i^1\cup \gamma _i^2\cup \cdots \cup \gamma _i^k$, we have the following corollary. 

\begin{cor}\label{cor_lift-t_{i,i+1}}
For $1\leq i\leq 2n+1$, $\widetilde{t}_{i,i+1}$ commutes with $\zeta $ (resp. $\zeta ^\prime $) in $\SMp $ (resp. $\SMb $). 
\end{cor}

By using Lemma~\ref{lift-a_i} and \ref{lift-t_{i,i+1}}, we have the following lemma. 

\begin{lem}\label{lift-h_i}
The relations 
\begin{enumerate}
\item $\widetilde{h}_i=t_{\gamma _i^1}t_{\gamma _{i+1}^1}t_{\gamma _i^2}t_{\gamma _{i+1}^2}\cdots t_{\gamma _i^{k-1}}t_{\gamma _{i+1}^{k-1}}t_{\gamma _i^{k}}$\quad for odd $1\leq i\leq 2n-1$,
\item $\widetilde{h}_i=t_{\gamma _i^{k}}t_{\gamma _{i+1}^{k}}t_{\gamma _i^{k-1}}t_{\gamma _{i+1}^{k-1}}\cdots t_{\gamma _i^{2}}t_{\gamma _{i+1}^{2}}t_{\gamma _i^{1}}$\quad for even $2\leq i\leq 2n-2$
\end{enumerate}
hold relative to $\widetilde{D}$ for some lift $\widetilde{h}_i$ of $h_i$ with respect to $p$, and the relation
\[
\widetilde{h}_{2n}=t_{\gamma _{2n}^{k}}t_{\gamma _{2n+1}^{k}}t_{\gamma _{2n}^{k-1}}t_{\gamma _{2n+1}^{k-1}}\cdots t_{\gamma _{2n}^{2}}t_{\gamma _{2n+1}^{2}}t_{\gamma _{2n}^{1}}
\] 
holds for some lift $\widetilde{h}_{2n}$ of $h_{2n}$ with respect to $p$.
\end{lem}

\begin{proof}
Since the isotopy class of a homeomorphism on $\Sigma _0$ relative to $\B $ (resp. $\B \cup D$) is determined by the isotopy class of the image of $L$ relative to $\B $ (resp. $\B \cup D$), we have $h_{2i-1}=t_{2i-1, 2i}a_i$ for $1\leq i\leq n$ and $h_{2i}=t_{2i, 2i+1}b_i$ for $1\leq i\leq n-1$ in $\LMp $ and $\LMb $, and $h_{2n}=t_{2n, 2n+1}b_n$ in $\LM $ by Figure~\ref{fig_proof_lift-h_i}. 
Thus the product $\widetilde{t}_{2i-1, 2i}\widetilde{a}_i$ (resp. $\widetilde{t}_{2i, 2i+1}\widetilde{b}_i$) is a lift of $h_{2i-1}$ (resp. $h_{2i}$) with respect to $p$. 
Since the products $t_{\alpha _i^1}t_{\alpha _i^2}\cdots t_{\alpha _i^{k-1}}$, $t_{\beta _i^{k-1}}t_{\beta _i^{k-2}}\cdots t_{\beta _i^1}$ for $1\leq i\leq n$ for $1\leq i\leq n$, and $t_{\gamma _j^1}t_{\gamma _j^2}\cdots t_{\gamma _j^{k}}$ for $1\leq j\leq 2n+1$ are lifts of $a_i$, $b_i$, and $t_{j,j+1}$ by Lemmas~\ref{lift-a_i} and \ref{lift-t_{i,i+1}}, respectively, we denote $\widetilde{a}_i=t_{\alpha _i^1}t_{\alpha _i^2}\cdots t_{\alpha _i^{k-1}}$, $\widetilde{b}_i=t_{\beta _i^{k-1}}t_{\beta _i^{k-2}}\cdots t_{\beta _i^1}$ for $1\leq i\leq n$, and $\widetilde{t}_{i,i+1}=t_{\gamma _i^1}t_{\gamma _i^2}\cdots t_{\gamma _i^{k}}$ for $1\leq i\leq 2n+1$. 
For $1\leq i\leq n$, $t_{\alpha _{i}^l}$ commutes with $t_{\gamma _{2i-1}^{l^\prime }}$ for $1\leq l\leq k-2$ and $l+2\leq l^\prime \leq k$, and the relation $t_{\gamma _{2i-1}^{l+1}}t_{\alpha _{i}^l}=t_{\gamma _{2i}^{l}}t_{\gamma _{2i-1}^{l+1}}$ holds for $1\leq l\leq k-1$. 
Hence, for $1\leq i\leq n$, we have
\begin{eqnarray*}
\widetilde{h}_{2i-1}&=&\widetilde{t}_{2i-1, 2i}\widetilde{a}_i\\
&=&t_{\gamma _{2i-1}^1}t_{\gamma _{2i-1}^2}\cdots t_{\gamma _{2i-1}^{k}}\cdot \underset{\leftarrow }{\underline{t_{\alpha _i^1}}}\cdots \underset{\leftarrow }{\underline{t_{\alpha _i^{k-2}}}}t_{\alpha _i^{k-1}}\\
&\overset{\text{COMM}}{\underset{}{=}}&t_{\gamma _{2i-1}^1}\underline{t_{\gamma _{2i-1}^2}t_{\alpha _i^1}}\ \underline{t_{\gamma _{2i-1}^3}t_{\alpha _i^2}}\cdots \underline{t_{\gamma _{2i-1}^{k}}t_{\alpha _i^{k-1}}}\\
&\overset{\text{CONJ}}{\underset{}{=}}&t_{\gamma _{2i-1}^1}t_{\gamma _{2i}^{1}}t_{\gamma _{2i-1}^2}t_{\gamma _{2i}^{2}}t_{\gamma _{2i-1}^3}\cdots t_{\gamma _{2i}^{k-1}}t_{\gamma _{2i-1}^{k}}.
\end{eqnarray*}
Thus we have completed the proof of Lemma~\ref{lift-h_i} for odd $1\leq i\leq 2n-1$. 

Similarly, for $1\leq i\leq n$, $t_{\beta _{i}^l}$ commutes with $t_{\gamma _{2i}^{l^\prime }}$ for $2\leq l\leq k-1$ and $1\leq l^\prime \leq l-1$, and the relation $t_{\gamma _{2i}^{l}}t_{\beta _{i}^l}=t_{\gamma _{2i+1}^{l+1}}t_{\gamma _{2i}^{l}}$ holds for $1\leq l\leq k-1$. 
Hence, for $1\leq i\leq n$, we have
\begin{eqnarray*}
\widetilde{h}_{2i}&=&\widetilde{t}_{2i, 2i+1}\widetilde{b}_i\\
&=&t_{\gamma _{2i}^{k}}t_{\gamma _{2i}^{k-1}}\cdots t_{\gamma _{2i}^1}\cdot \underset{\leftarrow }{\underline{t_{\beta _i^{k-1}}}}\cdots \underset{\leftarrow }{\underline{t_{\beta _i^2}}}t_{\beta _i^1}\\
&\overset{\text{COMM}}{\underset{}{=}}&t_{\gamma _{2i}^{k}}\underline{t_{\gamma _{2i}^{k-1}}t_{\beta _i^{k-1}}}\cdots \underline{t_{\gamma _{2i}^2}t_{\beta _i^2}}\ \underline{t_{\gamma _{2i}^1}t_{\beta _i^1}}\\
&\overset{\text{CONJ}}{\underset{}{=}}&t_{\gamma _{2i}^{k}}t_{\gamma _{2i+1}^{k}}t_{\gamma _{2i}^{k-1}}\cdots t_{\gamma _{2i+1}^3}t_{\gamma _{2i}^2}t_{\gamma _{2i+1}^2}t_{\gamma _{2i}^1}. 
\end{eqnarray*}
Therefore we have completed the proof of Lemma~\ref{lift-h_i}. 
\end{proof}

\begin{figure}[h]
\includegraphics[scale=0.95]{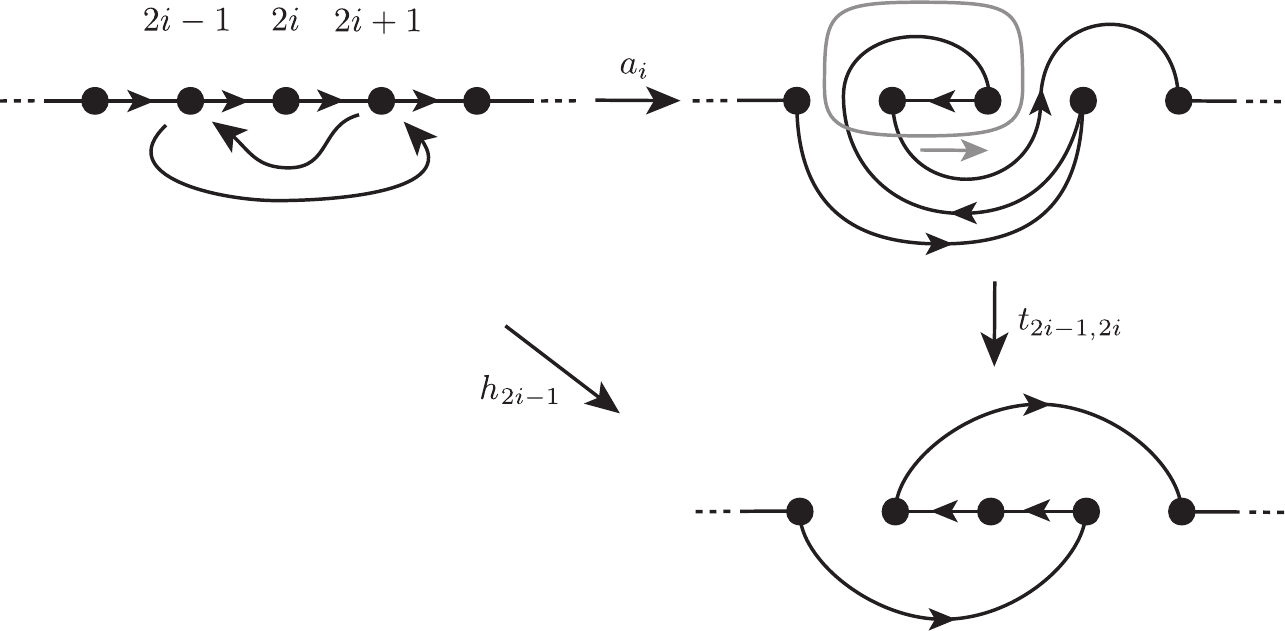}
\caption{The relation $h_{2i-1}=t_{2i-1, 2i}a_i$ for $1\leq i\leq n$ in $\LM $. Similarly, we have the relation $h_{2i}=t_{2i, 2i+1}b_i$ for $1\leq i\leq n$ in $\LM $. }\label{fig_proof_lift-h_i}
\end{figure}

By Corollaries~\ref{cor_lift_half-twist} and \ref{cor_lift-t_{i,i+1}}, and the fact that $\widetilde{h}_{2i-1}=\widetilde{t}_{2i-1, 2i}\widetilde{a}_i$ and $\widetilde{h}_{2i}=\widetilde{t}_{2i, 2i+1}\widetilde{b}_i$ for $1\leq i\leq n$ in the proof of Lemma~\ref{lift-h_i}, we have the following corollary. 

\begin{cor}\label{cor_lift-h_i}
For $1\leq i\leq 2n-1$, $\widetilde{h}_{i}$ commutes with $\zeta $ (resp. $\zeta ^\prime $) in $\SMp $ (resp. $\SMb $). 
$\widetilde{h}_{2n}$ commutes with $\zeta $ in $\SM $. 
\end{cor}

Let $\widetilde{r}$ be a homeomorphism on $\Sigma _g$ which satisfies $\widetilde{r}(\widetilde{l}_{2i-1}^l)=(\widetilde{l}_{2n-2i+3}^{k-l+2})^{-1}$ for $1\leq i\leq n+1$ and $2\leq l\leq k$ and $\widetilde{r}(\widetilde{l}_{2i}^{l})=(\widetilde{l}_{2n-2i+2}^{k-l+1})^{-1}$ for $1\leq i\leq n$ and $1\leq l\leq k$. 
Remark that $\widetilde{r}$ is uniquely determined up to isotopy and is described as the result of the $\pi $-rotation of $\Sigma _g$ as in Figure~\ref{fig_lift_r}. 
The rotation axis of $\widetilde{r}$ intersects with $\Sigma _g$ at two points in $S(\frac{n+1}{2})$ for odd $n$, at two points in the handle intersecting with $\widetilde{l}_{n+1}^{\frac{k+1}{2}}$ for even $n$ and odd $k$, and at no points for even $n$ and $k$. 
In actuality, the arc $\widetilde{l}_{n+1}^{\frac{k+1}{2}}$ in the case of even $n$ and odd $k$ must intersect with the rotation axis of $\widetilde{r}$, however, we describe $\widetilde{l}_{n+1}^{\frac{k+1}{2}}$ in Figure~\ref{fig_lift_r} as a parallel arc of $\widetilde{l}_{n+1}^{\frac{k+1}{2}}$ to see well. 
Since $r(l_i)=l_{2n-i+2}^{-1}$ for $1\leq i\leq 2n+1$, $\widetilde{r}$ is a lift of $r$ with respect to $p$. 
Since we have 
\begin{align*}
\widetilde{r}\zeta \widetilde{r}^{-1}(\widetilde{l}_{2i-1}^{l})&=\widetilde{r}\zeta \bigl( (\widetilde{l}_{2n-2i+3}^{k-l+1})^{-1}\bigr) =\widetilde{r}((\widetilde{l}_{2n-2i+3}^{k-l+2})^{-1})=\widetilde{l}_{2i-1}^{l-1}
\end{align*}
for $1\leq i\leq k$ mod $k$ and $1\leq i\leq 2n+1$, the following lemma holds. 

\begin{figure}[h]
\includegraphics[scale=1.1]{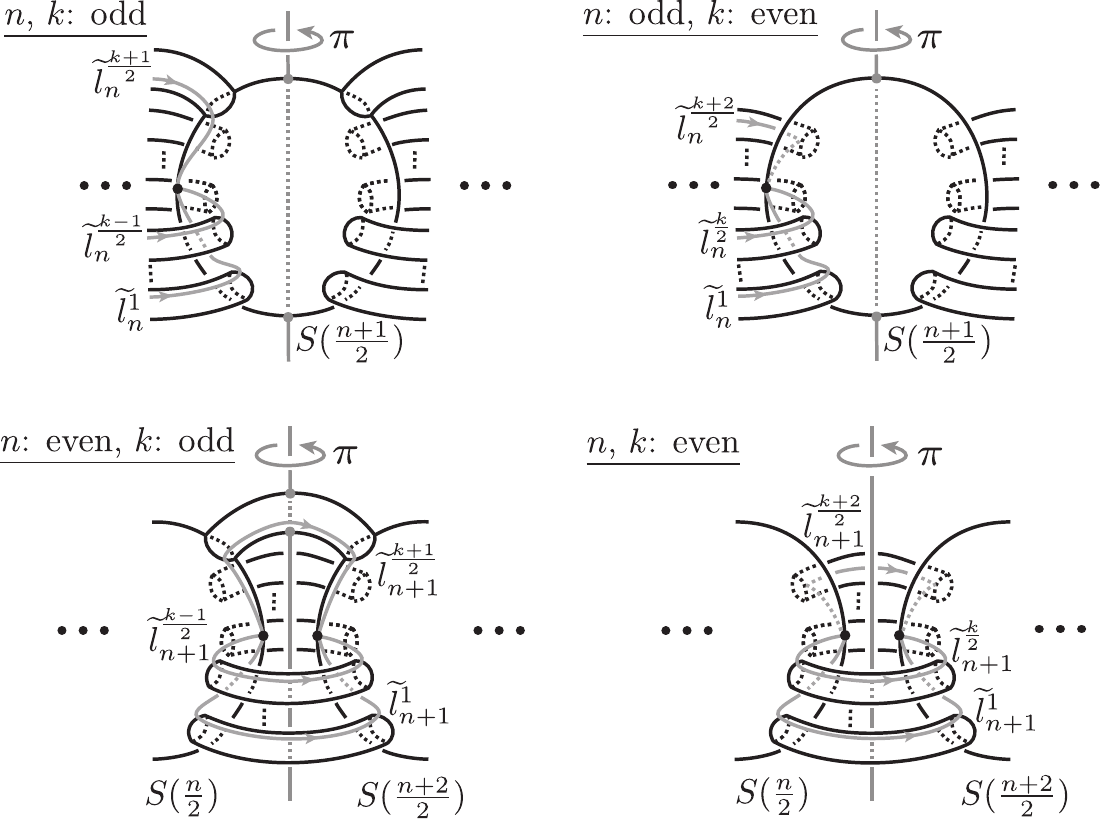}
\caption{A lift $\widetilde{r}$ of $r$ with respect to $p$.}\label{fig_lift_r}
\end{figure}

\begin{lem}\label{lift-r}
The relation $\widetilde{r}\zeta \widetilde{r}^{-1}=\zeta ^{-1}$ holds in $\SM $. 
\end{lem}

Let $\delta _i^l$ for $1\leq i\leq n$ and $1\leq l\leq k$ be a non-separating simple closed curve on $\Sigma _g$ which transversely intersects with $\widetilde{l}_{2i}^l$ at one point and does not intersect with other $\widetilde{l}_{j}^{l^\prime }$, and let $\delta _i$ for $2\leq i\leq n+1$ be a separating simple closed curve on $\Sigma _g$ which transversely intersects with $\widetilde{l}_{2i-1}^l$ for each $1\leq l\leq k$ at one point and does not intersect with other $\widetilde{l}_{j}^{l^\prime }$ as in Figure~\ref{fig_lift_t_1j}. 
Remark that $\delta _1^l=\gamma _1^l$ and $\delta _n^l=\gamma _{2n+1}^l$ for $1\leq l\leq k$, and $\delta _{n+1}=\partial \widetilde{D}$. 
Denote by $\zeta _i=\zeta _{g,k;i}$ for $2\leq i\leq n+1$ the self-homeomorphism on $\Sigma _g$ which is described as the result of $-(\frac{2\pi }{k})$-rotation of $\Sigma _g$ fixing the subsurface that is cut off by $\delta _{i}$ and includes $\widetilde{D}$ as in Figure~\ref{fig_lift_t_1j}. 
We also remark that $\zeta _{n+1}=\zeta ^\prime $ in $\SMb $ and $\zeta _{n+1}=\zeta $ in $\SM $. 
Since $\delta _i^l$ for $1\leq i\leq n$ and $1\leq l\leq k$ is a lift of $\gamma _{1,2i}$ with respect to $p$ and the isotopy class of a homeomorphism on $\Sigma _g $ is determined by the isotopy class of the image of $\widetilde{L}$, we have the following lemma. 

\begin{figure}[h]
\includegraphics[scale=1.28]{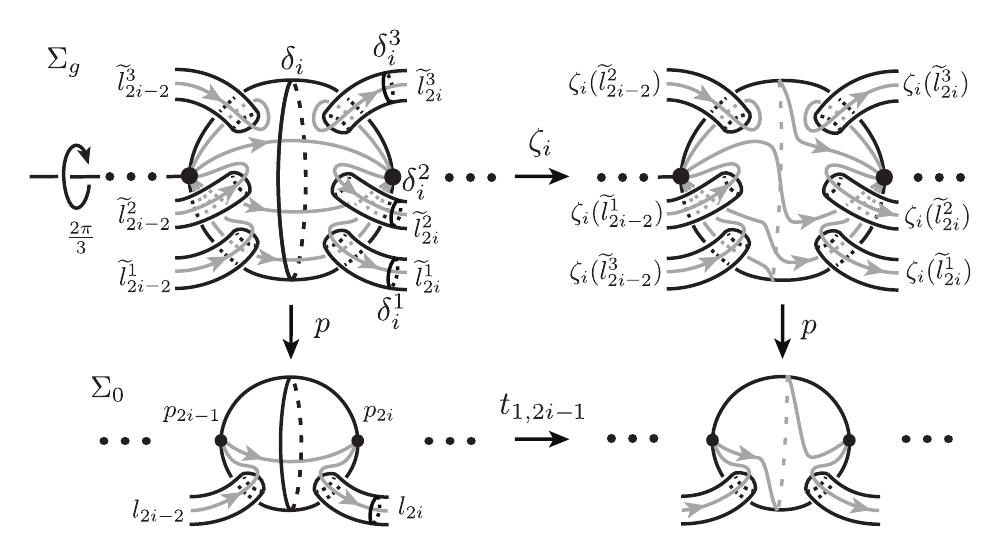}
\caption{Simple closed curves $\delta _i^l$ for $1\leq i\leq n$ and $1\leq l\leq k$ and $\delta _i$ for $2\leq i\leq n+1$ on $\Sigma _g$, and the homeomorphism $\zeta _i$ for $2\leq i\leq n+1$ on $\Sigma _g$ when $k=3$.}\label{fig_lift_t_1j}
\end{figure}

\begin{lem}\label{lem_lift_t{1,j}}
The relations 
\begin{enumerate}
\item $\widetilde{t}_{1,2i}=t_{\delta _i^{1}}t_{\delta _i^{2}}\cdots t_{\delta _i^{k}}$\quad for $1\leq i\leq n$, 
\item $\widetilde{t}_{1,2i-1}=\zeta _i$\quad for $2\leq i\leq n+1$
\end{enumerate}
hold relative to $\widetilde{D}$ for some lift $\widetilde{t}_{1,j}$ of $t_{1,j}$ $(1\leq j\leq 2n+1)$ with respect to $p$. 
\end{lem}

\subsection{Main theorems for the balanced superelliptic mapping class groups}\label{section_main_thm_smod}
We define $\widetilde{h}_i=t_{\gamma _i^1}t_{\gamma _{i+1}^1}t_{\gamma _i^2}t_{\gamma _{i+1}^2}\cdots t_{\gamma _i^{k-1}}t_{\gamma _{i+1}^{k-1}}t_{\gamma _i^{k}}$\quad for odd $1\leq i\leq 2n-1$, $\widetilde{h}_i=t_{\gamma _i^{k}}t_{\gamma _{i+1}^{k}}t_{\gamma _i^{k-1}}t_{\gamma _{i+1}^{k-1}}\cdots t_{\gamma _i^{2}}t_{\gamma _{i+1}^{2}}t_{\gamma _i^{1}}$\quad for even $2\leq i\leq 2n-2$, and $\widetilde{t}_{i,i+1}=t_{\gamma _i^1}t_{\gamma _i^2}\cdots t_{\gamma _i^{k}}$ for $1\leq i\leq 2n+1$ (see also Figure~\ref{fig_scc_c_il}). 
Main theorems for the balanced superelliptic mapping class groups are as follows. 

\begin{thm}\label{thm_pres_smodb}
For $n\geq 1$ and $k\geq 3$ with $g=n(k-1)$, $\SMb $ admits the presentation with generators $\widetilde{h}_i$ for $1\leq i\leq 2n-1$ and $\widetilde{t}_{i,j}$ for $1\leq i<j\leq 2n+1$, and the following defining relations: 
\begin{enumerate}
\item commutative relations
\begin{enumerate}
\item $\widetilde{h}_i \rightleftarrows \widetilde{h}_j$ \quad for $j-i>3$,
\item $\widetilde{t}_{i,j} \rightleftarrows \widetilde{t}_{k,l}$ \quad for $j<k$, $k\leq i<j\leq l$, or $l<i$, 
\item $\widetilde{h}_k \rightleftarrows \widetilde{t}_{i,j}$ \quad for  $k+2<i$, $i\leq k<k+2\leq j$, or $j<k$,
\end{enumerate}
\item relations among $\widetilde{h}_i$'s and $\widetilde{t}_{j,j+1}$'s: 
\begin{enumerate}
\item $\widetilde{h}_i ^{\pm 1}\widetilde{t}_{i,i+1}=\widetilde{t}_{i+1,i+2}\widetilde{h}_i^{\pm 1}$ \quad for $1\leq i\leq 2n-2$, 
\item $\widetilde{h}_i ^{\varepsilon }\widetilde{h}_{i+1}^{\varepsilon }\widetilde{t}_{i,i+1}=\widetilde{t}_{i+2,i+3}\widetilde{h}_i^{\varepsilon }\widetilde{h}_{i+1}^{\varepsilon }$ \quad for $1\leq i\leq 2n-2$ and $\varepsilon \in \{ 1, -1\}$,
\item $\widetilde{h}_{i}^{\varepsilon }\widetilde{h}_{i+1}^{\varepsilon }\widetilde{h}_{i+2}^{\varepsilon }\widetilde{h}_{i}^{\varepsilon }=\widetilde{h}_{i+2}^{\varepsilon }\widetilde{h}_{i}^{\varepsilon }\widetilde{h}_{i+1}^{\varepsilon }\widetilde{h}_{i+2}^{\varepsilon }$ \quad for $1\leq i\leq 2n-2$ and $\varepsilon \in \{ 1, -1\}$,
\item $\widetilde{h}_i\widetilde{h}_{i+1}\widetilde{t}_{i,i+1}=\widetilde{t}_{i,i+1}\widetilde{h}_{i+1}\widetilde{h}_i$ \quad for $1\leq i\leq 2n-2$,
\item $\widetilde{h}_i\widetilde{h}_{i+2}\widetilde{h}_i\widetilde{t}_{i+1,i+2}^{-1}=\widetilde{t}_{i+2,i+3}^{-1}\widetilde{h}_{i+2}\widetilde{h}_i\widetilde{h}_{i+2}$ \quad for $1\leq i\leq 2n-3$,
\end{enumerate}
\item lifts of the pentagonal relations $(P_{i,j,k,l,m})$ for $1\leq i<j<k<l<m\leq 2n+2$,\\
i.e. $\widetilde{t}_{j,m-1}^{-1}\widetilde{t}_{k,m-1}\widetilde{t}_{j,l-1}\widetilde{t}_{i,k-1}\widetilde{t}_{i,l-1}^{-1}=\widetilde{t}_{i,l-1}^{-1}\widetilde{t}_{i,k-1}\widetilde{t}_{j,l-1}\widetilde{t}_{k,m-1}\widetilde{t}_{j,m-1}^{-1}$, 
\item 
\begin{enumerate}
\item lifts of the relations $(T_{i,j})$, i.e.\\
$\widetilde{t}_{i,j}=\left\{
		\begin{array}{ll}
		\widetilde{h}_i^2 \quad \text{for }j-i=2,\\
		\widetilde{t}_{j-1,j}^{-\frac{j-i-3}{2}}\cdots \widetilde{t}_{i+2,i+3}^{-\frac{j-i-3}{2}}\widetilde{t}_{i,i+1}^{-\frac{j-i-3}{2}}(\widetilde{h}_{j-2}\cdots \widetilde{h}_{i+1}\widetilde{h}_{i})^{\frac{j-i+1}{2}}\\ \text{for }2n-1\geq j-i \geq 3\text{ is odd},\\
		\widetilde{t}_{j-2,j-1}^{-\frac{j-i-2}{2}}\cdots \widetilde{t}_{i+2,i+3}^{-\frac{j-i-2}{2}}\widetilde{t}_{i,i+1}^{-\frac{j-i-2}{2}}\widetilde{h}_{j-2}\cdots \widetilde{h}_{i+2}\widetilde{h}_{i}\widetilde{h}_{i}\widetilde{h}_{i+2}\cdots \widetilde{h}_{j-2}\\ \cdot (\widetilde{h}_{j-3}\cdots \widetilde{h}_{i+1}\widetilde{h}_{i})^{\frac{j-i}{2}}\quad \text{for }2n\geq j-i \geq 4\text{ is even}.	
		\end{array}
		\right.$
\end{enumerate}
\end{enumerate}
\end{thm}

Since $\theta ^1\colon \SMb \to \LMb $ is an isomorphism (the definitions is reviewed in Section~\ref{section_exact-seq_smod}), $\theta ^1(\widetilde{h}_i)=h_i$ for $1\leq i\leq 2n-1$ by Lemma~\ref{lift-t_{i,i+1}}, and $\theta ^1(\widetilde{t}_{i,i+1})=t_{i,i+1}$ for $1\leq i\leq 2n+1$ by Lemma~\ref{lift-h_i}, Theorem~\ref{thm_pres_smodb} is immediately obtained from Theorem~\ref{thm_pres_lmodb}. 

\begin{thm}\label{thm_pres_smodp}
For $n\geq 1$ and $k\geq 3$ with $g=n(k-1)$, $\SMp $ admits the presentation which is obtained from the finite presentation for $\SMb $ in Theorem~\ref{thm_pres_smodb} by adding the relation 
\begin{enumerate}
\item[(4)\ (b)] $\widetilde{t}_{1,2n+1}^k=1$. 
\end{enumerate}
\end{thm}

We remark that $\zeta =\widetilde{t}_{1,2n+1}$ by an argument in the proof of Proposition~\ref{prop_exact_lmodb} (see also the before the proof of Theorem~\ref{thm_pres_smodp}). 
We define $\widetilde{h}_{2n}=t_{\gamma _{2n}^{k}}t_{\gamma _{2n+1}^{k}}t_{\gamma _{2n}^{k-1}}t_{\gamma _{2n+1}^{k-1}}\cdots t_{\gamma _{2n}^{2}}t_{\gamma _{2n+1}^{2}}t_{\gamma _{2n}^{1}}$. 
By Lemma~\ref{lift-h_i} and an argument before Lemma~\ref{lift-r}, $\widetilde{h}_{2n}$ and $\widetilde{r}$ are preimeges of $h_{2n}$ and $r$ with respect to $\theta \colon \SM \to \LM$, respectively. 

\begin{thm}\label{thm_pres_smod}
For $n\geq 1$ and $k\geq 3$ with $g=n(k-1)$, $\SM $ admits the presentation which is obtained from the finite presentation for $\SMp $ in Theorem~\ref{thm_pres_smodp} by adding generators $\widetilde{h}_{2n}$ and $\widetilde{r}$, and the following relations:
\begin{enumerate}
\item[(1)] commutative relation
\begin{enumerate}
\item[(d)] $\widetilde{h}_{2n} \rightleftarrows \widetilde{t}_{1,2n+1}$, 
\end{enumerate}
\item[(4)]
\begin{enumerate}
\item[(c)] $\widetilde{t}_{1,2n}^{-n+1}\widetilde{t}_{2n-1,2n}^{-n+1}\cdots \widetilde{t}_{3,4}^{-n+1}\widetilde{t}_{1,2}^{-n+1}(\widetilde{h}_{2n}\cdots \widetilde{h}_{2}\widetilde{h}_{1})^{n+1}=1$, 
\end{enumerate}
\item[(5)] relations among $\widetilde{r}$ and $\widetilde{h}_i$'s or $\widetilde{t}_{j,j+1}$'s: 
\begin{enumerate}
\item $\widetilde{r}^2=1$, 
\item $\widetilde{r}\widetilde{h}_i=\widetilde{h}_{2n-i+1}\widetilde{r}$ \quad for $1\leq i\leq 2n$, 
\item $\widetilde{r}\widetilde{t}_{i,i+1}=\widetilde{t}_{2n-i+2,2n-i+3}\widetilde{r}$ \quad for $2\leq i\leq 2n$, 
\item $\widetilde{r}\widetilde{t}_{1,j}=\widetilde{t}_{1,2n-j+2}\widetilde{r}$ \quad for even $2\leq j\leq 2n$, 
\item $\widetilde{r}\widetilde{t}_{1,j}\widetilde{r}^{-1}\widetilde{t}_{1,2n-j+2}^{-1}=\widetilde{t}_{1,2n+1}^{-1}$ \quad for odd $3\leq j\leq 2n-1$,  
\item $\widetilde{r}\widetilde{t}_{1,2n+1}=\widetilde{t}_{1,2n+1}^{-1}\widetilde{r}$. 
\end{enumerate}
\end{enumerate}
\end{thm}

As a corollary of Theorems~\ref{thm_pres_smodb}, \ref{thm_pres_smodp}, and \ref{thm_pres_smod}, we have the following corollary. 

\begin{cor}\label{cor_gen_smod}
For $n\geq 1$ and $k\geq 2$ with $g=n(k-1)$,  
\begin{enumerate}
\item $\SMb $ and $\SMp$ are generated by $\widetilde{h}_1,\ \widetilde{h}_2,\ \dots ,\ \widetilde{h}_{2n-1}$, and $\widetilde{t}_{1,2}$, 
\item $\SM $ is generated by $\widetilde{h}_1,\ \widetilde{h}_3,\ \dots ,\ \widetilde{h}_{2n-1},\ \widetilde{t}_{1,2}$, and $\widetilde{r}$. 
\end{enumerate}
\end{cor}

\begin{proof}
We proceed by a similar argument in the proof of Corollary~\ref{cor_gen_lmod}. 
By Theorem~\ref{thm_pres_smodb}, $\LMb$ is generated by $\widetilde{h}_i$ for $1\leq i\leq 2n-1$ and $\widetilde{t}_{i,j}$ for $1\leq i<j\leq 2n+1$. 
By the relation~(4)~(a) in Theorem~\ref{thm_pres_smodb}, $\widetilde{t}_{i,j}$ for $1\leq i<j\leq 2n+1$ is a product of $\widetilde{h}_i$ $(1\leq i\leq 2n-1)$ and $\widetilde{t}_{i,i+1}$ $(1\leq i\leq 2n)$. 
By the relation~(2)~(a) in Theorem~\ref{thm_pres_smodb}, we show that $\widetilde{t}_{i,i+1}$ is a product of $\widetilde{h}_1,\ \widetilde{h}_2,\ \dots ,\ \widetilde{h}_{2n-1}$, and $\widetilde{t}_{1,2}$. 
By Lemma~\ref{surj_capping_smod}, we have $\widetilde{\iota }_\ast (\SMb )=\SMp $. 
Therefore, $\SMb $ and $\SMp $ are generated by $\widetilde{h}_1,\ \widetilde{h}_2,\ \dots ,\ \widetilde{h}_{2n-1}$, and $\widetilde{t}_{1,2}$. 

By Theorem~\ref{thm_pres_smod}, $\SM $ is generated by $\widetilde{h}_i$ for $1\leq i\leq 2n$, $\widetilde{t}_{i,j}$ for $1\leq i<j\leq 2n+1$, and $\widetilde{r}$. 
As a similar argument in the case~(1), every $\widetilde{t}_{i,j}$ is a product of $\widetilde{h}_i$ $(1\leq i\leq 2n-1)$ and $\widetilde{t}_{1,2}$. 
By the relation~(2)~(d) in Theorem~\ref{thm_pres_smod}, we show that $\widetilde{h}_2,\ \widetilde{h}_4,\ \dots ,\ \widetilde{h}_{2n}$ are products of $\widetilde{h}_1,\ \widetilde{h}_3,\ \dots ,\ \widetilde{h}_{2n-1},\ \widetilde{t}_{1,2}$, and $\widetilde{r}$. 
Therefore, $\SM $ is generated by $\widetilde{h}_1,\ \widetilde{h}_3,\ \dots ,\ \widetilde{h}_{2n-1},\ \widetilde{t}_{1,2}$, and $\widetilde{r}$, and we have completed Corollary~\ref{cor_gen_smod}. 
\end{proof}

Recall that $g=n(k-1)$ for $n\geq 1$ and $k\geq 3$. 
We remark that $t_{1,2n+1}=t_{\partial D}$ in $\LMb $, $\zeta ^\prime $ is the lift of $t_{1,2n+1}$ with respect to $p\colon \Sigma _g^1\to \Sigma _0^1$ (i.e. $\widetilde{t}_{1,2n+1}=\zeta ^\prime $), and ${(\zeta ^\prime )}^k=t_{\partial \widetilde{D}}$ in $\SMb $ (see the proof of Proposition~\ref{prop_exact_lmodb}).  

\begin{proof}[Proof of Theorem~\ref{thm_pres_smodp}]
By the remark above, the exact sequence~(\ref{exact_smodb}) in Proposition~\ref{prop_exact_smodb} induces the following exact sequence: 
\begin{eqnarray*}
1\longrightarrow \left< \widetilde{t}_{1,2n+1}^k\right>  \longrightarrow \SMb \stackrel{\widetilde{\iota }_\ast }{\longrightarrow }\SMp \longrightarrow 1. 
\end{eqnarray*}
Clearly, we have $\widetilde{\iota }_\ast (\widetilde{h}_i)=\widetilde{h}_i\in \SMp $ for $1\leq i\leq 2n-1$ and $\widetilde{\iota }_\ast (\widetilde{t}_{i,i+1})=\widetilde{t}_{i,i+1}\in \SMp $ for $1\leq i\leq 2n$. 
Put $\widetilde{t}_{i,j}=\widetilde{\iota }_\ast (\widetilde{t}_{i,j})\in \SMp $ for $2n\geq j-i\geq 2$. 
Then $\widetilde{\iota }_\ast $ maps the relations~(1)--(3) and (4) (a) in Theorem~\ref{thm_pres_smodb} to the relations~(1)--(3) and (4) (a) in Theorem~\ref{thm_pres_smodp}, respectively. 
Thus, a presentation for $\SMp $ is obtained from the presentation for $\SMb $ in Theorem~\ref{thm_pres_smodb} by adding the relation $\widetilde{t}_{1,2n+1}^k=1$, and this presentation coincides with the presentation in Theorem~\ref{thm_pres_smodp}. 
Therefore, we have completed the proof of Theorem~\ref{thm_pres_smodp}. 
\end{proof}

To prove Theorem~\ref{thm_pres_smod}, we prepare some lemmas. 
Since $\widetilde{r}(\gamma _{i}^l)=\gamma _{2n-i+2}^{k-l+1}$ for $1\leq l\leq k$ and odd $1\leq i\leq 2n+1$, $\widetilde{r}(\gamma _{i}^l)=\gamma _{2n-i+2}^{k-l}$ for $1\leq l\leq k-1$ and even $2\leq i\leq 2n$, and $\widetilde{r}(\gamma _{i}^k)=\gamma _{2n-i+2}^{k}$ for even $2\leq i\leq 2n$, by using conjugation relations, we have the following lemma.  

\begin{lem}\label{lem_conj_r-h_smod}
For $1\leq i\leq 2n$, the relation $\widetilde{r}\widetilde{h}_i\widetilde{r}^{-1}=\widetilde{h}_{2n-i+1}$ holds in $\SM $. 
\end{lem}

Since $\widetilde{r}(\gamma _{i}^1\cup \gamma _{i}^2\cup \cdots \cup \gamma _{i}^k)=\gamma _{2n-i+2}^1\cup \gamma _{2n-i+2}^2\cup \cdots \cup \gamma _{2n-i+2}^k$, we have the following lemma.  

\begin{lem}\label{lem_conj_r-t_smod}
For $1\leq i\leq 2n+1$, the relation $\widetilde{r}\widetilde{t}_{i,i+1}\widetilde{r}^{-1}=\widetilde{t}_{2n-i+2,2n-i+3}$ holds in $\SM $. 
\end{lem}

The conjugation $\widetilde{r}\zeta _i\widetilde{r}^{-1}$ for $2\leq i\leq n$ is the self-homeomorphism on $\Sigma _g$ which is described as the result of $\frac{2\pi }{k}$-rotation of $\Sigma _g$ fixing the subsurface that is cut off by $\delta _{n-i+2}$ and does not include $\widetilde{D}$. 
Thus the product $\zeta _{n-i+2}\widetilde{r}\zeta _i^{-1}\widetilde{r}^{-1}$ coincides with $\zeta $ in $\SM $. 
Since we can also see that $\widetilde{r}(\delta _{i}^1\cup \delta _{i}^2\cup \cdots \cup \delta _{i}^k)=\delta _{n-i+1}^1\cup \delta _{n-i+1}^2\cup \cdots \cup \delta _{n-i+1}^k$, we have the following lemma. 

\begin{lem}\label{lem_conj_t{1,j}-r}
The relations 
\begin{enumerate}
\item $\widetilde{r}\widetilde{t}_{1,2i}\widetilde{r}^{-1}=\widetilde{t}_{1,2n-2i+2}$\quad for $1\leq i\leq n$, 
\item $\zeta _{n-i+2}\widetilde{r}\zeta _i^{-1}\widetilde{r}^{-1}=\zeta $\quad for $2\leq i\leq n$
\end{enumerate}
hold in $\SM $. 
\end{lem}

By Lemma~\ref{lem_t_ij-braid}, we have $t_{1,2n}^{-n+1}t_{2n-1,2n}^{-n+1}\cdots t_{3,4}^{-n+1}t_{1,2}^{-n+1}(h_{2n}\cdots h_{2}h_{1})^{n+1}=t_{1,2n+2}$ in $\LM $. 
Since $\gamma _{1,2n+2}$ does not intersect with $L=l_1\cup l_2\cup \cdots \cup l_{2n+1}$, $\gamma _{1,2n+2}$ lifts a simple closed curve on $\Sigma _g$ which bounds a disk in $\Sigma _g$. 
Thus we have the following lemma. 

\begin{lem}\label{lem_lift_t{1,2n+2}}
The relation 
\[
\widetilde{t}_{1,2n}^{-n+1}\widetilde{t}_{2n-1,2n}^{-n+1}\cdots \widetilde{t}_{3,4}^{-n+1}\widetilde{t}_{1,2}^{-n+1}(\widetilde{h}_{2n}\cdots \widetilde{h}_{2}\widetilde{h}_{1})^{n+1}=1
\]
holds in $\SM $. 
\end{lem}

\begin{proof}[Proof of Theorem~\ref{thm_pres_smod}]
From the Birman-Hilden correspondence~\cite{Birman-Hilden2}, we have the exact sequence
\begin{eqnarray*}
1\longrightarrow \left< \zeta \right>  \longrightarrow \SM \stackrel{\theta }{\longrightarrow }\LM \longrightarrow 1. 
\end{eqnarray*}
We will apply Lemma~\ref{presentation_exact} to the exact sequence above and the presentation for $\LM $ in Theorem~\ref{thm_pres_lmod}. 

First, by Lemmas~\ref{lift-t_{i,i+1}} and \ref{lift-h_i}, we have $\theta (\widetilde{h}_i)=h_i$ for $1\leq i\leq 2n$, $\theta (\widetilde{t}_{i,i+1})=t_{i,i+1}$ for $1\leq i\leq 2n+1$, and $\theta (\widetilde{r})=r$. 
Since the forgetful homomorphism $\mathcal{F}\colon \SMp \hookrightarrow \SM $ is injective by Lemma~\ref{forget_smod_inj}, we regard $\SMp $ as a subgroup of $\SM $ and denote $\widetilde{t}_{i,j}=\mathcal{F}(\widetilde{t}_{i,j})\in \SM $ for $2n\geq j-i\geq 2$. 
Since $\widetilde{t}_{1,2i}=t_{\delta _i^{1}}t_{\delta _i^{2}}\cdots t_{\delta _i^{k}}$ for $1\leq i\leq n$ and $\widetilde{t}_{1,2i-1}=\zeta _i$ for $2\leq i\leq n+1$ in $\SMb $ by Lemma~\ref{lem_lift_t{1,j}} and the injectivity of $\theta ^1\colon \SMb \to \LMb $, we also have $\widetilde{t}_{1,2i}=t_{\delta _i^{1}}t_{\delta _i^{2}}\cdots t_{\delta _i^{k}}$ for $1\leq i\leq n$ and $\widetilde{t}_{1,2i-1}=\zeta _i$ for $2\leq i\leq n+1$ in $\SM $ (in particular,  $\widetilde{t}_{1,2n+1}=\zeta $ in $\SM $). 

By Theorem~\ref{thm_pres_smodp} and using $\mathcal{F}$, the relations~(1) (a)--(c), (2) (a)--(e), (3), (4) (a), and (b) in Theorem~\ref{thm_pres_smod} are hold in $\SM $. 
By Corollary~\ref{cor_lift-h_i} and Lemmas~\ref{lem_lift_t{1,2n+2}}, \ref{lem_conj_r-h_smod}, \ref{lem_conj_r-t_smod}, and \ref{lift-r}, the relations~(1) (d), (4) (c), (5) (b), (c), and (f) in Theorem~\ref{thm_pres_smod} hold in $\SM $. 
Since $\widetilde{r}$ is defined as a $\pi $-rotation of $\Sigma _g$, the relation~(5) (a) holds in $\SM $. 
By Lemma~\ref{lem_conj_t{1,j}-r}, the relations~(5) (d) and (5) (e) hold in $\SM $. 
By the argument above and applying Lemma~\ref{presentation_exact} to the exact sequence above and the presentation for $\LM $ in Theorem~\ref{thm_pres_lmod}, we have the presentation for $\SM $ whose generators are $\widetilde{h}_i$ for $1\leq i\leq 2n$, $\widetilde{t}_{i,j}$ for $1\leq i<j\leq 2n+1$, and $\widetilde{r}$, and the following defining relations:
\begin{enumerate}
\item[(A)] $\widetilde{t}_{1,2n+1}^k=1$, 
\item[(B)] the relations~(1) (a)--(c), (2) (a)--(e), (3), (4) (a)--(c), (5) (a)--(e)  in Theorem~\ref{thm_pres_smod}, 
\item[(C)]
\begin{enumerate}
\item $\widetilde{h}_{i}\widetilde{t}_{1,2n+1}\widetilde{h}_{i}^{-1}=\widetilde{t}_{1,2n+1}$ for $1\leq i\leq 2n$, 
\item $\widetilde{t}_{i,j}\widetilde{t}_{1,2n+1}\widetilde{t}_{i,j}^{-1}=\widetilde{t}_{1,2n+1}$ for $1\leq i<j\leq 2n+1$, 
\item $\widetilde{r}\widetilde{t}_{1,2n+1}\widetilde{r}^{-1}=\widetilde{t}_{1,2n+1}^{-1}$.
\end{enumerate}
\end{enumerate}
The generators of the presentation for $\SM $ above coincide with the generators of the presentation for $\SM $ in Theorem~\ref{thm_pres_smod}. 
The relation~(A) above coincides with the relation~(2) (e) in Theorem~\ref{thm_pres_smod}. 
The relations~(C) (a) and (b) above coincide with the relations~(1) (c) and (b) in Theorem~\ref{thm_pres_smod}, respectively. 
The relation~(C) (c) above coincides with the relation~(5) (f) in Theorem~\ref{thm_pres_smod}. 
Therefore, the presentation for $\SM $ above is equivalent to the presentation in Theorem~\ref{thm_pres_smod} and we have completed the proof of Theorem~\ref{thm_pres_smod}. 
\end{proof}

\subsection{The first homology groups of the balanced superelliptic mapping class groups}\label{section_abel-smod}

In this section, we will prove Theorem~\ref{thm_abel_smod}.  
For conveniences, we denote the equivalence class in $H_1(\SMb )$ (resp. $H_1(\SMp )$ and $H_1(\SM )$) of an element $h$ in $\SMb$ (resp. $\SMp $ and $\SM $) by $h$. 
Recall that $g=n(k-1)$ for $n\geq 1$ and $k\geq 3$. 
By an isomorphism $\theta ^1\colon \SMb \to \LMb $ and Theorem~\ref{thm_abel_lmod} for $\LMb $, we have 
\[
H_1(\SMb )\cong \left\{ \begin{array}{ll}
 \Z ^2&\text{if }  n=1,   \\
 \Z ^3&\text{if }  n\geq 2.
 \end{array} \right.
\]

\begin{proof}[Proof of Theorem~\ref{thm_abel_smod} for $\SMp $]
The exact sequence~(\ref{exact_smodb}) in Proposition~\ref{prop_exact_smodb} induces the exact sequence: 
\begin{eqnarray*}
1\longrightarrow \left< \widetilde{t}_{1,2n+1}^k\right>  \longrightarrow H_1(\SMb ) \stackrel{\widetilde{\iota }_\ast }{\longrightarrow }H_1(\SMp )\longrightarrow 1. 
\end{eqnarray*}
By Theorem~\ref{thm_abel_lmod} and its proof for $\LMp $, we have $H_1(\mathrm{LMod}_3^1)\cong \Z [\widetilde{h}_1]\oplus \Z [\widetilde{t}_{1,2}]$ and $H_1(\SMb )\cong \Z [\widetilde{h}_1]\oplus \Z [\widetilde{h}_2]\oplus \Z [\widetilde{t}_{1,2}]$ for $n\geq 2$. 
Since we have $\widetilde{t}_{1,3}=\widetilde{h}_1^2$ by the relation~(4) (a) in Theorem~\ref{thm_pres_smodb}, when $n=1$, as a presentation for an abelian group, we have
\[
H_1(\mathrm{SMod}_{k-1,\ast;k})\cong \left< \widetilde{h}_1, \widetilde{t}_{1,2}\middle| (\widetilde{h}_1^2)^k=1\right> \cong \Z [\widetilde{t}_{1,2}]\oplus \Z _{2k}[\widetilde{h}_1]. 
\]

Since we have 
\[
\widetilde{t}_{1,2n+1}=\widetilde{t}_{2n-1,2n}^{-n+1}\cdots \widetilde{t}_{3,4}^{-n+1}\widetilde{t}_{1,2}^{-n+1}\widetilde{h}_{2n-1}\cdots \widetilde{h}_3\widetilde{h}_1\widetilde{h}_1\widetilde{h}_3\cdots \widetilde{h}_{2n-1}(\widetilde{h}_{2n-2}\cdots \widetilde{h}_{2}\widetilde{h}_{1})^{n}
\]
by the relation~(4) (a) in Theorem~\ref{thm_pres_smodb}, the relation 
\[
\widetilde{t}_{1,2n+1}^k=(\widetilde{t}_{1,2}^{-n(n-1)}\widetilde{h}_1^{n(n+1)}\widetilde{h}_2^{n(n-1)})^k=(\widetilde{t}_{1,2}^{-n+1}\widetilde{h}_1^{n+1}\widetilde{h}_2^{n-1})^{kn}
\]
holds in $H_1(\LMp )$. 
Thus, when $n\geq 2$, as a presentation for an abelian group, we have
\begin{eqnarray*}
H_1(\SMp )
&\cong &\left< \widetilde{h}_1, \widetilde{h}_2, \widetilde{t}_{1,2} \middle| (\widetilde{t}_{1,2}^{-n+1}\widetilde{h}_1^{n+1}\widetilde{h}_2^{n-1})^{kn}=1 \right> \\
&\cong &\left< \widetilde{h}_1, \widetilde{h}_2, \widetilde{t}_{1,2}, X \middle| (X^{n-1}\widetilde{h}_1^{2})^{kn}=1, X=\widetilde{t}_{1,2}^{-1}\widetilde{h}_1\widetilde{h}_2 \right> \\
&\cong &\left< \widetilde{h}_1, \widetilde{h}_2, X \middle| (\widetilde{h}_1^{2}X^{n-1})^{kn}=1 \right> .
\end{eqnarray*}
By an argument similar to the proof of Theorem~\ref{thm_abel_lmod} for $\LMp $, we have
\[
H_1(\SMp )\cong \left\{ \begin{array}{ll}
 \Z [\widetilde{h}_2]\oplus \Z [X]\oplus \Z _{2kn}[Y_1]&\text{if }  n\geq 3\text{ is odd},   \\
 \Z [\widetilde{h}_2]\oplus \Z [Y_2]\oplus \Z _{kn}[Z]&\text{if }  n\geq 2\text{ is even},
 \end{array} \right.
\]
where $X=\widetilde{t}_{1,2}^{-1}\widetilde{h}_1\widetilde{h}_2$, $Y_1=\widetilde{t}_{1,2}^{-\frac{n-1}{2}}\widetilde{h}_1^{\frac{n+1}{2}}\widetilde{h}_2^{\frac{n-1}{2}}$, $Y_2=\widetilde{t}_{1,2}^{-\frac{n}{2}}\widetilde{h}_1^{\frac{n+2}{2}}\widetilde{h}_2^{\frac{n}{2}}$, and $Z=\widetilde{t}_{1,2}^{-n+1}\widetilde{h}_1^{n+1}\widetilde{h}_2^{n-1}$. 
Therefore, we have completed the proof of Theorem~\ref{thm_abel_smod} for $\SMp $.   
\end{proof}

\begin{proof}[Proof of Theorem~\ref{thm_abel_smod} for $\SM $]
We will calculate $H_1(\SM )$ by using the finite presentations for $\SM $ in Theorem~\ref{thm_pres_smod}. 
We proceed by an argument similar to the proof of Theorem~\ref{thm_abel_lmod} for $\LM $. 
By an argument similar to the proof of Theorem~\ref{thm_abel_lmod} for $\LM $, the relations~(1), (2) (a)--(e), (3), (5) (b), (c), and (f) in Theorem~\ref{thm_pres_smod} are equivalent to the relations $\widetilde{t}_{i,i+1}=\widetilde{t}_{i+1,i+2}$ for $1\leq i\leq 2n-1$ and $\widetilde{h}_{i}=\widetilde{h}_{i+1}$ for $1\leq i\leq 2n-1$ in $H_1(\LMb )$. 
By an argument similar to the proof of Theorem~\ref{thm_abel_lmod} for $\LM $, in $H_1(\SM )$, the relations~(5) (d), (4) (b), (c), and (5) (g) in Theorem~\ref{thm_pres_smod} are equivalent to the following relations: 
\begin{itemize}
\item[(A)] $(\widetilde{t}_{1,2}^{-n+1}\widetilde{h}_{1}^{2n})^{kn}=1$,  
\item[(B)] $(\widetilde{t}_{1,2}^{-n+1}\widetilde{h}_{1}^{2n})^{n+1}=1$, 
\item[(C)] $(\widetilde{t}_{1,2}^{-n+1}\widetilde{h}_{1}^{2n})^{n-2i+1}=1$ for $1\leq i\leq n$, and
\item[(D)] $(\widetilde{t}_{1,2}^{-n+1}\widetilde{h}_{1}^{2n})^{n-2i}=(\widetilde{t}_{1,2}^{-n+1}\widetilde{h}_{1}^{2n})^{n}\Leftrightarrow (\widetilde{t}_{1,2}^{-n+1}\widetilde{h}_{1}^{2n})^{2i}=1$ for $1\leq i\leq n-1$ 
\end{itemize}
Up to the relation~(B) above, the relation (C) for $1\leq i\leq n-1$ is equivalent to the relation~(D) for $1\leq i\leq n-1$. 
The relation (C) for $i=n$ coincides with the relation~(B). 
Up to the relation~(B), the relation (A) is equivalent to the relation $(\widetilde{t}_{1,2}^{-n+1}\widetilde{h}_{1}^{2n})^{k}=1$. 
The relations~(D) for $2\leq i\leq n-1$ are obtained from the relation~(D) for $i=1$. 
Thus the relations~(A)--(D) are equivalent to the relations
\begin{itemize}
\item[(A$^\prime $)] $(\widetilde{t}_{1,2}^{-n+1}\widetilde{h}_{1}^{2n})^{k}=1$, 
\item[(B$^\prime $)] $(\widetilde{t}_{1,2}^{-n+1}\widetilde{h}_{1}^{2n})^{n+1}=1$, and
\item[(C$^\prime $)] $(\widetilde{t}_{1,2}^{-n+1}\widetilde{h}_{1}^{2n})^{2}=1$.  
\end{itemize}
When $n$ is even or $k$ is odd, by the relation~(C$^\prime $), we have the relation $\widetilde{t}_{1,2}^{-n+1}\widetilde{h}_{1}^{2n}=1$ in $H_1(\SM )$. 
When $n$ is odd and $k$ is even, the relations~(A$^\prime $) and (B$^\prime $) are obtained from the relation~(C$^\prime $). 
Thus the relations~(A$^\prime $)--(C$^\prime $) above are equivalent to the relation $\widetilde{t}_{1,2}^{-n+1}\widetilde{h}_{1}^{2n}=1$ when $n$ is even or $k$ is odd, and $(\widetilde{t}_{1,2}^{-n+1}\widetilde{h}_{1}^{2n})^{2}=1$ when $n$ is odd and $k$ is even. 

When $n$ is even or $k$ is odd, by arguments above and in the proof of Theorem~\ref{thm_abel_lmod} for $\LM $, we have  
\begin{eqnarray*}
H_1(\SM )
&\cong &\left< \widetilde{h}_1, \widetilde{t}_{1,2}, \widetilde{r}\ \middle| \ \widetilde{r}^2=1, \widetilde{t}_{1,2}^{-n+1}\widetilde{h}_{1}^{2n}=1 \right> \\
&\cong &\left\{ \begin{array}{ll}
 \Z [\widetilde{t}_{1,2}]\oplus \Z _2[\widetilde{h}_1]\oplus \Z _2[\widetilde{r}]&\text{if }  n=1,   \\
 \Z [X]\oplus \Z _2[\widetilde{r}]\oplus \Z _2[Y_1]&\text{if }  n\geq 3\text{ is odd},   \\
 \Z [Y_2]\oplus \Z _2[\widetilde{r}] &\text{if }  n\geq 2\text{ is even},
 \end{array} \right.
\end{eqnarray*}
where $X=\widetilde{t}_{1,2}^{-1}\widetilde{h}_1^2$, $Y_1=\widetilde{t}_{1,2}^{-\frac{n-1}{2}}\widetilde{h}_1^{n}$, and $Y_2=\widetilde{t}_{1,2}^{-\frac{n}{2}}\widetilde{h}_1^{n+1}$. 

When $n$ is odd and $k$ is even, by an argument above, we have  
\begin{eqnarray*}
H_1(\SM )
&\cong &\left< \widetilde{h}_1, \widetilde{t}_{1,2}, \widetilde{r}\ \middle| \ \widetilde{r}^2=1, (\widetilde{t}_{1,2}^{-n+1}\widetilde{h}_{1}^{2n})^2=1 \right> \\
&\cong &\left< \widetilde{h}_1, \widetilde{t}_{1,2}, \widetilde{r}, X\ \middle| \ \widetilde{r}^2=1, (X^{n-1}\widetilde{h}_{1}^{2})^2=1, X=\widetilde{t}_{1,2}^{-1}\widetilde{h}_{1}^{2} \right> \\
&\cong &\left< \widetilde{h}_1, \widetilde{r}, X\ \middle| \ \widetilde{r}^2=1, (X^{n-1}\widetilde{h}_{1}^{2})^2=1 \right> \\
&\cong &\left< \widetilde{h}_1, \widetilde{r}, X, Y\ \middle| \ \widetilde{r}^2=1, Y^4=1, Y=X^{\frac{n-1}{2}}\widetilde{h}_{1} \right> \\
&\cong &\left< \widetilde{r}, X, Y\ \middle| \ \widetilde{r}^2=1, Y^4=1 \right> \\
&\cong &\Z [X]\oplus \Z _2[\widetilde{r}]\oplus \Z _4[Y],
\end{eqnarray*}
where $X=\widetilde{t}_{1,2}^{-1}\widetilde{h}_1^2$ and $Y=X^{\frac{n-1}{2}}\widetilde{h}_{1}=\widetilde{t}_{1,2}^{-\frac{n-1}{2}}\widetilde{h}_1^{n}$. 
Therefore we have completed the proof of Theorem~\ref{thm_abel_smod} for $\SM $
\end{proof}

\appendix

\par
{\bf Acknowledgement:} The authors would like to express their gratitude to Tyrone Ghaswala for helpful advices and telling them Proposition~5.3 in \cite{Ghaswala-McLeay} that is a generalized version of Lemma~\ref{lift_half-twist}. 
The first author was supported by JSPS KAKENHI Grant Numbers JP16K05156 and JP20K03618.
The second author was supported by JSPS KAKENHI Grant Numbers JP19K23409 and 21K13794.

\end{document}